\documentclass[a4paper, 11pt, reqno]{amsart}

\usepackage{geometry}
\usepackage[pagebackref,colorlinks]{hyperref}

\usepackage{graphicx, enumerate, euscript, float}
\usepackage{amssymb}
\usepackage{amsmath}
\usepackage{amsthm, bbold}
\usepackage{amscd}
\usepackage[all,2cell]{xy}

\usepackage[colorlinks]{hyperref}

\usepackage[ngerman,french,english]{babel}

\usepackage{mathtools}

\usepackage{tikz}
\usetikzlibrary{arrows,backgrounds,shadows,decorations.markings,matrix}

\UseAllTwocells \SilentMatrices
\newtheorem{thm}{Theorem}[section]

\newtheorem{cor}[thm]{Corollary}
\newtheorem{conj}[thm]{Conjecture}
\newtheorem{lem}[thm]{Lemma}

\newtheorem{prop}[thm]{Proposition}
\theoremstyle{definition}
\newtheorem{defn}[thm]{Definition}
\newtheorem{exm}[thm]{Example}

\theoremstyle{remark}
\newtheorem{rem}[thm]{\bf Remark}
\numberwithin{equation}{section}

\newcommand{\Barr}{\mathrm{Bar}}

\newcommand{\HH}{\mathrm{HH}}

\newcommand{\Hom}{\mathrm{Hom}}

\newcommand{\sg}{\mathrm{sg}}
\newcommand{\dg}{\mathrm{dg}}
\newcommand{\op}{\mathrm{op}}

\newcommand{\nc}{\mathrm{nc}}

\newcommand{\hooklongrightarrow}{\lhook\joinrel\longrightarrow}

\makeatletter
\newcommand{\smallotimes}{\mathbin{\mathpalette\make@small\otimes}}

\newcommand{\make@small}[2]{%
  \vcenter{\hbox{%
    $\m@th\ifx#1\displaystyle\scriptstyle\else\ifx#1\textstyle\scriptstyle
     \else\scriptscriptstyle\fi\fi#2$%
  }}%
}
\makeatother

\begin{document}
\sloppy

\title[Leavitt path algebras, $B_\infty$-algebras and Keller's conjecture]{Leavitt path algebras, $B_\infty$-algebras and Keller's conjecture for singular Hochschild cohomology}

\author{Xiao-Wu Chen, Huanhuan Li,  and Zhengfang Wang$^*$}

\thanks{$^*$ the corresponding author}

\date{\today}

\maketitle

\selectlanguage{english}

\begin{abstract}
For a finite quiver without sinks, we establish an isomorphism in the homotopy category $\mathrm {Ho}(B_\infty)$ of $B_{\infty}$-algebras between the Hochschild cochain complex of the Leavitt path algebra $L$ and the singular Hochschild cochain complex  of the corresponding radical square zero algebra $\Lambda$. Combining this isomorphism with a description of the dg singularity category of $\Lambda$ in terms of the dg perfect derived category of $L$, we verify  Keller's conjecture for the singular Hochschild cohomology of  $\Lambda$. More precisely, we prove that there is an isomorphism in $\mathrm{Ho}(B_\infty)$ between the singular Hochschild cochain complex of $\Lambda$ and the Hochschild cochain complex of the dg singularity category of $\Lambda$. One ingredient of the proof is  the following duality theorem on $B_\infty$-algebras: for any $B_\infty$-algebra, there is a natural  $B_\infty$-isomorphism between its opposite $B_\infty$-algebra  and its transpose $B_\infty$-algebra.

We prove that Keller's conjecture is invariant under one-point (co)extensions and singular equivalences with levels. Consequently, Keller's conjecture holds for those algebras obtained inductively from $\Lambda$ by one-point (co)extensions and singular equivalences with levels. These algebras include all finite dimensional gentle algebras.

\end{abstract}

\tableofcontents

\section{Introduction}

\subsection{The background}

 Let $\mathbb k$ be  a field and $\Lambda$ be a finite dimensional associative $\mathbb k$-algebra. Denote by $\Lambda\mbox{-mod}$ the abelian category of finite dimensional left $\Lambda$-modules and by $\mathbf{D}^b(\Lambda\mbox{-mod})$ its bounded derived category. The \emph{singularity category} $\mathbf{D}_{\rm sg}(\Lambda)$ of $\Lambda$ is by definition the Verdier quotient category of $\mathbf{D}^b(\Lambda\mbox{-mod})$ by  the full subcategory of perfect complexes. This notion is first introduced in \cite{Buc},  and then rediscovered in \cite{Orl} with motivations from homological mirror symmetry.  The singularity category measures the homological singularity of the algebra $\Lambda$, and reflects the asymptotic behaviour of syzygies of $\Lambda$-modules.

It is well known that triangulated categories are less rudimentary than dg categories as the former are inadequate to handle many basic algebraic and geometric operations.   The bounded dg derived category $\mathbf{D}_{\rm dg}^b(\Lambda\mbox{-mod})$ is a dg category whose zeroth cohomology coincides with $\mathbf{D}^b(\Lambda\mbox{-mod})$. Similarly, the \emph{dg singularity category} $\mathbf{S}_{\rm dg}(\Lambda)$ of $\Lambda$ \cite{Kel18, BRTV, BrDy} is defined to be the dg quotient category of $\mathbf{D}_{\rm dg}^b(\Lambda\mbox{-mod})$ by the full dg subcategory of perfect complexes. Then the zeroth cohomology of $\mathbf{S}_{\rm dg}(\Lambda)$ coincides with $\mathbf{D}_{\rm sg}(\Lambda)$. In other words, the dg singularity category provides a canonical dg enhancement for the singularity category.

As one of the advantages of working with dg categories, their Hochschild theory behaves  well with respect to various operations \cite{Kel03, LoVa, Toe}.   We consider the Hochschild cochain complex $C^*(\mathbf{S}_{\rm dg}(\Lambda), \mathbf{S}_{\rm dg}(\Lambda))$ of the dg singularity category $\mathbf{S}_{\rm dg}(\Lambda)$, which has a natural structure of a $B_\infty$-algebra \cite{GJ}. Moreover, it induces a Gerstenhaber algebra structure \cite{Ger} on the Hochschild cohomology $\HH^*(\mathbf{S}_{\rm dg}(\Lambda), \mathbf{S}_{\rm dg}(\Lambda))$. The $B_\infty$-algebra structures on the Hochschild cochain complexes play an essential role in the deformation theory \cite{LoVa} of categories. We mention that  $B_{\infty}$-algebras  are the key ingredients in the proof \cite{Tam} of Kontsevich's formality theorem. We  refer to \cite[Subsection 1.19]{MSS} for the relationship between $B_{\infty}$-algebras and Deligne's conjecture.

The \emph{singular Hochschild cohomology}  $\HH_{\sg}^*(\Lambda, \Lambda)$ of  $\Lambda$ is defined as
$$\HH_{\sg}^n(\Lambda, \Lambda):=\Hom_{\mathbf{D}_{\sg}(\Lambda^e)}(\Lambda, \Sigma^n(\Lambda)), \quad\quad \mbox{for any $n\in \mathbb Z$},$$
where $\Sigma$ is the suspension functor of the singularity category $\mathbf{D}_{\sg}(\Lambda^e)$ of the enveloping algebra $\Lambda^e=\Lambda\otimes \Lambda^{\op}$; see \cite{BJ, ZF, Kel18}. By \cite{Wan1}, there are two complexes $\overline{C}_{\sg, L}^*(\Lambda, \Lambda)$ and $\overline{C}_{\sg, R}^*(\Lambda, \Lambda)$ computing $\HH_{\sg}^*(\Lambda, \Lambda)$, called the \emph{left singular Hochschild cochain complex} and the \emph{right singular Hochschild cochain complex} of $\Lambda$, respectively. Moreover, both  $\overline{C}_{\sg, L}^*(\Lambda, \Lambda)$ and $\overline{C}_{\sg, R}^*(\Lambda, \Lambda)$ have  natural $B_{\infty}$-algebra structures, which induce the same Gerstenhaber algebra structure on $\HH_{\sg}^*(\Lambda, \Lambda)$; see Proposition \ref{prop:Ger-dual}.

There is a canonical isomorphism
\begin{align}\label{iso:duality}
\overline{C}_{\sg, L}^*(\Lambda^{\rm op}, \Lambda^{\rm op}) \simeq \overline{C}_{\sg, R}^*(\Lambda, \Lambda)^{\rm opp}
\end{align}
of $B_\infty$-algebras; see Proposition \ref{lemma-CL1}.  Here,  for a $B_\infty$-algebra $A$ we denote by $A^{\rm opp}$ its \emph{opposite} $B_\infty$-algebra; see Definition \ref{defnopposite}. We mention that the $B_{\infty}$-algebra structures on the singular Hochschild cochain complexes come  from a natural action of the cellular chains of the  spineless cacti operad introduced in \cite{Ka}.

The singular Hochschild cohomology is also called Tate-Hochschild cohomology in \cite{Wan, Wan1, Wan2}.
The result in \cite{RW} shows that the singular Hochschild cohomology can be viewed as an  algebraic formalism of  Rabinowitz-Floer homology \cite{CFO} in symplectic geometry.

\subsection{The main results}
Let $A=(A, m_n; \mu_{p, q})$ be a $B_\infty$-algebra, where $(A, m_n)$ is the underlying $A_\infty$-algebra and $\mu_{p,q}$ are the $B_\infty$-products. We denote by $A^{\rm tr}$ the \emph{transpose} $B_\infty$-algebra; see Definition~\ref{defn-transposeB}.

The first main result, a duality theorem on general $B_\infty$-algebras, might be viewed as a conceptual advance on $B_\infty$-algebras.

\begin{thm}[= Theorem \ref{thm:dualityB}]\label{thm:dualityintro} Let $(A, m_n; \mu_{p, q})$ be a $B_\infty$-algebra. Then there is a natural $B_\infty$-isomorphism between the opposite $B_\infty$-algebra $A^{\rm opp}$ and the transpose $B_\infty$-algebra $A^{\rm tr}$.
\end{thm}

 It is a standard fact that a $B_\infty$-algebra structure on $A$ is equivalent to a dg bialgebra structure on the tensor coalgebra $T^c(sA)$ over the $1$-shifted graded space $sA$. By a classical result,  the dg bialgebra $T^c(sA)$  admits a bijective antipode $S$.  The $B_\infty$-isomorphism in Theorem~\ref{thm:dualityintro} is precisely induced  by the antipode $S$.

 We mention that if $\mu_{p, q} = 0$ for any $p >1$, the antipode $S$ and thus the required $B_\infty$-isomorphism have an explicit graphic description from the Kontsevich-Soibelman minimal operad; see Remark~\ref{rem:antipode-brace}.

 Theorem~\ref{thm:dualityintro} is applied to establish $B_\infty$-isomorphisms between the (\emph{resp}.\ singular) Hochschild cochain complexes of  an algebra and of its opposite algebra; see Propositions~\ref{lemma-CL} and \ref{lemma-CL1} (= the isomorphism (\ref{iso:duality})), respectively.

\vskip 5pt

Recall that $\Lambda$ is a finite dimensional $\mathbb{k}$-algebra. Denote by $\Lambda_0$ the semisimple quotient algebra of $\Lambda$ modulo its Jacobson radical.
Recently,  Keller proves in \cite{Kel18} that if $\Lambda_0$ is separable over $\mathbb{k}$,  then there is a natural isomorphism of graded algebras
\begin{align}\label{equ:graded}
\HH_{\sg}^*(\Lambda, \Lambda)\stackrel{\sim}\longrightarrow \HH^*(\mathbf{S}_{\rm dg}(\Lambda), \mathbf{S}_{\rm dg}(\Lambda)).
\end{align}
This isomorphism plays a central role in \cite{HuKe}, which proves a weakened version of Donovan-Wemyss's conjecture \cite{DoWe}.

Denote by $\mathrm{Ho}(B_{\infty})$ the homotopy category of $B_\infty$-algebras \cite{Hinich, Kel03}. In \cite[{\bf Conjecture}~1.2]{Kel18}, Keller conjectures that there is an isomorphism in $\mathrm{Ho}(B_{\infty})$
\begin{align}\label{equ:Keller-conj-intr}
\overline{C}^*_{\sg, L}(\Lambda^{\rm op}, \Lambda^{\rm op})\simeq C^*(\mathbf{S}_{\rm dg}(\Lambda), \mathbf{S}_{\rm dg}(\Lambda)).
\end{align}
In particular, we have an induced isomorphism
$$\HH_{\sg}^*(\Lambda, \Lambda)\stackrel{\sim}\longrightarrow \HH^*(\mathbf{S}_{\rm dg}(\Lambda), \mathbf{S}_{\rm dg}(\Lambda))$$
respecting the Gerstenhaber structures. A slightly stronger version of the conjecture claims that the induced isomorphism above coincides with the natural isomorphism (\ref{equ:graded}).

 It is well known that a $B_\infty$-algebra induces a dg Lie algebra and that a $B_\infty$-quasi-isomorphism induces a quasi-isomorphism of dg Lie algebras; see Remark~\ref{Rem:dgliealgebra}. Keller's conjecture yields a zigzag of quasi-isomorphisms of dg Lie algebras, or equivalently, an $L_\infty$-quasi-isomorphism, between $\overline{C}^*_{\sg, L}(\Lambda^{\rm op}, \Lambda^{\rm op})$ and $C^*(\mathbf{S}_{\rm dg}(\Lambda), \mathbf{S}_{\rm dg}(\Lambda))$. From the general idea of deformation theory via dg Lie algebras \cite{Lu, Pri}, Keller's conjecture indicates that the deformation theory of the dg singularity category is controlled by the singular Hochschild cohomology, where the latter is usually much easier to compute than the Hochschild cohomology of the dg singularity category. For example, in view of the work \cite{BRTV,Dyc,Kel18}, it would be of interest to study the relationship between the singular Hochschild cohomology and the deformation theory of Landau-Ginzburg models.  We mention that Keller's conjecture is analogous to the isomorphism
\begin{align*}
C^*(\Lambda^{\rm op}, \Lambda^{\rm op})\simeq C^*(\mathbf{D}_{\rm dg}^b(\Lambda\mbox{-mod}), \mathbf{D}_{\rm dg}^b(\Lambda\mbox{-mod}))
\end{align*}
for the classical Hochschild cochain complexes; see \cite{Kel03, LoVa}.

We say that an algebra $\Lambda$ \emph{satisfies} Keller's conjecture, provided that there is an isomorphism (\ref{equ:Keller-conj-intr}) for $\Lambda$. The second main result, an invariance theorem,  justifies Keller's conjecture to some extent, as a reasonable conjecture should be invariant under reasonable equivalence relations.

\begin{thm}[= Theorem~\ref{thm:Keller-redu}]\label{thmintro-1}
Let $\Pi$ be another algebra. Assume that  $\Lambda$ and $\Pi$  are connected by a finite zigzag  of one-point (co)extensions and singular equivalences with levels. Then $\Lambda$ satisfies Keller's conjecture if and only if so does $\Pi$.
\end{thm}

Recall that a derived equivalence \cite{Ric} between two algebras naturally induces a singular equivalence with level. It follows that Keller's conjecture is invariant under derived equivalences.

We leave some comments on the proof of Theorem \ref{thmintro-1}. It is known that both one-point (co)extensions of algebras \cite{Chen11} and singular equivalences with levels \cite{Wan15} induce triangle equivalences between the singularity categories. We observe that these triangle equivalences can be enhanced to quasi-equivalences between the dg singularity categories.

On the other hand, we prove that the singular Hochschild cochain complexes, as $B_\infty$-algebras, are invariant under one-point (co)extensions and singular equivalences with levels. For the invariance under singular equivalences with levels, the idea using a triangular matrix algebra is adapted from \cite{Kel03}, while our argument is much more involved due to the colimits occurring in the consideration.  For example, analogous to the colimit construction \cite{Wan1} of the right singular Hochschild cochain complex,  we construct an explicit colimit complex for any $\Lambda$-$\Pi$-bimodule $M$. When  $M$ is projective on both sides,  the constructed colimit complex computes the Hom space from $M$ to $\Sigma^i(M)$ in the singularity category of $\Lambda$-$\Pi$-bimodules.

\vskip 5pt

Let $Q$ be a finite quiver without sinks. Denote by $\mathbb k Q/J^2$ the corresponding finite dimensional algebra with radical square zero. We aim  to verify Keller's conjecture for $\mathbb k Q/J^2$. However, our approach is indirect, using the \emph{Leavitt path algebra} $L(Q)$ over $\mathbb{k}$ in the sense of \cite{AA, AGGP, AMP}. We mention close connections of Leavitt path algebras with symbolic dynamic systems \cite{ALPS, Haz, Chen} and noncommutative geometry \cite{Smi}.

By the work \cite{Smi, CY, Li}, the singularity category of $\mathbb k Q/J^2$ is closely related to the Leavitt path algebra $L(Q)$. The Leavitt path algebra $L(Q)$ is infinite dimensional as $Q$ has no sinks, therefore its link to the finite dimensional algebra $\mathbb k Q/J^2$ is somehow unexpected. We mention that $L(Q)$ is naturally  $\mathbb{Z}$-graded, which will  be viewed as a dg algebra with trivial differential throughout this paper.

The third main result verifies Keller's conjecture for the algebra $\mathbb kQ/ J^2$.

\begin{thm}[= Theorem~\ref{thm-main}]\label{thmintro-2}
Let $Q$ be a finite quiver without sinks. Set $\Lambda=\mathbb k Q/J^2$. Then there are  isomorphisms in the homotopy category  $\mathrm{Ho}(B_{\infty})$ of $B_{\infty}$-algebras
$$\overline{C}_{\sg, L}^*(\Lambda^{\op}, \Lambda^{\op}) \stackrel{\Upsilon}\longrightarrow  C^*(L(Q), L(Q))  \stackrel{\Delta}\longrightarrow C^*(\mathbf{S}_{\rm dg}(\Lambda), \mathbf{S}_{\rm dg}(\Lambda)).$$
In particular, there are isomorphisms of Gerstenhaber algebras
$$\HH_{\sg}^*(\Lambda^{\op}, \Lambda^{\op}) \stackrel{}\longrightarrow  \HH^*(L(Q), L(Q))  \stackrel{}\longrightarrow \HH^*(\mathbf{S}_{\rm dg}(\Lambda), \mathbf{S}_{\rm dg}(\Lambda)).$$
\end{thm}

 In  Theorem~\ref{thmintro-2}, the isomorphism $\Delta$  between the Hochschild cochain complex of the Leavitt path algebra $L(Q)$ and the one of the dg singularity category $\mathbf{S}_{\rm dg}(\mathbb{k}Q/J^2)$  enhances the link  \cite{Smi, CY, Li} between $L(Q)$ and $\mathbb k Q/J^2$ to the $B_\infty$-level. The approach to obtain $\Delta$ is categorical, relying  on a description of $\mathbf{S}_{\dg}(\mathbb kQ/ J^2)$ via the dg perfect derived category of $L(Q)$. The isomorphism $\Upsilon$, which is inspired by \cite{Wan} and is of combinatoric  flavour, establishes a brand new link between $L(Q)$ and $\mathbb{k}Q/J^2$. The primary tool to obtain $\Upsilon$ is the homotopy transfer theorem \cite{Kad} for dg algebras.

 The composite isomorphism  $\Delta\circ \Upsilon$ verifies Keller's conjecture for the algebra $\mathbb{k}Q/J^2$, which seems to be the first confirmed case. Indeed, combining Theorems \ref{thmintro-1} and \ref{thmintro-2}, we verify Keller's conjecture for $\mathbb{k}Q/J^2$ for \emph{any} finite quiver $Q$ (possibly with sinks), and for any finite dimensional \emph{gentle algebra}. Let us mention that gentle algebras are of interest from many different perspectives \cite{GR, HKK}. It is unclear whether  the proof of Theorem~\ref{thmintro-2} can be generalized to a wider class of algebras, for example, Koszul algebras.

\vskip 3pt

Let us describe the key steps in the proof of Theorem~\ref{thmintro-2}, which are illustrated in the diagram \eqref{diagram1} in the proof of Theorem~\ref{thm-main}.

 Using the standard argument for dg quotient categories \cite{Kel99, Dri}, we prove first that the dg singularity category is essentially the same as the dg enhancement of the singularity category via acyclic complexes of injective modules \cite{Kra}; see Corollary \ref{cor:Krause}. Then using the explicit compact generator  \cite{Li} of the homotopy category of acyclic complexes of injective modules and the general results in \cite{Kel03} on Hochschild cochain complexes, we infer the isomorphism $\Delta$.

The isomorphism $\Upsilon$ is constructed in a very explicit but indirect manner. The main ingredients  are  the (non-strict) $B_\infty$-isomorphism (\ref{iso:duality}), two strict $B_\infty$-isomorphisms and  an explicit $B_\infty$-quasi-isomorphism $(\Phi_1, \Phi_2, \cdots)$.

We introduce two new explicit $B_\infty$-algebras, namely the \emph{combinatorial $B_\infty$-algebra} $\overline{C}^*_{\sg, R}(Q, Q)$  of $Q$ constructed by parallel paths in $Q$,   and the \emph{Leavitt $B_\infty$-algebra} $\widehat{C}^*(L, L)$ whose construction is inspired by an explicit projective bimodule resolution of $L=L(Q)$.

Set $E=\mathbb kQ_0$ to be the  semisimple subalgebra of $\Lambda$. We first observe that $\overline{C}_{\sg, R}^*(\Lambda, \Lambda)$ is strictly $B_\infty$-quasi-isomorphic to $\overline{C}_{\sg, R, E}^*(\Lambda, \Lambda)$, the $E$-relative right singular Hochschild cochain complex.  Using the explicit description \cite{Wan} of $\overline{C}_{\sg, R, E}^*(\Lambda, \Lambda)$ via parallel paths in $Q$, we obtain a strict $B_\infty$-isomorphism between $\overline{C}_{\sg, R, E}^*(\Lambda, \Lambda)$ and $\overline{C}^*_{\sg, R}(Q, Q)$. We prove that $\overline{C}^*_{\sg, R}(Q, Q)$ and $\widehat{C}^*(L, L)$  are strictly $B_\infty$-isomorphic.

 We construct  an explicit homotopy deformation retract between $\widehat{C}^*(L, L)$  and $\overline{C}^*_E(L, L)$, the normalized $E$-relative  Hochschild cochain complex of $L$.  Then the homotopy transfer theorem  for dg algebras yields an $A_{\infty}$-quasi-isomorphism
 $$(\Phi_1, \Phi_2, \cdots)\colon \widehat{C}^*(L, L)\longrightarrow \overline{C}^*_E(L, L).$$
  This $A_{\infty}$-morphism  is explicitly  given by the brace operation of $\widehat{C}^*(L, L)$. Using the higher pre-Jacobi identity, we prove that
 $$(\Phi_1, \Phi_2, \cdots)\colon \widehat{C}^*(L, L)\longrightarrow \overline{C}^*_E(L, L)^{\rm opp}$$
  is indeed a $B_{\infty}$-morphism. Since the natural embedding of $\overline{C}^*_E(L, L)$ into $C^*(L, L)$ is a strict $B_{\infty}$-quasi-isomorphism, we obtain the required isomorphism $\Upsilon$.

\subsection{The structure of the paper}

The paper is structured as follows. In Section~\ref{section2}, we review basic facts and results on dg quotient categories. We prove  in Subsection~\ref{subsec:one-point} that  dg singularity categories  are invariant under both one-point (co)extensions  and singular equivalences with levels.

 We enhance a result in \cite{Kra} to the dg level in Section~\ref{section3}. More precisely, we prove that the dg singularity category is essentially the same as the dg category of certain acyclic complexes of injective modules; see Proposition~\ref{prop:Krause}. The notion of Leavitt path algebras is recalled in Section~\ref{sec:lpa}. We prove that there is a zigzag of quasi-equivalences connecting the dg singularity category of $\Lambda=\mathbb kQ/J^2$ to the dg perfect derived category of the opposite dg algebra $L^{\rm op}=L(Q)^{\rm op}$; see Proposition~\ref{prop:CY-Li}. Here, $Q$ is a finite quiver without sinks.

 In Section~\ref{section5}, we give a brief introduction to $B_\infty$-algebras. We describe the axioms of $B_\infty$-algebras explicitly.  For any given  $B_\infty$-algebra, we introduce the opposite $B_\infty$ algebra and the transpose $B_\infty$-algebra. We prove that there is a natural $B_\infty$-isomorphism between the opposite $B_\infty$-algebra and the transpose $B_\infty$-algebra; see Theorem~\ref{thm:dualityB}.  We mainly focus on a special kind of $B_\infty$-algebras, the so-called {\it brace $B_\infty$-algebras}, whose underlying $A_\infty$-algebras are dg algebras as well as some of whose $B_\infty$-products  vanish.  We review some facts on Hochschild cochain complexes of dg categories and  (normalized) relative bar resolutions of dg algebras in Section~\ref{section6}.

 Slightly generalizing a result in \cite{He-Li-Li}, we provide a general construction of homotopy deformation retracts for dg algebras in Section~\ref{sectionforhdr}. Using this, we construct  an explicit homotopy deformation retract for the bimodule projective resolutions of Leavitt path algebras; see Proposition~\ref{prop:htrforLPA}.

  We recall from \cite{Wan1} the singular Hochschild cochain complexes and their $B_\infty$-structures in Section~\ref{section7}. We prove the $B_\infty$-isomorphism \eqref{iso:duality} in Proposition~\ref{lemma-CL1}, based on the general result in Theorem~\ref{thm:dualityB}.    We describe explicitly the brace operation on the singular Hochschild cochain complex and illustrate it with an example in Subsection~\ref{subsec:example-brace}.

  In Section~\ref{sectionforinvariance}, we prove that the (relative) singular Hochschild cochain complexes, as  $B_\infty$-algebras, are invariant under one-point (co)extensions of algebras and singular equivalences with levels. In Section~\ref{section8}, we prove that Keller's conjecture is invariant under one-point (co)extensions of algebras and singular equivalences with levels; see Theorem~\ref{thm:Keller-redu}. We give a proof of Theorem~\ref{thmintro-2}  (= Theorem \ref{thm-main}).

 In Section~\ref{subsection:Algebras-radical}, we give a combinatorial description for the singular Hochschild cochain complex of $\Lambda=\mathbb kQ/J^2$. We introduce the combinatorial $B_\infty$-algebra $\overline{C}^*_{\sg, R}(Q, Q)$ of $Q$, which is strictly $B_\infty$-isomorphic to the (relative) singular Hochschild cochain complex of $\Lambda$; see Theorem~\ref{thm:radical}.  We introduce the Leavitt  $B_{\infty}$-algebra $\widehat{C}^*(L, L)$ in Section~\ref{Section:9},  and show that it  is strictly $B_\infty$-isomorphic to   $\overline{C}^*_{\sg, R}(Q, Q)$, and thus to the (relative) singular Hoschild cochain complex of $\Lambda$; see Proposition~\ref{prop:interme}.

 In Section~\ref{Section:10}, we apply  the homotopy transfer theorem \cite{Kad} for dg algebras to obtain an explicit $A_{\infty}$-quasi-isomorphism $(\Phi_1, \Phi_2, \cdots )$ from  $\widehat{C}^*(L, L)$ to $\overline{C}_E^*(L, L)$; see Proposition~\ref{proposition-Phi}. In Section
 \ref{section:11}, we verify that $(\Phi_1, \Phi_2, \cdots)$ is indeed a $B_\infty$-morphism; see Theorem~\ref{prop-B4}.

Throughout this paper, we work over a fixed field $\mathbb k$. In other words, we require that all the algebras, categories and functors in the sequel are $\mathbb k$-linear; moreover, the unadorned Hom and $\otimes$ are over $\mathbb k$.  We use $\mathbf{1}_V$  to denote the identity endomorphism of the (graded) $\mathbb k$-vector space $V$. When no confusion arises, we simply write it as $\mathbf{1}$.

\section{DG categories and dg quotients}
\label{section2}

In this section, we recall basic facts and results on dg categories. The standard references are \cite{Kel94, Dri}. We prove that both one-point (co)extensions of algebras and singular equivalences with levels induce quasi-equivalences between  dg singularity categories.

For the fixed field  $\mathbb k$, we denote by $\mathbb k\mbox{-Mod}$ the abelian  category of $\mathbb k$-vector spaces.

\subsection{DG categories and dg functors}

Let $\mathcal{A}$ be a dg category over $\mathbb k$. For two objects $x$ and $y$, the Hom-complex is usually denoted by $\mathcal{A}(x, y)$ and its differential is denoted by $d_\mathcal{A}$. For a homogeneous morphism $a$, its degree is denoted by $|a|$.  Denote by $Z^0(\mathcal{A})$ the ordinary category of $\mathcal{A}$, which has the same objects as $\mathcal{A}$ and its Hom-space is given by $Z^0(\mathcal{A}(x, y))$, the zeroth cocycle of $\mathcal{A}(x, y)$. Similarly, the \emph{homotopy category} $H^0(\mathcal{A})$ has the same objects, but its Hom-space is given by the zeroth cohomology $H^0(\mathcal{A}(x, y))$.

Recall that a dg functor $F\colon  \mathcal{A}\rightarrow \mathcal{B}$ is \emph{quasi-fully faithful}, if the cochain map
$$F_{x, y}\colon  \mathcal{A}(x, y)\longrightarrow \mathcal{B}(Fx, Fy)$$
is a quasi-isomorphism for any objects $x, y$ in $\mathcal{A}$. Then $H^0(F)\colon  H^0(\mathcal{A})\rightarrow H^0(\mathcal{B})$ is fully faithful. A quasi-fully faithful dg functor $F$ is called a \emph{quasi-equivalence} if $H^0(F)$ is dense.

\begin{exm}\label{exm:Cdg}
{\rm Let $\mathfrak{a}$ be an additive category. Denote by $C_{\rm dg}(\mathfrak{a})$ the dg category of cochain complexes in $\mathfrak{a}$. A cochain complex in $\mathfrak{a}$ is usually denoted by $X=(\bigoplus_{p\in \mathbb{Z}} X^p, d_X)$ or $(X, d_X)$. The $p$-th component  of the Hom-complex $C_{\rm dg}(\mathfrak{a})(X, Y)$ is given by the following infinite product
    $$C_{\rm dg}(\mathfrak{a})(X, Y)^p=\prod_{n \in \mathbb{Z}} {\rm Hom}_\mathfrak{a}(X^n, Y^{n+p}),$$
    whose elements will be denoted by $f=\{f^n\}_{n\in \mathbb{Z}}$ with $f^n \in {\rm Hom}_\mathfrak{a}(X^n, Y^{n+p}).$ The differential $d$ acts on $f$ such that $d(f)^{n}=d_Y^{n+p}\circ f^n-(-1)^{|f|}f^{n+1}\circ d_X^n$ for each $n\in \mathbb{Z}$.

    We observe that the homotopy category $H^0(C_{\rm dg}(\mathfrak{a}))$ coincides with the classical  homotopy category $\mathbf{K}(\mathfrak{a})$ of cochain complexes in $\mathfrak{a}$.
    }\end{exm}

\begin{exm}
{\rm The dg category $C_{\rm dg}(\mathbb k\mbox{-Mod})$ is usually denoted by $C_{\rm dg}(\mathbb k)$. Let $\mathcal{A}$ be a small dg category. By a left dg $\mathcal{A}$-module, we mean a dg functor $M\colon  \mathcal{A}\rightarrow C_{\rm dg}(\mathbb k)$. The following notation will be convenient: for a morphism $a\colon  x\rightarrow y$ in $\mathcal{A}$ and $m\in M(x)$, the resulting element $M(a)(m)\in M(y)$ is written as $a.m$. Here,  the dot indicates the left $\mathcal{A}$-action on $M$. Indeed, we usually identify $M$ with the formal sum $\bigoplus_{x\in {\rm obj}(\mathcal{A})} M(x)$ with the above left $\mathcal{A}$-action. The differential $d_M$ means $\bigoplus_{x\in {\rm obj}(\mathcal{A})} d_{M(x)}$.

  We denote by $\mathcal{A}\mbox{-{\rm DGMod}}$ the dg category formed by left dg $\mathcal{A}$-modules. For two dg $\mathcal{A}$-modules $M$ and $N$, a morphism $\eta=(\eta_x)_{x\in {\rm obj}(\mathcal{A})}\colon  M\rightarrow N$ of degree $p$ consists of maps $\eta_x\colon  M(x)\rightarrow N(x)$ of degree $p$ satisfying
  $$N(a)\circ \eta_x=(-1)^{|a|\cdot p}\eta_y\circ M(a)$$
  for each morphism $a\colon  x\rightarrow y$ in $\mathcal{A}$. These morphisms form the $p$-th component of $\mathcal{A}\mbox{-{\rm DGMod}}(M, N)$. The differential is defined such that $d(\eta)_x=d(\eta_x)$. Here, $d(\eta_x)$ means the differential in $C_{\rm dg}(\mathbb k)$. In other words, $d(\eta_x)=d_{N(x)}\circ \eta_x-(-1)^p \eta_x\circ d_{M(x)}$.

For a left dg $\mathcal{A}$-module $M$, the \emph{suspended dg module} $\Sigma(M)$ is defined such that $\Sigma(M)(x)=\Sigma(M(x))$, the suspension of the complex $M(x)$. The left $\mathcal{A}$-action on $\Sigma(M)$ is given such that $a.\Sigma(m)=(-1)^{|a|}\Sigma(a.m)$, where $\Sigma(m)$ means the  element in $\Sigma(M(x))$ corresponding to $m\in M(x)$. This gives rise to a dg endofunctor $\Sigma$ on $\mathcal{A}\mbox{-DGMod}$, whose action on morphisms $\eta$ is given such that $\Sigma(\eta)_x=(-1)^{|\eta|} \eta_x$.
}\end{exm}

\begin{exm}
    Denote by $\mathcal{A}^{\rm op}$ the \emph{opposite dg category} of $\mathcal{A}$, whose composition is given by $a\circ^{\rm op} b=(-1)^{|a|\cdot |b|} b\circ a$. We identify a left $\mathcal{A}^{\rm op}$-module with a right dg $\mathcal{A}$-module. Then we obtain the dg category $\mbox{{\rm DGMod}-}\mathcal{A}$ of right dg $\mathcal{A}$-modules.

    For a right dg $\mathcal{A}$-module $M$, a morphism $a\colon  x\rightarrow y$ in $\mathcal{A}$  and $m\in M(y)$,  the right $\mathcal{A}$-action on $M$ is given such that $m.a=(-1)^{|a|\cdot |m|} M(a)(m)\in M(x)$. The suspended dg module $\Sigma(M)$ is defined similarly. We emphasize that the right $\mathcal{A}$-action on $\Sigma(M)$ is identical to the one on $M$.
\end{exm}

Let $\mathcal{A}$ be a small dg category. Recall that  $H^0(\mathcal{A}\mbox{-}{\rm DGMod})$ has a  canonical triangulated structure with the suspension functor induced by $\Sigma$. The \emph{derived category} $\mathbf{D}(\mathcal{A})$ is the Verdier quotient category of $H^0(\mathcal{A}\mbox{-}{\rm DGMod})$ by the triangulated subcategory of  acyclic dg modules.

Let $\mathcal{T}$ be a triangulated category with arbitrary coproducts. A triangulated subcategory $\mathcal{N}\subseteq \mathcal{T}$ is \emph{localizing} if it is closed under arbitrary coproducts. For a set $\mathcal{S}$ of objects, we denote by ${\rm Loc}(\mathcal{S})$ the localizing subcategory generated by $\mathcal{S}$, that is, the smallest localizing subcategory containing $\mathcal{S}$.

An object $X$ in $\mathcal{T}$ is compact if ${\rm Hom}_\mathcal{T}(X, -)\colon  \mathcal{T}\rightarrow \mathbb k\mbox{-Mod}$ preserves coproducts. Denote by $\mathcal{T}^c$ the full triangulated subcategory formed by compact objects. The category $\mathcal{T}$ is \emph{compactly generated}, provided that there is a set $\mathcal{S}$ of compact objects
such that $\mathcal{T}={\rm Loc}(\mathcal{S})$.

For example, the free dg $\mathcal{A}$-module $\mathcal{A}(x, -)$ is compact in $\mathbf{D}(\mathcal{A})$. Indeed,  $\mathbf{D}(\mathcal{A})$ is compactly generated by these modules. The \emph{perfect derived category} $\mathbf{per}(\mathcal{A})=\mathbf{D}(\mathcal{A})^c$ is the full subcategory formed by compact objects.

The Yoneda dg functor
$$\mathbf{Y}_\mathcal{A}\colon  \mathcal{A}\longrightarrow {\rm DGMod}\mbox{-}\mathcal{A}, \quad x\longmapsto \mathcal{A}(-, x)$$
is fully faithful. In particular, it induces a full embedding
$$H^0(\mathbf{Y}_\mathcal{A})\colon  H^0(\mathcal{A}) \longrightarrow H^0({\rm DGMod}\mbox{-}\mathcal{A}).$$
 The dg category $\mathcal{A}$ is said to be \emph{pretriangulated}, provided that the essential image of $H^0(\mathbf{Y}_{\mathcal{A}})$ is a triangulated subcategory of $H^0({\rm DGMod}\mbox{-}\mathcal{A})$. The terminology is justified by the evident fact: the homotopy category $H^0(\mathcal{A})$ of a pretriangulated dg category $\mathcal{A}$ has a canonical triangulated structure.

The following fact is well known; see \cite[Lemma 3.1]{CC}.

\begin{lem}\label{lem:quasi-equiv}
Let $F\colon  \mathcal{A}\rightarrow \mathcal{B}$ be a dg functor between two pretriangulated dg categories. Then $H^0(F) \colon   H^0(\mathcal{A})\rightarrow H^0(\mathcal{B})$ is naturally a triangle functor. Moreover, $F$ is a quasi-equivalence if and only if  $H^0(F)$ is a triangle equivalence. \hfill $\square$
\end{lem}

In this sequel, we will  identify quasi-equivalent dg categories. To be more precise, we work in the homotopy category $\mathbf{Hodgcat}$ \cite{Tab} of small dg categories, which is by definition the localization of $\mathbf{dgcat}$, the category of small dg categories, with respect to quasi-equivalences. The morphisms in $\mathbf{Hodgcat}$ are usually called \emph{dg quasi-functors}. Any dg quasi-functor from $\mathcal{A}$ to $\mathcal{B}$ can be realized as a roof
$$\mathcal{A} \stackrel{F_1} \longleftarrow \mathcal{C} \stackrel{F_2}\longrightarrow \mathcal{B}$$
of dg functors, where $F_1$ is a cofibrant replacement, in particular, it is a quasi-equivalence.  Recall that up to quasi-equivalences, every dg category might be identified with its cofibrant replacement; compare \cite[Appendix B.5]{Dri}.

Assume that $\mathcal{B}\subseteq \mathcal{A}$ is a full dg subcategory. We denote by $\pi\colon  \mathcal{A}\rightarrow \mathcal{A}/\mathcal{B}$ the \emph{dg quotient} of $\mathcal{A}$ by $\mathcal{B}$ \cite{Kel99, Dri}. Since we work over the field $\mathbb{k}$, the simple construction of $\mathcal{A}/\mathcal{B}$ is as follows: the objects of $\mathcal{A}/\mathcal{B}$ are the same as $\mathcal{A}$; we freely add new endomorphisms $\varepsilon_U$ of degree $-1$ for each object $U$ in $\mathcal{B}$, and set $d(\varepsilon_U)=1_U$. In other words, the added morphism $\varepsilon_U$ is a contracting homotopy for $U$; see \cite[Section 3]{Dri}.

The following fact follows immediately from the above simple construction.

\begin{lem}\label{lem:noether}
Assume that $\mathcal{C}\subseteq \mathcal{B}\subseteq \mathcal{A}$ are full dg subcategories. Then there is a canonical quasi-equivalence
\begin{equation*}
(\mathcal{A}/\mathcal{C})/{(\mathcal{B}/\mathcal{C})}\stackrel{\sim}\longrightarrow \mathcal{A}/\mathcal{B}.\tag*{$\square$}
\end{equation*}
\end{lem}

The following fundamental result follows immediately from \cite[Theorem 3.4]{Dri}; compare \cite[Theorem 1.3(i) and Lemma 1.5]{LO}.

\begin{lem}\label{lem:pre-qu}
Assume that both $\mathcal{A}$ and $\mathcal{B}$ are pretriangulated. Then $\mathcal{A}/\mathcal{B}$ is also pretriangulated. Moreover, $\pi\colon \mathcal{A}\rightarrow \mathcal{A}/\mathcal{B}$ induces a triangle equivalence
$$H^0(\mathcal{A})/H^0(\mathcal{B}) \stackrel{\sim}\longrightarrow H^0(\mathcal{A}/\mathcal{B}).$$
Here, $H^0(\mathcal{A})/H^0(\mathcal{B})$ denotes the Verdier quotient category of $H^0(\mathcal{A})$ by $H^0(\mathcal{B})$.  \hfill $\square$
\end{lem}

We will be interested in the following dg quotient categories.

\begin{exm}\label{exm:ddc}
For a small dg category $\mathcal{A}$, denote by $\mathcal{A}\mbox{-}{\rm DGMod}^{\rm ac}$ the full dg subcategory of $\mathcal{A}\mbox{-}{\rm DGMod}$ formed by acyclic modules. We have  the \emph{dg derived category} $$\mathbf{D}_{\rm dg}(\mathcal{A})=\mathcal{A}\mbox{-}{\rm DGMod}/{\mathcal{A}\mbox{-}{\rm DGMod}^{\rm ac}}.$$ The terminology is justified by the following fact: there is a canonical identification of $H^0(\mathbf{D}_{\rm dg}(\mathcal{A}))$ with $\mathbf{D}(\mathcal{A})$; see Lemma \ref{lem:pre-qu}. Then we have the \emph{dg perfect derived category} $\mathbf{per}_{\rm dg}(\mathcal{A})=\mathbf{D}_{\rm dg}(\mathcal{A})^c$, which is formed by modules becoming compact in $\mathbf{D}(\mathcal{A})$.

Here, we are sloppy about the precise definition of $\mathbf{D}_{\rm dg}(\mathcal{A})$, since neither of the dg categories $\mathcal{A}\mbox{-}{\rm DGMod}$ and $\mathcal{A}\mbox{-}{\rm DGMod}^{\rm ac}$ is small. However, by choosing a suitable universe $\mathbb{U}$ and restricting to $\mathbb{U}$-small dg modules, we can define the corresponding dg derived category $\mathbf{D}_{\rm dg, \mathbb{U}}(\mathcal{A})$; compare \cite[Remark 1.22 and Appendix A]{LO}. We then confuse $\mathbf{D}_{\rm dg}(\mathcal{A})$ with the well-defined category $\mathbf{D}_{\rm dg, \mathbb{U}}(\mathcal{A})$.
\end{exm}

\begin{exm}\label{exm:dsc}
Let $\Lambda$ be a $\mathbb k$-algebra,  which is a left noetherian ring. Denote by $\Lambda\mbox{-mod}$ the abelian category of finitely generated left $\Lambda$-modules. Denote by $C^b_{\rm dg}(\Lambda\mbox{-mod})$ the dg category of bounded complexes, and by $C^{b, {\rm ac}}_{\rm dg}(\Lambda\mbox{-mod})$ the full dg subcategory formed by acyclic complexes. The \emph{bounded dg derived category} is defined to be
  $$\mathbf{D}^b_{\rm dg}(\Lambda\mbox{-mod})=C^b_{\rm dg}(\Lambda\mbox{-mod})/{C^{b, {\rm ac}}_{\rm dg}(\Lambda\mbox{-mod})}.$$
  Similar as in Example~\ref{exm:ddc}, we identify $H^0(\mathbf{D}^b_{\rm dg}(\Lambda\mbox{-mod}))$ with the usual bounded derived category $\mathbf{D}^b(\Lambda\mbox{-mod})$.

 Denote by $\mathbf{per}(\Lambda)$ the full subcategory of  $\mathbf{D}^b(\Lambda\mbox{-mod})$  consisting of perfect complexes. The \emph{singularity category} \cite{Buc, Orl} of $\Lambda$ is defined to be the following Verdier quotient
 $$\mathbf{D}_{\rm sg}(\Lambda)=\mathbf{D}^b(\Lambda\mbox{-mod})/{\mathbf{per}(\Lambda)}.$$
  As its dg analogue, the \emph{dg singularity category} \cite{Kel18, BRTV} of $\Lambda$ is given by the following dg quotient category
    $$\mathbf{S}_{\rm dg}(\Lambda)=\mathbf{D}^b_{\rm dg}(\Lambda\mbox{-mod})/{\mathbf{per}_{\rm dg}(\Lambda)}.$$
    Here, $\mathbf{per}_{\rm dg}(\Lambda)$ denotes the full dg subcategory of $\mathbf{D}^b_{\rm dg}(\Lambda\mbox{-mod})$ formed by perfect complexes. This notation is consistent with the one in Example~\ref{exm:ddc},  if $\Lambda$ is viewed as  a dg category with a single object. By Lemma~\ref{lem:pre-qu}, we  identify $\mathbf{D}_{\rm sg}(\Lambda)$ with $H^0(\mathbf{S}_{\rm dg}(\Lambda))$.
\end{exm}

\subsection{One-point (co)extensions and singular equivalences with levels}\label{subsec:one-point}
In this subsection, we prove that both one-point (co)extensions  \cite[III.2]{ARS} and singular equivalences with levels \cite{Wan15} induce quasi-equivalences between dg singularity categories of the relevant algebras. For simplicity, we only consider finite dimensional algebras  and finite dimensional modules.

We first consider a  one-point coextension of an algebra.  Let $\Lambda$ be a finite dimensional $\mathbb{k}$-algebra, and $M$ be a finite dimensional right $\Lambda$-module. We view $M$ as a $\mathbb{k}$-$\Lambda$-bimodule  on which $\mathbb{k}$ acts centrally. The corresponding \emph{one-point coextension} is an upper triangular matrix algebra
$$\Lambda'=\begin{pmatrix}\mathbb{k} & M \\
                                  0 & \Lambda\end{pmatrix}.$$
As usual, a left $\Lambda'$-module is viewed as a column vector $\begin{pmatrix} V \\ X \end{pmatrix}$, where $V$ is a  $\mathbb{k}$-vector space and $X$ is a left $\Lambda$-module together with a $\mathbb{k}$-linear map $\psi\colon M\otimes_\Lambda X\rightarrow V$; see \cite[III.2]{ARS}. We usually suppress this $\psi$.

The obvious exact functor $j\colon \Lambda'\mbox{-mod}\rightarrow \Lambda\mbox{-mod}$ sends $\begin{pmatrix} V \\ X \end{pmatrix}$ to $X$. It induces a dg functor
$$j\colon \mathbf{D}^b_{\rm dg}(\Lambda'\mbox{-mod})\longrightarrow \mathbf{D}^b_{\rm dg}(\Lambda\mbox{-mod}).$$

\begin{lem}\label{lem:opce}
The above dg functor $j$ induces a quasi-equivalence $\bar{j} \colon \mathbf{S}_{\rm dg}(\Lambda')\stackrel{\sim}\longrightarrow \mathbf{S}_{\rm dg}(\Lambda)$.
\end{lem}

\begin{proof}
We observe that the functor  $j\colon \Lambda'\mbox{-mod}\rightarrow \Lambda\mbox{-mod}$ sends projective $\Lambda'$-modules to projective $\Lambda$-modules. It follows that the above dg functor $j$ respects perfect complexes. Therefore, we have the induced dg functor $\bar{j}$  between the dg singularity categories.  As in Example~\ref{exm:dsc}, we  identify $H^0(\mathbf{S}_{\rm dg}(\Lambda'))$ and $H^0(\mathbf{S}_{\rm dg}(\Lambda))$  with $\mathbf{D}_{\rm sg}(\Lambda')$ and $\mathbf{S}_{\rm sg}(\Lambda)$, respectively. Then we observe that $H^0(\bar{j})\colon \mathbf{D}_{\rm sg}(\Lambda')\rightarrow \mathbf{D}_{\rm sg}(\Lambda)$ coincides with the triangle equivalence in \cite[Proposition~4.2 and its proof]{Chen11}. By Lemma~\ref{lem:quasi-equiv}, we are done.
\end{proof}

 Let $N$ be a finite dimensional  left $\Lambda$-module. The \emph{one-point extension} is an upper triangular matrix algebra
 $$\Lambda''=\begin{pmatrix}\Lambda & N \\
                                  0 & \mathbb{k} \end{pmatrix}.$$
Similarly, a  left $\Lambda''$-module is denoted by a column vector $\begin{pmatrix} Y \\ U \end{pmatrix}$, where $U$ is a  $\mathbb{k}$-vector space and $Y$ is a left $\Lambda$-module endowed  with a left $\Lambda$-module morphism  $\phi\colon N\otimes U\rightarrow Y$.

The exact functor $i\colon \Lambda\mbox{-mod}\rightarrow \Lambda''\mbox{-mod}$ sends a  left $\Lambda$-module $Y$ to an evidently-defined $\Lambda''$-module $\begin{pmatrix} Y \\ 0 \end{pmatrix}$. It induces a dg functor
$$i\colon \mathbf{D}^b_{\rm dg}(\Lambda\mbox{-mod})\longrightarrow \mathbf{D}^b_{\rm dg}(\Lambda''\mbox{-mod}).$$

\begin{lem}\label{lem:ope}
The above dg functor $i$ induces a quasi-equivalence $\bar{i} \colon \mathbf{S}_{\rm dg}(\Lambda)\stackrel{\sim}\longrightarrow \mathbf{S}_{\rm dg}(\Lambda'')$.
\end{lem}

\begin{proof}
The argument here is similar to the one in the proof of  Lemma~\ref{lem:opce}. As the functor  $i\colon \Lambda\mbox{-mod}\rightarrow \Lambda''\mbox{-mod}$ sends projective $\Lambda$-modules to projective $\Lambda''$-modules,  the above dg functor $i$ respects perfect complexes. Therefore, we have the induced dg functor $\bar{i}$  between the dg singularity categories.  We observe that $H^0(\bar{i})\colon \mathbf{D}_{\rm sg}(\Lambda)\rightarrow \mathbf{D}_{\rm sg}(\Lambda'')$ coincides with the triangle equivalence in \cite[Proposition~4.1 and its proof]{Chen11}. Then we are done by applying Lemma~\ref{lem:quasi-equiv}.
\end{proof}

 Let $\Lambda$ and $\Pi$ be two finite dimensional  $\mathbb{k}$-algebras.  For a $\Lambda$-$\Pi$-bimodule, we always require that $\mathbb{k}$ acts centrally. Therefore, a $\Lambda$-$\Pi$-bimodule might be identified with a left module over $\Lambda\otimes \Pi^{\rm op}$.

 Denote by $\Lambda^e=\Lambda\otimes\Lambda^{\rm op}$ the \emph{enveloping algebra} of $\Lambda$.  Therefore, $\Lambda$-$\Lambda$-bimodules are viewed as left $\Lambda^e$-modules. Denote by $\Lambda^e\mbox{-\underline{mod}}$  the stable  category of $\Lambda^e\mbox{-{mod}}$ modulo projective $\Lambda^e$-modules \cite[IV.1]{ARS}, and by $\Omega_{\Lambda^e}^n(\Lambda)$ the $n$-th syzygy of $\Lambda$ for $n\geq 1$. By convention, we have $\Omega_{\Lambda^e}^0(\Lambda)=\Lambda$.

The following terminology is modified from \cite[Definition~2.1]{Wan15}.

\begin{defn}\label{defn:singequi}
Let $M$ and $N$ be a $\Lambda$-$\Pi$-bimodule and a $\Pi$-$\Lambda$-bimodule, respectively,  and let $n\geq 0$.  We say that the pair $(M, N)$ defines a \emph{singular equivalence with level $n$},  provided that the following conditions are fulfilled.
\begin{enumerate}
\item[(1)] The four one-sided modules $_\Lambda M$, $M_\Pi$, $_\Pi N$ and $N_\Lambda$ are all projective.
\item[(2)]  There are isomorphisms $M\otimes_\Pi N \simeq \Omega^n_{\Lambda^e}(\Lambda)$ and $N\otimes_\Lambda M\simeq \Omega^n_{\Pi^e}(\Pi)$ in $\Lambda^e\mbox{-\underline{mod}}$ and $\Pi^e\mbox{-\underline{mod}}$, respectively. \hfill $\square$
 \end{enumerate}
 \end{defn}

\begin{rem}\label{rem:singequi}
\begin{enumerate}
\item[(1)] A  stable equivalence of Morita type in the sense of \cite[Definition~5.A]{Brou} is naturally a singular equivalence with level zero.
\item[(2)] By \cite[Theorem~2.3]{Wan15}, a derived equivalence  induces a singular equivalence with a certain level.
    \item[(3)] By \cite[Proposition~2.6]{Sk}, a singular equivalence of Morita type, studied in  \cite{ZZ}, induces a singular equivalence with a certain level.
\end{enumerate}
\end{rem}

Assume that $M$ is a $\Lambda$-$\Pi$-bimodule such that both $_\Lambda M$ and $M_\Pi$ are projective. The obvious dg functor $M\otimes_\Pi-\colon \mathbf{D}_{\rm dg}^b(\Pi\mbox{-mod})\rightarrow \mathbf{D}_{\rm dg}^b(\Lambda\mbox{-mod})$ between the bounded dg derived categories preserves perfect complexes. Hence it induces a dg functor
$$M\otimes_\Pi- \colon \mathbf{S}_{\rm dg}(\Pi)\longrightarrow \mathbf{S}_{\rm dg}(\Lambda)$$
between the dg singularity categories.

Definition~\ref{defn:singequi} is justified by the following observation.

\begin{lem}\label{lem:sing-equi1}
Assume that  $(M, N)$ defines a singular equivalence with level $n$. Then the above dg functor $M\otimes_\Pi- \colon \mathbf{S}_{\rm dg}(\Pi)\rightarrow \mathbf{S}_{\rm dg}(\Lambda)$ is a quasi-equivalence.
\end{lem}

\begin{proof}
We identify $H^0(\mathbf{S}_{\rm dg}(\Pi))$ with $\mathbf{D}_{\rm sg}(\Pi)$, and $H^0(\mathbf{S}_{\rm dg}(\Lambda))$ with $\mathbf{D}_{\rm sg}(\Lambda)$; see Example~\ref{exm:dsc}. Then $H^0(M\otimes_\Pi-)$ is identified with the obvious tensor functor
$$M\otimes_\Pi-\colon \mathbf{D}_{\rm sg}(\Pi)\longrightarrow \mathbf{D}_{\rm sg}(\Lambda).$$
 As noted in \cite[Remark~2.2]{Wan15}, the latter  functor is a triangle equivalence, whose quasi-inverse is given by $\Sigma^n\circ (N\otimes_\Lambda-)$. Then we are done by Lemma~\ref{lem:quasi-equiv}.
\end{proof}

\section{The dg singularity category and acyclic complexes}
\label{section3}

In this section, we enhance a result in \cite{Kra} to show that the dg singularity category can be described as the dg category of certain acyclic complexes of injective modules.

We fix a $\mathbb k$-algebra $\Lambda$, which is a left noetherian ring.  We denote by $\Lambda\mbox{-Mod}$ the abelian category of left $\Lambda$-modules. For two complexes $X$ and $Y$ of $\Lambda$-modules, the Hom complex $C_{\rm dg}(\Lambda\mbox{-Mod})(X, Y)$ is usually denoted by ${\rm Hom}_\Lambda(X, Y)$.  Recall that the classical homotopy category $\mathbf{K}(\Lambda\mbox{-Mod})$  coincides with $H^0(C_{\rm dg}(\Lambda\mbox{-Mod}))$.

Denote by $\Lambda\mbox{-Inj}$ the category of injective $\Lambda$-modules,  and by $\mathbf{K}(\Lambda\mbox{-Inj})$ the homotopy category of complexes of injective modules. The full subcategory $\mathbf{K}^{\rm ac}(\Lambda\mbox{-Inj})$ is formed by acyclic complexes of injective modules.

For a bounded complex $X$ of $\Lambda$-modules, we denote by $\phi_X\colon  X\rightarrow \mathbf{i}X$ its injective resolution. Then we have the following isomorphism
\begin{align}\label{iso:inj1}
{\rm Hom}_{\mathbf{K}(\Lambda\mbox{-}{\rm Inj})}(\mathbf{i}X, I)\simeq {\rm Hom}_{\mathbf{K}(\Lambda\mbox{-}{\rm Mod})}(X, I), \quad  f\longmapsto f\circ \phi_X,
\end{align}
for each complex $I\in {\mathbf{K}(\Lambda\mbox{-Inj})}$. It follows that $\mathbf{i}X$ is compact in ${\mathbf{K}(\Lambda\mbox{-Inj})}$, if $X$ lies in $\mathbf{K}^b(\Lambda\mbox{-mod})$; see \cite[Lemma 2.1]{Kra}. In particular, we have
 \begin{align}\label{iso:inj}
 {\rm Hom}_{\mathbf{K}(\Lambda\mbox{-}{\rm Inj})}(\mathbf{i}\Lambda, I)\simeq {\rm Hom}_{\mathbf{K}(\Lambda\mbox{-}{\rm Mod})}(\Lambda, I)\simeq H^0(I).
 \end{align}
Here, we view the regular module $_\Lambda\Lambda$ as a stalk complex concentrated in degree zero. We denote by ${\rm Loc}(\mathbf{i}\Lambda)$ the localizing subcategory of $\mathbf{K}(\Lambda\mbox{-Inj})$ generated by $\mathbf{i}\Lambda$.

Denote by $C_{\rm dg}^{\rm ac}(\Lambda\mbox{-}{\rm Inj})$ the full dg subcategory of $C_{\rm dg}(\Lambda\mbox{-Mod})$ formed by acyclic complexes of injective $\Lambda$-modules. We identify $H^0(C_{\rm dg}^{\rm ac}(\Lambda\mbox{-}{\rm Inj}))$ with $\mathbf{K}^{\rm ac}(\Lambda\mbox{-Inj})$. Then $C_{\rm dg}^{\rm ac}(\Lambda\mbox{-}{\rm Inj})^c$ means the full dg subcategory formed by complexes which become compact in $\mathbf{K}^{\rm ac}(\Lambda\mbox{-Inj})$.

The following result enhances \cite[Corollary 5.4]{Kra} to the dg level.

\begin{prop}\label{prop:Krause}
There is a dg quasi-functor
$$\Phi\colon  \mathbf{S}_{\rm dg}(\Lambda)\longrightarrow C_{\rm dg}^{\rm ac}(\Lambda\mbox{-}{\rm Inj})^c,$$
such that
$$H^0(\Phi)\colon  \mathbf{D}_{\rm sg}(\Lambda)\longrightarrow \mathbf{K}^{\rm ac}(\Lambda\mbox{-}{\rm Inj})^c$$
 is a triangle equivalence up to direct summands.
\end{prop}

The following immediate consequence will be useful.

\begin{cor}\label{cor:Krause}
Assume that the $\mathbb k$-algebra $\Lambda$ is finite dimensional. Then there is a zigzag of quasi-equivalences connecting $\mathbf{S}_{\rm dg}(\Lambda)$ to $C_{\rm dg}^{\rm ac}(\Lambda\mbox{-}{\rm Inj})^c$.
\end{cor}

\begin{proof}
By \cite[Corollary 2.4]{Chen11}, the singularity category $\mathbf{D}_{\rm sg}(\Lambda)$ has split idempotents. It follows that $H^0(\Phi)$ is actually  a triangle equivalence. In view of Lemma~\ref{lem:quasi-equiv},  the required result follows immediately.
\end{proof}

Let $\mathcal{T}$ be a triangulated category. For a triangulated subcategory $\mathcal{N}$, we have the right orthogonal subcategory $\mathcal{N}^\perp=\{X\in \mathcal{T}\mid {\rm Hom}_\mathcal{T}(N, X)=0 \mbox{ for all } N\in \mathcal{N}\}$ and the left orthogonal subcategory  $^\perp \mathcal{N}=\{Y\in \mathcal{T}\mid {\rm Hom}_\mathcal{T}(Y, N)=0 \mbox{ for all } N\in \mathcal{N}\}$. The subcategory $\mathcal{N}$ is right admissible (\emph{resp}. left admissible) provided that the inclusion $\mathcal{N}\hookrightarrow \mathcal{T}$ has a right adjoint (\emph{resp}. left adjoint); see \cite{Bon}.

The following lemma is well  known; see \cite[Lemma 3.1]{Bon}.

\begin{lem}\label{lem:adm}
Let $\mathcal{N}\subseteq \mathcal{T}$ be left admissible. Then the natural functor $\mathcal{N}\rightarrow \mathcal{T}/{^\perp\mathcal{N}}$ is an equivalence.  Moreover, the left orthogonal subcategory $^\perp\mathcal{N}$ is right admissible satisfying $\mathcal{N}=(^\perp\mathcal{N})^\perp$. \hfill $\square$
\end{lem}

Denote by $\mathcal{L}$ the full dg subcategory of $C_{\rm dg}(\Lambda\mbox{-Mod})$ consisting of those complexes $X$ such that ${\rm Hom}_\Lambda(X, I)$ is acyclic for each $I\in C_{\rm dg}(\Lambda\mbox{-Inj})$. Similarly, denote by $\mathcal{M}$ the full dg subcategory formed by $Y$ satisfying that ${\rm Hom}_\Lambda(Y, J)$ is acyclic for each $J\in C_{\rm dg}^{\rm ac}(\Lambda\mbox{-Inj})$.

\begin{lem}\label{lem:4equiv}
The following canonical functors are all equivalences
\begin{enumerate}
\item[(1)] $\mathbf{K}(\Lambda\mbox{-}{\rm Inj}) \stackrel{\sim}\longrightarrow \mathbf{K}(\Lambda\mbox{-}{\rm Mod})/{H^0(\mathcal{L})}$;

  \item[(2)] $\mathbf{K}^{\rm ac}(\Lambda\mbox{-}{\rm Inj}) \stackrel{\sim}\longrightarrow \mathbf{K}(\Lambda\mbox{-}{\rm Mod})/{H^0(\mathcal{M})}$;

      \item[(3)] $\mathbf{K}^{\rm ac}(\Lambda\mbox{-}{\rm Inj}) \stackrel{\sim}\longrightarrow \mathbf{K}(\Lambda\mbox{-}{\rm Inj})/{{\rm Loc}(\mathbf{i}\Lambda)}$;

          \item[(4)] $  \mathbf{K}(\Lambda\mbox{-}{\rm Inj})/{{\rm Loc}(\mathbf{i}\Lambda)} \stackrel{\sim}\longrightarrow  \mathbf{K}(\Lambda\mbox{-}{\rm Mod})/{H^0(\mathcal{M})}$,
\end{enumerate}
which send any complex $I$ to itself, viewed as an object in the target categories.
\end{lem}

\begin{proof}
 The Brown representability theorem and its dual version yield the following useful fact: for a triangulated category $\mathcal{T}$ with arbitrary coproducts and a localizing subcategory $\mathcal{N}$ which is compactly generated, then the subcategory $\mathcal{N}$ is right admissible; if furthermore $\mathcal{N}$ is closed under products, then $\mathcal{N}$ is also left admissible; see \cite[Proposition 3.3]{Kra}.

Recall from \cite[Proposition 2.3 and Corollary 5.4]{Kra} that both $\mathbf{K}(\Lambda\mbox{-}{\rm Inj})$ and $\mathbf{K}^{\rm ac}(\Lambda\mbox{-}{\rm Inj})$ are compactly generated, which are both closed under coproducts and products in $\mathbf{K}(\Lambda\mbox{-}{\rm Mod})$. Moreover, we observe that $^\perp \mathbf{K}(\Lambda\mbox{-}{\rm Inj})=H^0(\mathcal{L})$ and $^\perp\mathbf{K}^{\rm ac}(\Lambda\mbox{-}{\rm Inj})=H^0(\mathcal{M})$, where the orthogonal is taken in $\mathbf{K}(\Lambda\mbox{-}{\rm Mod})$. Then the above fact and Lemma \ref{lem:adm} yield (1) and (2).

By the isomorphism (\ref{iso:inj}), we infer that $\mathbf{K}^{\rm ac}(\Lambda\mbox{-}{\rm Inj})={\rm Loc}(\mathbf{i}\Lambda)^\perp$, where the orthogonal is taken in $\mathbf{K}(\Lambda\mbox{-}{\rm Inj})$. Since $\mathbf{i}\Lambda$ is compact in $\mathbf{K}(\Lambda\mbox{-}{\rm Inj})$, the subcategory ${\rm Loc}(\mathbf{i}\Lambda)$ is right admissible. It follows from the dual version of Lemma \ref{lem:adm} that $\mathbf{K}^{\rm ac}(\Lambda\mbox{-}{\rm Inj})\subseteq \mathbf{K}(\Lambda\mbox{-}{\rm Inj})$ is left admissible satisfying $^\perp\mathbf{K}^{\rm ac}(\Lambda\mbox{-}{\rm Inj})={\rm Loc}(\mathbf{i}\Lambda)$. Then (3) follows from  Lemma \ref{lem:adm}.

The functor in (4) is well defined, since ${\rm Loc}(\mathbf{i}\Lambda)\subseteq H^0(\mathcal{M})$. Then (4) follows by combining (2) and (3).
\end{proof}

Denote by $\mathcal{P}$ the full dg subcategory of $C_{\rm dg}^b(\Lambda\mbox{-mod})$ formed by those complexes which are  isomorphic to  bounded complexes of projective $\Lambda$-modules in $\mathbf{D}^b(\Lambda\mbox{-mod})$. Therefore, we might identify the singularity category $\mathbf{D}_{\rm sg}(\Lambda)$ with $\mathbf{K}^b(\Lambda\mbox{-mod})/{H^0(\mathcal{P})}$.

\begin{lem}\label{lem:sg-inj}
The canonical functor $\mathbf{K}^b(\Lambda\mbox{-}{\rm mod})/{H^0(\mathcal{P})}\rightarrow  \mathbf{K}(\Lambda\mbox{-}{\rm Mod})/{H^0(\mathcal{M})}$ is fully faithful, which induces a triangle equivalence up to direct summands
$$\mathbf{K}^b(\Lambda\mbox{-}{\rm mod})/{H^0(\mathcal{P})}\stackrel{\sim}\longrightarrow (\mathbf{K}(\Lambda\mbox{-}{\rm Mod})/{H^0(\mathcal{M})})^c.$$
\end{lem}

\begin{proof}
The functor is well defined since we have $\mathcal{P}\subseteq \mathcal{M}$. The assignment $X\mapsto \mathbf{i}X$ of injective resolutions yields a triangle functor $\mathbf{i}\colon  \mathbf{K}^b(\Lambda\mbox{-}{\rm mod})\rightarrow \mathbf{K}(\Lambda\mbox{-}{\rm Inj})$. It induces the following horizontal functor.
\[
\xymatrix{
\mathbf{K}^b(\Lambda\mbox{-}{\rm mod})/{H^0(\mathcal{P})}\ar[rd] \ar[rr]^-{\mathbf{i}} && \mathbf{K}(\Lambda\mbox{-}{\rm Inj})/{{\rm Loc}(\mathbf{i}\Lambda)} \ar[ld]\\
& \mathbf{K}(\Lambda\mbox{-}{\rm Mod})/{H^0(\mathcal{M})}
}\]
The unnamed arrows are canonical functors. By \cite[Corollary 5.4]{Kra} the horizontal functor $\mathbf{i}$ induces a triangle equivalence up to direct summands
 $$\mathbf{K}^b(\Lambda\mbox{-}{\rm mod})/{H^0(\mathcal{P})}\stackrel{\sim}\longrightarrow (\mathbf{K}(\Lambda\mbox{-}{\rm Inj})/{{\rm Loc}(\mathbf{i}\Lambda)})^c.$$

 We claim that the diagram is commutative up to a natural isomorphism. Then we are done by Lemma \ref{lem:4equiv}(4).

 For the claim, we take $X\in \mathbf{K}^b(\Lambda\mbox{-}{\rm mod})$ and consider its injective resolution $\phi_X\colon  X\rightarrow {\bf i}X$. We have the exact triangle
 $$X\stackrel{\phi_X}\longrightarrow \mathbf{i}X \longrightarrow {\rm Cone}(\phi_X)\longrightarrow \Sigma(X).$$
 The isomorphism (\ref{iso:inj1}) implies that ${\rm Cone}(\phi_X)$ lies in $H^0(\mathcal{L})\subseteq H^0(\mathcal{M})$. Therefore, $\phi_X$ becomes an isomorphism in  $\mathbf{K}(\Lambda\mbox{-}{\rm Mod})/{H^0(\mathcal{M})} $, proving the claim.
 \end{proof}

We are now in a position to prove Proposition~\ref{prop:Krause}.

\vskip 5pt

\noindent \emph{Proof of Proposition~\ref{prop:Krause}.}\quad By the equivalence in Lemma \ref{lem:4equiv}(2), the canonical  dg functor
$$C_{\rm dg}^{\rm ac}(\Lambda \mbox{-} {\rm Inj}) \stackrel{\sim}\longrightarrow C_{\rm dg}(\Lambda \mbox{-}{\rm Mod})/\mathcal{M}$$
is a quasi-equivalence, which restricts to a quasi-equivalence on compact objects
$$C_{\rm dg}^{\rm ac}(\Lambda \mbox{-} {\rm Inj})^c \stackrel{\sim}\longrightarrow (C_{\rm dg}(\Lambda \mbox{-}{\rm Mod})/\mathcal{M})^c.$$
 Here, for the precise definition of the dg quotient category $C_{\rm dg}(\Lambda \mbox{-}{\rm Mod})/\mathcal{M}$, we have to consult \cite[Remark 1.22]{LO}; compare Example~\ref{exm:ddc}.

By Lemma \ref{lem:noether}, we may identify $\mathbf{S}_{\rm dg}(\Lambda)$ with $C_{\rm dg}^{b}(\Lambda\mbox{-}{\rm mod})/\mathcal{P}$. By Lemma \ref{lem:sg-inj}, the following canonical dg functor
$$C_{\rm dg}^{b}(\Lambda\mbox{-}{\rm mod})/\mathcal{P} \longrightarrow (C_{\rm dg}(\Lambda \mbox{-}{\rm Mod})/\mathcal{M})^c $$
is quasi-fully faithful, which induces a triangle equivalence up to direct summands between the homotopy categories. Combining them, we obtain the required dg quasi-functor.  \hfill $\square$

\section{Quivers and Leavitt path algebras}\label{sec:lpa}

In this section, we recall basic facts on quivers and Leavitt path algebras. Using the main result in \cite{Li}, we relate the dg singularity category of the finite dimensional algebra with radical square zero to the dg perfect derived category of the Leavitt path algebra. We obtain an explicit graded derivation over the Leavitt path algebra, which will be used in Subsection~\ref{subsection:hdrforLPA}.

Recall that a quiver $Q=(Q_{0}, Q_{1}; s, t)$ consists of a set $Q_{0}$ of vertices, a set $Q_{1}$ of arrows and two maps $s, t\colon  Q_{1}\xrightarrow []{}Q_{0}$, which associate to each arrow $\alpha$ its starting vertex $s(\alpha)$ and its terminating vertex $t(\alpha)$, respectively. A vertex $i$ of $Q$ is a \emph{sink} provided that the set $s^{-1}(i)$ is empty.

  A path of length $n$ is a sequence $p=\alpha_{n}\dotsb\alpha_{2}\alpha_{1}$ of arrows with $t(\alpha_{j})=s(\alpha_{j+1})$ for $1\leq j\leq n-1$. Denote by $l(p)=n$.  The starting vertex of $p$, denoted by $s(p)$, is $s(\alpha_{1})$ and the terminating vertex of $p$, denoted by $t(p)$, is $t(\alpha_{n})$. We identify an arrow with a path of length one. We associate to each vertex $i\in Q_{0}$ a trivial path $e_{i}$ of length zero. Set $s(e_i)=i=t(e_i)$. Denote by $Q_n$ the set of paths of length $n$.

The \emph{path algebra} $\mathbb k Q=\bigoplus_{n\geq 0}\mathbb k Q_{n}$ has a basis given by all paths in $Q$, whose multiplication is given as follows: for two paths $p$ and $q$ satisfying $s(p)=t(q)$, the product $pq$ is their concatenation; otherwise, we set the product $pq$ to be zero. Here, we write the concatenation of paths from right to left. For example,  $e_{t(p)}p=p=pe_{s(p)}$ for each path $p$. Denote by $J=\bigoplus_{n\geq 1} \mathbb k Q_n$ the two-sided ideal generated by arrows.

  We denote by $\overline{Q}$ the \emph{double quiver} of $Q$, which is obtained by adding for each arrow $\alpha\in Q_1$ a new arrow $\alpha^*$ in the opposite direction. Clearly, we have $s(\alpha^*)=t(\alpha)$ and $t(\alpha^*)=s(\alpha)$. The added arrows $\alpha^*$ are called the ghost arrows.

In what follows, we assume that $Q$ is a finite quiver without sinks. We set $\Lambda=\mathbb k Q/J^2$ to be the corresponding finite dimensional algebra with radical square zero. Observe that $J^2$ is the two-sided ideal of $\mathbb{k}Q$ generated by the set of all paths of length two.

The \emph{Leavitt path algebra} $L=L(Q)$ \cite{AA,AGGP, AMP} is by definition the quotient algebra of $\mathbb k\overline{Q}$ modulo the two-sided ideal generated by the following set
\begin{align*}
\big\{\alpha\beta^*-\delta_{\alpha, \beta}e_{t(\alpha)}\mid \alpha, \beta\in Q_1 \mbox{ with }s(\alpha)=s(\beta)\}\cup \{\sum_{\{\alpha\in Q_1\mid s(\alpha)=i\}}\alpha^*\alpha-e_i\mid i\in Q_0 \big\}.
\end{align*}
These elements are known as the \emph{first Cuntz-Krieger relations} and the \emph{second Cuntz-Krieger relations}, respectively.

If $p=\alpha_{n}\cdots\alpha_2\alpha_{1}$ is a path in $Q$ of length $n\geq 1$, we define $p^{*}=\alpha_{1}^*\alpha_2^*\cdots\alpha_{n}^*$. We have $s(p^*)=t(p)$ and $t(p^*)=s(p)$. For convention, we set $e_i^*=e_i$.  We observe that for paths $p, q$ in $Q$ satisfying $t(p)\neq t(q)$, $p^*q=0$ in $L$. Recall that the Leavitt path algebra $L$ is spanned by the following set
$$\big\{e_i, p, p^*, \gamma^*\eta\mid i\in Q_0, \; p, \gamma, \text{~and~} \eta \text{~are ~nontrivial paths in ~} Q \text{~with~} t(\gamma)=t(\eta)\big\};$$
see \cite[Corollary 3.2]{Tom}. In general, this set is not $\mathbb k$-linearly independent. For an explicit basis, we refer to \cite[Theorem 1]{AAJZ}.

The Leavitt path algebra $L$ is naturally $\mathbb{Z}$-graded by  $|e_{i}|=0$, $|\alpha|=1$ and $|\alpha^{*}|=-1$ for $i\in Q_{0}$ and $\alpha\in Q_{1}$. We write $L=\bigoplus_{n\in\mathbb{Z}} L^{n}$, where $L^n$ consists of homogeneous elements of degree $n$.

For each $i\in Q_0$ and $m\geq 0$, we consider the following subspace of $e_iLe_i$
$$X_{i, m}={\rm Span}_\mathbb k\{\gamma^*\eta\mid t(\gamma)=t(\eta), s(\gamma)=i=s(\eta), l(\eta)=m\}.$$
We observe that $X_{i, m}\subseteq X_{i, m+1}$, since we have
\begin{align}\label{equ:CK2}
\gamma^*\eta=\sum_{\{\alpha\in Q_1\mid s(\alpha)=t(\eta)\}} (\alpha\gamma)^*\alpha\eta.
\end{align}

\begin{lem} \label{lemma:basis}
The following facts hold.
\begin{enumerate}
\item The set $\{\gamma^*\eta\mid t(\gamma)=t(\eta), s(\gamma)=i=s(\eta), l(\eta)=m\}$ is $\mathbb k$-linearly independent.
 \item We have $e_iLe_i=\bigcup_{m\geq 0} X_{i, m}$.
\end{enumerate}
\end{lem}

\begin{proof}
Using the grading of $L$, the first statement follows from \cite[Proposition~4.1]{Chen}. The second one is trivial.
\end{proof}

The following result is based on the main result of \cite{Li}. We will always view the $\mathbb{Z}$-graded algebra $L=L(Q)$ as a dg algebra with trivial differential. Then $L^{\rm op}$ denotes the opposite dg algebra. We view $\Lambda=\mathbb k Q/J^2$ as a dg algebra concentrated in degree zero.

\begin{prop}\label{prop:CY-Li}
Keep the notation as above. Then there is a zigzag of  quasi-equivalences connecting $\mathbf{S}_{\rm dg}(\Lambda)$ to $\mathbf{ per}_{\rm dg}(L^{\rm op})$.
\end{prop}

\begin{proof}
 Recall that the {\it injective Leavitt complex} $\mathcal I$ is constructed in \cite{Li}, which is a dg $\Lambda$-$L^{\rm op}$-bimodule. Moreover, it induces a triangle equivalence
 $${\rm Hom}_\Lambda(\mathcal I, -) \colon  \mathbf{K}^{\rm ac}(\Lambda\mbox{-Inj}) \stackrel{\sim}\longrightarrow \mathbf{D}(L^{\rm op}),$$
 which restricts to an equivalence
 $$\mathbf{K}^{\rm ac}(\Lambda\mbox{-Inj})^c \stackrel{\sim}\longrightarrow \mathbf{per}(L^{\rm op}).$$

Recall the identifications  $H^0( C_{\rm dg}^{\rm ac}(\Lambda\mbox{-Inj})^c)=\mathbf{K}^{\rm ac}(\Lambda\mbox{-Inj})^c$ and $H^0(\mathbf{per}_{\rm dg}(L^{\rm op}))=\mathbf{per}(L^{\rm op})$.  Then combining the above restricted equivalence and   Lemma~\ref{lem:quasi-equiv},  we infer that the dg functor
 $${\rm Hom}_\Lambda(\mathcal I, -)\colon  C_{\rm dg}^{\rm ac}(\Lambda\mbox{-Inj})^c\longrightarrow \mathbf{per}_{\rm dg}(L^{\rm op})$$
is  a quasi-equivalence. Then we are done by Corollary \ref{cor:Krause}.
\end{proof}

Set $E=\mathbb kQ_0=\bigoplus_{i\in Q_0} \mathbb ke_i$, which is viewed as a semisimple subalgebra of $L^0$. Let $M$ be a graded $L$-$L$-bimodule. A graded map $D\colon  L\rightarrow M$ of degree $-1$ is called a \emph{graded derivation} provided that it satisfies the graded Leibniz rule
$$D(xy)=D(x)y+(-1)^{|x|}xD(y)$$
 for $x,y\in L$; if furthermore it satisfies $D(e_i)=0$ for each $i\in Q_0$, it is called a \emph{graded $E$-derivation}.

Let $s\mathbb k$ be the $1$-shifted space of $\mathbb k$, that is, $s\mathbb k$ is concentrated in degree $-1$. The element $s1_\mathbb k$ of degree $-1$  will be simply denoted by $s$. Then we have the graded $L$-$L$-bimodule $\bigoplus_{i\in Q_0} Le_i\otimes s\mathbb k\otimes e_i L$, which is clearly isomorphic to $L \otimes_E sE \otimes_E L$.

\begin{lem}\label{lem:deri}
Keep the notation as above. Then there is a unique graded $E$-derivation
$$D\colon  L\longrightarrow \bigoplus_{i\in Q_0}Le_i\otimes s\mathbb k\otimes e_i L $$
satisfying  $D(\alpha)=-\alpha\otimes s\otimes e_{s(\alpha)}$ and $D(\alpha^*)=-e_{s(\alpha)}\otimes s\otimes  \alpha^*$ for each $\alpha\in Q_1$.
\end{lem}

\begin{proof}
It is well known that there is a unique graded $E$-derivation
$$\overline{D}\colon  \mathbb k \overline{Q}\longrightarrow \bigoplus_{i\in Q_0}Le_i\otimes s\mathbb k\otimes e_i L$$
satisfying $\overline{D}(\alpha)=-\alpha\otimes s\otimes e_{s(\alpha)}$ and $\overline{D}(\alpha^*)=-e_{s(\alpha)}\otimes s \otimes \alpha^*$; consult the explicit bimodule projective resolution in \cite[Chapter 2, Proposition 2.6]{Cohn}. It is routine to verify that $\overline{D}$ vanishes on the Cuntz-Krieger relations. Therefore, by the graded Leibniz rule, it vanishes on the whole defining ideal. Then $\overline{D}$ induces uniquely the required derivation $D$.
\end{proof}

The following observation will be useful in the proof of Proposition~\ref{proposition-Phi}.

 \begin{rem}\label{rem:derivation}
By the graded Leibniz rule, the graded $E$-derivation $D$ has the following explicit description: for  nontrivial paths $\eta=\alpha_m\cdots \alpha_2\alpha_1$ and $\gamma=\beta_p\cdots \beta_2\beta_1$ satisfying $t(\eta)=t(\gamma)$, we have
\begin{equation*}
\begin{split}
D(\gamma^*\eta) = & -e_{s(\gamma)}\otimes s\otimes \gamma^*\eta -
\sum_{l=1}^{p-1}(-1)^l  \beta_1^*\cdots \beta_l^*\otimes s\otimes \beta_{l+1}^*\cdots \beta_p^*\alpha_m\cdots \alpha_1+{}\\
&\!\!\!\! \!\!  \sum_{l=1}^{m-1} (-1)^{m+p-l}\beta_1^*\cdots \beta_p^*\alpha_m\cdots \alpha_{l+1} \otimes s\otimes \alpha_l\cdots \alpha_1 +(-1)^{m+p} \gamma^*\eta\otimes s\otimes e_{s(\eta)}.
\end{split}
\end{equation*}
Similarly,  we have
\begin{equation*}
\begin{split}
D(\gamma^*)&=-e_{s(\gamma)}\otimes s\otimes \gamma^*-\sum^{p-1}_{l=1} (-1)^l \beta_1^*\cdots \beta_l^*\otimes s\otimes \beta_{l+1}^*\cdots \beta_p^*, \; {\rm and} \\
D(\eta)&=\sum^{m-1}_{l=1}(-1)^{m-l} \alpha_m\cdots\alpha_{l+1}\otimes s\otimes \alpha_{l}\cdots \alpha_1+(-1)^{m} \eta\otimes s\otimes e_{s(\eta)}.
\end{split}
\end{equation*}

\end{rem}

\section{An introduction to $B_\infty$-algebras}
\label{section5}

In this section, we give a brief self-contained introduction to  $B_{\infty}$-algebras and $B_{\infty}$-morphisms.  We introduce  the  opposite $B_\infty$-algebra and the  transpose $B_\infty$-algebra of   any given  $B_\infty$-algebra. There is a natural $B_\infty$-isomorphism between them; see Theorem~\ref{thm:dualityB}. We are mainly interested in a class of $B_{\infty}$-algebras, called brace $B_{\infty}$-algebras, whose underlying $A_{\infty}$-algebras are dg algebras and some of whose $B_\infty$-products vanish.

\subsection{$A_\infty$-algebras and morphisms}\label{subsec:A-infinity}

Let us  start by recalling $A_{\infty}$-algebras and $A_{\infty}$-morphisms. For details,  we refer  to \cite{Kel01}. For two graded maps $f\colon  U\rightarrow V$ and $f'\colon  U'\rightarrow V'$ between graded spaces, the tensor product $f\otimes f'\colon  U\otimes U'\rightarrow V\otimes V'$ is defined such that $$(f\otimes f')(u\otimes u')=(-1)^{|f'|\cdot |u|} f(u)\otimes f'(u'),$$
 where the sign $(-1)^{|f'|\cdot |u|}$ is given by the {\it Koszul sign rule}.   We use $\mathbf{1}$  to denote the identity endomorphism.

\begin{defn}
\label{defn:A-infinity}
 An {\it  $A_{\infty}$-algebra} is a graded $\mathbb k$-vector space $A=\bigoplus_{p\in \mathbb Z} A^p$ endowed with graded $\mathbb k$-linear maps
$$m_n\colon  A^{\otimes n}\longrightarrow A, \quad n\geq 1,$$
of degree $2-n$ satisfying the following relations
\begin{equation}\label{relationforA}
\sum_{j=0}^{n-1}\sum_{s=1}^{n-j}(-1)^{j+s(n-j-s)} \; m_{n-s+1} (\mathbf{1}^{\otimes j}\otimes m_s\otimes \mathbf{1}^{\otimes(n-j-s)})=0,
\quad \mbox{for $n\geq 1$.}
\end{equation}
In particular, $(A, m_1)$ is a cochain complex of $\mathbb k$-vector spaces.

For two $A_{\infty}$-algebras $A$ and $A'$, an {\it $A_{\infty}$-morphism} $f=(f_n)_{\geq 1}\colon  A\rightarrow A'$ is given by a collection of graded maps
$f_n\colon  A^{\otimes n}\rightarrow A'$
of degree $1-n$ such that, for all $n\geq 1$, we have
\begin{equation}\label{equationforAmorph}
\sum_{\substack{a+s+t=n\\ a, t\geq 0, s\geq 1}} (-1)^{a+st} f_{a+1+t}(\mathbf{1}^{\otimes
a}\otimes m_s\otimes \mathbf{1}^{\otimes t})=\sum_{\substack{r\geq 1\\ i_1+\dotsb + i_r=n}} (-1)^{\epsilon} m'_r (f_{i_1}\otimes \cdots\otimes f_{i_r}),
\end{equation}
where $\epsilon=(r-1)(i_1-1)+(r-2)(i_2-1)+\cdots+2(i_{r-2}-1)+(i_{r-1}-1)$; if $r=1$, we set $\epsilon=0$.  In particular, $f_1\colon (A, m_1)\rightarrow (A', m'_1)$ is a cochain map.

The {\it composition} $g\circ_\infty f$ of two $A_{\infty}$-morphisms $f\colon  A\rightarrow A'$ and $g\colon  A'\rightarrow A''$ is given by
$$(g\circ_{\infty} f)_n=\sum_{r\geq 1, \; i_1+\cdots+i_r=n} (-1)^{\epsilon}g_r  (f_{i_1}\otimes \cdots\otimes f_{i_r}), \quad n\geq 1,$$
where $\epsilon$ is defined as above. \hfill $\square$
\end{defn}

An $A_{\infty}$-morphism  $f\colon   A\rightarrow A'$ is  \emph{strict} provided that $f_i=0$ for all $i\neq 1$. The identity morphism is the strict morphism $f$ given by $f_1=\mathbf{1}_A$.  An $A_{\infty}$-morphism $f\colon   A\rightarrow A'$ is an {\it $A_\infty$-isomorphism} if  there exists an $A_{\infty}$-morphism $g\colon   A'\rightarrow A$ such that the composition
$f\circ_{\infty} g$ coincides with the identity morphism of $A'$ and $g\circ_{\infty} f$ coincides with the identity morphism of $A$. In general,  an $A_\infty$-isomorphism is not necessarily strict; see Theorem~\ref{thm:dualityB} for an example.

An $A_{\infty}$-morphism  $f\colon   A\rightarrow A'$  is called an {\it $A_\infty$-quasi-isomorphism} provided that $f_1\colon   (A, m_1)\rightarrow (A', m'_1)$ is a quasi-isomorphism between the underlying complexes. An $A_\infty$-isomorphism is necessarily an $A_\infty$-quasi-isomorphism.

\begin{rem}\label{rem:A-inf}
Let $A$ be a graded $\mathbb k$-space and let $sA$ be the  $1$-shifted graded  space: $(sA)^i=A^{i+1}$. Denote by $(T^c(sA), \Delta)$ the tensor coalgebra over $sA$. It is well known that   an $A_{\infty}$-algebra structure on $A$ is equivalent to a dg coalgebra structure $(T^c(sA), \Delta, D)$ on $T^c(sA)$, where  $D$ is a coderivation of degree one satisfying $D^2=0$ and $D(1)=0$. Accordingly, $A_\infty$-morphisms $f\colon  A\rightarrow A'$ correspond bijectively to dg coalgebra homomorphisms $T^c(sA)\rightarrow T^c(sA')$. Under this  bijection, the above composition $f\circ_{\infty} g$ of the $A_{\infty}$-morphisms $f$ and $g$  corresponds to the usual composition of the induced dg coalgebra homomorphisms; see \cite[Lemma 3.6]{Kel01}.
 \end{rem}

 We mention that any dg algebra $A$ is viewed as an $A_\infty$-algebra with $m_n=0$ for $n\geq 3$. In Subsection~\ref{subsection:A-quasi-brace},  we will construct an explicit $A_\infty$-quasi-isomorphism between two concrete dg algebras, which is a non-strict $A_\infty$-morphism, that is, not a dg algebra homomorphism between the dg algebras.

\subsection{$B_\infty$-algebras and morphisms}\label{subsection:B-infinity}
The notion of $B_\infty$-algebras\footnote{We remark that the letter \lq B' stands for Baues,  who showed in \cite{BAU} that the normalized cochain complex  $C^*(X)$ of any simplicial set $X$ carries a natural $B_\infty$-algebra.} is due to \cite[Subsection~5.2]{GJ}. We unpack the definition therein and write the axioms explicitly. We are mainly concerned with a certain kind of $B_\infty$-algebras, called \emph{brace $B_\infty$-algebras}; see Definition~\ref{defn:special}.  We mention other references \cite{Vor,  Kel03} for $B_\infty$-algebras.

Let $A=\bigoplus_{p\in \mathbb{Z}} A^p$ be a graded space, and  let $r\geq 1$ and $l, n\geq 0$.  For any two sequences of nonnegative integers $(l_1, l_2, \dotsc, l_r)$ and $(n_1, n_2, \dotsc, n_r)$ satisfying  $l=l_1+\cdots+l_r$ and $n=n_1+\cdots+n_r$, we define a $\mathbb k$-linear map
$$\tau_{(l_1, \dotsc, l_r; n_1, \dotsc, n_r)}\colon   A^{\otimes l}\bigotimes A^{\otimes n} \longrightarrow (A^{\otimes l_1}\bigotimes A^{\otimes n_1}) \otimes \cdots\otimes (A^{\otimes l_r}\bigotimes A^{\otimes n_r})$$
 by sending $(a_1\otimes \cdots \otimes a_l)\bigotimes (b_1\otimes \cdots\otimes b_n)\in A^{\otimes l}\bigotimes A^{\otimes n}$ to
 \begin{equation*}
\begin{split}
(-1)^{\epsilon'} &(a_1\otimes \cdots \otimes a_{l_1} \bigotimes b_1\otimes \cdots\otimes b_{n_1})\otimes \cdots \otimes \\
&(a_{l_1+\cdots+l_{r-1}+1} \otimes \dotsb\otimes a_l\bigotimes b_{n_1+\cdots+n_{r-1}+1}\otimes \cdots\otimes b_n),
\end{split}
\end{equation*}
where $\epsilon'= \sum_{i=0}^{r-2} (|b_{n_1+\cdots+n_i+1}|+\cdots+|b_{n_1+\cdots+n_{i+1}}|)(|a_{l_1+\cdots+ l_{i+1}+1}|+\cdots+|a_l|)$ with  $n_0=0$. If $l_i=0$ for some $1\leq i \leq r$ we set $A^{\otimes l_i} = \mathbb k$ and
$a_{l_1+\cdots+l_{i-1}+1}\otimes \cdots\otimes a_{l_1+\cdots+l_i}=1\in \mathbb k$; similarly,  if $n_i=0$ we set $A^{\otimes n_i} = \mathbb k$ and $b_{n_1+\cdots+n_{i-1}+1}\otimes \cdots\otimes b_{n_1+\cdots+n_i}=1\in \mathbb{k}$. Here and later, we use the big tensor product $\bigotimes$ to distinguish from the usual $\otimes$ and to specify the space where the tensors belong to.

\begin{defn}\label{definition-B-infinity}
A {\it $B_{\infty}$-algebra} is an $A_{\infty}$-algebra $(A, m_1, m_2, \cdots)$ together with a collection of graded  maps (called {\it $B_\infty$-products})
$$\mu_{p, q}\colon  A^{\otimes p}\bigotimes A^{\otimes q}\longrightarrow A, \quad p, q\geq 0$$
of degree $1-p-q$ satisfying the following relations.
\begin{enumerate}
\item The unital condition:
\begin{equation}\label{equation-B01}
\mu_{1, 0}=\mathbf{1}_A= \mu_{0, 1}, \quad \mu_{k, 0}=0=\mu_{0, k}\quad  \mbox{for $k\neq 1$}.
\end{equation}

\item The associativity of $\mu_{p, q}$: for any fixed $k, l, n\geq 0$, we have
\begin{equation}\label{equation-B02}
\begin{split}
&\sum_{r=1}^{l+n}\sum_{\substack{l_1+\cdots+l_r=l\\ n_1+\cdots+n_r=n}} (-1)^{\epsilon_1}\mu_{k, r} (\mathbf{1}^{\otimes k}\bigotimes (\mu_{l_1, n_1}\otimes \cdots\otimes \mu_{l_r, n_r})\circ \tau_{(l_1, \dotsc, l_r; n_1, \dotsc, n_r)}) \\
={}& \sum_{s=1}^{k+l}\sum_{\substack{k_1+\cdots+k_s=k\\ l_1+\cdots+l_s=l}} (-1)^{\eta_1}\mu_{s, n}((\mu_{k_1, l_1}\otimes \cdots\otimes \mu_{k_s, l_s})\circ \tau_{(k_1, \dotsc, k_s; l_1, \dotsc, l_s)}\bigotimes \mathbf{1}^{\otimes n}),
\end{split}
\end{equation}
where
$$
\epsilon_1=\sum_{i=1}^{r-1}(l_i+n_i-1)(r-i)+\sum_{i=1}^{r-1}n_i(l_{i+1}+\cdots+l_r),
$$
$$\text{and} \quad \eta_1=\sum_{i=1}^s(k_i+l_i-1)(n+s-i)+\sum_{i=1}^{s-1} l_i(k_{i+1}+\cdots+k_s).
$$

\item The Leibniz rule for $m_n$ with respect to  $\mu_{p, q}$: for any fixed $k, l\geq 0$, we have
\begin{equation}\label{equation-B03}
\begin{split}
&\sum_{r=1}^{k+l} \sum_{\substack{k_1+\cdots+k_r=k\\ l_1+\cdots+l_r=l}} (-1)^{\epsilon_2} m_r (\mu_{k_1, l_1}\otimes \cdots\otimes  \mu_{k_r, l_r})\circ \tau_{(k_1, \dotsc, k_r; l_1, \dotsc, l_r)}\\
={}& \sum_{r=1}^{k} \sum_{i=0}^{k-r} (-1)^{\eta_2'} \mu_{k-r+1, l} (\mathbf{1}^{\otimes i}\otimes  m_r\otimes \mathbf{1}^{\otimes k-r-i}\bigotimes \mathbf{1}^{\otimes l})\\
&+ \sum_{s=1}^l \sum_{i=0}^{l-s} (-1)^{\eta_2''} \mu_{k, l-s+1}(\mathbf{1}^{\otimes k}\bigotimes  \mathbf{1}^{\otimes i}\otimes m_s\otimes \mathbf{1}^{\otimes l-i-s} ),
\end{split}
\end{equation}
where $$\epsilon_2=\sum_{i=1}^r (k_i+l_i-1)(r-i)+\sum_{i=1}^rl_i(k-k_{1}-\cdots-k_i),$$ $$\eta_2'=r(k-r-i+l)+i, \quad \text{and} \quad
\eta_2''=s(l-i-s)+k+i.$$
\end{enumerate}
We usually denote a $B_{\infty}$-algebra by $(A, m_n; \mu_{p,q})$.

 A {\it $B_{\infty}$-morphism}  from   $(A, m_n; \mu_{p, q})$ to  $(A', m'_n; \mu_{p, q}')$ is an $A_{\infty}$-morphism
 $$f=(f_n)_{n\geq 1}\colon   A\longrightarrow A'$$
 satisfying the following identity for any $p, q\geq 0$:
 \begin{equation}\label{equation-B-new}
\begin{split}
&\sum_{r, s\geq 0}\sum_{\substack{i_1+ i_2+ \dotsb + i_r=p\\ j_1+j_2+ \cdots+ j_s=q}}(-1)^{\epsilon} \mu'_{r, s}  (f_{i_1}\otimes  \cdots\otimes f_{i_r}\bigotimes f_{j_1} \otimes   \cdots\otimes f_{j_s})\\
={}&\sum_{t\geq 1}\sum_{\substack{l_1+ l_2+ \dotsb + l_t=p\\ m_1+m_2+ \cdots+ m_t=q}} (-1)^{\eta}f_t\circ  (\mu_{l_1, m_1}\otimes  \cdots\otimes\mu_{l_t, m_t})\circ \tau_{(l_1, \dotsc, l_t; m_1, \dotsc, m_t)},
\end{split}
\end{equation}
where
\begin{equation*}
\begin{split}
\epsilon={}&\sum_{k=1}^r (i_k-1)(r+s-k)+\sum_{k=1}^s( j_k-1)(s-k), \; {\rm and} \\
\eta={}&\sum_{k=1}^t m_k(p-l_1-\cdots-l_k)+\sum_{k=1}^t (l_k+m_k-1) (t-k).\end{split}
\end{equation*}
The composition of $B_\infty$-morphisms is the same as the one of $A_\infty$-morphisms. \hfill $\square$
\end{defn}

  A $B_{\infty}$-morphism $f\colon A\rightarrow A'$ is \emph{strict} if $f_i=0$ for each $i\neq 1$. A $B_{\infty}$-morphism $f\colon A\rightarrow A'$ is a {\it $B_\infty$-isomorphism}, if there exists an $B_{\infty}$-morphism $g\colon A'\rightarrow A$ such that the compositions $f\circ_{\infty} g=\mathbf{1}_{A'}$ and $g\circ_{\infty} f=\mathbf{1}_{A}.$  A $B_{\infty}$-morphism $f\colon A\rightarrow A'$ is a {\it $B_\infty$-quasi-isomorphism} if $f_1\colon (A, m_1)\rightarrow (A', m'_1)$ is a quasi-isomorphism.

Consider the category of $B_{\infty}$-algebras, whose objects are $B_{\infty}$-algebras and whose morphisms are $B_{\infty}$-morphisms. It follows from \cite{Hinich, Kel03} that the category of $B_{\infty}$-algebras admits a model structure, whose weak equivalences are precisely $B_{\infty}$-quasi-isomorphisms.  We denote by $\mathrm{Ho}(B_{\infty})$ the \emph{homotopy category} associated with this model structure. In particular, each isomorphism in $\mathrm{Ho}(B_{\infty})$ comes from a zigzag of $B_{\infty}$-quasi-isomorphisms.

\begin{rem}\label{remark-B-infinity}
		Similar to Remark~\ref{rem:A-inf},  a $B_{\infty}$-algebra structure on $A$  is equivalent to a dg bialgebra structure $(T^c(sA), \Delta, D, \mu)$ on the tensor coalgebra $T^c(sA)$  such that  $1\in \mathbb k= (sA)^{\otimes 0}$ is the unit of the algebra $(T^c(sA), \mu)$; compare \cite{BAU} and \cite[Subsection~1.4]{LoRo2006}.

More precisely, for a $B_{\infty}$-algebra $(A, m_n; \mu_{p, q})$, we may define two families of graded maps $M_n$ and $M_{p, q}$ on $sA$ via the following two commutative diagrams:
		\begin{equation*}\label{} \xymatrix@C=4pc{
A^{\otimes n}\ar[d]_-{s^{\otimes n}} \ar[r]^-{m_n} & A\ar[d]_{s}\\
(sA)^{\otimes n} \ar[r]^-{M_n}&sA
} \end{equation*} and
\begin{equation*}\label{} \xymatrix@C=4pc{
A^{\otimes p}\bigotimes A^{\otimes q}\ar[d]_-{s^{\otimes p+q}} \ar[r]^-{\mu_{p, q}} & A\ar[d]_{s}\\
(sA)^{\otimes p}\bigotimes (sA)^{\otimes q} \ar[r]^-{M_{p, q}}&sA,
} \end{equation*} where $s\colon A\xrightarrow{} sA$ denotes the canonical map $a\mapsto sa$ of degree $-1$. The maps $M_n$ and $M_{p, q}$ induce, respectively,  the differential $D$ and the multiplication $\mu$ on $T^c(sA)$; compare Remark~\ref{rem:A-inf} and \cite[Proposition~1.6]{LoRo2006}.

Accordingly, an $A_\infty$-morphism between two $B_\infty$-algebras is a $B_\infty$-morphism if and only if its induced dg coalgebra homomorphism is a dg bialgebra homomorphism.
	\end{rem}

\subsection{A duality theorem on $B_\infty$-algebras}
We introduce the opposite $B_\infty$-algebra and the transpose $B_\infty$-algebra of any given $B_\infty$-algebra. We show that there is a natural $B_\infty$-isomorphism between them.
\begin{defn}
\label{defnopposite}
The {\it opposite $B_\infty$-algebra} of a $B_{\infty}$-algebra $(A, m_n; \mu_{p, q})$ is defined to be the $B_{\infty}$-algebra  $(A, m_n; \mu_{p, q}^{\rm opp})$, where
$$\mu_{p, q}^{\rm opp}(a_1\otimes \dotsb \otimes a_p \bigotimes b_1 \otimes \dotsb\otimes b_q) = (-1)^{pq+\epsilon} \mu_{q, p}(b_1 \otimes \dotsb \otimes b_q \bigotimes a_1 \otimes \dotsb \otimes a_p).
$$
Here we denote
$\epsilon:=(|b_1|+\cdots+|b_q|)(|a_1|+\cdots+|a_p|).$
\end{defn}

\begin{rem}\label{Rem:oppositeB}
Let $(T^c(sA), \Delta, D, \mu)$ be the dg bialgebra associated to $(A, m_n; \mu_{p, q})$.  Then the corresponding dg bialgebra to $(A, m_n; \mu_{p, q}^{\rm opp})$ is given by $(T^c(sA), \Delta, D, \mu^{\rm opp})$ where
$
\mu^{\rm opp} = \mu\circ \tau$. Here $\tau$ is the swapping map
\begin{align}\label{algin:swappingmap}
\tau\colon T^c(sA) \otimes T^c(sA) \xrightarrow{\simeq} T^c(sA) \otimes T^c(sA), \quad x \otimes y \longmapsto (-1)^{|x||y|} y \otimes x.
\end{align}

We will simply denote $(A, m_n; \mu_{p, q}^{\rm opp})$ by $A^{\rm opp}$ when no confusion can arise. By definition,  $A^{\rm opp}$ and $A$ have the same $A_{\infty}$-algebra structure. Note that $(A^{\rm opp})^{\rm opp} = A$ as $B_{\infty}$-algebras.
\end{rem}

We also need the following notion of the {\it transpose $B_{\infty}$-algebra} $A^{\rm tr}$ of $A$. We mention that the transpose arises naturally when one considers the opposite algebra; see the proofs of Propositions~\ref{lemma-CL} and \ref{lemma-CL1}.

\begin{defn}
\label{defn-transposeB}
The {\it transpose  $B_{\infty}$-algebra}    of a $B_\infty$-algebra $(A, m_n; \mu_{p, q})$ is defined  to be the $B_{\infty}$-algebra   $A^{\rm tr}= (A, m_n^{\rm tr}; \mu_{p, q}^{\rm tr})$, where
\begin{align*}
\label{B-}
 m_n^{\rm{tr}}(a_1 \otimes a_2\otimes \dotsb \otimes a_n) & := (-1)^{\epsilon_n} m_n(a_n \otimes a_{n-1}\otimes\dotsb \otimes a_1),\nonumber\\
\mu^{\rm tr}_{p, q}(a_1\otimes \dotsb\otimes a_p\bigotimes b_1\otimes \dotsb \otimes b_q) &:= (-1)^{\epsilon} \mu_{p,q}(a_p \otimes \dotsb \otimes a_1\bigotimes b_q \otimes \dotsb \otimes b_1),
\end{align*}
for any $a_1, \dotsb, a_p, b_1, \dotsc, b_q \in A$. Here
\begin{align*}
\epsilon_n &= \frac{(n-1)(n-2)}{2} + \sum_{j=1}^{n-1} |a_j|(|a_{j+1}| + \dotsb + |a_{n}|)\\
\epsilon & = 1+ \frac{p(p+1)}{2} +\frac{q(q+1)}{2} + \sum_{j=1}^{p-1} |a_j|(|a_{j+1}| + \dotsb + |a_{p}|) + \sum_{j=1}^{q-1} |b_j|(|b_{j+1}|+ \dotsb + |b_q|).
\end{align*} \end{defn}

\begin{rem}\label{rem:antipode}
Let $(T^c(sA), \Delta, D, \mu)$ be the dg bialgebra associated to $(A, m_n; \mu_{p, q})$.  Then the corresponding dg bialgebra to $A^{\rm tr}$ is  $(T^c(sA), \Delta, D^{\rm tr}, \mu^{\rm tr})$. By Remark~\ref{remark-B-infinity}, $D^{\rm tr}$ and $\mu^{\rm tr}$ are uniquely determined by
\begin{align*}
M_n^{\rm tr}(sa_1\otimes \dotsb \otimes sa_n)& =(-1)^{n-1+\epsilon} M_n(sa_n\otimes \dotsb \otimes sa_1)\\
M_{p, q}^{\rm tr} (sa_1\otimes \dotsb \otimes sa_p \bigotimes sb_1 \otimes \dotsb \otimes sb_q)&  = (-1)^{p+q-1+\epsilon'} M_{p, q}(sa_p \otimes \dotsb \otimes sa_1 \bigotimes sb_q \otimes \dotsb \otimes sb_1)
\end{align*}
where
\begin{align*}
\epsilon &= \sum_{i=1}^{n-1} (|a_i|-1)((|a_{i+1}|-1) + \dotsb + (|a_n|-1))\\
\epsilon' &= \sum_{i=1}^{p-1} (|a_i|-1)((|a_{i+1}|-1) + \dotsb + (|a_p|-1)) + \sum_{j=1}^{q-1}(|b_j|-1)((|b_{j+1}|-1)+\dotsb +( |b_q|-1)).
\end{align*}
Here, the maps $M_n$ and $M_{p, q}$ are given in Remark~\ref{remark-B-infinity}, and are uniquely determined by $D$ and $\mu$, respectively.
\end{rem}

\begin{lem}\label{lem:transposeB}
There is an isomorphism of dg bialgebras
$$
O \colon (T^c(sA), \Delta^{\rm op}, D, \mu)  \xrightarrow{\simeq} (T^c(sA), \Delta, D^{\rm tr}, \mu^{\rm tr})
$$
where $\Delta^{\rm op}  = \tau\circ  \Delta$ and $\tau$ is given in \eqref{algin:swappingmap}.
\end{lem}
\begin{proof}
We first define
$$
O(sa_1 \otimes \dotsb \otimes sa_n) = (-1)^{n+\epsilon} sa_n \otimes \dotsb \otimes sa_1
$$
where $\epsilon = \sum_{j=1}^{n-1}(|a_j|-1)((|a_{j+1}|-1)+\dotsb +(|a_n|-1))$. Clearly, we have $O \circ O = \mathbf 1_{T^c(sA)}$.

We observe that $O$ is a morphism of coalgebras, that is,  $(O\otimes O)\circ \Delta^{\rm op} = \Delta \circ O$. Denote by ${\rm pr}\colon T^c(sA)\rightarrow sA$ the projection. By the cofreeness of $T^c(sA)$, to prove $O \circ D = D^{\rm tr} \circ O$, it suffices to prove
$${\rm pr}\circ O \circ D = {\rm pr}\circ D^{\rm tr} \circ O.$$
The above identity is equivalent to $O_1\circ M_n= M_n^{\rm tr} \circ O_n$ for each $n\geq 1$. Here, $O_n = O|_{(sA)^{\otimes n}}$. But the last identity follows from the very definition of $M_n^{\rm tr}$.

We observe that $O_1\circ M_{p, q} =  M_{p, q}^{\rm tr} \circ (O_p\otimes O_q)$. By a similar argument as above, we infer that $ O \circ \mu = \mu^{\rm tr} \circ (O \otimes O)$.
\end{proof}

Let us prove the following duality theorem. For applications, we refer to Propositions~\ref{lemma-CL} and \ref{lemma-CL1} below.

\begin{thm}\label{thm:dualityB}
Let $(A, m_n;\mu_{p, q})$ be a $B_\infty$-algebra. Then there is a natural $B_\infty$-isomorphism between the opposite $B_\infty$-algebra $A^{\rm opp}$ and the transpose $B_\infty$-algebra $A^{\rm tr}$.
\end{thm}

\begin{proof}
Denote by $(T^c(sA), \Delta, D, \mu)$ the corresponding dg bialgebra of the given $B_\infty$-algebra $(A, m_n;\mu_{p, q})$. By Remark~\ref{Rem:oppositeB}, the opposite $B_\infty$-algebra $A^{\rm opp}$ corresponds to the dg bialgebra $(T^c(sA), \Delta, D, \mu^{\rm opp})$.  By Lemma~\ref{lem:transposeB} the transpose $B_\infty$-algebra $A^{\rm tr}$ corresponds to the dg bialgebra $(T^c(sA), \Delta^{\rm op}, D, \mu)$.

Recall that the category of $B_\infty$-algebras is equivalent to the category of dg cofree bialgebras; compare Remark~\ref{remark-B-infinity}.  Therefore, it suffices to prove that there is a natural isomorphism of dg bialgebras between $(T^c(sA), \Delta, D, \mu^{\rm opp})$ and $(T^c(sA), \Delta^{\rm op}, D, \mu)$.

It is well known that as a connected bialgebra, the bialgebra $(T^c(sA), \Delta,\mu)$ admits an \emph{antipode} $S$; see \cite[Subsection~1.2]{LoRo2006}. Moreover, the antipode $S$ is bijective by a classical result \cite[Proposition~1.2]{Masu}.

Indeed, the antipode $S$ is given inductively  as follows: $S(sa)=-sa$ for $sa\in sA$; for $n\geq 2$ and $x\in (sA)^{\otimes n}$, we use Sweedler's notation to  write
$$\Delta(x)=1\bigotimes x + x\bigotimes 1 + \sum x'\bigotimes x''$$
 with $x'$ and $x''$ having smaller tensor-length, and then set
\begin{equation}
\label{equ-indu}
S(x)=-x-\sum \mu(x'\bigotimes S( x'')).
\end{equation}
A similar inductive formula holds for the inverse of $S$.

The antipode $S$ gives automatically a bialgebra isomorphism
$$
S \colon (T^c(sA), \Delta, \mu^{\rm opp}) \longrightarrow (T^c(sA), \Delta^{\rm op}, \mu).
$$
To complete the proof, it remains to show $S\circ D=D\circ S$. It is clear that $SD(sa)=DS(sa)$ for $sa\in sA$, as both equal $-sm_1(a)$. We prove  the general case by induction on the tensor-length of elements in $T^c(sA)$.

For $n\geq 2$ and $x\in (sA)^{\otimes n}$, we use  (\ref{equ-indu}) and the fact that $D$ is a derivation with respect to $\mu$, and obtain the first equality of the following identity.
\begin{align*}
DS(x)&=-D(x)-\sum \mu(D(x')\bigotimes S(x''))-\sum (-1)^{|x'|}\mu(x'\otimes DS(x''))\\
     &=-D(x)-\sum \mu(D(x')\bigotimes S(x''))-\sum (-1)^{|x'|}\mu(x'\otimes SD(x''))\\
     &=SD(x)
\end{align*}
Here, the second equality uses the induction hypothesis. For the last one, we use the fact that $D$ is a coderivation with respect to $\Delta$. Then $\Delta D(x)=1\bigotimes D(x) + D(x)\bigotimes 1 + \sum D(x')\bigotimes x''+ \sum (-1)^{|x'|} x'\bigotimes D(x'')$. Applying  (\ref{equ-indu}) to $D(x)$, we infer the last equality. This completes the proof. \end{proof}

\begin{rem}\label{rem:antipode-brace}
To obtain an explicit $B_\infty$-isomorphism from $A^{\rm opp}$ to $A^{\rm tr}$, one has to compute ${\rm pr}\circ O\circ S$, where $O$ is the isomorphism in Lemma~\ref{lem:transposeB} and ${\rm pr}\colon T^c(sA)\rightarrow sA$ is the projection. By the inductive formula (\ref{equ-indu}), it seems possible to describe the antipode $S$ explicitly. However, if the multiplication $\mu$ is arbitrary, we do not have a closed formula for $S$.

We assume that  $\mu_{p, q}=0$ for  any $p>1$ (for example, the condition holds for any brace $B_\infty$-algebra; see Subsection~\ref{subsec:brace}). We might compute explicitly $S$ and then ${\rm pr}\circ O\circ S$. It turns out that  the above $B_\infty$-isomorphism is given by
$$
\Theta_k \colon (sA)^{\otimes k} \longrightarrow sA, \quad k \geq 1,
$$
where  $\Theta_k = \mathbf 1$ and for $k >1$ \begin{align}
\label{recursive}
\Theta_{k} = \sum_{(i_1, \dotsc, i_r) \in \mathcal I_{k-1}}M^{\rm tr}_{1,r} \circ(\mathbf{1} \otimes \Theta_{i_1}\otimes  \Theta_{i_2}\otimes \cdots \otimes  \Theta_{i_r}).
\end{align}
Here, $M_{1, r}^{\rm tr}$ is given in Remark \ref{rem:antipode} and the sum on the right hand side is taken over the set
$$
\mathcal I_{k-1} = \{(i_1, i_2,  \dotsc, i_r)\mid r\geq 1\ \text{and} \    i_1, i_2, \dotsc, i_r \geq 1\ \ \text{such that} \ \  i_1+ i_2 + \dotsb +  i_r = k-1\}.
$$
For instance, we have $\Theta_2 =M^{\rm tr}_{1,1}, \ \Theta_3 =M^{\rm tr}_{1,2}+M^{\rm tr}_{1,1}\circ (\mathbf 1 \otimes M^{\rm tr}_{1,1})$ and
\begin{align*}
\Theta_4 ={} &M^{\rm tr}_{1, 3}+M^{\rm tr}_{1, 2}\circ (\mathbf 1 \otimes M^{\rm tr}_{1, 1} \otimes \mathbf 1) + M^{\rm tr}_{1, 2}\circ (\mathbf 1 \otimes \mathbf 1 \otimes M^{\rm tr}_{1, 1}) \\
& + M^{\rm tr}_{1, 1}\circ (\mathbf 1 \otimes M^{\rm tr}_{1, 2})+ M^{\rm tr}_{1, 1}\circ \big(\mathbf 1 \otimes M^{\rm tr}_{1, 1}\circ(\mathbf 1 \otimes M^{\rm tr}_{1, 1})\big).
\end{align*}

The construction of the maps $(\Theta_k)_{k\geq 1}$ might be also obtained  from the Kontsevich-Soibelman minimal operad $\mathcal{M}$ introduced in  \cite[Section~5]{KoSo}. Roughly speaking, the $n$-th space $\mathcal M(n)$ for $n \geq 1$ is the  $\mathbb k$-linear space spanned by planar rooted trees with $n$-vertices labelled by $1, 2, \dotsc, n$ and some (possibly zero) number of unlabelled vertices. Note that  the summands of $\Theta_k$ correspond bijectively to those trees $T$ without unlabelled vertices in $\mathcal M(k)$  whose vertices are labelled in counterclockwise order. Since such labelling is unique, the number of summands in $\Theta_k$ equals the Catalan number $C_{k-1}=\frac{1}{k} {{2k-2} \choose{k-1}}$. For instance, the three trees in Figure \ref{minimal-operad}  correspond respectively to the following three summands in $\Theta_6$:
\begin{align*}
M_{1, 3}^{\rm tr} \circ ( \mathbf 1 \otimes M_{1, 1}^{\rm tr} \otimes \mathbf 1 \otimes M_{1, 1}^{\rm tr})\\
M^{\rm tr}_{1, 2}\circ  (\mathbf 1 \otimes M^{\rm tr}_{1, 2} \otimes M_{1, 1}) \\
M_{1, 2}^{\rm tr}\circ (\mathbf 1 \otimes \mathbf 1 \otimes M^{\rm tr}_{1, 2}(\mathbf 1 \otimes M^{\rm tr}_{1, 1}\otimes \mathbf 1))
\end{align*}
\begin{figure}[h]
\centering{
 \includegraphics[height=35mm]{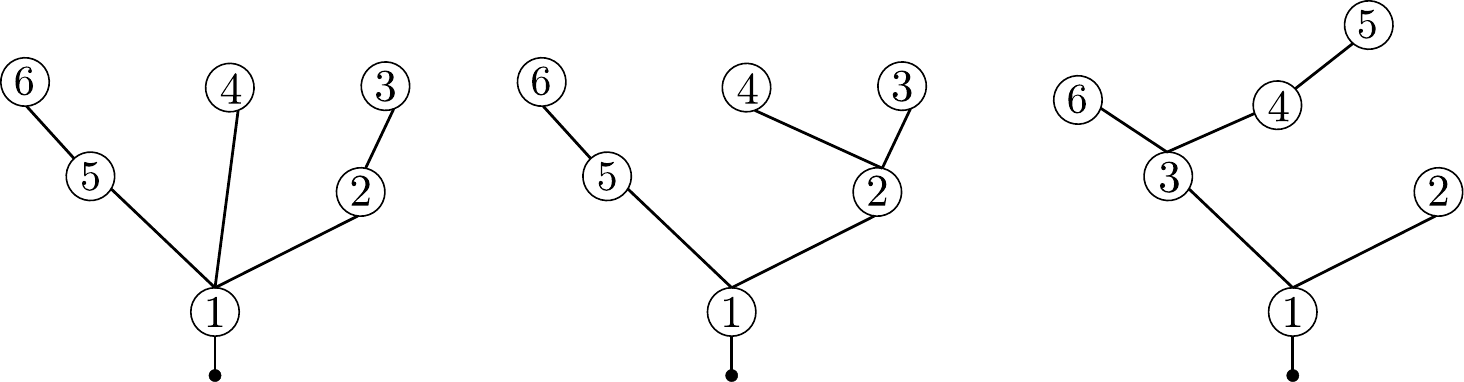}
 \caption{Three of the summands in $\Theta_6$}
  \label{minimal-operad}
  }
\end{figure}
\end{rem}

\subsection{Brace $B_\infty$-algebras}\label{subsec:brace}

The following new terminology will be convenient for us.

\begin{defn}\label{defn:special}
A $B_{\infty}$-algebra $(A, m_n; \mu_{p, q})$ is called a {\it brace $B_\infty$-algebra},  provided that $m_n=0$ for $n>2$  and that $\mu_{p, q}=0$ for $p> 1$. \hfill $\square$
\end{defn}

We mention that a brace $B_\infty$-algebra is called  a {\it homotopy $G$-algebra} in \cite{GV} or  a {\it Gerstenhaber-Voronov algebra} in \cite{LoRo, GTV, BI}. The notion is introduced mainly as an algebraic model to  unify the rich algebraic structures on the Hochschild cochain complex of an algebra.

The underlying $A_\infty$-algebra structure of a brace $B_{\infty}$-algebra  is just a dg algebra. For a brace $B_\infty$-algebra, we usually use the following notation, called the {\it brace operation} \cite{GV, Vor}:
\begin{align*}
a \{b_1, \dotsc, b_p\}:=(-1)^{p|a|+(p-1)|b_1|+(p-2)|b_2|+\cdots+|b_{p-1}|} \mu_{1, p}(a\bigotimes b_1\otimes \cdots\otimes b_p)
\end{align*}
for any $a, b_1, \dotsc, b_p\in A$. In particular, $a\{\emptyset\}=\mu_{1, 0}(a\bigotimes 1)=a $ by (\ref{equation-B01}). We will abbreviate $a\{b_1, \dotsc, b_p\}$  and $a'\{c_1, \dotsc, c_q\}$ as $a\{b_{1, p}\}$ and $a'\{c_{1, q}\}$, respectively.

The $B_\infty$-algebras occurring in this paper  are all brace $B_{\infty}$-algebras; see Subsections~\ref{subsection-dg-HH} and \ref{subsec:sHcc}. In the following remark, we describe the axioms for brace $B_\infty$-algebras explicitly, which will be useful later.

\begin{rem}
\label{rem-specialB} Let $(A, m_n; \mu_{p,q})$ be a brace  $B_{\infty}$-algebra.
  Then the above $B_{\infty}$-relation (\ref{equation-B02}) is simplified as (\ref{higher-pre-Jacobi}) below,  and the $B_{\infty}$-relation (\ref{equation-B03}) splits into (\ref{distributivity}) and (\ref{higher-homotopy}) below (corresponding to the cases $k=2$ and $k=1$, respectively).
\begin{enumerate}
\item \label{higher-pre-Jacobi} The higher pre-Jacobi identity:
\begin{equation*}
\begin{split}
& (a\{b_{1, p}\}) \{c_{1, q}\}\\
={}&\sum (-1)^{\epsilon}a\{c_{1, i_1}, b_1\{c_{i_1+1, i_1+l_1}\}, c_{i_1+l_1+1, i_2}, b_2\{c_{i_2+1, i_2+l_2}\}, \dotsc, c_{i_p}, b_p\{c_{i_p+1, i_p+l_p}\}, c_{i_p+l_p+1, q}\},
\end{split}
\end{equation*}
where the sum is taken over all  sequences of nonnegative integers $(i_1, \dotsc, i_p; l_1, \dotsc, l_p)$ such that $$0\leq i_1\leq i_1+l_1\leq i_2 \leq i_2+l_2\leq i_3 \leq \cdots \leq i_p+l_p\leq q$$ and $$\epsilon=\sum\limits_{l=1}^p \left(\left(|b_l|-1\right)\sum\limits_{j=1}^{i_l}\left(|c_j|-1\right)\right).$$
\item \label{distributivity} The distributivity:
\begin{align*}
m_2(a_1\otimes a_2)\{b_{1, q}\}=\sum_{j=0}^q(-1)^{|a_2|\sum\limits_{i=1}^j (|b_i|-1)} m_2((a_1\{b_{1, j}\})\otimes (a_2\{b_{j+1, q}\})).
\end{align*}

\item \label{higher-homotopy} The higher homotopy:
\begin{align*}
&  m_1(a\{b_{1, p}\}  )-(-1)^{|a|(|b_1|-1)} m_2(b_1\otimes (a \{b_{2, p}\}))+(-1)^{\epsilon_{p-1}} m_2((a\{b_{1, p-1}\})\otimes b_p) \\
={} &  m_1(a) \{b_{1, p}\}-\sum_{i=0}^{p -1} (-1)^{\epsilon_i}a\{b_{1, i}, m_1(b_{i+1}),b_{i+2, p}\}  + \sum_{i=0}^{p -2} (-1)^{\epsilon_{i+1}}a \{b_{1, i}, m_2(b_{i+1,i+2}), b_{i+3, p}\},
\end{align*}
where $\epsilon_0 = |a|$ and $\epsilon_i=|a|+\sum\limits_{j=1}^i(|b_j|-1)$ for $i \geq 1$.
\end{enumerate}
\end{rem}

\begin{rem}\label{remark-B}
The opposite $B_\infty$-algebra $(A, m_n; \mu_{p, q}^{\rm opp})$ of a brace $B_{\infty}$-algebra $A$  is given by
$$\mu^{\rm opp}_{0, 1}=\mu^{\rm opp}_{1, 0}=\mathbf{1}_A, \ \mu^{\rm opp}_{p,1} (b_1\otimes \cdots\otimes b_p\bigotimes  a)=(-1)^{\epsilon}\mu_{1, p}(a\bigotimes b_1\otimes \cdots\otimes b_p),$$ and $\mu^{\rm opp}_{p, q}=0$ for other cases,
where $\epsilon=|a|(|b_1|+\cdots+|b_p|)+p$. In general, the opposite $B_\infty$-algebra $A^{\rm opp}$ is not a brace $B_{\infty}$-algebra.

The transpose $B_{\infty}$-algebra $(A^{\rm tr}, m_1^{\rm tr}, m_2^{\rm tr}; -\{-, \dotsc, -\}^{\rm tr})$ of a brace $B_\infty$-algebra $A$ is also a brace $B_{\infty}$-algebra given by
\begin{align}
\label{tr1}
m_1^{\rm tr} = m_1,\quad m_2^{\rm tr}(a \otimes b) =(-1)^{|a||b|} m_2(b \otimes a),\nonumber\\
a\{b_1, b_2, \dotsc, b_k\}^{\rm tr} =(-1)^{\epsilon'} a \{b_k, b_{k-1}, \dotsc, b_1\}
\end{align}
where $\epsilon'=k+\sum_{j=1}^{k-1} (|b_j|-1)\big((|b_{j+1}|-1)+ (|b_{j+2}|-1) + \dotsb + (|b_k|-1)\big)$.  As dg algebras, $(A^{\rm tr}, m_1^{\rm tr}, m_2^{\rm tr})$ coincides with the (usual) opposite dg algebra $A^{\rm op}$ of $A$.
\end{rem}

The following observation follows directly from    Definition~\ref{definition-B-infinity}.

\begin{lem}\label{lemma-infinity-morphism1}
Let $A$ and $A'$ be two brace  $B_{\infty}$-algebras. A homomorphism of dg algebras $f\colon  (A, m_1, m_2)\rightarrow (A', m_1', m_2')$ becomes a strict $B_{\infty}$-morphism if
and only if $f$ is compatible with $-\{-, \cdots,-\}_A$ and $-\{-, \dotsc, -\}_{A'}$, namely
$$f(a\{b_1, \dotsc, b_p\}_A)=f(a)\{ f(b_1),\dotsc,  f(b_p)\}_{A'}$$
for any $p\geq 1$ and $a, b_1, \dotsc, b_p \in A$. \hfill $\square$
\end{lem}

Let $f=(f_n)_{\geq 1}\colon A\to A'$ be an $A_{\infty}$-morphism. We define $\widetilde{f}_n\colon (sA)^{\otimes n} \to A'$ by the following commutative diagram.
\begin{equation*}\label{equation:shift} \xymatrix@C=4pc{
A^{\otimes n}\ar[d]_-{s^{\otimes n}} \ar[r]^-{f_n} & A'\\
(sA)^{\otimes n} \ar[ru]_-{\widetilde{f}_n}
} \end{equation*}
Namely, we have
\begin{equation}
\label{formula-sec}
\widetilde{f}_n(sa_1\otimes sa_2\otimes \cdots\otimes sa_n)=(-1)^{\sum_{i=1}^n (n-i)|a_i|}
f_n(a_1\otimes a_2\otimes \cdots\otimes a_n).\end{equation}
 The advantage of using $(\widetilde{f}_n)_{n\geq 1}$ in Lemma \ref{lemma-infinity-morphism} below,  instead of  using $(f_n)_{n\geq 1}$,  is that the signs become much simpler.

The following lemma will be used in the proofs of Theorem~\ref{prop-B4}. We will abbreviate $sa_1\otimes \cdots \otimes  sa_n$ as $sa_{1, n}$, and $a\{b_1, \dotsc, b_m\}$ as $a\{b_{1, m}\}$.

\begin{lem}\label{lemma-infinity-morphism}
Let $A$ and $A'$ be two brace $B_{\infty}$-algebras. Assume that $(f_n)_{n\geq 1}\colon  A \rightarrow A'$ is an $A_\infty$-morphism. Then  $(f_n)_{n\geq 1} \colon   A\rightarrow  A'^{\rm opp}$ is a $B_{\infty}$-morphism  if and only if  the following identities hold for any  $p, q\geq 0$ and  $a_1, \dotsc, a_p, b_1, \dotsc, b_q\in A$
\begin{align}\label{equation-infinity-morphism}
&\sum_{r\geq 1} \sum_{i_1+\dotsb+i_r=p} (-1)^{\epsilon}\widetilde{f}_q(sb_{1, q})\{\widetilde{f}_{i_1}(sa_{1, i_1}), \widetilde{f}_{i_2}(sa_{i_1+1, i_1+i_2}), \dotsb, \widetilde{f}_{i_r}(sa_{i_1+\dotsb+i_{r-1}+1, p})\}_{A'}\nonumber\\
={}&\sum (-1)^{\eta}
\widetilde{f}_t(sb_{1, j_1}\otimes s(a_1\{b_{j_1+1, j_1+l_1}\}_A)\otimes  sb_{j_1+l_1+1, j_2}\otimes s(a_2\{b_{j_2+1, j_2+l_2}\}_A) \otimes\nonumber\\
& \quad\quad \dotsb\otimes sb_{j_p} \otimes s(a_{p}\{b_{j_p+1, j_p+l_p}\}_A)\otimes sb_{j_p+l_p+1, q}).
\end{align}
Here, the maps $\widetilde{f}_q$ and $\widetilde{f}_t$ are defined in  (\ref{formula-sec});  the sum on the right hand side is taken over all the  sequences of nonnegative integers  $(j_1, \dotsc, j_{p}; l_1, \dotsc, l_p)$ such that  $$0 \leq j_1 \leq j_1+l_1 \leq j_2 \leq j_2+l_2\leq \dotsb \leq j_p \leq j_p+l_p \leq q,$$ and $t=p+q-l_1-\cdots-l_p$;   the signs are determined by the identities
\begin{align*}
\epsilon &= (|a_1|+\cdots+|a_p|-p)(|b_1|+\cdots+|b_q|-q),\; and \\
\eta& =\sum_{i=1}^p (|a_i|-1)\big((|b_1|-1)+(|b_2|-1) + \dotsb +(|b_{j_i}|-1)\big).
\end{align*}
\end{lem}

\begin{proof}
Since $\mu_{s, r}' = 0$ for $s > 1$ and $(\mu'_{r, 1})^{\rm opp}=(-1)^{r}\mu'_{1, r}\circ \tau_{r, 1}$, the identity (\ref{equation-B-new}) becomes
 \begin{align}
 \label{202003}
&\sum_{r \geq 1}\sum_{\substack{i_1+ i_2+ \dotsb + i_r=p}}(-1)^{\epsilon_1 + r} \mu_{1, r}' \circ \tau_{r, 1}(f_{i_1}\smallotimes  \cdots\smallotimes f_{i_r}\otimes f_{q})  \\
={}&\sum_{\substack{t\geq 1\\ m_1+\cdots+m_t=p\\ n_1+\cdots+n_t=q}}(-1)^{\eta_1}f_t\circ (\mu_{m_1,n_1}\smallotimes \cdots \smallotimes \mu_{m_t, n_t})\circ \tau_{(m_1, \dotsc, m_t; n_1, \dotsc, n_t)} \nonumber \\
={}&\sum (-1)^{\eta_1}  f_t\circ  (\mu_{0, 1}^{\smallotimes j_1} \smallotimes  \mu_{1, l_1} \smallotimes \mu_{0, 1}^{\smallotimes j_2-j_1-l_1} \smallotimes \mu_{1, l_2} \smallotimes \!\dotsb\! \smallotimes\mu_{1, l_p} \smallotimes \mu_{0, 1}^{\smallotimes q-l_p-j_p})\circ \tau_{(m_1, \dotsc, m_t; n_1, \dotsc, n_t)},\nonumber
\end{align}
where the sum on the right hand side of the last identity is taken over all the  sequences of nonnegative integers  $(j_1, \dotsc, j_{p}; l_1, \dotsc, l_p)$ such that  $$0 \leq j_1 \leq j_1+l_1 \leq j_2 \leq j_2+l_2\leq \dotsb \leq j_p \leq j_p+l_p \leq q,$$ and $t=p+q-l_1-\cdots-l_p$.    The signs are determined by
\begin{align*}
\epsilon_1 & =\sum_{k=1}^r (i_k-1)(r+1-k), \; \mbox{\rm and}\\
\eta_1 & =\sum_{k=1}^t n_k(p-m_1-\cdots-m_k)+\sum_{k=1}^t(m_k+n_k-1)(t-k)\\
&=\sum_{i=1}^p j_i+\sum_{i=1}^p l_i (t-j_i-l_1-\dotsb- l_{i-1}+i).
\end{align*}
We apply \eqref{202003} to the element $(-1)^{\sum_{i=1}^p|a_i|(p+q-i)+\sum_{j=1}^q |b_j|(q-j)}(a_1\otimes \cdots \otimes a_p\bigotimes b_1\otimes \cdots \otimes b_q)$, where the sign $(-1)^{\sum_{i=1}^p|a_i|(p+q-i)+\sum_{j=1}^q |b_j|(q-j)} $ is added just in order to simplify the sign computation.
Using  (\ref{formula-sec}), we obtain the required identity (\ref{equation-infinity-morphism}).
\end{proof}

\subsection{Gerstenhaber algebras}

In this subsection, we recall the well-known relationship between $B_{\infty}$-algebras and Gerstenhaber algebras.

\begin{defn}
A {\it Gerstenhaber algebra} is the triple $(G, -\cup-, [-,-])$, where   $G=\bigoplus_{n\in \mathbb Z} G^n$ is a graded $\mathbb k$-space equipped with two graded maps: a cup product
$$-\cup- \colon   G\otimes G \longrightarrow G$$
of degree zero, and a Lie bracket of degree $-1$
$$[-, -]\colon  G \otimes G\longrightarrow G$$
satisfying the following conditions:
\begin{enumerate}
\item $(G, -\cup-)$ is a graded commutative associative algebra;
\item $(G^{*+1}, [-, -])$ is a graded Lie algebra, that is
$$[\alpha, \beta]=-(-1)^{(|\alpha|-1)(|\beta|-1)}[\beta, \alpha]$$
and
\begin{align}\label{jacobiidentityg}(-1)^{(|\alpha|-1)(|\gamma|-1)}[[\alpha, \beta], \gamma]+(-1)^{(|\beta|-1)(|\alpha|-1)}[[\beta, \gamma], \alpha]
+(-1)^{(|\gamma|-1)(|\beta|-1)}[[\gamma, \alpha], \beta]=0;\end{align}
\item the operations $-\cup-$ and $[-, -]$ are compatible through the graded Leibniz rule
\begin{equation*}[\alpha, \beta\cup \gamma]=[\alpha, \beta]\cup \gamma+(-1)^{(|\alpha|-1)|\gamma|} \beta\cup[\alpha, \gamma].\tag*{$\square$}
\end{equation*}
\end{enumerate}
\end{defn}

The following well-known result is contained in \cite[Subsection 5.2]{GJ}.

\begin{lem}\label{lem:Ger}
Let $(A, m_n; \mu_{p, q})$ be a $B_{\infty}$-algebra. Then there is a natural Gerstenhaber algebra structure $(H^*(A, m_1), -\cup-, [-, -])$ on its cohomology, where the cup product $-\cup-$ and the Lie bracket $[-, -]$ of degree $-1$ are given by
\begin{equation*}
\begin{split}
\alpha\cup \beta={}&m_2(\alpha, \beta);\\
[\alpha, \beta]={}&(-1)^{|\alpha|} \mu_{1, 1}(\alpha,  \beta)-(-1)^{(|\alpha|-1)(|\beta|-1)+|\beta|} \mu_{1,1}(\beta, \alpha).
\end{split}
\end{equation*}
Moreover, a $B_{\infty}$-quasi-isomorphism between two $B_{\infty}$-algebras $A$ and $A'$ induces an isomorphism of Gerstenhaber algebras between $H^*(A)$ and $H^*(A')$. \hfill $\square$
\end{lem}

\begin{rem}\label{Rem:dgliealgebra}
A prior, the Lie bracket $[-, -]$ in Lemma~\ref{lem:Ger} is defined on $A$ at the cochain complex level. By definition, we have $[\alpha,\beta] = -(-1)^{(|\alpha|-1)(|\beta|-1)} [\beta, \alpha].$
It follows from \eqref{equation-B02} that $[-, -]$  satisfies the graded Jacobi identity \eqref{jacobiidentityg}. By \eqref{equation-B03} we have
$$
m_1([\alpha, \beta]) = [m_1(\alpha), \beta] +(-1)^{|\alpha|-1} [\alpha, m_1(\beta)],
$$
which ensures that $[-,-]$ descends to $H^*(A)$. That is,  $(A, m_1, [-, -])$ is a {\it dg Lie algebra of degree $-1$}; see \cite[Subsection~5.2]{GJ}.  By \eqref{equation-B-new} we see that a $B_{\infty}$-morphism induces a morphism of dg Lie algebras between the associated dg Lie algebras.

We mention that the associated dg Lie algebras to $B_\infty$-algebras play a crucial role in deformation theory; see e.g. \cite{LoVa}.
\end{rem}

\section{The Hochschild cochain complexes}
\label{section6}
In this section, we  recall basic results on Hochschild cochain complexes of dg categories and  (normalized) relative bar resolutions of dg algebras.

 \subsection{The Hochschild cochain complex of a dg category} \label{subsection-dg-HH}

 Recall that for  a cochain complex $(V, d_V)$, we denote by $sV$ the $1$-shifted complex. For a homogeneous element $v\in V$, the degree of the corresponding element $sv\in sV$ is given by $|sv|=|v|-1$ and $d_{sV}(sv)=-sd_V(v)$. Indeed, we have $sV=\Sigma(V)$, where $\Sigma$ is the suspension functor.

Let $\mathcal A$ be a small dg category over $\mathbb{k}$. The  {\it Hochschild cochain complex} of $\mathcal A$ is the complex
{\small $$C^*(\mathcal A, \mathcal A)=\prod_{n\geq 0} \prod_{A_0, \dotsc, A_n\in {\rm obj}(\mathcal A)} {\rm Hom}( s\mathcal A(A_{n-1}, A_n)\otimes s\mathcal A(A_{n-2}, A_{n-1}) \otimes \cdots \otimes s\mathcal A(A_0, A_1), \mathcal A(A_0, A_n))$$}
 with differential $\delta=\delta_{in}+\delta_{ex}$ defined as follows. For any $\varphi\in {\rm Hom}( s\mathcal A(A_{n-1}, A_n)\smallotimes \dotsb \smallotimes s\mathcal A(A_0, A_1), \mathcal A(A_0, A_n))$ the {\it internal differential} $\delta_{in}$ is
  \begin{eqnarray*}
\delta_{in} (\varphi)(sa_{1, n})=d_\mathcal{A}\varphi(sa_{1, n})+\sum_{i=1}^n(-1)^{\epsilon_i}\varphi(sa_{1, i-1}\otimes sd_\mathcal{A}(a_i)\otimes sa_{i+1, n})
\end{eqnarray*}
and the {\it external differential} is
\begin{equation*}
\begin{split}
\delta_{ex}(\varphi)(sa_{1, n+1})={}&-(-1)^{(|a_1|-1)|\varphi|} a_1\circ \varphi(sa_{2, n+1})+(-1)^{\epsilon_{n+1}}\varphi(sa_{1, n})\circ a_{n+1}\\
&-\sum_{i=2}^{n+1}(-1)^{\epsilon_i} \varphi(sa_{1, i-2}\otimes s(a_{i-1}\circ a_i)\otimes sa_{i+1, n+1}).\end{split}
\end{equation*}
Here, $\epsilon_i=|\varphi|+\sum_{j=1}^{i-1}(|a_j|-1)$ and $sa_{i, j}:=sa_i\otimes \cdots\otimes sa_j\in s\mathcal A(A_{n-i}, A_{n-i+1})\otimes \cdots \otimes s\mathcal A(A_{n-j}, A_{n-j+1})$ for $i\leq j$.

For any $n\geq 0$, we define the following subspace of $C^*(\mathcal{A}, \mathcal{A})$
{\small $$C^{*, n}(\mathcal A, \mathcal A):=\prod_{A_0, \dotsc, A_n\in {\rm obj}(\mathcal A)}{\rm Hom}( s\mathcal A(A_{n-1}, A_n)\otimes s\mathcal A(A_{n-2}, A_{n-1}) \otimes\cdots \otimes s\mathcal A(A_0, A_1), \mathcal A(A_0, A_n)).$$}
We observe  $C^{*,0}(\mathcal{A}, \mathcal{A})=\prod_{A_0\in {\rm obj}(\mathcal{A})}{\rm Hom}(\mathbb{k}, \mathcal{A}(A_0, A_0))\simeq \prod_{A_0\in {\rm obj}(\mathcal{A})} \mathcal{A}(A_0, A_0)$.

There are two basic operations on $C^*(\mathcal A, \mathcal A)$.   The first one is the  {\it cup product}
$$-\cup- \colon  C^*(\mathcal A, \mathcal A)\otimes C^*(\mathcal A, \mathcal A)\longrightarrow C^*(\mathcal A, \mathcal A).$$
For $\phi\in C^{*, p}(\mathcal A, \mathcal A)$ and $\varphi\in C^{*, q}(\mathcal A, \mathcal A)$, we define
$$\phi\cup \varphi(sa_{1, p+q})=(-1)^{\epsilon} \phi(sa_{1, p})\circ \varphi(sa_{p+1, p+q}),$$
where $\epsilon=(|a_1|+\cdots+|a_p|-p)|\varphi|$.

The second one is the {\it brace operation}
$$-\{-, \dotsc, -\}\colon   C^*(\mathcal A, \mathcal A)\otimes C^*(\mathcal A,\mathcal A)^{\otimes k}\longrightarrow C^*(\mathcal A, \mathcal A)$$
defined as follows. Let $k \geq 1$.  For
$\varphi\in C^{*, m}(\mathcal A, \mathcal A)$ and $\phi_i\in C^{*, n_i}(\mathcal A,\mathcal  A)\ (1\leq i\leq k)$,
\begin{equation}\label{equation:brace}
\varphi\{\phi_1, \dotsc, \phi_k\}=\sum  \varphi (\mathbf{1}^{\otimes i_1} \otimes (s\circ\phi_1) \otimes \mathbf{1}^{\otimes i_2} \otimes (s\circ\phi_2)\otimes \cdots\otimes  \mathbf{1}^{\otimes i_k}\otimes (s\circ\phi_k)\otimes \mathbf{1}^{\otimes i_{k+1}}),
\end{equation}
where the summation is taken over the set
$$ \{ (i_1, i_2, \dotsc, i_{k+1})\in \mathbb Z_{\geq 0}^{\times (k+1)} \mid i_1+i_2+\cdots+ i_{k+1}=m-k \}.$$
If the set  is empty, we define $\varphi\{\phi_1, \dotsc, \phi_k\}=0$. Here, $s\circ \phi_j$ means the composition of $\phi_j$ with the natural isomorphism $s\colon  \mathcal{A}(A, A')\rightarrow s\mathcal{A}(A, A')$
of degree $-1$ for suitable  $A, A'\in {\rm obj}(\mathcal{A})$. For $k = 0$,  we set $-\{\emptyset\} = \mathbf 1$. Observe that the cup product and the brace operation extend naturally to the whole space $C^*(\mathcal A, \mathcal A) = \prod_{n \geq 0} C^{*, n}(\mathcal A, \mathcal A)$.

It is well known that   $C^*(\mathcal A, \mathcal A)$ is a brace $B_{\infty}$-algebra with
$$ \quad m_1=\delta, \quad m_2=-\cup-, \quad \mbox{and}\quad m_i=0 \quad \mbox{for $i>2$};  $$
$$\mu_{0, 1}=\mu_{1, 0}=\mathbf{1}, \ \mu_{1, k} (\varphi, \phi_1, \dotsc, \phi_k)=\varphi\{\phi_1, \dotsc, \phi_k\}, \ \mbox{and}\ \mu_{p, q}=0 \ \mbox{otherwise}.$$
We refer to \cite[Subsections 5.1 and 5.2]{GJ} for  details.

The following useful lemma is contained in \cite[Theorem 4.6 b)]{Kel03}.

\begin{lem}\label{lem:C-quasi-iso}
Let $F\colon  \mathcal{A}\rightarrow \mathcal{B}$ be a quasi-equivalence between two small dg categories. Then there is an isomorphism
$$C^*(\mathcal A, \mathcal A)\longrightarrow C^*(\mathcal B, \mathcal B)$$
 in the homotopy category ${\rm Ho}(B_\infty)$ of $B_\infty$-algebras. \hfill $\square$
\end{lem}

Let $A$ be a dg algebra. We view $A$  as a dg category with a single object, still denoted by $A$. In particular, the Hochschild cochain complex $C^*(A, A)$ is defined as above.  The dg category $A$ might be identified as a full dg subcategory of $\mathbf{per}_{\rm dg}(A^{\rm op})$ by taking the right regular dg $A$-module $A_A$.  Then the following result follows from \cite[Theorem 4.6 c)]{Kel03}; compare \cite[Theorem~4.4.1]{LoVa}.

\begin{lem}\label{lem:C-dga}
Let $A$ be a dg algebra. Then the restriction map
$$C^*(\mathbf{ per}_{\rm dg}(A^{\rm op}), \mathbf{ per}_{\rm dg}(A^{\rm op})) \longrightarrow C^*(A, A)$$ is an isomorphism in $\mathrm{Ho}(B_{\infty})$. \hfill $\square$
\end{lem}

\subsection{The relative bar resolutions}\label{subsection-bar}

Let $A$ be a dg algebra with its differential $d_A$. Let $E=\bigoplus_{i\in \mathcal I} \mathbb ke_i \subseteq A^0 \subseteq A$ be a  semisimple subalgebra satisfying $d_A(e_i)=0$ and $e_ie_j=\delta_{i,j}e_i$ for any $i, j \in \mathcal I$.   Set $(sA)^{\otimes_E 0}=E$ and  $T_E(sA):=\bigoplus_{n \geq 0} (sA)^{\otimes_E n}$.

Recall from \cite{Abb} that the {\it $E$-relative bar resolution} of $A$  is the dg $A$-$A$-bimodule $$\Barr_E(A):=A\otimes_E T_E(sA)\otimes_E A $$
 with the differential $d=d_{in}+d_{ex}$, where $d_{in}$ is the \emph{internal differential} given by
\begin{equation*}
\begin{split}
d_{in}(a \otimes_E sa_{1, n} \otimes_E b)={}&d_A(a) \otimes_E sa_{1, n}\otimes_E b+(-1)^{\epsilon_{n+1}}a\otimes_E sa_{1, n}\otimes_E d_A(b)\\
&-\sum_{i=1}^n(-1)^{\epsilon_i}a\otimes_E sa_{1, i-1}\otimes_E sd_A(a_i)\otimes_E sa_{i+1, n}\otimes_E b
\end{split}
\end{equation*}
and $d_{ex}$ is the \emph{external differential} given by
\begin{equation*}
\begin{split}
 d_{ex}(a\otimes_E sa_{1, n}\otimes_E b)={}&(-1)^{\epsilon_1} aa_1\otimes_E sa_{2, n}\otimes_E b-(-1)^{\epsilon_n} a\otimes_E sa_{1, n-1}\otimes_E a_nb\\
 &+\sum_{i=2}^n(-1)^{\epsilon_i}a\otimes_E sa_{1, i-2}\otimes_E sa_{i-1}a_{i}\otimes_E sa_{i+1, n}\otimes_E b.
\end{split}
\end{equation*}
Here,  $\epsilon_i=|a|+\sum_{j=1}^{i-1} (|a_j|-1)$, and  for simplicity,  we denote $sa_i\otimes_E sa_{i+1}\otimes_E \cdots \otimes_E sa_j$ by $sa_{i, j}$ for $i<j$.    The degree of  $a\otimes_E sa_{1, n}\otimes_E b\in A\otimes_E (sA)^{\otimes_E n} \otimes_E A$ is $$|a|+\sum_{j=1}^n (|a_j|-1)+|b|.$$
The graded $A$-$A$-bimodule structure on $A\otimes_E (sA)^{\otimes_E n}\otimes_E A$ is given by the {\it outer} action
$$a(a_0\otimes_E sa_{1, n}\otimes_E a_{n+1})b:=aa_0\otimes_E sa_{1, n}\otimes_E a_{n+1}b.$$
There is a natural  morphism of dg $A$-$A$-bimodules $ \varepsilon \colon  \Barr_E(A)\rightarrow A$ given by the composition
\begin{align}\label{equ:bar}
\Barr_E(A)\longrightarrow{} A\otimes_E A\stackrel{\mu}\longrightarrow{} A,
\end{align}
where the first map is the canonical projection and $\mu$ is the multiplication  of $A$.  It is well known that  $\varepsilon$ is a quasi-isomorphism.

Set $\overline{A}$ to be the quotient dg $E$-$E$-bimodule $A/(E\cdot 1_A)$. We  have the notion of {\it normalized $E$-relative bar resolution} $\overline{\Barr}_E(A)$  of $A$.  By definition, it is the dg $A$-$A$-bimodule
$$\overline{\Barr}_E(A)=A\otimes_E T_E(s\overline{A})\otimes_E A$$
 with the induced differential from $\Barr(A)$.  It is also well known that the natural projection $\Barr_E(A)\rightarrow \overline{\Barr}_E(A)$ is a quasi-isomorphism.

Let $\mathbf{D}(A^e)$ be the derived category of dg $A$-$A$-bimodules. Let $M$ be a dg $A$-$A$-bimodule. The Hochschild cohomology group with coefficients in $M$ of degree $p$, denoted by $\HH^p(A, M)$,  is defined as $\Hom_{\mathbf{D}(A^e)}(A, \Sigma^p(M))$, where $\Sigma$ is the suspension functor in $\mathbf{D}(A^e)$.  Since $\Barr_E(A)$ is a dg-projective bimodule resolution of $A$, we obtain that
$$\HH^p(A, M)\cong H^p({\rm Hom}_{A\text{-}A}(\Barr_E(A), M), \delta),  \quad \text{for $p \in \mathbb Z$} $$
where $\delta(f):=d_M\circ f-(-1)^{|f|} f\circ d$.  We observe that there is a natural isomorphism, for each $i\geq 0$,
 \begin{align}
 \label{identification-bimodule}
 {\rm Hom}_{\text{$E$-$E$}}((sA)^{\otimes_E i}, M) \stackrel{\sim}\longrightarrow {\rm Hom}_{A\text{-}A}(A\otimes_E (sA)^{\otimes_E i}\otimes_E A, M)
 \end{align}
which sends $f$ to the map $a_0\otimes_E sa_{1, i}\otimes_E a_{i+1}\mapsto (-1)^{|a_0|\cdot |f|}a_0f(sa_{1, i})a_{i+1}$. It follows that
$$\HH^p(A, M)\cong H^p({\rm Hom}_{\text{$E$-$E$}}(T_E(sA), M), \delta=\delta_{in}+\delta_{ex}),$$
where the differentials  $\delta_{in}$ and $ \delta_{ex}$ are defined as in Subsection~\ref{subsection-dg-HH}.

We call $C_E^*(A, M):=({\rm Hom}_{\text{$E$-$E$}}(T_E(sA), M), \delta)$ the {\it $E$-relative Hochschild cochain complex} of $A$ with coefficients in $M$. In particular, $C_E^*(A, A)$ is called the {\it $E$-relative Hochschild cochain complex} of $A$. Similarly,  the {\it normalized $E$-relative  Hochschild cochain complex} $\overline{C}_E^*(A, M)$ is defined as ${\rm Hom}_{\text{$E$-$E$}}(T_E(s\overline A), M)$ with the induced differential. When $E=\mathbb k$, we simply write $C_\mathbb k^*(A, M)$ as $C^*(A, M)$ and write $\overline C_\mathbb k^*( A, M)$ as $\overline C^*( A, M)$.

When the dg algebra $A$ is viewed as a dg category $\mathcal{A}$ with a single object,   $C^*(\mathcal A, \mathcal A)$ coincides with $C^*(A, A)$. Thus, from Subsection~\ref{subsection-dg-HH}, $C^*(A, A)$ has a   $B_{\infty}$-algebra structure induced by the cup product $-\cup- $ and the brace operation $-\{-, \dotsc, -\}$ .

 We have the following commutative diagram of injections.
$$
\xymatrix{
\overline C_E^*(A, A)\ar@{^(->}[r] \ar@{^(->}[d]&  C_E^*(A, A)\ar@{^(->}[d]\\
\overline C^*( A, A)\ar@{^(->}[r] & C^*(A, A)
}
$$

  \begin{lem}\label{lemma6.1-split}
 The $B_{\infty}$-algebra structure on $C^*(A, A)$  restricts to the other three smaller complexes $C_E^*(A, A), \overline C_E^*( A, A)$ and $\overline C^*( A, A)$. In particular, the above injections are strict $B_{\infty}$-quasi-isomorphisms.
 \end{lem}

 \begin{proof}
 It is straightforward to check that the cup product and brace operation on $C^*(A, A)$ restrict to the subcomplexes $C_E^*(A, A), \overline C_E^*( A, A)$ and $\overline C^*( A, A)$. Moreover, the injections preserve the two operations.   Thus by Lemma \ref{lemma-infinity-morphism1},  the injections are strict $B_{\infty}$-morphisms. Clearly, the injections are quasi-isomorphisms since all the complexes compute  $\HH^*(A, A)$. This proves the lemma.
  \end{proof}

Let $A$ be a dg $\mathbb k$-algebra. Consider the $B_{\infty}$-algebra $(C^*(A, A), \delta, -\cup-; -\{-, \dotsc, -\})$ of Hochschild cochain complex; compare Subsection~\ref{subsection-dg-HH}. Let $A^{\op}$ be the opposite dg algebra of $A$.

\begin{prop}\label{lemma-CL}
There is a $B_\infty$-isomorphism from the opposite  $B_{\infty}$-algebra $C^*(A, A)^{\rm opp}$  to the $B_{\infty}$-algebra $C^*(A^{\op}, A^{\op})$.
\end{prop}
\begin{proof}
 Consider the \emph{swap isomorphism}
 \begin{align*}
T\colon   C^*(A, A) \longrightarrow C^*(A^{\op}, A^{\op})
\end{align*}
which sends $f\in C^*(A, A)$ to
$$T(f)(sa_1\otimes sa_2\otimes \cdots\otimes sa_m)=(-1)^{\epsilon}f(sa_m\otimes \cdots\otimes sa_2\otimes sa_1),$$
for any $a_1, a_2, \dotsc, a_m \in A$,
where $\epsilon=|f|+\sum_{i=1}^{m-1} (|a_i|-1)(|a_{i+1}|-1+\cdots+|a_m|-1)$. Here, we use the identification $A^{\op}= A$ as dg $\mathbb k$-modules.

We claim that $T$ is a strict $B_\infty$-isomorphism from the transpose $B_\infty$-algebra $C^*(A, A)^{\rm tr}$ to $C^*(A^{\op}, A^{\op})$. By Lemma~\ref{lemma-infinity-morphism1} it suffices to verify the following two identities
\begin{align}
\label{lemma-tr}
T(g_1\cup^{\rm tr} g_2) & = T(g_1) \cup T(g_2) \nonumber\\
 T(f\{g_1, \dotsc, g_k\}^{\rm tr}) & =T(f)\{T(g_1), \dotsc, T(g_k)\}.
\end{align}

By definition,  $g_1 \cup^{\rm tr} g_2 = (-1)^{|g_1||g_2|} g_2 \cup g_1$ and $f\{g_1, \dotsc, g_k\}^{\rm tr} = (-1)^{\epsilon} f\{g_k, \dotsc, g_1\}$, where
$\epsilon = k+\sum_{i=1}^{k-1} (|g_i|-1)((|g_{i+1}|-1)+ (|g_{i+2}|-1) + \dotsb + (|g_k|-1))$. By a straightforward computation, we have
\begin{align*}
T(g_1) \cup T(g_2) & = (-1)^{|g_1||g_2|}\;  T(g_2 \cup g_1)\\
T(f)\{T(g_1), \dotsc, T(g_k)\} & = (-1)^{\epsilon} \; T(f \{g_k, \dotsc, g_1\}).
\end{align*}
This proves the claim.

By Theorem~\ref{thm:dualityB} there is a $B_\infty$-isomorphism between $C^*(A, A)^{\rm tr}$ and $C^*(A, A)^{\rm opp}$. Thus, we obtain a $B_\infty$-isomorphism between $C^*(A, A)^{\rm opp}$ and $C^*(A^{\op}, A^{\op}).$
\end{proof}

\begin{rem}
In a private communication (March 2019), Bernhard Keller pointed out a proof of Proposition~\ref{lemma-CL} using the intrinsic description of the $B_{\infty}$-algebra structures on Hochschild cochain complexes; compare \cite[Subsection~5.7]{Kel06}. We are very grateful to him for sharing his  intuition on $B_\infty$-algebras, which essentially  leads to the general result Theorem~\ref{thm:dualityB}.
\end{rem}

\section{A homotopy deformation retract and the homotopy transfer theorem}
\label{sectionforhdr}

In this section, we provide an explicit homotopy deformation retract for the Leavitt path algebra. We begin by recalling a construction of homotopy deformation retracts between resolutions.

\subsection{A construction for homotopy deformation retracts}

We will generalize a result in \cite{He-Li-Li}, which provides a general construction of homotopy deformation retracts between the bar resolution and a smaller projective resolution for a dg algebra.

The following notion is standard; see \cite[Subsection~1.5.5]{LodV}.

\begin{defn}
\label{defn-retract}
Let $(V, d_V)$ and $(W, d_W)$ be two cochain complexes. A {\it homotopy deformation retract} from $V$ to $W$ is a triple $(\iota, \pi, h)$, where
$\iota\colon  V\rightarrow W$ and $\pi\colon W\rightarrow V$ are cochain maps satisfying  $\pi\circ \iota=\mathbf{1}_V$,  and $h\colon  W\rightarrow W$ is a homotopy of degree $-1$ between $\mathbf{1}_W$ and $\iota\circ \pi$, that is,
$\mathbf{1}_W = \iota\circ \pi +d_W\circ h+h\circ d_W$.

The homotopy deformation retract $(\iota, \pi, h)$ is usually depicted by the following diagram
\begin{equation*}\xymatrix@C=0.0000000000001pc{
(V, d_V)   \ar@<0.5ex>[rrrrrrrrrr]^-{\iota}&&&&&&&&&&(W, d_W)\ar@<0.5ex>[llllllllll]^-{\pi} & \ar@(dr, ur)_-{h}}
\end{equation*}
\end{defn}

Let $A$ be a dg algebra with a semisimple subalgebra $E=\bigoplus_{i=1}^n \mathbb ke_i \subseteq A^0 \subseteq A$ satisfying $d_A(e_i)=0$ and $e_ie_j = \delta_{i, j}e_i$ for any $i, j\in \mathcal I$. We consider the (normalized) $E$-relative  bar resolution $\overline{\rm Bar}_E(A)$, whose differential is denoted by $d$. The \emph{tensor-length} of a  typical element $y=a_0\otimes_E s \overline{a_{1, n}}\otimes_Eb\in A\otimes_E (s\overline A)^{\otimes_E n}\otimes_E A$ is defined to be $n+2$, where $s\overline{a_{1, n}}$ means $s\overline{a_1}\otimes_E s\overline{a_2} \otimes_E\cdots \otimes_E s\overline{a_n}$.  The following natural map
\begin{align}\label{equ:natproj}
s\colon  A\otimes_E (s\overline A)^{\otimes_E n}\otimes_E A  &\longrightarrow (s\overline A)^{\otimes_E n+1}\otimes_E A\\
y=a_0\otimes_E s \overline{a_{1, n}}\otimes_Eb &\longmapsto s(y)=s\overline{a_{0, n}}\otimes_E b \nonumber
\end{align}
is of degree $-1$.

The following result is inspired by \cite[Proposition 3.3]{He-Li-Li}.

\begin{prop}\label{prop:hdr}
Let $A$ be a dg algebra with a semisimple subalgebra $E =\bigoplus_{i=1}^n \mathbb ke_i \subseteq A^0\subseteq A$ satisfying $d_A(e_i)=0$ and $e_ie_j = \delta_{i, j}e_i$. Assume that $\omega\colon  \overline{\rm Bar}_E(A)\rightarrow \overline{\rm Bar}_E(A)$ is a morphism of dg $A$-$A$-bimodules satisfying $\omega(a\otimes_E b)=a\otimes_E b$ for all $a, b\in A$. Define a $\mathbb k$-linear map $h\colon  \overline{\rm Bar}_E(A)\rightarrow \overline{\rm Bar}_E(A)$ of degree $-1$ as follows
  \begin{equation*}
  \begin{split}
h&(a_0\otimes_E s\overline{a_{1, n}} \otimes_E b)\\
={} &\begin{cases}
0 & \mbox{if $n=0$; }\\
\sum_{i=1}^n (-1)^{\epsilon_i+1} a_0\otimes_E  s\overline{a_{1, i-1}}\otimes_E \overline{\omega}(1\otimes_E s\overline{a_{i, n}}\otimes_E b) & \mbox{if $n>0$.}
\end{cases}
\end{split}
\end{equation*}
Here,  $\epsilon_{i}=|a_0|+|a_1|+\cdots+|a_{i-1}|+i-1$,  and  $\overline{\omega}$ denotes the composition of $\omega$ with the natural map $s$ in (\ref{equ:natproj}). Then we have $d\circ h+h\circ d={\bf 1}_{\overline{\rm Bar}_E(A)}-\omega$.
\end{prop}

\begin{proof}
We use induction on the tensor-length. Let $a\in A$ and $y\in A\otimes_E (s\overline A)^{\otimes_E n}\otimes_E A$. Then $a\otimes_E s(y)$ lies in $A\otimes_E (s\overline A)^{\otimes_E n+1} \otimes_E A$. To save the space, we write $a\otimes_E s(y)$ as $a\otimes_E \overline{y}$.

 Recall from Subsection~\ref{subsection-bar} that $d = d_{in} + d_{ex}$, where $d_{in}$ is the internal differential and $d_{ex}$ is the external differential. We observe that $d_{in}(a\otimes_E \overline{y})=d_A(a)\otimes_E \overline{y}+(-1)^{|a|+1} a\otimes_E \overline{d_{in}}(y)$ and that $d_{ex}(a\otimes_E \overline{y})=(-1)^{|a|}(ay-a\otimes_E\overline{d_{ex}}(y))$. Here, $ay$ denotes the left action of $a$ on $y$, and $\overline{d_{in}}$ (resp. $\overline{d_{ex}}$) is the composition of $d_{in}$ (resp. $d_{ex}$) with the map $s$ in (\ref{equ:natproj}). Then we have
\begin{align}\label{equ:diff}
d(a\otimes_E \overline{y})=d_A(a)\otimes_E\overline{y}+(-1)^{|a|+1}a\otimes_E \overline{d}(y) +(-1)^{|a|} ay.
\end{align}
From the very definition, we observe
\begin{align*}
h(a\otimes_E \overline{y})=(-1)^{|a|+1}(a\otimes_E \overline{h}(y)+a\otimes_E \overline{\omega}(1\otimes_E \overline{y})).
\end{align*}

Using the above two identities, we obtain
\begin{align*}
d\circ h(a\otimes_E\overline{y})&=(-1)^{|a|+1} d_A(a)\otimes_E \overline{h}(y) + a\otimes_E \overline{d\circ h}(y) -ah(y)\\
&+(-1)^{|a|+1} d_A(a)\otimes_E \overline{\omega}(1\otimes_E \overline{y})+a\otimes_E \overline{d\circ \omega}(1\otimes_E\overline{y}) -a \omega(1\otimes_E \overline{y}),
\end{align*}
and
\begin{align*}
h\circ d(a\otimes_E\overline{y})&=(-1)^{|a|} d_A(a)\otimes_E \overline{h}(y) + (-1)^{|a|} d_A(a)\otimes_E \overline{\omega}(1\otimes_E\overline{y})\\
&+a\otimes_E\overline{h\circ d}(y)+a\otimes_E\overline{\omega}(1\otimes_E \overline{d}(y))+(-1)^{|a|}h(ay).
\end{align*}
Using the fact $ah(y)=(-1)^{|a|}h(ay)$, we infer the first equality of the following identities
\begin{align*}
&(d\circ h+h\circ d)(a\otimes_E\overline{y})\\
={}& a\otimes_E \overline{(d\circ h+h\circ d)}(y) + a\otimes_E \overline{d\circ \omega}(1\otimes_E\overline{y})
+a\otimes_E\overline{\omega}(1\otimes_E \overline{d}(y)) -a \omega(1\otimes_E \overline{y})\\
={}& a\otimes_E \overline{y}-a\otimes_E\overline{\omega}(y) + a\otimes_E \overline{d\circ \omega}(1\otimes_E\overline{y})+a\otimes_E\overline{\omega}(1\otimes_E \overline{d}(y)) -a \omega(1\otimes_E \overline{y})\\
={}&a\otimes_E \overline{y}-a\otimes_E\overline{\omega}(y) + a\otimes_E \overline{ \omega\circ d}(1\otimes_E\overline{y})+a\otimes_E\overline{\omega}(1\otimes_E \overline{d}(y)) -\omega(a\otimes_E \overline{y})\\
={}& a\otimes_E \overline{y}-\omega(a\otimes_E \overline{y}).
\end{align*}
Here, the second equality uses the induction hypothesis,  and the third one uses the fact that $\omega$ respects the differentials and the left $A$-module structure. The last equality uses the following special case of (\ref{equ:diff})
$$-y+d(1\otimes_E \overline{y})+1\otimes_E \overline{d}(y)=0.$$
This completes the proof.
\end{proof}

\begin{rem}
We observe that the obtained homotopy $h$ respects the $A$-$A$-bimodule structures. More precisely, $h\colon  \overline{\Barr}_E(A)\rightarrow \Sigma^{-1}\overline {\Barr}_E(A)$ is a morphism of graded $A$-$A$-bimodules.
\end{rem}

The following immediate consequence of Proposition~\ref{prop:hdr} is a slight generalization of \cite[Proposition 3.3]{He-Li-Li}, which might be a useful tool in many fields to construct explicit homotopy deformation retracts. We recall from (\ref{equ:bar}) the quasi-isomorphism $\varepsilon\colon  \overline{\rm Bar}_E(A)\rightarrow A$.

\begin{cor}\label{cor:hdr}
Let $A$ be a dg algebra with a semisimple subalgebra $E=\bigoplus_{i=1}^n \mathbb ke_i\subseteq A^0\subseteq A$ satisfying $d_A(e_i)=0$ and  $e_ie_j = \delta_{i,j}e_i$. Assume that $P$ is a dg $A$-$A$-bimodule and that there are two morphisms of dg $A$-$A$-bimodules
 $$\iota\colon   P\longrightarrow \overline{\Barr}_E(A), \quad \pi\colon   \overline{\Barr}_E(A)\longrightarrow P$$
satisfying $\pi\circ \iota=\mathbf{1}_P$ and $\iota \circ \pi|_{A\otimes_EA}=\mathbf{1}_{A\otimes_E A}$. Then the pair $(\iota, \pi)$ can be extended to a homotopy deformation retract $(\iota, \pi, h)$, where $h\colon  \overline{\Barr}_E(A)\rightarrow \overline{\Barr}_E(A)$ is given as in Proposition~\ref{prop:hdr} with $\omega=\iota\circ \pi$.

In particular, the composition
$$P\stackrel{\iota}\longrightarrow \overline{\Barr}_E(A)\stackrel{\varepsilon}\longrightarrow A$$
is a quasi-isomorphism of dg $A$-$A$-bimodules. \hfill $\square$
\end{cor}

\subsection{A homotopy deformation retract for the Leavitt path algebra}\label{subsection:hdrforLPA}

In this subsection, we apply the above construction to  Leavitt path algebras. We obtain a homotopy deformation retract between the normalized $E$-relative bar resolution and an explicit bimodule projective resolution.

Let $Q$ be a finite quiver without sinks. Let $L=L(Q)$ be the Leavitt path algebra viewed as a dg algebra with trivial differential; see Section~\ref{sec:lpa}. Set $E=\bigoplus_{i\in Q_0} \mathbb ke_i\subseteq L^0\subseteq L$. We write $\overline L = L / (E\cdot \mathbf 1_{L})$.   In what follows, we will construct an explicit homotopy deformation retract.
\begin{equation}\label{equation-hdr-1}\xymatrix@C=0.00000000001pc{
( P, \partial) \ar@<0.5ex>[rrrrrr]^-{\iota}&&&&& &(\overline{\Barr}_E(L), d)\ar@<0.5ex>[llllll]^-{\pi}& \ar@(dr,ur)_-{h}}
\end{equation}

Let us first describe the dg $L$-$L$-bimodule  $(P, \partial)$. As a graded $L$-$L$-bimodule, $$P=\bigoplus_{i\in Q_0}\left( Le_i\otimes s\mathbb k\otimes e_i L\right) \oplus \bigoplus_{i\in Q_0} Le_i\otimes e_i L.$$
The differential $\partial$ of $P$ is given by
\begin{equation*}
\begin{split}
\partial(x\otimes s\otimes y)={}&(-1)^{|x|} x\otimes y-(-1)^{|x|} \sum_{\{\alpha\in Q_1\mid s(\alpha)=i \}} x\alpha^*\otimes \alpha y,\\
\partial(x\otimes y)={}&0,
\end{split}
\end{equation*}
for $x\in Le_i$, $y\in e_iL$ and $i\in Q_0$. Here, $s\mathbb k$ is the $1$-dimensional graded $\mathbb k$-vector space concentrated in degree $-1$, and the element $s1_{\mathbb k}\in s\mathbb k$ is abbreviated as $s$.

The homotopy deformation retract (\ref{equation-hdr-1}) is defined  as follows.
\begin{enumerate}
\item The injection $\iota\colon  P\rightarrow \overline{\Barr}_E(L)$ is given by
 \begin{equation*}
 \begin{split}
 \iota(x\otimes y)&=x\otimes_E  y,\\
 \iota(x\otimes s \otimes y)&=-\sum_{\{\alpha\in Q_1\mid s(\alpha)=i \}}x\alpha^*\otimes_E s\alpha\otimes_E   y,
 \end{split}
 \end{equation*}
 for  $x\in Le_i$, $y\in e_iL$ and $i\in Q_0$.

 \vskip 3pt

 \item The surjection $\pi\colon  \overline{\Barr}_E(L)\rightarrow P$ is given by

 \begin{equation}
 \label{pi}
 \begin{split}
 \pi(a'\otimes_E b')&=a'\otimes b',\\
 \pi(a\otimes_E s\overline{z} \otimes_E b)&= aD(z)b, \\
 \pi|_{L\otimes_E (s\overline{L})^{\otimes_E >1}\otimes_E L}&=0,
 \end{split}
 \end{equation}
for  $a'=a'e_i$ and $b'=e_ib'$ for some $i\in Q_0$, and any $a, b, z\in L$,  where $D\colon  L\rightarrow\bigoplus_{i\in Q_0}\left( Le_i\otimes s\mathbb k\otimes e_i L\right)$ is the graded $E$-derivation of degree $-1$ in Lemma \ref{lem:deri}. Here and also in the proof of Proposition~\ref{prop:htrforLPA}, we use the canonical identification
$$\bigoplus_{i\in Q_0} Le_i\otimes e_i L=L\otimes_E L, \quad x\otimes y\longmapsto x\otimes_E y.$$

 \item The homotopy $h\colon  \overline{\Barr}_E(L)\rightarrow  \overline{\Barr}_E(L)$  is given by
 \begin{eqnarray*}
\lefteqn{h(a_0\otimes_E s\overline{a_1}\otimes_E \cdots\otimes_E s\overline{a_n}\otimes_E b)}\\
={} &\begin{cases}
0 & \mbox{if $n=0$;}\\
(-1)^{\epsilon_n+1} a_0\otimes_E s\overline{a_1}\otimes_E\cdots \otimes_E s\overline{a_{n-1}}\otimes_E \overline{\iota\circ \pi}(1\otimes_E s\overline{a_n}\otimes_E b) & \mbox{if $n>0$},
\end{cases}
\end{eqnarray*}
where $\epsilon_{n}=|a_0|+|a_1|+\cdots+|a_{n-1}|+n-1$, and  $\overline{\iota\circ \pi}$ is the composition of $\iota\circ \pi$ with the natural isomorphism  $s\colon  L\otimes_E s\overline{L}\otimes_E L\rightarrow s\overline L\otimes_E s\overline{L}\otimes_E L$ of degree $-1$.
 \end{enumerate}

\begin{prop}\label{prop:htrforLPA}
The above triple $(\iota, \pi, h)$ defines a homotopy deformation retract in the abelian category of dg $L$-$L$-bimodules. In particular, the dg $L$-$L$-bimodule $P$ is a dg-projective bimodule resolution of $L$.
\end{prop}

 \begin{proof}
We first observe that $\iota$ and $\pi$ are  morphisms of $L$-$L$-bimodules. Recall that the differential  of $\overline{\Barr}_E(L)$ is given by the external differential $d_{ex}$ since the internal differential $d_{in}$ is zero; see Subsection~\ref{subsection-bar}. We claim that both $\iota$ and $\pi$ respect the differential.  It suffices to prove the commutativity of the following diagram.
\begin{equation*}
\xymatrix@C=1.5pc{
\cdots \ar[r] & 0\ar[d]\ar[r] & \bigoplus_{i\in Q_0} Le_i\otimes s\mathbb k\otimes e_i L \ar[d]^-{\iota}\ar[r]^-{\partial} & L\otimes_E L \ar@{=}[d] \\
\cdots \ar[r]& L\otimes_E (s\overline{L})^{\otimes_E 2}\otimes_EL \ar[d] \ar[r]^-{d_{ex}} & L\otimes_E s\overline{L} \otimes_E L \ar[d]^-{\pi}\ar[r]^-{d_{ex}} & L\otimes_E L  \ar@{=}[d] \\
\cdots \ar[r]& 0\ar[r] & \bigoplus_{i\in Q_0} Le_i\otimes s\mathbb k\otimes e_i L \ar[r]^-{\partial} & L\otimes_E L
}\end{equation*}

For the northeast square, we have
 \begin{equation*}
 \begin{split}
 d_{ex}\circ \iota(x\otimes s\otimes y)={}&-\sum_{\{\alpha\in Q_1\mid s(\alpha)=i \}}d_{ex}(x\alpha^*\otimes_Es\alpha\otimes_E  y)\\
 ={}&\sum_{\{\alpha\in Q_1\mid s(\alpha)=i \}}  -(-1)^{|x|+1} x\alpha^*\alpha\otimes_E y-(-1)^{|x|} x\alpha^*\otimes_E \alpha y\\
 ={}&
 \partial(x\otimes s\otimes y), \end{split}
 \end{equation*}
 where the third equality follows from the second Cuntz-Krieger relations.

 For the southwest square, we have
\begin{equation*}
\begin{split}
&\pi \circ d_{ex}(a\otimes_E  s\overline{y}\otimes_E  s\overline{z}\otimes_E  b)\\
={}&(-1)^{|a|} \pi(ay\otimes_E  s\overline{z}\otimes_E  b)+(-1)^{|a|+|y|-1} (\pi(a\otimes_E  s
\overline{yz}\otimes_E  b)- \pi (a\otimes_E  s\overline{y}\otimes_E  zb))\\
={}&(-1)^{|a|} ay D(z)b+(-1)^{|a|+|y|-1} a D(yz)w-(-1)^{|a|+|y|-1} a D(y) zb\\
={}&0,
\end{split}
\end{equation*}
where the last equality follows from the graded Leibniz rule of $D$.

It remains to verify that the southeast square commutes, namely $\partial\circ \pi=d_{ex}.$  For this, we first note that
\begin{align*}
\partial \circ \pi(a \otimes_E  s\overline{\alpha}\otimes_E  b)= &-(-1)^{|a|+1} a\alpha\otimes  b +(-1)^{|a|+1} \sum_{\{\beta\in Q_1\mid s(\beta)=s(\alpha) \}}a\alpha\beta^*\otimes  \beta b \\
={}& (-1)^{|a|} a\alpha\otimes  b -(-1)^{|a|} a \otimes  \alpha b\\
={}&d_{ex}(a\otimes_E  s\alpha\otimes_E b),
\end{align*}
where $\alpha\in Q_1$ is an arrow, $a\in Le_{t(\alpha)}$ and $b\in e_{s(\alpha)}L$. For the second equality, we use the first Cuntz-Krieger relations $\alpha\beta^*=\delta_{\alpha, \beta}e_{t(\alpha)}$. Similarly, we have $\partial \circ \pi(a\otimes_E s\alpha^*\otimes_E b)=d_{ex}(a \otimes_E s\alpha^*\otimes_E b). $

For the general case, we use induction on the length of the path $w$ in $a\otimes_E sw\otimes_E b$. By the \emph{length} of a path $w$ in $L$, we mean the number of  arrows in $w$, including the ghost arrows.  We write $w=\gamma\eta$ such that the lengths of $\gamma$ and $\eta$ are both strictly smaller than that of $w$. We have  \begin{equation*}
\begin{split}
\partial\circ  \pi(a\otimes_E s\overline{\gamma\eta} \otimes_E b)
={}&\partial\big(a D(\gamma)\eta b+(-1)^{|\gamma|} a\gamma D(\eta )b\big)\\
={}& \partial \circ \pi\big(a\otimes_E s\overline{\gamma}\otimes_E \eta b+(-1)^{|\gamma|} a\gamma\otimes_E s\overline{\eta} \otimes_E b\big)\\
={}&d_{ex}\big(a\otimes_E s\overline{\gamma} \otimes_E \eta b+(-1)^{|\gamma|} a\gamma\otimes_E s\overline{\eta} \otimes_E b\big)\\
={}&d_{ex}(a\otimes_E s\overline{\gamma\eta} \otimes_E b),
\end{split}
\end{equation*}
where the third equality uses the induction hypothesis, and the fourth one follows from $d_{ex}^2(a\otimes_Es\overline{\gamma} \otimes_E s\overline{\eta} \otimes_E b)=0$. This proves the required commutativity and the claim.

The fact $\pi\circ \iota=\mathbf{1}_P$ follows
from the second Cuntz-Krieger relations. By Corollary \ref{cor:hdr},  it follows that $(\iota, \pi)$ extends to a homotopy deformation retract $(\iota, \pi, h)$; moreover, the obtained $h$ coincides with the given one.
\end{proof}

\begin{rem}
\begin{enumerate}
\item From the $L$-$L$-bimodule resolution $P$ above, it follows that the Leavitt path algebra $L$ is {\it quasi-free} in the sense of \cite[Section 3]{CQ}; this result can be also proved along the way of the proof of \cite[Proposition 5.3(2)]{CQ}.
\item The following comment is due to Bernhard Keller: the above explicit projective bimodule resolution  $P$  might be used to give a shorter proof of the computation of the Hochschild homology of $L$  in \cite[Theorem~4.4]{AC}.
\end{enumerate}
\end{rem}

\subsection{The homotopy transfer theorem for dg algebras}

We recall the homotopy transfer theorem for dg algebras, which will be used in the next section.

\begin{thm}[\cite{Kad}]\label{thm-hdr}
Let $(A, d_A, \mu_A)$ be a dg algebra. Let \begin{equation*}\xymatrix@C=0.0000000000001pc{
(V, d_V)   \ar@<0.5ex>[rrrrrrrrrr]^-{\iota}&&&&&&&&&&(A, d_A)\ar@<0.5ex>[llllllllll]^-{\pi} & \ar@(dr, ur)_-{h}}
\end{equation*}
be a homotopy deformation retract between cochain complexes (cf. Definition \ref{defn-retract}). Then there is an $A_{\infty}$-algebra structure $(m_1=d_V, m_2, m_3, \cdots)$ on $V$, where $m_k$ is depicted in  Figure~\ref{A-infinity-product}. Moreover, the map $\iota\colon   V\rightarrow A$ extends to an $A_{\infty}$-quasi-isomorphism $(\iota_1=\iota, \iota_2, \cdots)$ from the resulting $A_{\infty}$-algebra $V$ to the dg algebra $A$, where $\iota_k$ is depicted in Figure~\ref{A-infinity-product}.
\end{thm}
\begin{figure}[H]
\centering
  \includegraphics[height=30mm]{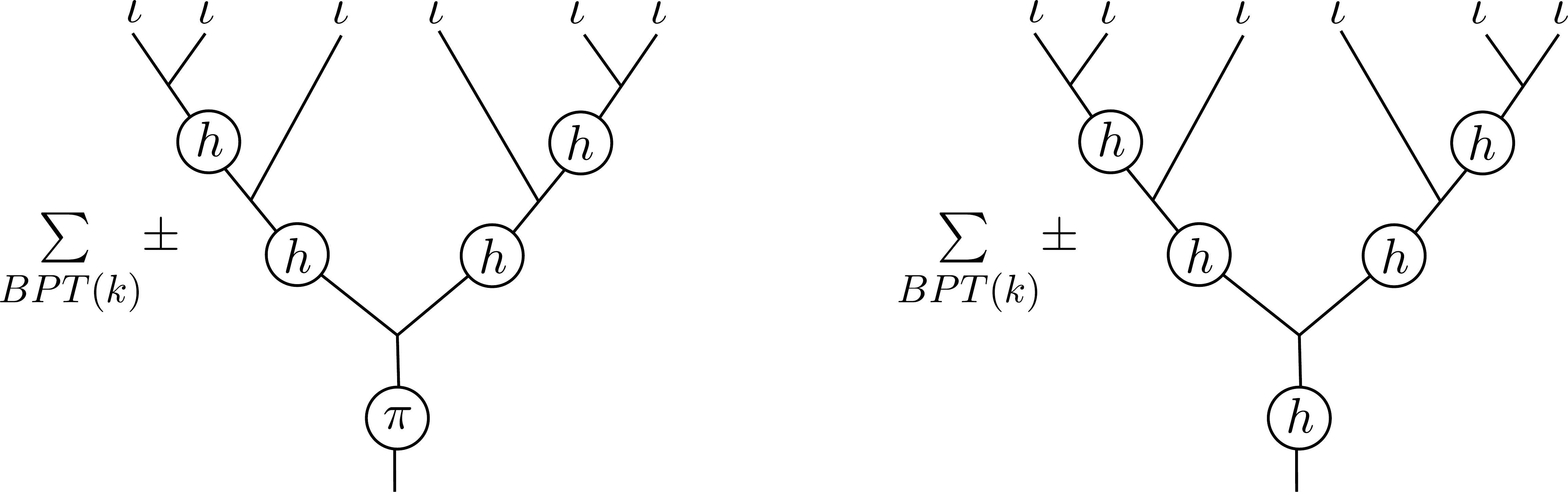}
  \caption{The $A_{\infty}$-product $m_k$ is on the left and the $A_{\infty}$-quasi-isomorphism $\iota_k$ is on the right,  where  the sums are taken over $BPT(k)$, the set of all planar rooted binary trees with $k$ leaves. }
  \label{A-infinity-product}
\end{figure}

In this paper, we only need the following special case of Theorem \ref{thm-hdr}.

\begin{cor}\label{corollary-hdr}
Let $(A, d_A, \mu_A)$ be a dg algebra. Let \begin{equation*}\xymatrix@C=0.0000000000001pc{
(V, d_V)   \ar@<0.5ex>[rrrrrrrrrr]^-{\iota}&&&&&&&&&&(A, d_A)\ar@<0.5ex>[llllllllll]^-{\pi} & \ar@(dr, ur)_-{h}}
\end{equation*}
be a homotopy deformation retract between cochain complexes. We further assume that
\begin{flalign}\label{assumptionhomtopytransfer}
& &h\mu_A(a\otimes  h(b))=0=\pi \mu_A(a\otimes h(b))&&  \mbox{for any $a, b\in A$}.&&
\end{flalign}
  Then the resulting $A_{\infty}$-algebra $(V, m_1=d_V, m_2, m_3, \cdots)$ is simply given by (cf. Figure \ref{A-infinity-special1})
\begin{align*}
m_2(a_1\otimes a_2) &=\pi (\iota(a_1) \iota(a_2)),\\ m_k(a_1\otimes \cdots \otimes a_k) & = \pi( h(\cdots(h(h(\iota(a_1)\iota(a_2))\iota(a_3))\cdots)\iota(a_k)), \quad \quad k > 2,
\end{align*}
where we simply write $\iota(a)\iota(b)=\mu_A(\iota(a)\otimes \iota(b)).$

 Moreover, the $A_{\infty}$-quasi-isomorphism $(\iota_1=\iota, \iota_2, \cdots)$ is  given by (cf. Figure  \ref{A-infinity-special1})
 $$ \iota_k(a_1\otimes \cdots \otimes a_k)=(-1)^{\frac{k(k-1)}{2}}h(h(\cdots(h(h(\iota(a_1)\iota(a_2))\iota(a_3))\cdots)\iota(a_k)), \quad\quad  k \geq 2.$$
\end{cor}

\begin{figure}[H]
\centering
  \includegraphics[height=30mm]{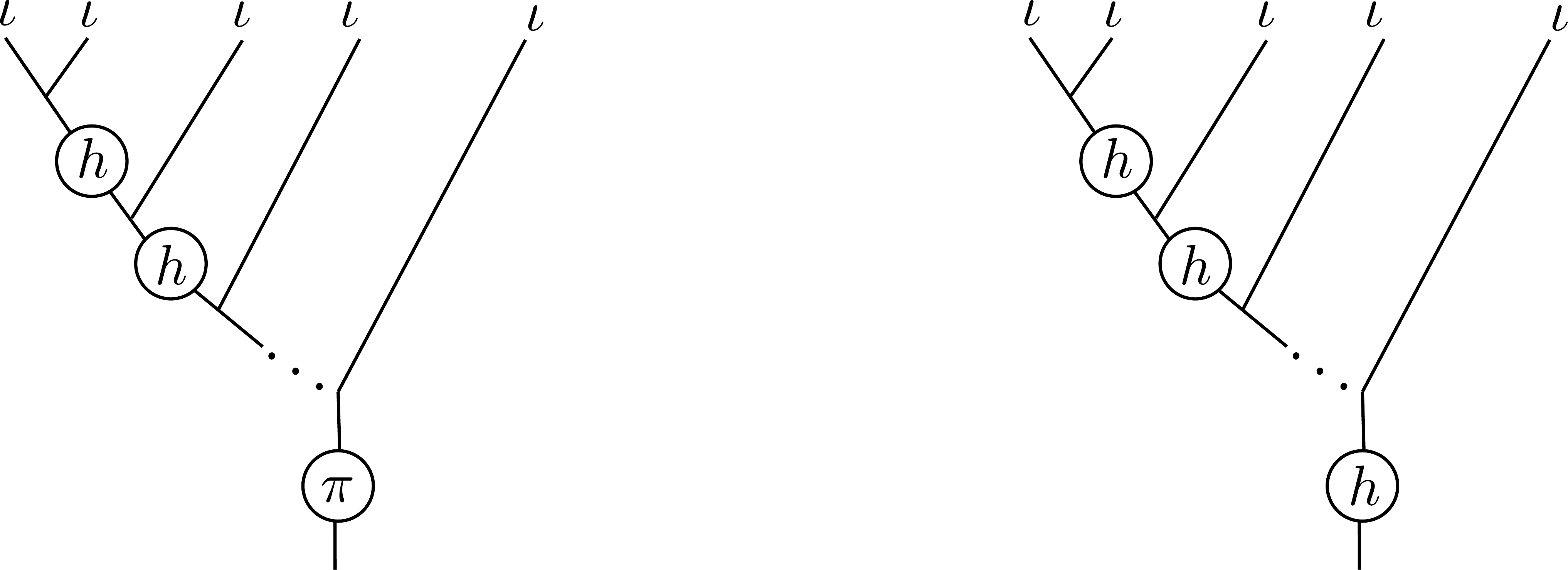}
  \caption{The  $A_{\infty}$-product $m_k$ and $A_{\infty}$-quasi-isomorphism $\iota_k$. }
  \label{A-infinity-special1}
\end{figure}

\begin{rem}\label{ind}
Note that under the assumption \eqref{assumptionhomtopytransfer}, the formulae for the resulting $A_\infty$-algebra and $A_\infty$-morphism are highly simplified.

For $k\geq 2$, we have the following recursive formula
\begin{align*}
\iota_k(a_1\otimes \cdots \otimes a_k)   =(-1)^{k-1}h(\iota_{k-1}(a_1\otimes \cdots\otimes a_{k -1}) \iota(a_k))
\end{align*}
and the following identity
\begin{align*}
m_k(a_1\otimes \cdots\otimes a_k)  = (-1)^{\frac{(k-1)(k-2)}{2}}\pi( \iota_{k-1}(a_1\otimes \cdots\otimes a_{k-1})\iota(a_k)).
\end{align*}	
\end{rem}

\section{The singular Hochschild cochain complexes}
\label{section7}

In this section, we recall the singular Hochschild cochain complexes and their $B_\infty$-structures. We describe explicitly the brace operation on the singular Hochschild cochain complex and illustrate it with an example.

\subsection{The left and right singular Hochschild cochain complexes}\label{subsec:sHcc}

Let $\Lambda$ be a finite dimensional $\mathbb k$-algebra. Denote by $\Lambda^e=\Lambda\otimes \Lambda^{\op}$ its enveloping algebra. Let $\mathbf{D}_{\sg}(\Lambda^e)$ be the singularity category of $\Lambda^e$. Following \cite{BJ, ZF, Kel18},  the {\it singular Hochschild cohomology} of $\Lambda$ is defined  as
$$\HH_{\sg}^n(\Lambda, \Lambda):=\Hom_{\mathbf{D}_{\sg}(\Lambda^e)}(\Lambda, \Sigma^n(\Lambda)), \quad \text{for $n\in \mathbb Z$}.$$
Recall from \cite[Section 3]{Wan1} that the singular Hochschild cohomology $\HH_{\sg}^*(\Lambda, \Lambda)$ can be computed by the so-called  singular Hochschild cochain complex.

There are two kinds of  singular Hochschild cochain complexes: the {\it left singular Hochschild cochain complex} and the {\it right singular Hochschild cochain complex}, which are  constructed by using  the left noncommutative differential forms and the right noncommutative differential forms, respectively. We mention that only the left one is considered in \cite{Wan1} with slightly different notation; see \cite[Definition~3.2]{Wan1}. We will first recall the right singular Hochschild cochain complex $\overline C_{\sg, R}^*(\Lambda, \Lambda)$.

Throughout this subsection, we denote  $\overline{\Lambda} = \Lambda/ (\mathbb k \cdot 1_{\Lambda})$.  Recall from \cite{ZF} that  the \emph{graded $\Lambda$-$\Lambda$-bimodule of right noncommutative differential $p$-forms} is defined as
 $$\Omega_{\nc, R}^p(\Lambda)=(s\overline \Lambda)^{\otimes p}\otimes \Lambda.$$
 Observe that $\Omega_{\nc,R}^p(\Lambda)$ is concentrated in degree $-p$ and that its bimodule structure is given by
\begin{equation}\label{rightblacktriangle}
\begin{split}
a_0\blacktriangleright( s \overline{a_1}\otimes \cdots\otimes s\overline{a_p}\otimes a_{p+1})b={}&\sum_{i=0}^{p-1}(-1)^{i}  s\overline{a_0}\otimes \cdots\otimes s\overline{a_ia_{i+1}} \otimes \cdots\otimes s\overline{a_{p}}\otimes a_{p+1}b\\
& +(-1)^p s\overline{a_0}\otimes s\overline{a_1}\otimes \cdots\otimes s\overline{a_{p-1}}\otimes a_pa_{p+1}b
\end{split}
\end{equation}
for $b, a_{0}\in \Lambda$ and $ s\overline{a_1}\otimes \cdots\otimes s\overline{a_p}\otimes a_{p+1}\in \Omega_{\nc,R}^p(\Lambda)$. Note that there is a $\mathbb k$-linear isomorphism between $\Omega^p_{\nc, R}(\Lambda)$ and the cokernel of the $(p+1)$-th differential
$$\Lambda \otimes (s\overline\Lambda)^{\otimes p+1} \otimes \Lambda \xrightarrow{d_{ex}} \Lambda \otimes (s\overline\Lambda)^{\otimes p} \otimes \Lambda$$ in $\overline\Barr(\Lambda)$ defined in  Subsection~\ref{subsection-bar}. Then the above bimodule structure on $\Omega_{\nc, R}^p(\Lambda)$ is inherited from this cokernel; compare \cite[Lemma~2.5]{Wan1}. For ungraded noncommutative differential forms, we refer to \cite[Sections~1 and 3]{CQ}.

 We have a short exact sequence of graded bimodules
\begin{align}\label{shortexactsequence00}
0\longrightarrow \Sigma^{-1}\Omega^{p+1}_{\nc, R}(\Lambda) \stackrel{d'}\longrightarrow \Lambda \otimes (s\bar\Lambda)^{\otimes p} \otimes \Lambda   \stackrel{d''}\longrightarrow \Omega^{p}_{\nc, R}(\Lambda)\longrightarrow 0
\end{align}
where $d'(s^{-1}x) = d_{ex} (1 \otimes x)$ for any $x \in \Omega^{p+1}_{\nc, R}(\Lambda)$, and $d''=(\varpi \otimes \mathbf 1_{s\overline \Lambda}^{\otimes p-1} \otimes \mathbf 1_\Lambda)\circ  d_{ex}$. Here, $\varpi \colon \Lambda \to s\overline \Lambda$ is the natural projection of degree $-1$. We observe that $d_{ex}$ factors as
$$ \Lambda \otimes (s\bar\Lambda)^{\otimes p+1} \otimes \Lambda \stackrel{d''}\longrightarrow \Omega^{p+1}_{{\rm nc}, R}(\Lambda) \stackrel{\Sigma(d')} \longrightarrow \Lambda \otimes (s\bar\Lambda)^{\otimes p} \otimes \Lambda, $$
and that
$$d''(a\otimes s\overline{a_{1,p}}\otimes a_{p+1})=a\blacktriangleright(s\overline{a_{1, p}}\otimes a_{p+1}).$$

Let $\overline{C}^*(\Lambda, \Omega_{\nc,R}^p(\Lambda))$ be the normalized Hochschild cochain complex of $\Lambda$ with coefficients in the graded bimodule $\Omega_{\nc,R}^p(\Lambda)$; see Subsection~\ref{subsection-bar}. Here,  $\Lambda$ is viewed as a dg algebra concentrated in degree zero.

For each $p\geq 0$, we define a morphism (of degree zero) of complexes
\begin{equation}\label{theta-p}
\theta_{p, R}\colon   \overline{C}^*(\Lambda, \Omega_{\nc,R}^p(\Lambda))\longrightarrow \overline{C}^*(\Lambda, \Omega_{\nc,R}^{p+1}(\Lambda)), \quad f \longmapsto  \mathbf{1}_{s\overline{\Lambda}}\otimes f. 
\end{equation}
Here, we recall that  $\overline{C}^m(\Lambda, \Omega_{\nc,R}^p(\Lambda))={\rm  Hom}((s\overline \Lambda)^{\otimes m+p}, \Omega_{\nc,R}^p(\Lambda))$, the Hom-space between non-graded spaces. Then for $f\in \overline{C}^m(\Lambda, \Omega_{\nc,R}^p(\Lambda))$, the map $\mathbf{1}_{s\overline{\Lambda}}\otimes f$ naturally lies in $\overline{C}^m(\Lambda, \Omega_{\nc,R}^{p+1}(\Lambda))$,  using the following identification
$$\Omega_{\nc, R}^{p+1}(\Lambda)=s\overline{\Lambda}\otimes \Omega_{\nc, R}^p(\Lambda).$$
 We mention that when $\mathbf{1}_{s\overline{\Lambda}}\otimes f$ is applied to  elements in $(s\overline{\Lambda})^{\otimes m+p+1}$, an additional  sign $(-1)^{|f|}$  appears  due to the Koszul sign rule.

The  {\it right singular Hochschild cochain complex} $\overline{C}_{\sg,R}^*(\Lambda, \Lambda)$ is defined to be the colimit of the inductive system
\begin{equation}\label{equation-inductive} \overline{C}^*(\Lambda, \Lambda)\xrightarrow{\theta_{0, R}} \overline{C}^*(\Lambda, \Omega_{\nc, R}^1(\Lambda))\xrightarrow{\theta_{1, R}}  \cdots \xrightarrow{\theta_{p-1, R}}  \overline{C}^*(\Lambda, \Omega_{\nc, R}^p(\Lambda))\xrightarrow{\theta_{p, R}} \cdots.
\end{equation}
We mention that all the maps $\theta_{p, R}$ are injective.

The above terminology is justified by the following observation.

\begin{lem}\label{lem:HH-colimit}
For each $n\in \mathbb{Z}$, we have an isomorphism
$$\HH_{\rm sg}^n(\Lambda, \Lambda) \simeq H^n(\overline{C}_{\sg,R}^*(\Lambda, \Lambda)).$$\end{lem}

\begin{proof}
The proof is  analogous to that of \cite[Theorem 3.6]{Wan1} for the left singular Hochschild cochain complex.   For  the convenience of the reader,  we give a complete proof.

Since  the direct colimit commutes with the cohomology functor, we obtain that
$$H^n(\overline{C}_{\sg,R}^*(\Lambda, \Lambda))\simeq \varinjlim_{\widetilde \theta_{p,R}} \; \HH^n(\Lambda, \Omega^p_{\nc, R}(\Lambda)),$$
where the maps $\widetilde \theta_{p, R}$ are induced by the above cochain maps $\theta_{p, R}$.

Applying the functor $\HH^*(\Lambda, -)$ to the short exact sequence \eqref{shortexactsequence00}, we obtain a long exact sequence
$$
\dotsb  \to \HH^n(\Lambda, \Lambda \otimes (s\bar\Lambda)^{\otimes p} \otimes \Lambda)   \to \HH^n(\Lambda, \Omega^{p}_{\nc, R}(\Lambda)) \stackrel{c}\to \HH^{n+1}(\Lambda, \Sigma^{-1}\Omega^{p+1}_{\nc, R}(\Lambda))\to \dotsb.
$$
Since $\HH^{n+1}(\Lambda, \Sigma^{-1}\Omega^{p+1}_{\nc, R}(\Lambda)$ is naturally isomorphic to  $\HH^{n}(\Lambda, \Omega^{p+1}_{\nc, R}(\Lambda))$, the connecting morphism $c$ in the long exact sequence  induces a map
$$
\widehat \theta_{p, R} \colon \HH^n(\Lambda, \Omega^{p}_{\nc, R}(\Lambda))\longrightarrow  \HH^{n}(\Lambda, \Omega^{p+1}_{\nc, R}(\Lambda)).
$$
By  \cite[Subsection~2.3]{KV} or \cite[Lemma~2.4]{Kel18},  we have a natural isomorphism
$${\rm HH}^n_{\rm sg}(\Lambda, \Lambda) \simeq \varinjlim_{\widehat \theta_{p,R}} \; \HH^n(\Lambda, \Omega^p_{\nc, R}(\Lambda)).$$

It remains to  show that $\widetilde \theta_{p, R}=\widehat \theta_{p, R}$. Indeed, let $f \in  \HH^n(\Lambda, \Omega^{p}_{\nc, R}(\Lambda))$. Assume that it is represented by a linear map  $f\colon (s\overline{\Lambda})^{\otimes n+p}\rightarrow \Omega^p_{\nc, R}(\Lambda)$. As $f$ is a cocycle,  the induced map
$$f'\colon \Omega^{n+p}_{\nc, R}(\Lambda)\longrightarrow \Omega^p_{\nc, R}(\Lambda), \quad x\otimes a\mapsto f(x)a$$
is a bimodule homomorphism of degree $n$; here, we recall that $\Omega^{n+p}_{\nc, R}(\Lambda)=(s\overline{\Lambda})^{\otimes n+p}\otimes \Lambda$. We have the following commutative diagram of bimodules with exact rows.
\begin{equation*}
\xymatrix{
0\ar[r]& \Sigma^{-1}\Omega^{n+p+1}_{\nc, R}(\Lambda) \ar[r]^-{d'}\ar[d]_-{(-1)^n\Sigma^{-1}({\bf 1}_{s\overline{\Lambda}}\otimes f')} &  \Lambda \otimes (s\bar\Lambda)^{\otimes n+p} \otimes \Lambda  \ar[d]^-{{\bf 1}_\Lambda\otimes f'} \ar[r]^-{d''} & \Omega^{n+p}_{\nc, R}(\Lambda)  \ar[d]^-{f'} \ar[r] &0\\
0\ar[r] & \Sigma^{-1}\Omega^{p+1}_{\nc, R}(\Lambda) \ar[r]^-{d'} & \Lambda \otimes (s\bar\Lambda)^{\otimes p} \otimes \Lambda \ar[r]^-{d''} & \Omega^p_{\nc, R}(\Lambda)  \ar[r] & 0
}
\end{equation*}
For the sign in the leftmost vertical arrow, we recall that $\Sigma^{-1}$ acts on any morphism of degree $n$ by $(-1)^n$. The bimodule homomorphism ${\bf 1}_{s\overline{\Lambda}}\otimes f'$ corresponds to the linear map
$${\bf 1}_{s\overline{\Lambda}}\otimes f\colon (s\overline{\Lambda})^{\otimes n+p+1}\longrightarrow \Omega^{p+1}_{\nc, R}(\Lambda).$$
By the construction of the connecting morphism, we infer that $\widehat \theta_{p, R}(f)={\bf 1}_{s\overline{\Lambda}}\otimes f$. In view of the  very definition of $\theta_{p, R}$, we deduce $\widetilde \theta_{p, R}=\widehat \theta_{p, R}$.
\end{proof}

There are two basic  operations on   $\overline{C}_{\sg , R}^*(\Lambda, \Lambda)$. The first one is the cup product
$$-\cup_R-\colon   \overline{C}_{\sg, R}^*(\Lambda, \Lambda)\otimes \overline{C}_{\sg, R}^*(\Lambda, \Lambda)\longrightarrow \overline{C}_{\sg, R}^*(\Lambda, \Lambda)$$
which is defined as follows: for $\varphi\in \overline C^{m-p}(\Lambda, \Omega^p_{\nc, R}(\Lambda))$ and $\phi\in \overline C^{n-q}(\Lambda, \Omega^q_{\nc, R}(\Lambda))$, we define
\begin{equation}
\label{equation-defcupforsg}
\varphi\cup_R \phi:=\left(\mathbf{1}_{s\overline{\Lambda}}^{\otimes p+q}\otimes \mu\right)\circ \left( \mathbf{1}_{s\overline{\Lambda}}^{\otimes q}\otimes \varphi\otimes \mathbf{1}_\Lambda\right) \circ  \left( \mathbf{1}_{s\overline{\Lambda}}^{\otimes m}\otimes \phi\right) \in \overline{C}^{m+n-p-q}(\Lambda, \Omega^{p+q}_{\nc, R}(\Lambda)),
\end{equation}
where $\mu$ denotes the multiplication of $\Lambda$. When $\varphi\cup_R \phi$ is applied to  elements in $(s\overline{\Lambda})^{\otimes m+n}$, an additional  sign $(-1)^{mn+pq}$  appears  due to the Koszul sign rule. In particular, if  $p=q=0$ we get  the classical cup product on $\overline{C}^*(\Lambda, \Lambda)$. Note that  $-\cup_R-$ is compatible with the colimit, hence it is well defined on $\overline{C}_{\sg , R}^*(\Lambda,\Lambda)$.

The second one is the brace operation
$$\begin{array}{lcccr}
-\{-, \dotsc, -\}_R\colon   & \overline{C}_{\sg, R}^*(\Lambda, \Lambda) \otimes \overline{C}_{\sg, R}^*(\Lambda, \Lambda)^{\otimes k}  &  \longrightarrow  & \overline{C}_{\sg, R}^*(\Lambda, \Lambda), & \mbox{for $k\geq 1$}, \\
 & x \otimes (y_1\otimes \cdots\otimes y_k) &  \longmapsto & x\{y_1, \dotsc, y_k\}_R,&
\end{array}
$$
 which is defined in Subsection~\ref{subsec:example-brace} below; see Definition~\ref{defn:brace-A}. It restricts to the classical brace operation on $\overline{C}^*(\Lambda, \Lambda)$.

The following result is a right-sided version of \cite[Theorem 5.1]{Wan1}.

 \begin{thm}\label{thm-Wan2}
 The right singular Hochschild cochain complex  $\overline{C}_{\sg, R}^*(\Lambda, \Lambda)$,
 equipped with $\cup_R$ and $-\{-,\dotsc,-\}_R$, is a
 brace $B_{\infty}$-algebra.   Consequently, $(\HH_{\sg}^*(\Lambda, \Lambda),-\cup_R-, [-,-]_R)$ is a Gerstenhaber algebra\hfill $\square$\end{thm}

We now recall the left singular Hochschild cochain complex $\overline{C}_{\sg, L}^*(\Lambda, \Lambda)$. The \emph{graded $\Lambda$-$\Lambda$-bimodule of left noncommutative differential $p$-forms} is
$$\Omega_{\nc, L}^p(\Lambda)=\Lambda\otimes (s\overline \Lambda)^{\otimes p},$$
whose bimodule structure is given by
\begin{equation*}
\begin{split}
b( a_0 \otimes s\overline{a_1} \cdots\otimes s\overline{a_p})\blacktriangleleft a_{p+1}={}& (-1)^{p}ba_0a_1\otimes s\overline{a_2}\otimes \cdots \otimes s\overline{a_p}\otimes s\overline{a_{p+1}}+{}\\
& \sum_{i=1}^p (-1)^{p-i} ba_0\otimes s\overline{a_1}\otimes \cdots\otimes s\overline{a_ia_{i+1}}\otimes \cdots \otimes s\overline{a_{p+1}}
\end{split}
\end{equation*}
for $b, a_{p+1}\in \Lambda$ and $ a_0\otimes s\overline{a_1}\otimes \cdots\otimes s\overline{a_p}\in \Omega_{\nc, L}^p(\Lambda)$. It follows from \cite[Lemma~2.5]{Wan1} that $\Omega_{\nc, L}^p(\Lambda)$ is also isomorphic, as graded $\Lambda$-$\Lambda$-bimodules,  to the cokernel of the $(p+1)$-th differential
$$\Lambda \otimes (s\overline\Lambda)^{\otimes p+1} \otimes \Lambda \xrightarrow{d_{ex}} \Lambda \otimes (s\overline\Lambda)^{\otimes p} \otimes \Lambda$$
in $\overline\Barr(\Lambda)$. In particular, we infer that  $\Omega_{\nc, L}^p(\Lambda)$ and $\Omega_{\nc, R}^p(\Lambda)$ are isomorphic as graded $\Lambda$-$\Lambda$-bimodules.

The {\it left singular Hochschild cochain complex} $\overline{C}_{\sg, L}^*( \Lambda, \Lambda)$ is defined as the colimit of the inductive system
$$ \overline{C}^*(\Lambda, \Lambda)\xrightarrow{\theta_{0, L}} \overline{C}^*(\Lambda, \Omega_{\nc, L}^1(\Lambda))\xrightarrow{\theta_{1, L}} \cdots \xrightarrow{\theta_{p-1, L}}  \overline{C}^*(\Lambda, \Omega_{\nc, L}^p(\Lambda))\xrightarrow{\theta_{p, L}} \dotsc, $$
where $$\theta_{p, L}\colon   \overline{C}^*(\Lambda, \Omega_{\nc, L}^p(\Lambda))\longrightarrow \overline{C}^*(\Lambda, \Omega_{\nc, L}^{p+1}(\Lambda)), \quad f \longmapsto f\otimes \mathbf{1}_{s\overline{\Lambda}}. $$

The cup product and brace operation on $\overline C_{\sg,L}^*(\Lambda, \Lambda)$ are defined in \cite[Subsections~4.1 and 5.2]{Wan1}. Let us denote them by $- \cup_L -$ and $-\{-, \dotsc, -\}_L$, respectively.

\begin{thm}{\rm (\cite[Theorem~5.1]{Wan1})}
The left singular Hochschild cochain complex $\overline{C}_{\sg , L}^*(\Lambda, \Lambda)$, equipped with the mentioned cup product and brace operation, is a brace $B_{\infty}$-algebra. Consequently,  $(\HH_{\sg}^*(\Lambda, \Lambda), -\cup_L-, [-, -]_L)$ is a Gerstenhaber algebra. \hfill $\square$
\end{thm}

The above two Gerstenhaber algebra structures on $\HH_{\sg}^*(\Lambda, \Lambda)$ are actually the same.

\begin{prop}\label{prop:Ger-dual}
The above two Gerstenhaber algebras $(\HH_{\sg}^*(\Lambda, \Lambda), -\cup_L-, [-, -]_L)$ and $(\HH_{\sg}^*(\Lambda, \Lambda), -\cup_R-, [-, -]_R)$ coincide.
\end{prop}

\begin{proof}
By \cite[Proposition 4.7]{Wan1}, both $-\cup_L-$ and $-\cup_R-$ coincide with the Yoneda product on $\HH_{\sg}^*(\Lambda, \Lambda)$. Then  we have $-\cup_L-  = -\cup_R-$. By \cite[Corollary 5.10]{Wan2}, we infer that $[-, -]_R$ is isomorphic to a subgroup $G_{\Lambda}$ of the singular derived Picard group of $\Lambda$. Similarly,  one  proves that $[-, -]_L$ is also isomorphic to $G_{\Lambda}$. For more details, we refer to \cite{Wan2}.
\end{proof}

Let $\Lambda^{\op}$ be the opposite algebra of $\Lambda$. Consider the following two  $B_{\infty}$-algebras $$(\overline{C}_{\sg , L}^*(\Lambda, \Lambda), \delta, \cup_L; -\{-, \dotsc, -\}_L)$$ and $$(\overline{C}_{\sg , R}^*(\Lambda^{\op}, \Lambda^{\op})
, \delta, \cup_R; -\{-,\dotsc, -\}_R).$$
The following result is analogous to Proposition~\ref{lemma-CL}.

\begin{prop}
\label{lemma-CL1}
Let $\Lambda$ be a $\mathbb{k}$-algebra,  and $\Lambda^{\op}$ be the opposite algebra of $\Lambda$. Then there is a $B_\infty$-isomorphism between the opposite $B_\infty$-algebra $\overline{C}_{\sg , L}^*(\Lambda, \Lambda)^{\rm opp}$ and the $B_\infty$-algebra $\overline{C}_{\sg , R}^*(\Lambda^{\op}, \Lambda^{\op})$.
\end{prop}
\begin{proof}
Consider the {\it swap isomorphism} (note that $\Lambda = \Lambda^{\op}$ as $\mathbb k$-modules)
\begin{align}
\label{align-T}
T \colon  \overline{C}_{\sg , L}^*(\Lambda, \Lambda)\stackrel{}\longrightarrow  \overline{C}_{\sg , R}^*(\Lambda^{\op}, \Lambda^{\op})
\end{align}
which sends  $f\in \Hom((s\overline \Lambda)^{\otimes m}, \Lambda\otimes (s\overline{\Lambda})^{\otimes p})$ to $T(f)\in \Hom((s\overline{\Lambda})^{\otimes m}, (s\overline{\Lambda})^{\otimes p}\otimes \Lambda)$ with
$$T(f)(s\overline{a_1}\otimes s\overline{a_2} \otimes \cdots\otimes s\overline{a_m})= (-1)^{m-p+\frac{m(m-1)}{2}} \tau_p(f(s\overline{a_m}\otimes \cdots \otimes s\overline{a_2} \otimes s\overline{a_1})).$$
Here, the $\mathbb k$-linear map
$\tau_p \colon  \Lambda \otimes  (s\overline \Lambda)^{\otimes p} \rightarrow (s\overline{\Lambda})^{\otimes p} \otimes \Lambda $ is defined as
$$
\tau_p(b_0 \otimes s\overline{b_1} \otimes s\overline{b_2}\otimes  \cdots \otimes s\overline{b_p})=(-1)^{\frac{p(p-1)}{2}} s\overline{b_p} \otimes \cdots \otimes s\overline{b_2}\otimes s\overline{b_1} \otimes b_0.
$$

It is straightforward to verify the following two identities
\begin{align*}
T(g_1) \cup_R T(g_2) & = (-1)^{|g_1||g_2|} T(g_2 \cup_L g_1), \\
 T(f)\{T(g_1), \dotsc, T(g_k)\}_R & = (-1)^{\epsilon}T( f\{g_k, \dotsc, g_1\}_L),
\end{align*}
where $\epsilon =k+ \sum_{i=1}^{k-1} (|g_i|-1)((|g_{i+1}|-1)+ (|g_{i+2}|-1) + \dotsb + (|g_k|-1))$.
By (\ref{tr1}) we have
\begin{align*}
T(g_1 \cup^{\rm tr}_L g_2) & = (-1)^{|g_1||g_2|} T(g_2 \cup_L g_1),  \\
T(f\{g_1, \dotsc, g_k\}_L^{\rm tr}) & = (-1)^{\epsilon}T( f\{g_k, \dotsc, g_1\}_L).
\end{align*}
Combining the above identities, from Lemma~\ref{lemma-infinity-morphism1} we obtain that  $T$ is a strict $B_{\infty}$-isomorphism from
$\overline{C}_{\sg , L}^*(\Lambda, \Lambda)^{\rm tr}$ to $ \overline{C}_{\sg , R}^*(\Lambda^{\op}, \Lambda^{\op})$.

By Theorem~\ref{thm:dualityB} there is a $B_\infty$-isomorphism between $\overline{C}_{\sg , L}^*(\Lambda, \Lambda)^{\rm tr}$ and $\overline{C}_{\sg , L}^*(\Lambda, \Lambda)^{\rm opp}$. We obtain a $B_\infty$-isomorphism between $\overline{C}_{\sg , L}^*(\Lambda, \Lambda)^{\rm opp}$ and $ \overline{C}_{\sg , R}^*(\Lambda^{\op}, \Lambda^{\op})$.
\end{proof}

\begin{rem}\label{rem:left-right}
From Proposition \ref{lemma-CL1} it follows that there is  a (non-strict) $B_{\infty}$-isomorphism
$$\overline{C}^*_{\sg, L}(\Lambda, \Lambda)\cong \overline{C}_{\sg, R}^*(\Lambda^{\mathrm{op
}}, \Lambda^{\mathrm{op}})^{\rm opp}.$$
 In particular, this $B_\infty$-isomorphism induces an isomorphism of Gerstenhaber algebras
 $$(\HH_{\sg}^*(\Lambda, \Lambda),-\cup_L-, [-,-]_L) \simeq  (\HH_{\sg}^*(\Lambda^{\mathrm{op}}, \Lambda^{\mathrm{op}}), -\cup_R-, [-, -]^{\mathrm{opp}}_R),$$
 where $[f, g]^{\mathrm{opp}}_R = -[f, g]_R$.

In contrast to Proposition \ref{prop:Ger-dual}, we do not know whether the $B_{\infty}$-algebras $\overline{C}^*_{\sg, L}(\Lambda, \Lambda)$ and $\overline{C}^*_{\sg, R}(\Lambda, \Lambda)$ are isomorphic in $\mathrm{Ho}(B_{\infty})$. Actually, it seems that there is  even no obvious natural quasi-isomorphism of complexes between them,  although both of them compute the same $\HH_{\sg}^*(\Lambda, \Lambda)$.
\end{rem}

\subsection{The relative singular Hochschild cochain complexes}\label{subsec:rel-sinHo}

We will need the relative version of the singular Hochschild cochain complexes.

Let $E=\bigoplus_{i=1}^n \mathbb ke_i\subseteq \Lambda$ be a semisimple subalgebra of $\Lambda$ with a decomposition $e_1 +\dotsb + e_n=1_{\Lambda}$ of the unity into orthogonal idempotents. Assume that $\varepsilon \colon   \Lambda \twoheadrightarrow E$ is a split surjective algebra homomorphism such that the inclusion map $E\hookrightarrow \Lambda$ is a section of $\varepsilon$.

The following notion is slightly different from the one in Subsection~\ref{subsec:sHcc}.  We will denote the quotient $E$-$E$-bimodule $\Lambda/(E\cdot 1_\Lambda)$ by $\overline{\Lambda}$. The quotient $\mathbb{k}$-module $\Lambda/(\mathbb{k}\cdot 1_\Lambda)$ will be temporarily denoted by $\overline{\overline \Lambda}$ in this subsection. Identifying $\overline{\Lambda}$ with ${\rm Ker}(\varepsilon)$, we obtain a natural injection
$$\xi\colon \overline{\Lambda}\longrightarrow \overline{\overline \Lambda},\quad  x+(E\cdot 1_\Lambda)\longmapsto x+(\mathbb{k}\cdot 1_\Lambda)$$
for each $x\in {\rm Ker}(\varepsilon)$.

Consider the \emph{graded $\Lambda$-$\Lambda$-bimodule of $E$-relative right noncommutative differential $p$-forms} $$\Omega_{\nc, R, E}^p(\Lambda)=(s\overline \Lambda)^{\otimes_E p}\otimes_E \Lambda.$$
Similarly, $\Omega_{\nc, R, E}^p(\Lambda)$ is isomorphic to the cokernel of the differential in $\overline{\Barr}_E(\Lambda)$
$$\Lambda \otimes_E (s\overline{\Lambda})^{\otimes_E  p+1}\otimes_E \Lambda\xrightarrow{d_{ex}} \Lambda \otimes_E (s\overline{\Lambda})^{\otimes_E  p}\otimes_E \Lambda. $$

The {\it $E$-relative right singular Hochschild cochain complex} $\overline{C}_{\sg, R, E}^*(\Lambda, \Lambda)$ is defined  to be  the colimit of the inductive system
$$ \overline C_E^*( \Lambda, \Lambda)\xrightarrow{\theta_{0, R, E}} \overline{C}_E^*( \Lambda, \Omega_{\nc, R, E}^1(\Lambda)) \xrightarrow{\theta_{1, R, E}}  \cdots \rightarrow \overline C_E^*( \Lambda, \Omega_{\nc, R, E}^p(\Lambda))\xrightarrow{\theta_{p, R, E}} \dotsc, $$
where
\begin{align}\label{equ:thetaRE}
\theta_{p, R, E}\colon  \overline C_E^*( \Lambda, \Omega_{\nc, R, E}^p(\Lambda))\longrightarrow\overline C_E^*( \Lambda, \Omega_{\nc, R, E}^{p+1}(\Lambda)), \quad f \longmapsto \mathbf{1}_{s\overline{\Lambda}}\otimes_E f. \end{align}

 We have the natural ($\mathbb{k}$-linear) projections
$${\varpi}^{m}\colon (s\overline{\overline \Lambda})^{\otimes m}\longrightarrow (s\overline{\Lambda})^{\otimes_E  m}, \quad \mbox{for  all } m\geq 0.$$
  Denote by  $t_p$ the natural injection
   $$\Omega_{\nc, R,E}^p(\Lambda)\hooklongrightarrow \Omega^p_{\nc, R}(\Lambda),$$
 induced by $\xi$. We have  inclusions
{\small
\begin{align*}
{\rm Hom}_{\text{$E$-$E$}}((s\overline{\Lambda})^{\otimes_E m+p}, \Omega^p_{\nc, R, E}(\Lambda))
&\hooklongrightarrow {\rm Hom}_{}((s\overline{\overline \Lambda})^{\otimes m+p}, \Omega^p_{\nc, R, E}(\Lambda))\hooklongrightarrow {\rm Hom}_{}((s\overline{\overline \Lambda})^{\otimes m+p}, \Omega^p_{\nc, R}(\Lambda)),
\end{align*}
}
where the first inclusion is induced by the projection $\varpi^{m+p}$, and  the second one  is given by ${\rm Hom}((s\overline{\overline \Lambda})^{\otimes m+p}, t_{p})$.   Therefore,  we have the injection $$\overline{C}^m_E(\Lambda, \Omega^p_{\nc, R,E}(\Lambda))\hooklongrightarrow \overline{C}^m(\Lambda, \Omega^p_{\nc, R}(\Lambda)).$$

 For any $m\in \mathbb Z$, we have the following commutative diagram.
\begin{equation*}
\xymatrix@C=1.5pc{
\overline{C}^m_E(\Lambda, \Lambda)\ar[r]^-{\theta_{0, R,E}}\ar@{^{(}->}[d]& \overline{C}^m_E(\Lambda, \Omega^1_{\nc, R,E}(\Lambda)) \ar@{^{(}->}[d]\ar[r]^-{\theta_{1, R,E}} & \cdots \ar[r]^-{} & \overline{C}^m_E(\Lambda, \Omega^p_{\nc, R,E}(\Lambda)) \ar@{^{(}->}[d]\ar[r]^-{\theta_{p, R,E}} & \cdots \\
\overline{C}^m(\Lambda, \Lambda)\ar[r]^-{\theta_{0,R}}& \overline{C}^m(\Lambda, \Omega^1_{\nc, R}(\Lambda)) \ar[r]^-{\theta_{1,R}} & \cdots\ar[r]^-{} & \overline{C}^m(\Lambda, \Omega^p_{\nc, R}(\Lambda))   \ar[r]^-{\theta_{p,R}} & \cdots
}\end{equation*}
It gives rise to an  injection of complexes
$$\iota\colon  \overline{C}_{\sg, R,  E}^*(\Lambda, \Lambda)\hooklongrightarrow \overline{C}_{\sg, R}^*(\Lambda, \Lambda). $$
We observe that the cup product and the brace operation on $\overline{C}_{\sg, R}^*(\Lambda, \Lambda)$ restrict to  $\overline{C}_{\sg, R, E}^*(\Lambda, \Lambda)$. Thus $\overline{C}_{\sg, R, E}^*(\Lambda, \Lambda)$ inherits a brace  $B_{\infty}$-algebra structure.

\begin{lem}\label{lemma7.3-split}
The injection $\iota\colon  \overline{C}_{\sg, R,  E}^*(\Lambda, \Lambda)\hookrightarrow \overline{C}_{\sg, R}^*(\Lambda, \Lambda)$ is a strict $B_{\infty}$-quasi-isomorphism.
\end{lem}

 \begin{proof}
Since $\iota$ preserves the cup products and brace operations,  it follows from Lemma~\ref{lemma-infinity-morphism1} that  $\iota$ is a strict  $B_{\infty}$-morphism.

It remains to prove that $\iota$ is a quasi-isomorphism of complexes. The injection $\xi\colon \overline{\Lambda} \rightarrow \overline{\overline \Lambda}$ induces an injection of complexes of $\Lambda$-$\Lambda$-bimodules
 $$\overline{\Barr}_E(\Lambda)\hooklongrightarrow \overline{\Barr}(\Lambda)= \bigoplus_{n\geq 0} \Lambda \otimes (s\overline{\overline \Lambda})^{\otimes n}\otimes \Lambda.$$
Recall that  $\Omega_{\nc, R}^p(\Lambda)$ is isomorphic to the cokernel of the external differential $d_{ex}$ in $\overline{\Barr}(\Lambda)$ and that $\Omega_{\nc, R, E}^p(\Lambda)$ is isomorphic to the cokernel of $d_{ex}$ in $\overline{\Barr}_E(\Lambda)$.
We infer that both $\overline{C}_{\sg, R,  E}^*(\Lambda, \Lambda)$ and $\overline{C}_{\sg, R}^*(\Lambda, \Lambda) $ compute $\HH_{\sg}^*(\Lambda, \Lambda)$; compare \cite[Theorem~3.6]{Wan1}. Therefore, the injection $\iota$ is a quasi-isomorphism.
 \end{proof}

Similar, we define  the {\it $E$-relative left singular Hochschild cochain complex} $\overline{C}_{\sg, E, L}^*( \Lambda, \Lambda)$ as the colimit of the inductive system
$$ \overline{C}_E^*(\Lambda, \Lambda)\xrightarrow{\theta_{0, L, E}} \overline{C}_E^*(\Lambda, \Omega_{\nc, L, E}^1(\Lambda))\xrightarrow{\theta_{1, L, E}} \cdots \xrightarrow{\theta_{p-1, L, E}}  \overline{C}_E^*(\Lambda, \Omega_{\nc, L, E}^p(\Lambda))\xrightarrow{\theta_{p, L}} \dotsc, $$
where $\Omega_{\nc, L, E}^p(\Lambda)=\Lambda \otimes_E (s\overline \Lambda)^{\otimes_E p}$ is the {\it graded $\Lambda$-$\Lambda$-bimodule of $E$-relative left noncommutative differential $p$-forms} and the maps
\begin{align}\label{equ:thetaLE}
\theta_{p, L, E}\colon   \overline{C}_E^*(\Lambda, \Omega_{\nc, L, E}^p(\Lambda))\longrightarrow \overline{C}_E^*(\Lambda, \Omega_{\nc, L, E}^{p+1}(\Lambda)), \quad f \longmapsto f\otimes_E \mathbf{1}_{s\overline{\Lambda}}.
\end{align}

We have an analogous result of Lemma~\ref{lemma7.3-split}.

\begin{lem}\label{lem:L-inclu}
There is a natural injection $ \overline{C}_{\sg, L,  E}^*(\Lambda, \Lambda)\hookrightarrow \overline{C}_{\sg, L}^*(\Lambda, \Lambda)$, which is a strict $B_{\infty}$-quasi-isomorphism. \hfill $\square$
\end{lem}

\subsection{The brace operation on the right singular Hochschild cochain complex}\label{subsec:example-brace}

We will recall the  brace operation $-\{-, \dotsb, -\}_R$  on $\overline{C}^*_{\sg, R}(\Lambda, \Lambda)$.
It might be carried over word by word from the left case, studied in  \cite[Section 5]{Wan1}, but with different graph presentations. We mention that, similar to the left case,  the brace operation
$-\{-, \dotsb, -\}_R$ is  induced from a natural action of the cellular chain dg operad of the spineless cacti operad \cite{Ka}.

Similar to  \cite[Figure 1]{Wan1}, any element $$f\in \overline{C}^{m-p}(\Lambda, \Omega_{\nc, R}^p(\Lambda))=\Hom((s\overline \Lambda)^{\otimes m}, (s\overline\Lambda)^{\otimes p}\otimes \Lambda)$$ can be depicted by a tree-like graph and a cactus-like graph (cf. Figure \ref{Tree-presentation}):
\begin{itemize}
\item   The tree-like presentation is the usual graphic presentation of morphisms in tensor categories (cf. e.g. \cite{JoSt}). We read the graph from top to bottom and left to right. We use the color blue to distinguish the {\it special output} $\Lambda$ and the other black outputs represent $s \overline \Lambda$.  The inputs $(s\overline{\Lambda})^{\otimes m}$ are ordered from left to right at the top but are labelled by $1, 2, \dotsc, m$ from right to left. Similarly,  the outputs $(s\overline{\Lambda})^{\otimes p}\otimes \Lambda$ are ordered  from left to right at the bottom but are labelled by $0, 1, 2, \dotsc, p$ from right to left. The above labelling is convenient when taking the colimit \eqref{colimitrepresent}; see Figure~\ref{Theta-map}.

\item The cactus-like presentation is read as follows. The image of $0\in \mathbb R$ in the red circle $S^1=\mathbb R/\mathbb Z$ is decorated by a blue dot, called the zero point of $S^1$. The center of $S^1$ is decorated by $f$. The blue radius  represents the special output $\Lambda$. The inputs $(s\overline{\Lambda})^{\otimes m}$ are represented by $m$ black radii (called {\em inward radii}) on the right semicircle pointing  towards the center in clockwise. Similarly, the outputs $(s\overline \Lambda)^{\otimes p}$ are represented by $p$ black radii (called {\em outward radii}) on the left semicircle pointing outwards the center in counterclockwise. The cactus-like presentation is inspired by the spineless cacti operad.
\end{itemize}

 \begin{figure}[h]
\centering{
 \includegraphics[height=40mm]{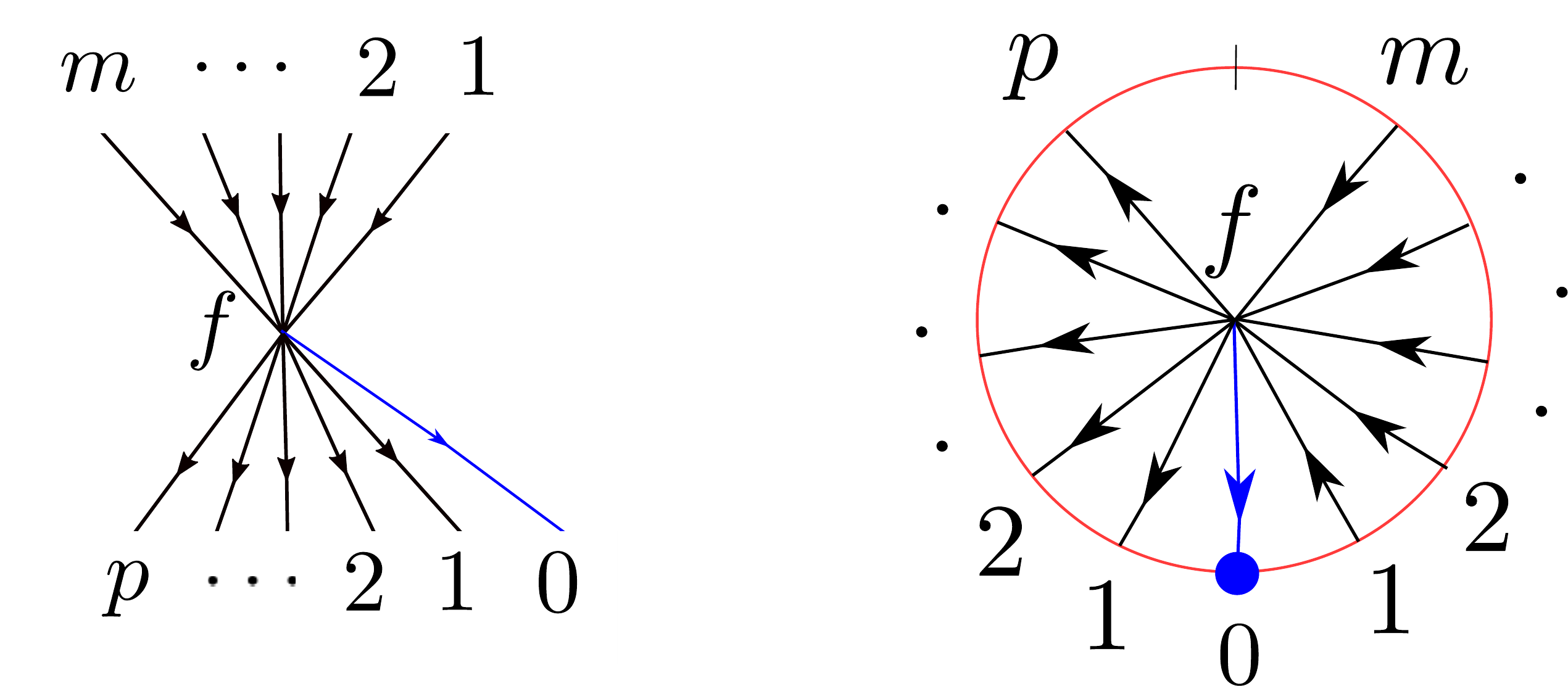}
 \caption{The tree-like and cactus-like presentations of $f\in \overline{C}^{m-p}(\Lambda, \Omega_{\nc, R}^p(\Lambda))$.}
  \label{Tree-presentation}
  }
\end{figure}

Recall that the  maps in the inductive system (\ref{equation-inductive}) of $\overline{C}_{\sg, R}^*(\Lambda, \Lambda)$ are given by $$\theta_{p, R}\colon   \overline{C}^*(\Lambda, \Omega_{\nc, R}^p(\Lambda))\stackrel{}\longrightarrow  \overline{C}^*(\Lambda, \Omega_{\nc, R}^{p+1}(\Lambda)), \quad f \longmapsto \mathbf 1\otimes f.$$ That is, for any $f\in \overline{C}^*(\Lambda,\Omega_{\nc, R}^p(\Lambda))$ we have \begin{align}\label{colimitrepresent}
f=\mathbf 1\otimes f = \mathbf 1^{\otimes 2}\otimes f =\cdots= \mathbf 1^{\otimes m}\otimes f =\cdots
\end{align}
 in  $\overline{C}_{\sg, R}^*(\Lambda, \Lambda)$. Thus, any element $f\in \overline{C}_{\sg, R}^{m-p}(\Lambda, \Lambda)$ is depicted by Figure \ref{Theta-map}, where the straight line represents the identity map of $s\overline \Lambda$.  Thanks to \eqref{colimitrepresent}, we can freely add or remove the straight lines from the left side and from the top, respectively.

The tree-like and cactus-like presentations have their own advantages: it is much easier to read off the corresponding morphisms from the tree-like presentation (as we have seen from tensor categories), while it is more convenient to construct the brace operation using the cactus-like presentation as you will see in the sequel.

\begin{figure}[h]
\centering{
 \includegraphics[height=40mm]{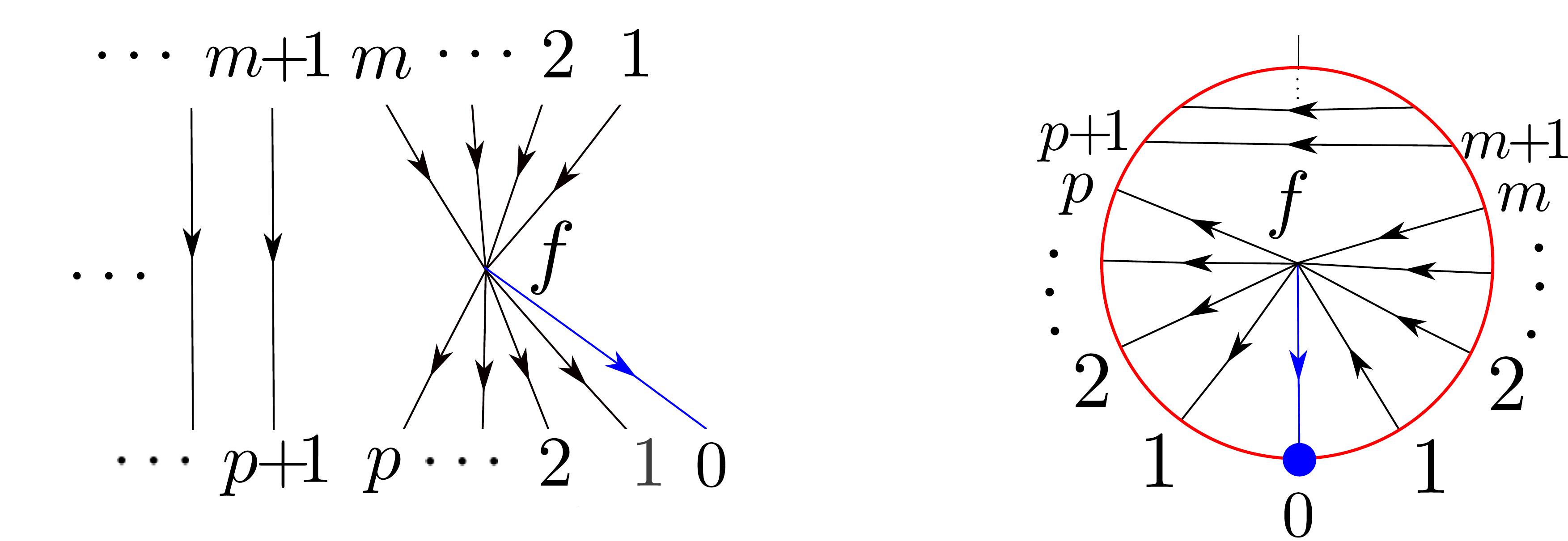}
 \caption{The colimit maps $\theta_{*, R}$, where the straight line represents  the identity map of $s \overline \Lambda$.}
  \label{Theta-map}
  }
\end{figure}

For any $k \geq 0$, let us define the brace operation of degree $-k$
$$-\{-, \dotsc, -\}_R\colon   \overline{C}_{\sg, R}^*(\Lambda, \Lambda)\otimes \overline{C}_{\sg, R}^*(\Lambda, \Lambda)^{\otimes k} \stackrel{}\longrightarrow  \overline{C}_{\sg, R}^{*}(\Lambda, \Lambda).$$

\begin{defn}\label{defn:brace-A}
Let  $x\in \overline{C}^{m-p}(\Lambda, \Omega_{\nc, R}^p(\Lambda))$ and $y_i\in \overline{C}^{n_i-q_i}(\Lambda, \Omega_{\nc, R}^{q_i}(\Lambda))$ for $1\leq i\leq k$. Set $m'=m-p$ and $n_r'=n_r-q_r-1 $ for $1\leq r\leq k.$ Then we define $$x\{y_1, \dotsc, y_k\}_R\in \Hom((s\overline \Lambda)^{\otimes m+n_1+n_2+ \dotsb +  n_k-k},  \Omega_{\nc, R}^{ p+q_1+\cdots+q_k}(\Lambda))$$ as follows: for $k=0$, we set $x\{\emptyset\}=x$;  for $k\geq 1$, we set
\begin{equation}
x\{y_1, \dotsc, y_k\}_R=\sum\limits_{\substack{0\leq j\leq k\\1\leq i_1< i_2 <\cdots < i_j\leq m\\ 1\leq l_1\leq l_2\leq  \cdots \leq l_{k-j}\leq p }}(-1)^{k-j} B^{(i_1, \dotsc, i_j)}_{(l_1, \dotsc, l_{k-j})}(x; y_1, \dotsc, y_k),
\end{equation}
where the summand $B^{(i_1, \dotsc, i_j)}_{(l_1, \dotsc, l_{k-j})}(x; y_1, \dotsc, y_k)$ is illustrated in Figure \ref{Radical-brace-operation-1}; where the extra sign $(-1)^{k-j}$ is added in order to make sure that the brace operation is compatible with the colimit maps $\theta_{*, R}$.   When the operation $B^{(i_1, \dotsc, i_j)}_{(l_1, \dotsc, l_{k-j})}(x; y_1, \dotsc, y_k)$  applies to  elements, an additional sign $(-1)^{\epsilon}$ appears due to Koszul sign rule, where
\begin{align*}
\epsilon&:=\Big(m'+\sum_{i=1}^kn_i'\Big)\Big(p+\sum_{i=1}^k q_i\Big)+m'p+\sum_{i=1}^kn_i'q_i\\
&\quad +\sum_{r=1}^{k-j} (n_1'+\cdots+n_{r}'+l_{r}-1)n_{r}' +\sum_{s=1}^{j} (n_1'+\cdots+n_{k-s+1}'+m'-i_s-1)n_{k-s+1}'.\end{align*}
\end{defn}

\begin{figure}[h]
\centering
  \includegraphics[height=40mm]{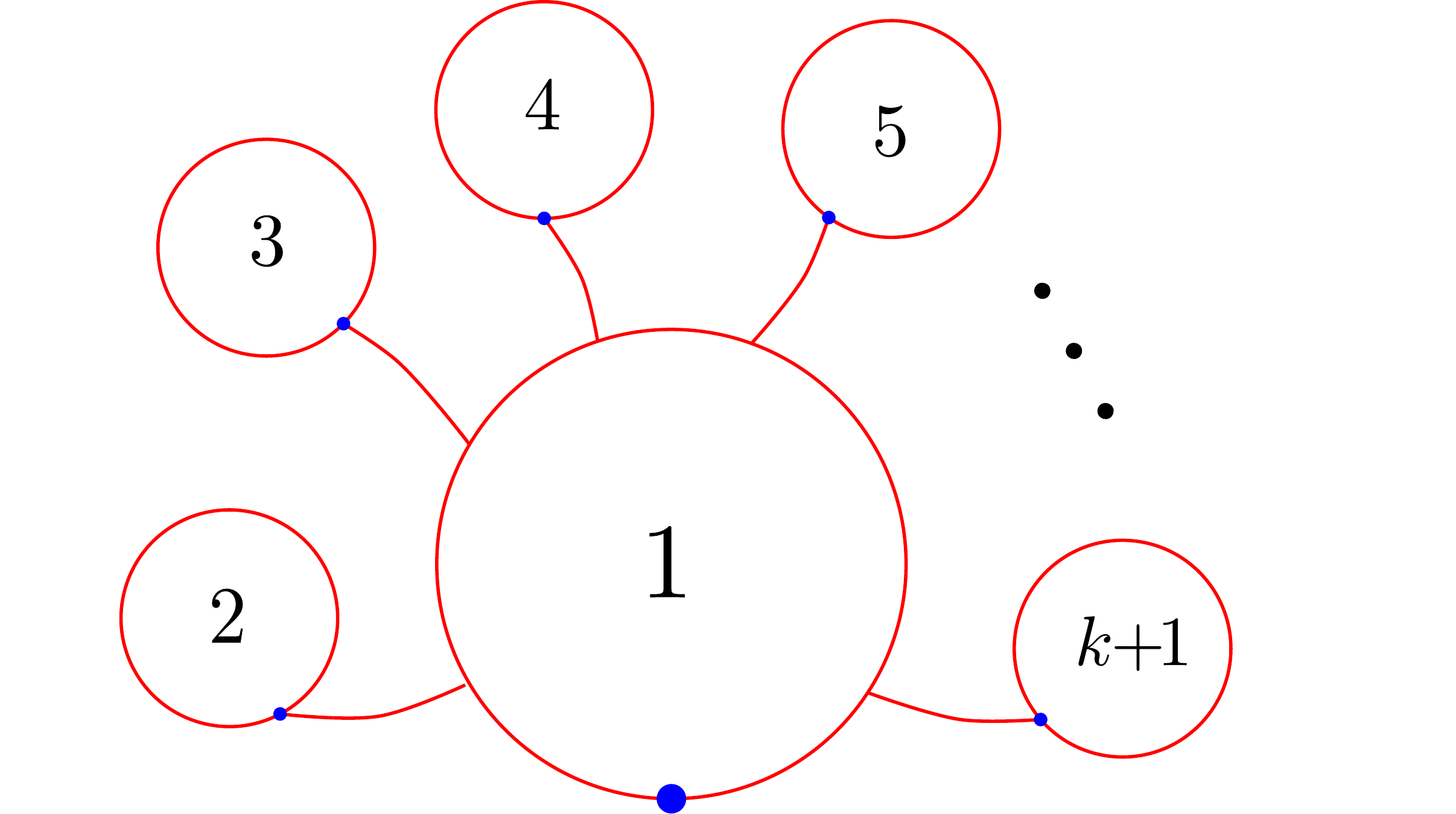}
  \caption{A cell in  the spineless cacti operad. }
  \label{Bracell-Sep}
\end{figure}

Let us now describe Figure \ref{Radical-brace-operation-1} in detail and   how to read off $B^{(i_1, \dotsc, i_j)}_{(l_1, \dotsc, l_{k-j})}(x; y_1, \dotsc, y_k)$.
\begin{enumerate}[(i)]
\item We start with  the cell depicted in    Figure~\ref{Bracell-Sep} of the spineless cacti operad. As in Figure~\ref{Theta-map}, we use the element $x$ to decorate  the circle $1$ of Figure~\ref{Bracell-Sep} and similarly use the element $y_i$ to decorate  the circle $i+1$ for $1\leq i\leq k$.

 \item The left semicircle of the circle $1$  is divided into $p+1$ arcs by the outward radii of $x$. For each $1\leq r \leq k-j$, the red curve of the circle $r$ (decorated by $y_r$)  intersects with the circle $1$   at the open arc   between the   $(l_{r}-1)$-th   and $l_{r}$-th outward radii of $x$. The red curves are  not allowed to intersect with each other.
  \item On the right semicircle of the circle $1$, we have $m$ intersection points of the $m$ inward radii of $x$ with the circe $1$. Unlike (ii),  for each $1\leq r\leq j$ the red curve of the  circle $k-r+1$ (decorated by $y_{k-r+1}$) intersects with the  circle $1$  exactly at the  $i_r$-th intersection point.
 \item We  connect some inputs with outputs using the following rule.
 \begin{itemize}
 \item For each $1\leq r\leq j$, connect the blue output of $y_{k-r+1}$ with the  $i_r$-th inward radius of the circle $1$ on the right semi-circle of the circe $1$. Then starting from the blue dot (i.e. the zero point) of circle $1$, walk counterclockwise along the red path (i.e.  the outside of the red circles and the red  curves)  and record the inward  and outward radii  (including the blue radii) in order as a sequence $\mathcal S$. When an outward radius is found closely behind an inward radius in $\mathcal S$, we call this pair {\it in-out}.
\item Let us define the following operation.
\vskip 3pt

{\bf Deletion Process:} Once the pair in-out appears in the sequence $\mathcal S$, we connect the outward radius with the inward radius  via a dashed arrow in Figure \ref{Radical-brace-operation-1}.  Delete this pair and renew the sequence $\mathcal S$. Then repeat the above operations iteratively until no pair in-out left in $\mathcal S$.
 \end{itemize}

 \item After applying the above Deletion Process, we obtain a final sequence $\mathcal S$ with all outward radii  preceding all inward radii. Finally, we translate the updated cactus-like graph  into a tree-like graph by putting the inputs (in the final sequence) on the top and outputs on the bottom.  We therefore get the $\mathbb k$-linear map
     $$B^{(i_1, \dotsc, i_j)}_{(l_1, \dotsc, l_{k-j})}(x; y_1, \dotsc, y_k)\colon  (s\overline{\Lambda})^{\otimes u}\longrightarrow  (s\overline\Lambda)^{\otimes v}\otimes \Lambda,$$ where $u$ and $v$ are respectively the numbers of the inward radii and  outward radii  in the final sequence $\mathcal S$. See  Example \ref{Example-brace1} below.
 \end{enumerate}

 \begin{figure}[h]
\centering
  \includegraphics[height=80mm]{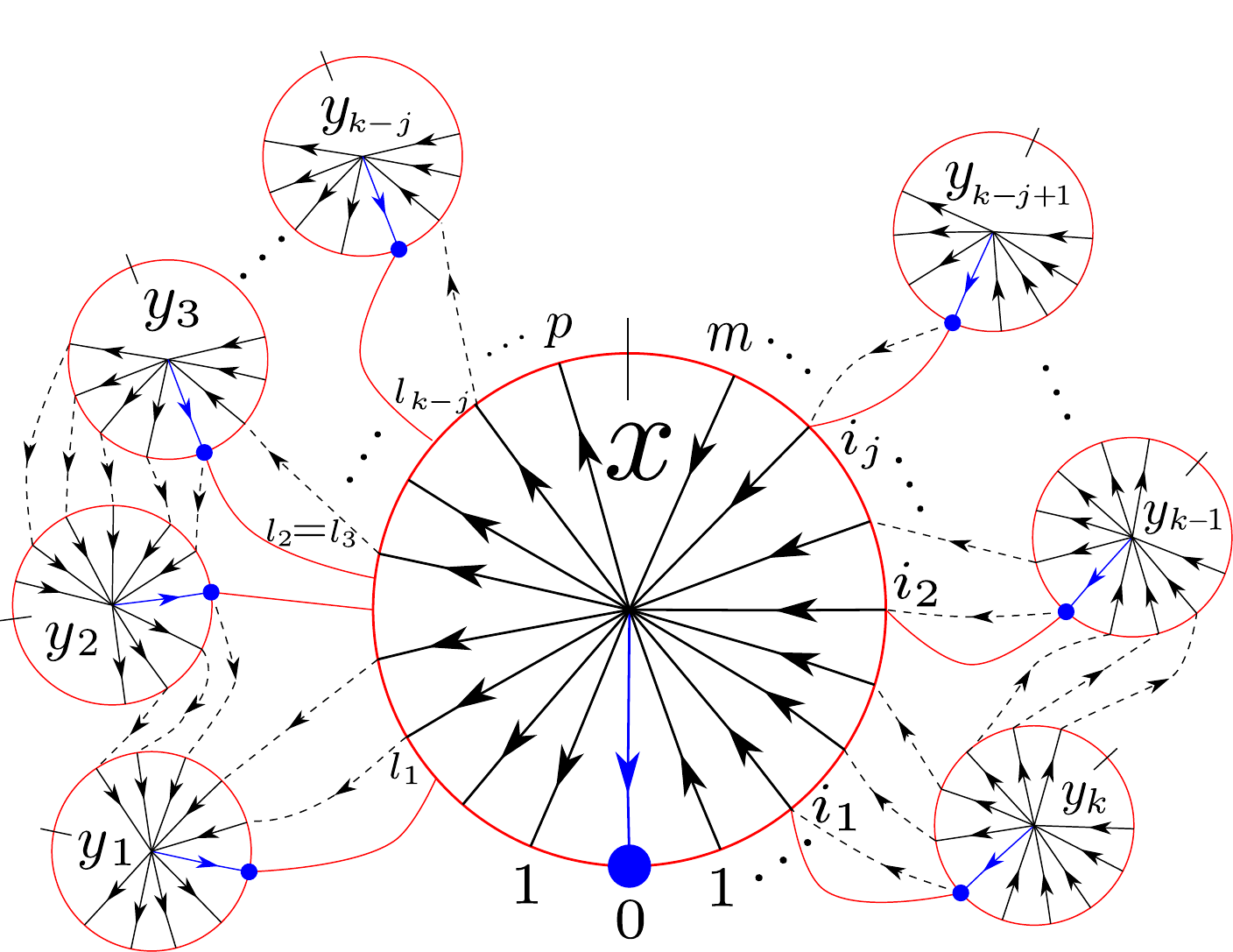}
  \caption{The summand $B^{(i_1, \dotsc, i_j)}_{(l_1, \dotsc, l_{k-j})}(x; y_1, \dotsc, y_k)$ of  $x\{y_1, \dotsc, y_k\}_R$. }
  \label{Radical-brace-operation-1}
\end{figure}

 Note that $x\{y_1, \dotsc, y_k\}_R$ is compatible with the colimit maps $\theta_{*, R}$ and thus it induces a  well-defined operation (still denoted by $-\{-, \dotsc, -\}_R$)  on $\overline{C}_{\sg, R}^*(\Lambda, \Lambda)$.
When $p=q_1=\cdots=q_k=0$, the above $x\{y_1, \dotsc, y_k\}_R$ coincides with the usual brace operation on $\overline{C}^*(\Lambda, \Lambda)$; compare (\ref{equation:brace}).

\begin{exm}\label{Example-brace1}Let
\begin{align*}
f & \in \overline{C}^{2}(\Lambda, \Omega^3_{\nc, R}(\Lambda))= \Hom((s\overline \Lambda)^{\otimes 5}, (s\overline\Lambda)^{\otimes 3}\otimes \Lambda)\\
 g_1 & \in \overline{C}^{2}(\Lambda, \Omega_{\nc, R}^1(\Lambda))=\Hom((s\overline \Lambda)^{\otimes 3}, s\overline\Lambda\otimes \Lambda)\\
  g_2  & \in \overline{C}^{0}(\Lambda, \Omega_{\nc, R}^3(\Lambda))=\Hom((s\overline \Lambda)^{\otimes 3}, (s\overline\Lambda)^{\otimes 3}\otimes \Lambda) \\
   g_3  & \in \overline{C}^{-1}(\Lambda, \Omega_{\nc, R}^3(\Lambda))=\Hom((s\overline \Lambda)^{\otimes 2}, (s\overline\Lambda)^{\otimes 3}\otimes \Lambda).
\end{align*}
 Then the operation $B_{(2)}^{(2, 4)}(f; g_1, g_2, g_3)$ is depicted in Figure~\ref{Brace-action-Sep}. It is represented by the following composition of maps (Here, we ignore the identity map $\mathbf 1_{s\overline \Lambda}^{\otimes 6}$ on the left)
$$
(\mathbf 1_{s\overline \Lambda}\otimes \overline g_1\otimes \mathbf 1_{s\overline \Lambda}\otimes \mathbf 1_{\Lambda})(\mathbf 1_{s\overline \Lambda}^{\otimes 2}\otimes f)(\overline g_2\otimes \mathbf 1_{s\overline \Lambda}^{\otimes 3})(\mathbf 1_{s\overline \Lambda} \otimes \overline  g_3\otimes \mathbf 1_{s\overline \Lambda}) \colon   (s\overline\Lambda)^{\otimes 4}\stackrel{}\longrightarrow  (s\overline\Lambda)^{\otimes 4}\otimes \Lambda
$$
where $\overline g\colon (s\overline\Lambda)^{\otimes m}\xrightarrow{g} (s\overline\Lambda)^{\otimes p}\otimes \Lambda\xrightarrow{\mathbf 1_{s\overline \Lambda}^{\otimes p}\otimes \pi} (s\overline{\Lambda})^{\otimes p+1}$ and $\pi\colon \Lambda \to s\overline \Lambda$ is the natural projection $a \mapsto s\overline a$ of degree $-1$.
\end{exm}

\begin{figure}[h]
\centering
  \includegraphics[height=53mm]{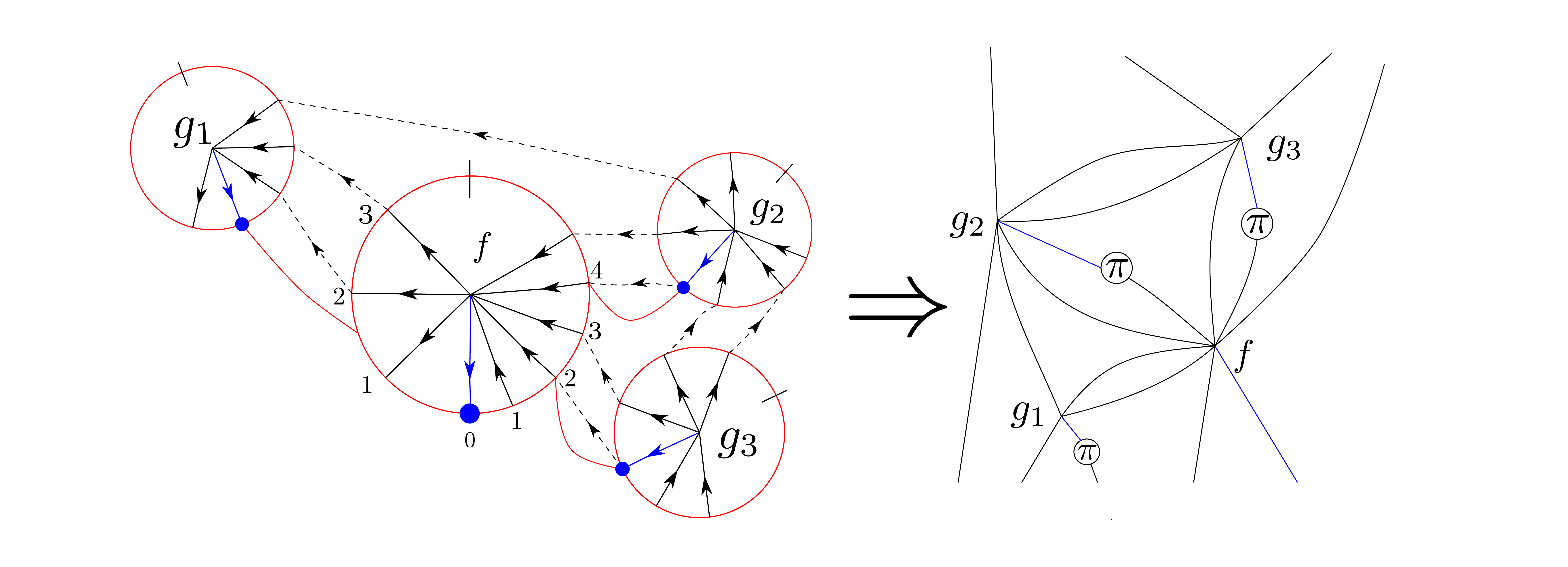}
  \caption{An example of  $B_{(2)}^{(2, 4)}(f; g_1, g_2, g_3)$.}
  \label{Brace-action-Sep}
\end{figure}

\section{$B_\infty$-quasi-isomorphisms induced by  one-point (co)extensions and bimodules}

\label{sectionforinvariance}

In this section, we prove that the (relative) singular Hochschild cochain complexes, as  $B_\infty$-algebras, are invariant under one-point (co)extensions of algebras and singular equivalences with levels.

 These invariance results are analogous to the ones in Subsection~\ref{subsec:one-point}. However, the proofs here are much harder, since the colimit construction of the singular Hochschild cochain complex is involved.

Throughout this section, $\Lambda$ and $\Pi$ will be finite dimensional $\mathbb{k}$-algebras.

\subsection{Invariance under one-point (co)extensions}

Let $E=\bigoplus_{i=1}^n  \mathbb{k}e_i\subseteq \Lambda$ be a semisimple subalgebra of $\Lambda$. Set $\overline{\Lambda}=\Lambda/{(E\cdot 1_\Lambda)}$. We have the $B_\infty$-algebra $\overline{C}_{\sg, R,  E}^*(\Lambda, \Lambda)$ of the  $E$-relative right singular Hochschild cochain complex of $\Lambda$.

Consider the one-point coextension $\Lambda'=\begin{pmatrix} \mathbb{k} & M \\ 0 & \Lambda\end{pmatrix}$ in Subsection~\ref{subsec:one-point}. Set $e'=\begin{pmatrix} 1 & 0\\ 0 & 0\end{pmatrix}$, and identify $\Lambda$ with $(1_{\Lambda'}-e')\Lambda'(1_{\Lambda'}-e')$.  We take $E'=\mathbb{k}e'\oplus E$, which is a semisimple subalgebra of $\Lambda'$. Set $\overline{\Lambda'}=\Lambda'/{(E'\cdot 1_{\Lambda'})}$.

 To consider the $E'$-relative  right singular Hochschild cochain complex $\overline{C}_{\sg, R,  E'}^*(\Lambda', \Lambda')$, we naturally  identity $\overline{\Lambda'}$ with $\overline{\Lambda}\oplus M$. Then we have a natural isomorphism for each $m \geq 1$
 \begin{align}\label{equ:nat-opce}
 (s\overline{\Lambda'})^{\otimes_{E'}m}\simeq (s\overline{\Lambda})^{\otimes_{E}m}\oplus sM\otimes_E (s\overline{\Lambda})^{\otimes_{E}m-1},
 \end{align}
 where we use the fact that $s\overline{\Lambda'}\otimes_{E'}sM=0$.
The following decomposition follows immediately from \eqref{equ:nat-opce}.
 \begin{equation}
\begin{split}
&\Hom_{\text{$E'$-$E'$}}((s \overline{\Lambda'})^{\smallotimes_{E'} m}, (s\overline{\Lambda'} )^{\smallotimes_{E'} p}\smallotimes_{E'} \Lambda') \\
\simeq{} &   \Hom_{\text{$E$-$E$}}((s \overline{\Lambda})^{\smallotimes_{E} m}, (s\overline{\Lambda} )^{\smallotimes_{E} p}\smallotimes_{E} \Lambda)  \oplus  \Hom_{\text{$\mathbb{k}$-$E$}}(sM \smallotimes_E (s\overline{\Lambda})^{\smallotimes_E {m-1}},  sM \smallotimes_E  (s\overline{\Lambda})^{\smallotimes_E{p-1}}\smallotimes_E \Lambda) \nonumber
\end{split}
\end{equation}
We take the colimits along $\theta_{p, R, E'}$ for $\Lambda'$,  and along $\theta_{p, R, E}$ for $\Lambda$ in (\ref{equ:thetaRE}). Then the above decomposition yields a restriction of complexes
$$\overline{C}_{\sg, R, E'}^*(\Lambda', \Lambda') \twoheadrightarrow \overline{C}_{\sg, R, E}^*(\Lambda, \Lambda).$$
It is routine to check that the above restriction preserves the cup products and brace operations, i.e. it is a strict $B_\infty$-morphism.

The following two lemmas show the invariance of the left and right singular Hochschild cochain complexes under one-point coextensions.

\begin{lem}\label{lem:strict-R}
Let $\Lambda'$ be the one-point coextension as above. Then the restriction map $\overline{C}_{\sg, R, E'}^*(\Lambda', \Lambda') \twoheadrightarrow \overline{C}_{\sg, R, E}^*(\Lambda, \Lambda)$ is a strict $B_\infty$-isomorphism.
\end{lem}

\begin{proof}
The crucial fact is that $s\overline{\Lambda'}\otimes_{E'}sM=0$. Then by the very definition, $\theta_{p, R, E'}$ vanishes on the following component
$$\Hom_{\text{$\mathbb{k}$-$E$}}(sM \otimes_E (s\overline{\Lambda})^{\otimes_E {m-1}},  sM \otimes_E  (s\overline{\Lambda})^{\otimes_E{p-1}}\otimes_E \Lambda).$$
It follows that taking the colimits, the restriction  becomes an actual  isomorphism.
\end{proof}

We now consider the $E$-relative left  singular Hochschild cochain complex $\overline{C}_{\sg, L,  E}^*(\Lambda, \Lambda)$, and the $E'$-relative left  singular Hochschild cochain complex $\overline{C}_{\sg, L,  E'}^*(\Lambda', \Lambda')$. Using the natural isomorphism (\ref{equ:nat-opce}), we have a decomposition
 \begin{align}\label{decompositioncoextension}
&\Hom_{\text{$E'$-$E'$}}((s \overline{\Lambda'})^{\smallotimes_{E'} m}, \Lambda'\smallotimes_{E'}(s\overline{\Lambda'} )^{\smallotimes_{E'} p}) \nonumber \\
\simeq{} &   \Hom_{\text{$E$-$E$}}((s \overline{\Lambda})^{\smallotimes_{E} m}, \Lambda\smallotimes_E (s\overline{\Lambda} )^{\smallotimes_{E} p})  \oplus  \Hom_{\text{$\mathbb{k}$-$E$}}(sM \smallotimes_E (s\overline{\Lambda})^{\smallotimes_E {m-1}}, \mathbb ke'\smallotimes  sM \smallotimes_E  (s\overline{\Lambda})^{\smallotimes_E{p-1}})\nonumber\\
& \oplus  \Hom_{\text{$\mathbb{k}$-$E$}}(sM \smallotimes_E (s\overline{\Lambda})^{\smallotimes_E {m-1}}, M \smallotimes_E  (s\overline{\Lambda})^{\smallotimes_E{p}}). \end{align}
Similar as above, the decomposition will give rise to a restriction of complexes
$$\overline{C}_{\sg, L, E'}^*(\Lambda', \Lambda') \twoheadrightarrow  \overline{C}_{\sg, L, E}^*(\Lambda, \Lambda) ,$$
which is a strict $B_\infty$-morphism.

Unlike the isomorphism in Lemma~\ref{lem:strict-R}, this restriction is only a quasi-isomorphism.

\begin{lem}\label{lem:strict-L-quasi}
Let $\Lambda'$ be the one-point coextension. Then the above  restriction map $ \overline{C}_{\sg, L, E'}^*(\Lambda', \Lambda') \twoheadrightarrow  \overline{C}_{\sg, L, E}^*(\Lambda, \Lambda)$ is a strict $B_\infty$-quasi-isomorphism.
\end{lem}

\begin{proof}
It suffices to show that the kernel  of the restriction map is acyclic. For this, we observe that  the decomposition \eqref{decompositioncoextension} induces a decomposition of {\it graded vector spaces}
\begin{align}\label{decompositionsg}
\overline{C}_{\sg, L, E'}^*(\Lambda', \Lambda') \simeq  \overline{C}_{\sg, L, E}^*(\Lambda, \Lambda) \oplus X^* \oplus Y^*.
\end{align}
Here, the $(m-p)$-th component $X^{m-p}$ of $X^*$ is the colimit along the maps
$$
 \Hom_{\text{$\mathbb k $-$E$}}(sM \smallotimes_E (s\overline{\Lambda})^{\smallotimes_E {m-1}}, \mathbb k  e'\smallotimes  sM \smallotimes_E  (s\overline{\Lambda})^{\smallotimes_E{p-1}}) \to  \Hom_{\text{$\mathbb{k}$-$E$}}(sM \smallotimes_E (s\overline{\Lambda})^{\smallotimes_E {m}}, \mathbb k  e'\smallotimes  sM \smallotimes_E  (s\overline{\Lambda})^{\smallotimes_E{p}})
$$
which sends $f$ to $f \otimes_E \mathbf {1}_{s\overline\Lambda}$. Similarly,  the $(m-p)$-th component $Y^{m-p}$ of $Y^*$ is the colimit along the maps
$$
 \Hom_{\text{$\mathbb{k}$-$E$}}(sM \smallotimes_E (s\overline{\Lambda})^{\smallotimes_E {m-1}},   M \smallotimes_E  (s\overline{\Lambda})^{\smallotimes_E{p}}) \to  \Hom_{\text{$\mathbb{k}$-$E$}}(sM \smallotimes_E (s\overline{\Lambda})^{\smallotimes_E {m}}, M \smallotimes_E  (s\overline{\Lambda})^{\smallotimes_E{p+1}})
$$
sending $f$ to $f \otimes_E \mathbf{1}_{s\overline\Lambda}$.

We observe that $X^*$ is, as a graded vector space, isomorphic to the $1$-shift of $Y^*$ by identifying $\mathbb k  e'\otimes sM$ with $sM$. Then we have
\begin{equation}\label{secondthirdcomponent}
X^* \simeq \Sigma(Y^*).
\end{equation}

The differential of $\overline{C}_{\sg, L, E'}^*(\Lambda', \Lambda')$ induces a differential on the decomposition \eqref{decompositionsg}. Namely we have the following commutative diagram.
\begin{equation}
\label{comcom}
\xymatrix@C=0.9pc@R=4.5pc{
\Hom_{\text{$E'$-$E'$}}(s{\overline{\Lambda'}}^{\smallotimes_{E'} m}, \Lambda'\smallotimes_{E'} s{\overline{\Lambda'}}^{\smallotimes_{E'} p})
 \ar[d]_-{\delta_{\Lambda'}}\ar[r]^-{\sim} & {\begin{matrix} \Hom_{\text{$E$-$E$}}((s \overline{\Lambda})^{\smallotimes_{E} m}, \Lambda\smallotimes_E (s\overline{\Lambda} )^{\smallotimes_{E} p}) \\  \oplus \Hom_{\text{$\mathbb k $-$E$}}(sM  \smallotimes_E  (s\overline{\Lambda})^{\smallotimes_E {m-1}}, \mathbb k  e'\smallotimes  sM  \smallotimes_E   (s\overline{\Lambda})^{\smallotimes_E{p-1}})  \\ \oplus  \Hom_{\text{$\mathbb k $-$E$}}(sM \smallotimes_E (s\overline{\Lambda})^{\smallotimes_E {m-1}}, M \smallotimes_E  (s\overline{\Lambda})^{\smallotimes_E{p}})\end{matrix}} \ar[d]^-{\tiny{\begin{pmatrix}
\delta_\Lambda & 0 &0 \\
0 & \Sigma(\delta_{Y}) &0 \\
\widetilde \delta & \theta & \delta_Y
\end{pmatrix}}} \\
\Hom_{\text{$E'$-$E'$}}(s{\overline{\Lambda'}}^{\otimes_{E'} {m+1}}, \Lambda'\otimes_{E'}s\overline{\Lambda'}^{\otimes_{E'} {p}}) \ar[r]^-{\sim}& {\begin{matrix} \Hom_{\text{$E$-$E$}}((s \overline{\Lambda})^{\smallotimes_{E} m+1}, \Lambda\smallotimes_E (s\overline{\Lambda} )^{\smallotimes_{E} p}) \\ \oplus  \Hom_{\text{$\mathbb k $-$E$}}(sM \smallotimes_E (s\overline{\Lambda})^{\smallotimes_E {m}}, \mathbb k  e' \smallotimes  sM \smallotimes_E  (s\overline{\Lambda})^{\smallotimes_E{p-1}}) \\ \oplus  \Hom_{\text{$\mathbb k $-$E$}}(sM \smallotimes_E (s\overline{\Lambda})^{\smallotimes_E {m}}, M \smallotimes_E  (s\overline{\Lambda})^{\smallotimes_E{p}})\end{matrix}}}
 \end{equation}
 where we write elements in the decomposition \eqref{decompositionsg} as column vectors.

Let us explain the entries of the $3\times 3$-matrix in \eqref{comcom}.
\begin{enumerate}[(i)]
\item We observe that $\delta_{\Lambda'}$ restricts to a differential, denoted by $\delta_Y$, of $Y^*$. That is, $(Y^*, \delta_Y)$ is a cochain complex. The differential on the second component $X^*$ is given by $\Sigma(\delta_Y)$ under the natural isomorphism $X^* \simeq \Sigma(Y^*)$ in \eqref{secondthirdcomponent}. 

\item The differential $\delta_\Lambda$ is the external differential of $\overline C_{E}^*(\Lambda,\Lambda \otimes_{E}s\overline{\Lambda}^{\otimes_{E} {p}})$.
\item The map $\widetilde \delta$ is given as follows: for any $f \in \Hom_{E\text{-}E}((s \overline{\Lambda})^{\smallotimes_{E} m}, \Lambda\smallotimes_E (s\overline{\Lambda} )^{\smallotimes_{E} p})$, the element $\widetilde \delta (f) \in  \Hom_{\text{$\mathbb k $-$E$}}(sM \smallotimes_E (s\overline{\Lambda})^{\smallotimes_E {m}}, M \smallotimes_E  (s\overline{\Lambda})^{\smallotimes_E{p}})$ is defined by 
$$
\widetilde\delta(f) (sx \otimes_E s\overline{a}_{1, m}) = -(-1)^{m-p}x \otimes_\Lambda f(s\overline{a}_{1, m}).
$$
\item The map $\theta$ is given as follows: for any $f  \in \Hom_{\mathbb k\text{-}E}(sM \smallotimes_E (s\overline{\Lambda})^{\smallotimes_E {m-1}}, \mathbb k  e'\smallotimes  sM \smallotimes_E  (s\overline{\Lambda})^{\smallotimes_E{p-1}})$, the corresponding element  $\theta(f)\in \Hom_{\text{$\mathbb k $-$E$}}(sM \smallotimes_E (s\overline{\Lambda})^{\smallotimes_E {m}}, M \smallotimes_E  (s\overline{\Lambda})^{\smallotimes_E{p}})$ is defined by
$$
\theta(f)(sx \otimes_E s\overline{a}_{1, m}) =   f(sx \otimes_E s\overline{a}_{1, m-1})\otimes_E s\overline a_m.
$$
 Here,  we use the natural isomorphism  $\mathbb k  e' \otimes sM \rightarrow M$ of degree one, and thus $\theta$ is a map of degree one. We observe that after taking the colimits,  $\theta$ becomes the identity map
$$
{\bf 1}\colon X^*  \to Y^*, \quad \Sigma(y) \mapsto y
$$
using the identification \eqref{secondthirdcomponent}.

\end{enumerate}

 Thus, the kernel of the restriction map is identified with  the subcomplex
$$
 \left( X^* \oplus Y^*,\left(\begin{smallmatrix}
 \Sigma(\delta_Y) & 0\\
 \mathbf{1} & \delta_Y
\end{smallmatrix}\right)\right),
$$
which is exactly the mapping cone of the identity of the complex $(Y^*, \delta_Y)$.  It follows that this kernel is acyclic, as required.
\end{proof}

\begin{rem}
The decomposition \eqref{decompositionsg}  induces an embedding of graded vector spaces
$$ \overline{C}_{\sg, L, E}^*(\Lambda, \Lambda) \longrightarrow \overline{C}_{\sg, L, E'}^*(\Lambda', \Lambda').
$$
However, it is in general {\it not} a cochain map,  since the differential $\widetilde{\delta}$ in the matrix of \eqref{comcom} is nonzero.
\end{rem}

  Let us consider the one-point extension $\Lambda''=\begin{pmatrix} \Lambda & N \\
                                                                     0 & \mathbb{k}\end{pmatrix}$ in Subsection~\ref{subsec:one-point}.
  We set $e''=\begin{pmatrix}0 & 0\\ 0 & 1\end{pmatrix}$ and $E''=E\oplus \mathbb{k}e''\subseteq \Lambda''$. Set $\overline{\Lambda''}=\Lambda''/{(E''\cdot 1_{\Lambda''})}$, which is identified with $\overline{\Lambda}\oplus N$.

  We first consider the  $E$-relative left singular Hochschild cochain complex $\overline{C}_{\sg, L,  E}^*(\Lambda, \Lambda)$ and $E''$-relative left  singular Hochschild cochain complex $\overline{C}_{\sg, L,  E''}^*(\Lambda'', \Lambda'')$.

  The following result is analogous to Lemma~\ref{lem:strict-R}.
  \begin{lem}
  \label{onepointext}
  Let $\Lambda''$ be the one-point extension as above. Then we have a strict $B_\infty$-isomorphism
  $$  \overline{C}_{\sg, L, E''}^*(\Lambda'', \Lambda'')  \longrightarrow  \overline{C}_{\sg, L, E}^*(\Lambda, \Lambda).$$
  \end{lem}

  \begin{proof}
  The proof is completely similar to that of Lemma~\ref{lem:strict-R}. We have a similar decomposition
   \begin{equation}
\begin{split}
&\Hom_{\text{$E''$-$E''$}}((s \overline{\Lambda''})^{\smallotimes_{E''} m}, \Lambda'' \smallotimes_{E''} (s\overline{\Lambda''} )^{\smallotimes_{E''} p}) \\
\simeq{} &   \Hom_{\text{$E$-$E$}}((s \overline{\Lambda})^{\smallotimes_{E} m}, \Lambda \smallotimes_{E} (s\overline{\Lambda} )^{\smallotimes_{E} p})  \oplus  \Hom_{\text{$E$-$\mathbb{k}$}}((s\overline{\Lambda})^{\smallotimes_E {m-1}} \smallotimes_E sN, \Lambda \smallotimes_E  (s\overline{\Lambda})^{\smallotimes_E{p-1}}) \smallotimes_E sN). \nonumber
\end{split}
\end{equation}
We observe the crucial fact $sN\otimes_{E''}s\overline{\Lambda''}=0$. Then taking the colimit along $\theta_{p, L, E''}$ in (\ref{equ:thetaLE}), the above rightmost component will vanish. This gives rise to the desired $B_\infty$-isomorphism.
  \end{proof}

The following result is analogous to Lemma~\ref{lem:strict-L-quasi}. We omit the same argument.

\begin{lem}
Let $\Lambda''$ be the one-point extension as above.  Then the obvious restriction
  $$   \overline{C}_{\sg, R, E''}^*(\Lambda'', \Lambda'')  \longrightarrow  \overline{C}_{\sg, R, E}^*(\Lambda, \Lambda)$$
is a strict $B_\infty$-quasi-isomorphism. \hfill $\square$
\end{lem}

\subsection{$B_\infty$-quasi-isomorphisms induced by a bimodule}

We will prove that the $B_{\infty}$-algebra structures on  singular Hochschild cochain complexes are invariant under singular equivalences with  levels. Indeed, a slightly stronger statement will be established in Theorem~\ref{thm:quasi-iso-bimod}.

We fix a $\Lambda$-$\Pi$-bimodule $M$, over which $\mathbb{k}$ acts centrally. Therefore, $M$ is also viewed a left $\Lambda\otimes \Pi^{\rm op}$-module. We require further that the underlying left $\Lambda$-module $_\Lambda M$ and the right $\Pi$-module $M_\Pi$  are both projective.

Denote by $\mathbf{D}_{\rm sg}(\Lambda^e)$, $\mathbf{D}_{\rm sg}(\Pi^e)$ and $\mathbf{D}_{\rm sg}(\Lambda\otimes \Pi^{\rm op})$ the singularity categories of the algebras $\Lambda^e$, $\Pi^e$ and $\Lambda\otimes \Pi^{\rm op}$, respectively.  The projectivity assumption on $M$ guarantees that the following two triangle functors are well defined.
\begin{equation}
\label{twofunctors}
\begin{split}
-\otimes_\Lambda M \colon \mathbf{D}_{\rm sg}(\Lambda^e) \longrightarrow \mathbf{D}_{\rm sg}(\Lambda \otimes \Pi^{\op})\\
M\otimes_\Pi - \colon  \mathbf{D}_{\rm sg}(\Pi^e) \longrightarrow \mathbf{D}_{\rm sg}( \Lambda\otimes \Pi^{\op})
\end{split}
\end{equation}
The functor $-\otimes_\Lambda M $ sends $\Lambda$ to $M$, and $M\otimes_\Pi -$ sends $\Pi$ to $M$.
Consequently, they induce  the following maps
\begin{align} \label{equ:twomaps}
\HH_{\sg}^i(\Lambda, \Lambda) \stackrel{\alpha_{\sg}^i}{\longrightarrow} \Hom_{\mathbf{D}_{\rm sg}( \Lambda\otimes \Pi^{\op})}(M, \Sigma^i(M)) \stackrel{\beta_{\sg}^i}{\longleftarrow}\HH_{\sg}^i(\Pi, \Pi)
\end{align}
for all $i\in \mathbb{Z}$.  Here, we recall that the singular Hochschild cohomology groups are defined as
$$\HH_{\sg}^i(\Lambda, \Lambda)={\rm Hom}_{\mathbf{D}_{\rm sg}(\Lambda^e)}(\Lambda, \Sigma^i(\Lambda)) \quad \mbox{and} \quad \HH_{\sg}^i(\Pi, \Pi)={\rm Hom}_{\mathbf{D}_{\rm sg}(\Pi^e)}(\Pi, \Sigma^i(\Pi)).$$
Moreover, these groups are computed by the  the right singular Hochschild cochain complexes $\overline C_{\sg, R}^*(\Lambda, \Lambda)$  and $\overline C_{\sg, R}^*(\Pi, \Pi)$, respectively; see Subsection~\ref{subsec:sHcc} for details.

Under reasonable conditions, the bimodule $M$ induces an isomorphism between the above two right singular Hocschild cochain complexes.

\begin{thm}\label{thm:quasi-iso-bimod}
Let $M$ be a $\Lambda$-$\Pi$-bimodule such that it is projective both as a left $\Lambda$-module and as a right $\Pi$-module. Suppose that the two maps  in  (\ref{equ:twomaps}) are isomorphisms for each $i\in \mathbb{Z}$. Then we have an isomorphism
$$ \overline{C}^*_{\sg, R}(\Lambda, \Lambda) \simeq  \overline{C}^*_{\sg, R}(\Pi, \Pi) $$
in the homotopy category $\mathrm{Ho}(B_\infty)$ of $B_\infty$-algebras.
\end{thm}

We postpone the proof of Theorem \ref{thm:quasi-iso-bimod} until the end of this section, whose argument  is adapted from  the one developed in \cite{Kel03}; see also \cite{LoVa1}. We will consider a triangular matrix algebra $\Gamma$, using which we construct  two strict $B_\infty$-quasi-isomorphisms  connecting $\overline{C}^*_{\sg, R}(\Lambda, \Lambda)$ to $\overline{C}^*_{\sg, R}(\Pi, \Pi)$.

We now apply Theorem~\ref{thm:quasi-iso-bimod} to singular equivalences with levels, in which case the two maps  in  (\ref{equ:twomaps}) are indeed isomorphisms for each $i\in \mathbb{Z}$.

\begin{prop}\label{prop:sing-equi2}
Assume that  $(M, N)$ defines a singular equivalence with level $n$ between $\Lambda$ and $\Pi$.  Then the maps $\alpha_{\sg}^i$ and $\beta_{\sg}^i$  in  (\ref{equ:twomaps})  are isomorphisms for all $i\in \mathbb{Z}$. Consequently, there is an isomorphism $\overline{C}^*_{\sg, R}(\Lambda, \Lambda) \simeq  \overline{C}^*_{\sg, R}(\Pi, \Pi)$ in $\mathrm{Ho}(B_\infty)$.
\end{prop}

It follows that a singular equivalence with  level gives rise to an isomorphism of Gerstenhaber algebras
$$\HH_{\sg}^*(\Lambda, \Lambda) \simeq \HH_{\sg}^*(\Pi, \Pi).$$
We refer to  \cite{Wan2} for an alternative proof of this isomorphism. We mention that the above isomorphism as graded algebras is also obtained in \cite[Theorem~3 and Remark~2]{BRyD} via a similar argument as \cite{Kel03}.

\begin{proof}[Proof of Proposition~\ref{prop:sing-equi2}]
By Theorem~\ref{thm:quasi-iso-bimod}, it suffices to prove that both $\alpha_{\sg}^i$ and $\beta_{\sg}^i$ are isomorphisms. We only prove that the maps $\beta_{\sg}^i$ are isomorphisms, since a similar argument works for $\alpha_{\sg}^i$.

Indeed, we will prove a slightly stronger result. Let $\mathcal{X}$ (\emph{resp}. $\mathcal{Y}$) be the full subcategory of $\mathbf{D}_{\rm sg}(\Pi^e)$ (\emph{resp}. $\mathbf{D}_{\rm sg}(\Lambda\otimes \Pi^{\rm op})$) consisting of those complexes $X$, whose underlying complexes $X_\Pi$ of right $\Pi$-modules are perfect. The triangle functors
$$M\otimes_\Pi-\colon \mathcal{X}\longrightarrow \mathcal{Y} \mbox{ and } N\otimes_\Lambda-\colon \mathcal{Y}\longrightarrow \mathcal{X}$$
 are well defined. We claim that they are equivalences. This claim clearly implies that $\beta_{\sg}^i$ are isomorphisms.

For the proof of the claim, we observe that for a bounded complex $P$ of projective $\Pi^e$-modules and an object $X$ in $\mathcal{X}$, the complex $P\otimes_\Pi X$ is perfect, that is, isomorphic to zero in $\mathcal{X}$. There is a canonical exact triangle in $\mathbf{D}^b(\Pi^e\mbox{-mod})$
$$\Sigma^{n-1}\Omega^n_{\Pi^e}(\Pi)\longrightarrow P \longrightarrow \Pi \longrightarrow \Sigma^n \Omega^n_{\Pi^e}(\Pi),$$
where $P$ is a bounded complex of projective $\Pi^e$-modules with length precisely $n$. Applying $-\otimes_\Pi X$ to this triangle, we infer a natural isomorphism
$$X\simeq \Sigma^n \Omega^n_{\Pi^e}(\Pi)\otimes_\Pi X$$
in $\mathcal{X}$. By the second condition in Definition~\ref{defn:singequi}, we have
 $$N\otimes_\Lambda (M\otimes_\Pi X)\simeq \Omega^n_{\Pi^e}(\Pi)\otimes_\Pi X\simeq \Sigma^{-n}(X).$$
 Similarly, we infer that $M\otimes_\Pi(N\otimes_\Lambda Y)\simeq \Sigma^{-n}(Y)$ for any object $Y\in \mathcal{Y}$. This proves the claim.
\end{proof}

\subsection{A non-standard resolution and liftings}\label{subsection:homotopycartesiansquare}
In this subsection, we make preparation for the proof of Theorem~\ref{thm:quasi-iso-bimod}. We study a non-standard resolution of $M$, and lift certain maps between cohomological groups to cochain complexes.

Recall from Subsection~\ref{subsection-bar} the normalized bar resolution $\overline\Barr(\Lambda)$.  It is well known that $\overline\Barr(\Lambda)\otimes_\Lambda M\otimes_\Pi \overline\Barr(\Pi)$ is a projective $\Lambda$-$\Pi$-bimodule resolution of $M$, even without the projectivity assumption on $M$. However,  we will  need another \emph{non-standard} resolution of $M$; this resolution requires the projectivity assumption on the $\Lambda$-$\Pi$-bimodule $M$.

We denote by $\widetilde \Barr(\Lambda)$ the undeleted bar resolution
\begin{align}\label{equ:abc}
\dotsb \to \Lambda \otimes (s\overline \Lambda)^{\otimes m} \otimes \Lambda \xrightarrow{d_{ex}} \dotsb \xrightarrow{d_{ex}} \Lambda\otimes  (s\overline \Lambda) \otimes \Lambda   \xrightarrow{d_{ex}} \Lambda \otimes \Lambda \xrightarrow{\mu} s^{-1}\Lambda \rightarrow 0,
\end{align}
where $\mu$ is the multiplication and $d_{ex}$ is the external differential; see Subsection \ref{subsection-bar}. Here, we use $s^{-1}\Lambda$ to emphasize that it is of cohomological degree one.  Similarly, we have the undeleted bar resolution $\widetilde \Barr(\Pi)$ for $\Pi$.

 Consider  the following complex of $\Lambda$-$\Pi$-bimodules
 $$\mathbb{B}=\mathbb B(\Lambda, M, \Pi):=\widetilde \Barr(\Lambda) \otimes_\Lambda sM \otimes_\Pi \widetilde\Barr(\Pi).$$
 We observe that $\mathbb{B}$ is acyclic.  By using the natural isomorphisms  $$s^{-1}\Lambda \otimes_\Lambda sM \simeq M,\quad \text{and}\quad sM \otimes_{\Pi} s^{-1}\Pi \simeq M,$$  we obtain that the $(-p)$-th component of $\mathbb B$ is given by
  $$
\mathbb B^{-p}= \bigoplus_{\substack{ i+j=p-1\\ i, j \geq 0}} \Lambda \otimes (s\overline\Lambda)^{\otimes i} \otimes sM \otimes (s\overline \Pi)^{\otimes j} \otimes \Pi \; \bigoplus\;  \Lambda \otimes (s\overline\Lambda)^{\otimes p} \otimes  M\;  \bigoplus\;   M \otimes (s\overline \Pi)^{\otimes p} \otimes \Pi
 $$
 for any $p \geq 0$, and that  $\mathbb{B}^{1}= s^{-1}\Lambda \otimes_\Lambda sM\otimes_\Pi s^{-1}\Pi \simeq s^{-1}M$.
 In particular,  we have
 \begin{equation*}
 \begin{split}
 \mathbb {B}^0&\simeq (\Lambda\smallotimes M)\bigoplus (M\smallotimes \Pi), \\
  \mathbb{B}^{-1}&=(\Lambda\smallotimes sM\smallotimes \Pi)\bigoplus (\Lambda\smallotimes s\overline{\Lambda}\smallotimes M)\bigoplus (M\smallotimes s\overline{\Pi}\smallotimes \Pi).
 \end{split}
 \end{equation*}
The differential $\partial^{-p}\colon \mathbb{B}^{-p}\rightarrow \mathbb{B}^{-(p-1)}$ is induced by the differentials of $\widetilde{\Barr}(\Lambda)$ and $\widetilde{\Barr}(\Pi)$ in \eqref{equ:abc} via tensoring with $sM$. For instance, the differential $\partial^0\colon \mathbb B^0\xrightarrow{}\mathbb B^{1}$ is given by
$$
\Lambda \otimes M \bigoplus M \otimes \Pi \longrightarrow{} M, \quad\quad (a \otimes x  \longmapsto ax, \quad  x' \otimes b \longmapsto x'b); $$
 the differential $\partial^{-1}:\mathbb B^{-1}\xrightarrow{} \mathbb B^{0}$ is given by the maps
 \begin{flalign*}
&& \Lambda\otimes sM\otimes \Pi &\longrightarrow{} (\Lambda\otimes M)\bigoplus (M\otimes \Pi), && (a\otimes sx\otimes b\longmapsto -a\otimes xb+ax\otimes b)&& \\
&& \Lambda\otimes s\overline{\Lambda}\otimes M & \longrightarrow{} \Lambda\otimes M, &&  (a\otimes s\overline{a}_1\otimes x\longmapsto aa_1\otimes x-a\otimes a_1x)&&\\
&& M\otimes s\overline{\Pi}\otimes \Pi & \longrightarrow{}  M\otimes \Pi, &&  (x\otimes s\overline{b}_1\otimes b\longmapsto xb_1\otimes b-x\otimes b_1b).&&
 \end{flalign*}

  Since $M$ is projective as a  left $\Lambda$-module and as a right $\Pi$-module, it follows that all the direct summands of $\mathbb B^{-p}$ are projective as $\Lambda$-$\Pi$-bimodules for $p \geq 0$. We infer that  $\mathbb{B}$ is an undeleted $\Lambda$-$\Pi$-bimodule projective resolution of $M$.

\begin{lem}\label{lem:differentialforms}
For each $p\geq 1$, the cokernel $\mathrm{Cok}(\partial^{-p-1})$ is isomorphic to
$$
\Omega_{\text{$\Lambda$-$\Pi$}}^{p}(M) : = \bigoplus_{\substack{ i+j =p-1\\ i, j \geq 0}} (s\overline\Lambda)^{\otimes i} \otimes sM \otimes (s\overline \Pi)^{\otimes j} \otimes \Pi
{}  \; \bigoplus {} (s\overline\Lambda)^{\otimes p} \otimes M.$$
In particular, $\Omega_{\text{$\Lambda$-$\Pi$}}^{p}(M)$ inherits a $\Lambda$-$\Pi$-bimodule structure from $\mathrm{Cok}(\partial^{-p-1})$.
\end{lem}

\begin{proof}
We have a $\mathbb k $-linear map
$$\gamma^{-p} \colon \Omega_{\text{$\Lambda$-$\Pi$}}^{p}(M) \xrightarrow{1\otimes \mathbf 1} \mathbb B^{-p}\longrightarrow{}\mathrm{Cok}(\partial^{-p-1}), $$
where the unnamed arrow is the natural projection and  the first map $1\otimes \mathbf 1$ is given by
\begin{equation}\label{1mathbf1}
\begin{split}
s\overline a_{1, i} \otimes sx \otimes s\overline b_{1, j} \otimes b_{j+1}&  \longmapsto 1 \otimes s\overline a_{1, i} \otimes sx \otimes s\overline b_{1, j} \otimes b_{j+1}\\
s\overline a_{1, p} \otimes x & \longmapsto 1 \otimes s\overline a_{1, p} \otimes x.
\end{split}
\end{equation}

We claim that $\gamma^{-p}$ is surjective. Indeed, under the projection $\mathbb B^{-p}\twoheadrightarrow \mathrm{Cok}(\partial^{-p-1})$,  the image of a typical element $a_0\otimes s\overline{a}_{1, i}\otimes y \in  \mathbb B^{-p}$ equals the image of the following element
$$z:= \sum_{k=0}^{i-1}(-1)^k 1\otimes s\overline{a}_{0, k-1}\otimes s\overline{a_ka_{k+1}}\otimes  s\overline{a}_{k+2, i}\otimes y+ (-1)^{i}1\otimes s\overline{a}_{0, i-1}\otimes a_iy \in \mathbb B^{-p},$$
where $y$ lies in $sM\otimes (s\overline{\Pi})^{\otimes j}\otimes \Pi$ or $M$, since $\partial^{-p-1} ( 1 \otimes s\overline{a}_0\otimes s\overline{a}_{1, i}\otimes y) =a_0\otimes s\overline{a}_{1, i}\otimes y -z$.   Similarly, the image of a typical element $x\otimes s\overline{b}_{1, p}\otimes b_{p+1}\in M \otimes (s\overline{\Pi})^{\otimes p}\otimes \Pi $ equals the image of  
\begin{align*}
z':={}& 1\otimes s(xb_1) \otimes s\overline{b}_{2, p}\otimes b_{p+1} +\sum_{k=1}^{p-1} (-1)^k 1\otimes sx \otimes s\overline{b}_{1, k-1}\otimes s\overline{b_kb_{k+1}}\otimes s\overline{b}_{k+2, p}\otimes b_{p+1}
\\
&+ (-1)^p 1\otimes sx\otimes s\overline{b}_{1, p-1}\otimes b_pb_{p+1} \in \mathbb B^{-p},
\end{align*}
since $\partial^{-p-1} ( 1 \otimes sx\otimes s\overline{b}_{1, p}\otimes b_{p+1}) =x\otimes s\overline{b}_{1, p}\otimes b_{p+1} -z'$. In both cases, the latter elements $z$ and $z'$ lie in the image of the map $1\otimes \mathbf 1$. This shows that $\gamma^{-p}$ is surjective.

On the other hand, we have a projection of degree $-1$
$$\varpi^{-p+1} \colon \mathbb B^{-p+1}  \twoheadrightarrow \Omega_{\text{$\Lambda$-$\Pi$}}^{p}(M)$$
given by
\begin{align*}
a_0 \otimes s\overline a_{1, i} \otimes sm \otimes s\overline b_{1, j} \otimes b_{j+1} & \longmapsto s\overline a_0\otimes s\overline a_{1, i} \otimes sm \otimes s\overline b_{1, j} \otimes b_{j+1}\\
 a_0 \otimes s\overline a_{1, p-1} \otimes m & \longmapsto s\overline a_0\otimes  s\overline a_{1, p-1} \otimes m \\
m \otimes s \overline b_{1, p-1} \otimes b_p & \longmapsto sm \otimes s \overline b_{1, p-1} \otimes b_p
\end{align*}
We define a $\mathbb k $-linear map
$$\widetilde \eta^{-p}=\varpi^{-p+1} \circ \partial^{-p} \colon\mathbb B^{-p} \longrightarrow  \Omega_{\text{$\Lambda$-$\Pi$}}^{p}(M).$$
In view of $\widetilde \eta^{-p} \circ \partial^{-p-1} =0$, we have a unique induced map
$$
 \eta^{-p} \colon \mathrm{Cok}(\partial^{-p-1}) \longrightarrow{}\Omega_{\text{$\Lambda$-$\Pi$}}^{p}(M).
$$
One checks easily that $\eta^{-p}\circ \gamma^{-p}$ equals the identity. By the surjectivity of $\gamma^{-p}$, we infer that $\gamma^{-p}$ is an isomorphism.
\end{proof}

\begin{rem}\label{remshortexactM}
The right $\Pi$-module structure on $\Omega_{\text{$\Lambda$-$\Pi$}}^{p}(M)$ is induced by the right action of $\Pi$ on $M$ and $\Pi$. The left $\Lambda$-module structure is given by; compare \eqref{rightblacktriangle}
\begin{align*}
 a_0\blacktriangleright (s\overline a_{1, i} \otimes sx \otimes s\overline b_{1, j} \otimes b_{j+1}) & := (\pi\otimes \mathbf 1^{\otimes p})\circ \partial^{-p}(a_0 \otimes s\overline a_{1, i} \otimes sx \otimes s\overline b_{1, j} \otimes b_{j+1}),\\
a_0\blacktriangleright (s\overline a_{1, p} \otimes x) &: = (\pi\otimes \mathbf 1^{\otimes p})\circ \partial^{-p}(a_0 \otimes s\overline a_{1, p} \otimes x),
 \end{align*}
where $\pi \colon \Lambda \to s\overline\Lambda$ is the natural projection $a \mapsto s\overline a$ of degree $-1$.

We  have a short exact sequence of $\Lambda$-$\Pi$-modules; compare \eqref{shortexactsequence00}
\begin{align}\label{equ:bimod-B}
0 \longrightarrow \Sigma^{-1} \Omega_{\text{$\Lambda$-$\Pi$}}^{p+1}(M)\xrightarrow{\partial^{-p-1} \circ (1 \otimes \mathbf 1)} \mathbb B^{-p} \xrightarrow{ \quad \widetilde \eta^{-p} \; } \Omega_{\text{$\Lambda$-$\Pi$}}^{p}(M) \longrightarrow 0,
 \end{align}
 where the map $1 \otimes \mathbf 1$ is given in  \eqref{1mathbf1}. Here,  we always view $\Omega_{\text{$\Lambda$-$\Pi$}}^{p}(M)$ as a graded $\Lambda$-$\Pi$-bimodule concentrated in degree $-p$. By convention, we have $\Omega^0_{\text{$\Lambda$-$\Pi$}}(M)=M$.
\end{rem}

\vskip 2pt

Fix $p\geq 0$. Applying the functor ${\rm Hom}_{\text{$\Lambda$-$\Pi$}}(-,\Omega_{\text{$\Lambda$-$\Pi$}}^{p}(M))$ to the resolution  $\overline\Barr(\Lambda) \otimes_\Lambda M\otimes_\Pi \overline\Barr(\Pi)$ of $M$; see the proof of \cite[Proposition 4.1]{Cib}, we obtain a cochain complex
$$\overline{C}^*(M, \Omega_{\text{$\Lambda$-$\Pi$}}^{p}(M))$$
 computing ${\rm Ext}^*_{\text{$\Lambda$-$\Pi$}}(M, \Omega_{\text{$\Lambda$-$\Pi$}}^{p}(M))$. The space $\overline C^m(M, \Omega_{\text{$\Lambda$-$\Pi$}}^{p}(M))$ in degree $m$ is as follows:
 \begin{align*}
&\bigoplus_{\substack{ i+j=m+p\\i, j\geq 0}} \Hom
\left((s\overline\Lambda)^{\otimes i} \otimes M \otimes (s\overline \Pi)^{\otimes j} , \ \bigoplus_{
\substack{k+l=p-1\\ k, l \geq 0}}(s\overline\Lambda)^{\otimes k} \otimes sM \otimes (s\overline \Pi)^{\otimes l} \otimes \Pi \bigoplus (s\overline\Lambda)^{\otimes p} \otimes M \right).
\end{align*}

Recall that $\Omega^p_{\nc, R}(\Lambda)=(s\overline \Lambda)^{\otimes p}\otimes \Lambda$ is the graded $\Lambda$-$\Lambda$-bimodule of right noncommutative differential $p$-forms.  We have a natural identification
 $$\HH^*(\Lambda, \Omega_{\nc, R}^p(\Lambda)) \simeq \mathrm{Ext}^*_{\Lambda^e}(\Lambda, \Omega_{\nc, R}^p(\Lambda)).$$

Consider the following triangle functor
$$-\otimes_\Lambda M \colon \mathbf{D}(\Lambda^e) \longrightarrow \mathbf{D}(\Lambda \otimes \Pi^{\op}).$$
Then we have a map
 \begin{align*}
\alpha^*_p\colon \HH^*(\Lambda, \Omega_{\nc, R}^p(\Lambda))  \xrightarrow{-\otimes_\Lambda M} \mathrm{Ext}^*_{\Lambda\mbox{-}\Pi}(M, \Omega_{\nc, R}^p(\Lambda) \otimes_\Lambda M ) \longrightarrow   \mathrm{Ext}^*_{\Lambda\mbox{-}\Pi}(M, \Omega_{\text{$\Lambda$-$\Pi$}}^{p}(M)),
\end{align*}
where the second map is induced by the natural inclusion
 \begin{align}\label{alpha-p}
  \Omega_{\nc, R}^p(\Lambda) \otimes_\Lambda M \xrightarrow{\simeq} (s\overline \Lambda)^{\otimes p} \otimes M\hookrightarrow \Omega_{\text{$\Lambda$-$\Pi$}}^{p}(M).
\end{align}

We define a cochain map
\begin{align}\label{equ:lift-alpha}
 \widetilde \alpha_p  \colon  \overline{C}^*(\Lambda, \Omega_{\nc, R}^p(\Lambda))  \stackrel{}{\longrightarrow}   \overline C^{*}(M, \Omega_{\text{$\Lambda$-$\Pi$}}^{p}(M))
\end{align}
as follows: for any $f \in  \Hom((s\overline \Lambda)^{\otimes m}, (s\overline \Lambda)^{\otimes p}\otimes \Lambda)$ with $m\geq 0$, the corresponding map $\widetilde \alpha_p(f) \in \overline C^{m-p} (M, \Omega_{\text{$\Lambda$-$\Pi$}}^{p}(M))$ is given by
\begin{align*}
\widetilde \alpha_p(f)|_{(s\overline \Lambda)^{\otimes m-i} \otimes M \otimes (s\overline \Pi)^{\otimes i}}& = 0 &&\text{ if $i \neq 0$}\\
\widetilde \alpha_p(f)(s\overline{a}_{1, m} \otimes x) &=f(s\overline{a}_{1, m})\otimes_\Lambda x
\end{align*}
for any $s\overline{a}_{1, m} \otimes x \in (s\overline \Lambda)^{\otimes m} \otimes M$.

Recall that the cochain complexes $\overline{C}^*(\Lambda, \Omega_{\nc, R}^p(\Lambda)) $ and $\overline C^{*}(M, \Omega_{\text{$\Lambda$-$\Pi$}}^{p}(M))$ compute $\HH^*(\Lambda, \Omega_{\nc, R}^p(\Lambda))$ and  $\mathrm{Ext}^*_{\Lambda\mbox{-}\Pi}(M, \Omega_{\text{$\Lambda$-$\Pi$}}^{p}(M))$, respectively.

\begin{lem}\label{lem:liftingalpha}
The cochain map $\widetilde \alpha_p$ is a lifting of $\alpha_p^*$.
\end{lem}

\begin{proof}
Since $M$ is projective as a right $\Pi$-module, it follows that the tensor functor $-\otimes_\Lambda M$ sends the projective resolution $\overline{\Barr}(\Lambda)$ of $\Lambda$ to a projective resolution $\overline{\Barr}(\Lambda)\otimes_\Lambda M$ of $M$.

Denote $\Omega_{\nc, R}^{p}(M) =\Omega_{\nc, R}^p(\Lambda) \otimes_\Lambda M$.  Consider the complex
$$\overline C_{\text{$\mathbb k $-$\Pi$}}^*(M, \Omega_{\nc, R}^{p}(M))  =\prod_{m \geq 0} \Hom_{\text{$\mathbb k $-$\Pi$}}((s\overline \Lambda)^{\otimes m}\otimes M, \Omega_{\nc, R}^{p}(M))$$ whose differential is induced by the differential of $\Hom_{\text{$\Lambda$-$\Pi$}}(\overline{\Barr}(\Lambda)\otimes_\Lambda M, \Omega_{\nc, R}^{p}(M))$ under the natural isomorphism
$$\begin{array}{cccr}
\Hom_{\text{$\Lambda$-$\Pi$}}(\Lambda\otimes  (s\overline \Lambda)^{\otimes m}\otimes  M, \Omega_{\nc, R}^{p}(M)) & \stackrel{\simeq}{\longrightarrow}  & \Hom_{\text{$\mathbb k $-$\Pi$}}((s\overline \Lambda)^{\otimes m}\otimes M, \Omega_{\nc, R}^{p}(M)) \\
f & \longmapsto & ( s\overline{a}_{1, m} \otimes x \longmapsto   f(1_\Lambda \otimes s\overline{a}_{1, m} \otimes x)).
\end{array}
$$

Note that $\overline C_{\text{$\mathbb k $-$\Pi$}}^*(M, \Omega_{\nc, R}^{p}(M))$ also computes $\mathrm{Ext}^*_{\Lambda\mbox{-}\Pi}(M, \Omega_{\nc, R}^{p}(M))$.  The first map  $ \HH^*(\Lambda, \Omega_{\nc, R}^p(\Lambda))  \xrightarrow{-\otimes_\Lambda M} \mathrm{Ext}^*_{\Lambda\mbox{-}\Pi}(M, \Omega_{\nc, R}^{p}(M))$ in defining $\alpha_p^*$ has the following lifting
$$
\alpha_p' \colon  \overline{C}^*(\Lambda, \Omega_{R}^p(\Lambda)) \stackrel{}{\longrightarrow} \overline C_{\text{$\mathbb k $-$\Pi$}}^*(M, \Omega_{\nc, R}^{p}(M)), $$
which sends  $f \in  \Hom((s\overline \Lambda)^{\otimes m}, \Omega_{\nc, R}^p(\Lambda))$ to $\alpha'(f)\in \Hom_{\text{$\mathbb k $-$\Pi$}}((s\overline \Lambda)^{\otimes m}\otimes M, \Omega_{\nc, R}^{p}(M))$
given by
$$
\alpha_p'(f) (s\overline a_{1, m}\otimes x) = f(s\overline a_{1, m}) \otimes_\Lambda x.
$$

The second map $\mathrm{Ext}^*_{\Lambda\mbox{-}\Pi}(M, \Omega_{\nc, R}^{p}(M))\rightarrow   \mathrm{Ext}^*_{\Lambda\mbox{-}\Pi}(M, \Omega_{\text{$\Lambda$-$\Pi$}}^{p}(M))$ in defining $\alpha_p^*$ has the following lifting
$$
\iota \colon\overline C_{\text{$\mathbb k $-$\Pi$}}^*(M, \Omega_{\nc, R}^{p}(M))   \hookrightarrow   \overline C^{*}(M, \Omega_{\text{$\Lambda$-$\Pi$}}^{p}(M))  $$
which is induced by the natural  inclusion
$$
\Hom_{\text{$\mathbb k $-$\Pi$}}((s\overline \Lambda)^{\smallotimes m}\smallotimes \!M, \!\Omega_{\nc, R}^{p}\!(M)) \!\hookrightarrow \!\Hom((s\overline \Lambda)^{\smallotimes m}\smallotimes M, \!\Omega_{\nc, R}^{p}\!(M))\! \hookrightarrow \!\Hom((s\overline \Lambda)^{\smallotimes m}\smallotimes M, \!\Omega_{\text{$\Lambda$-$\Pi$}}^{p}(M)).
$$
Observe that $\widetilde \alpha_p = \iota \circ \alpha_p'$. It follows that $\widetilde \alpha_p$ is a lifting of $\alpha_p^*$.
\end{proof}

Similarly, we have  the following triangle functor
\begin{equation*}
M\otimes_\Pi - \colon  \mathbf{D}(\Pi^e) \longrightarrow \mathbf{D}( \Lambda\otimes \Pi^{\op}),
\end{equation*}
and  the corresponding  map
\begin{align*}
\beta_p^* \colon \HH^*(\Pi, \Omega_{\nc, R}^p(\Pi)) \xrightarrow{M \otimes_{\Pi}-}  \mathrm{Ext}^*_{\Lambda\otimes \Pi^{\op}}(M, M \otimes_{\Pi}  \Omega_{\nc, R}^p(\Pi) )  \stackrel{}{\longrightarrow}  \mathrm{Ext}^*_{\Lambda\otimes \Pi^{\op}}(M, \Omega_{\text{$\Lambda$-$\Pi$}}^{p}(M)), \end{align*}
where the second map is induced by  the following bimodule homomorphism
 \begin{align}\label{mpi}
   M \otimes_{\Pi}  \Omega_{\nc, R}^p(\Pi)  \hookrightarrow  \Omega_{\text{$\Lambda$-$\Pi$}}^{p}(M), \quad x \otimes_{\Pi} (s\overline b_{1, p} \otimes b_{p+1})\longmapsto
x \triangleright (s \overline b_{1, p} \otimes b_{p+1}).
  \end{align}
  Here, the action $\triangleright$  is given by; compare \eqref{rightblacktriangle}
\begin{align}\label{righttriangle}
x \triangleright (s \overline b_{1, p} \otimes b_{p+1}) ={}& s(xb_1) \otimes s \overline b_{2, p} \otimes b_{p+1} +\sum_{i=1}^{p-1} (-1)^isx \otimes s \overline b_{1,i-1} \otimes s\overline{b_ib_{i+1}} \otimes s\overline b_{i+2, p} \otimes b_{p+1})\nonumber\\
&+(-1)^p sx \otimes s \overline b_{1, p-1} \otimes b_pb_{p+1}.
\end{align}
We stress that the injection \eqref{mpi} differs from the natural inclusion \eqref{alpha-p}.   This actually leads to a  tricky argument in the proof of Proposition \ref{prop:HH-sg-long}; For more explanations, see Remark \ref{rem:tricky}.

   We define a cochain map
    \begin{equation}\label{equ:lift-beta}
 \widetilde \beta_p \colon  \overline{C}^*(\Pi, \Omega_{\nc, R}^p(\Pi)) \stackrel{}{\longrightarrow} \overline C^{*}(M, \Omega_{\text{$\Lambda$-$\Pi$}}^{p}(M))
\end{equation}
 as follows: for any map  $g \in  \Hom((s\overline \Pi)^{\otimes m}, (s\overline \Pi)^{\otimes p}\otimes \Pi)$, the corresponding map $\widetilde \beta_p(g) \in \overline{C}^{m-p} (M, \Omega_{\text{$\Lambda$-$\Pi$}}^{p}(M))$ is given by
\begin{align}\label{triangleright}
\widetilde \beta_p(g)|_{(s\overline \Lambda)^{\otimes i} \otimes M \otimes (s\overline \Pi)^{\otimes m-i}}&=0 && \text{if $i \neq 0$}\nonumber\\
\widetilde \beta_p(g)( x\otimes s\overline{b}_{1, m} ) &= x \triangleright g(s\overline{b}_{1, m})
\end{align}
for any $x \otimes s\overline{b}_{1, m} \in M \otimes (s\overline \Pi)^{\otimes m}$, where the action $\triangleright$  is defined in \eqref{righttriangle}.

We have the following analogous result of Lemma~\ref{lem:liftingalpha}.

\begin{lem}\label{lem:liftingbeta}
The map $\widetilde \beta_p$ is a lifting of $\beta_p^*$.
\end{lem}

\begin{proof}
The tensor functor $M \otimes_\Pi -$ sends the projection resolution $\overline \Barr(\Pi)$ of $\Pi$ to the projective resolution $M \otimes_\Pi \overline \Barr(\Pi)$ of $M$.

Consider the complex  $$\overline C_{\text{$\Lambda$-$\mathbb k $}}^*(M, M\otimes_{\Pi} \Omega_{\nc, R}^{p}(\Pi))  =\prod_{m \geq 0} \Hom_{\text{$\Lambda$-$\mathbb k $}}(M \otimes (s\overline \Pi)^{\otimes m},M\otimes_{\Pi} \Omega_{\nc, R}^{p}(\Pi)),$$
which is naturally isomorphic to $\Hom_{\text{$\Lambda$-$\Pi$}}(M \otimes_\Pi \overline \Barr(\Pi), M\otimes_{\Pi} \Omega_{\nc, R}^{p}(\Pi))$. Therefore, both complexes compute $\mathrm{Ext}^*_{\Lambda\mbox{-}\Pi}(M, M \otimes_{\Pi}  \Omega_{\nc, R}^p(\Pi) )$.
Then the map
$$ \HH^*(\Pi, \Omega_{\nc, R}^p(\Pi)) \xrightarrow{M \otimes_{\Pi}-}  \mathrm{Ext}^*_{\Lambda\mbox{-} \Pi}(M, M \otimes_{\Pi}  \Omega_{\nc, R}^p(\Pi) )$$
 has a lifting
$$
\beta_p' \colon \overline C^*(\Pi, \Omega_{\nc, R}^p(\Pi))  \longrightarrow \overline C_{\text{$\Lambda$-$\mathbb k $}}^*(M, M \otimes_{\Pi}  \Omega_{\nc, R}^p(\Pi)),
$$
which sends $g \in  \Hom((s\overline \Pi)^{\otimes m}, \Omega_{\nc, R}^p(\Pi))$ to $\beta_p'(g) \in  \Hom_{\text{$\Lambda$-$\mathbb k $}}(M \otimes (s\overline \Pi)^{\otimes m}, M \otimes_{\Pi}  \Omega_{\nc, R}^p(\Pi))$ given by
$$
\beta_p'(g)( x \otimes s\overline{b}_{1, m}) =x \otimes_{\Pi} g(s\overline{b}_{1, m}).
$$
We have an injection of complexes
$$
\iota \colon \overline C_{\text{$\Lambda$-$\mathbb k $}}^*(M, M \otimes_{\Pi}  \Omega_{\nc, R}^p(\Pi)) \longrightarrow{} \overline C^{*}(M, \Omega_{\text{$\Lambda$-$\Pi$}}^{p}(M))
$$
induced by the injection \eqref{mpi}. By $\widetilde \beta_p = \iota \circ \beta_p'$, we conclude that $\widetilde \beta_p$ is a lifting of $\beta_p^*$.
\end{proof}

\subsection{A triangular matrix algebra and colimits}

Denote by $\Gamma = \left(\begin{matrix} \Lambda & M \\ 0 & \Pi \end{matrix}\right)$  the  upper triangular matrix algebra. Set $e_1 =  \left(\begin{matrix}1_\Lambda & 0 \\ 0 & 0\end{matrix}\right)$ and  $e_2 =  \left(\begin{matrix} 0 & 0\\ 0 & 1_\Pi\end{matrix}\right)$.
Then we have the following natural identifications:
\begin{align}\label{idempotent}
e_1 \Gamma e_1 \simeq \Lambda, \quad e_2 \Gamma e_2 \simeq \Pi, \quad e_1\Gamma e_2 \simeq M, \quad {\rm and}\;  e_2 \Gamma e_1 =0.
\end{align}

Denote by $E = \mathbb k  e_1 \oplus \mathbb k  e_2$ the semisimple subalgebra of $\Gamma$. Set $\overline{\Gamma}=\Gamma/(E\cdot1_\Gamma)$. Consider the $E$-relative right singular Hochschild cochain complex $\overline{C}_{\sg, R, E}^*(\Gamma, \Gamma)$.

 Using \eqref{idempotent}, we identify $\overline{\Gamma}$ with $\overline{\Lambda}\oplus \overline{\Pi} \oplus M$. Here, we agree that $\overline{\Lambda}=\Lambda/(\mathbb k  \cdot 1_{\Lambda})$ and $\overline{\Pi}=\Lambda/(\mathbb k  \cdot 1_{\Pi})$.  Then we have
 \begin{align*}
s\overline \Gamma^{\otimes_E m}\cong s\overline\Lambda^{\otimes m}\bigoplus s\overline\Pi^{\otimes m}\bigoplus \left(\bigoplus_{\substack{i,j\geq 0\\ i+j=m-1}}s\overline\Lambda^{\otimes i}\otimes sM\otimes s\overline\Pi^{\otimes j}\right).
\end{align*}
For each $m, p\geq 0$, we have the following natural decomposition of vector spaces
\begin{equation}\label{decompositionmp}
\begin{split}
&\Hom_{\text{$E$-$E$}}((s\overline \Gamma)^{\smallotimes_E m}, (s\overline \Gamma)^{\smallotimes _E p}\smallotimes_E \Gamma) \\
\simeq{} &  \Hom((s\overline \Lambda)^{\smallotimes m}, (s\overline \Lambda)^{\smallotimes p}\smallotimes \Lambda) \; \bigoplus  \Hom((s\overline \Pi)^{\smallotimes m}, (s\overline \Pi)^{\smallotimes p}\smallotimes \Pi) \; \bigoplus  \\
&\!\!\! \bigoplus_{\substack{i, j, \geq 0\\ i + j=m-1}} \Hom\Big((s\overline \Lambda)^{\smallotimes i} \smallotimes sM \smallotimes (s\overline \Pi)^{\smallotimes j},  \bigoplus_{\substack{i', j' \geq 0\\ i'+j'=p-1}}(s\overline \Lambda)^{\smallotimes i'}\smallotimes sM \smallotimes  (s\overline \Pi)^{\smallotimes j'}\smallotimes \Pi \bigoplus (s\overline \Lambda)^{\smallotimes p}\smallotimes M\Big),
\end{split}
\end{equation}
which induces the following decomposition of graded vector spaces
\begin{align}\label{splitdecomposition1}
\overline{C}_{ E}^*(\Gamma, \Omega_{\nc, R, E}^p(\Gamma)) \simeq    \overline{C}^*(\Lambda, \Omega_{\nc, R}^p(\Lambda)) \oplus \overline{C}^*(\Pi, \Omega_{\nc,R}^p(\Pi))\oplus \Sigma^{-1}\overline C^{*}(M, \Omega_{\text{$\Lambda$-$\Pi$}}^{p}(M)).
\end{align}

 We write elements on the right hand side of  \eqref{splitdecomposition1} as column vectors.  The differential $\delta_\Gamma$ of $\overline{C}_{ E}^*(\Gamma, \Omega_{\nc, R, E}^p(\Gamma))$ induces a differential $\overline\delta$  on the right hand side of \eqref{splitdecomposition1}. Similar to \eqref{comcom}, we obtain that $\overline\delta$ has the following form
  \begin{equation}\label{alphabeta}
\overline\delta=\begin{pmatrix}\delta_{\Lambda}&0&0\\
0&\delta_{\Pi}&0\\
 -s^{-1}\circ  \widetilde \alpha_p &s^{-1}\circ  \widetilde \beta_p &\Sigma^{-1}(\delta_M)
\end{pmatrix},
\end{equation}
where $\delta_\Lambda, \delta_\Pi$ and $\delta_M$ are  the Hochschild differentials of  $ \overline{C}^*(\Lambda, \Omega_{\nc, R}^p(\Lambda)),  \overline{C}^*(\Pi, \Omega_{\nc, R}^p(\Pi))$ and $\overline C^{*}(M, \Omega_{\text{$\Lambda$-$\Pi$}}^{p}(M))$, respectively. The entry
$$
s^{-1}\circ  \widetilde \alpha_p \colon
\overline{C}^*(\Lambda, \Omega_{\nc, R}^p(\Lambda)) \longrightarrow \Sigma^{-1}\overline C^{*}(M, \Omega_{\text{$\Lambda$-$\Pi$}}^{p}(M))
$$
 is a map of degree one, which is the composition of $\widetilde \alpha_p$ with the natural identification $s^{-1} \colon \overline C^{*}(M, \Omega_{\text{$\Lambda$-$\Pi$}}^{p}(M))\to \overline \Sigma^{-1}C^{*}(M, \Omega_{\text{$\Lambda$-$\Pi$}}^{p}(M))$ of degree one. A similar remark holds for $s^{-1}\circ \widetilde \beta_p$.

The decomposition \eqref{splitdecomposition1} induces a short exact sequence of complexes
\begin{align}\label{shortexactsequence}
0 \longrightarrow{} \Sigma^{-1}\overline C^{*}(M, \Omega_{\text{$\Lambda$-$\Pi$}}^{p}(M))
\stackrel{\rm inc}\longrightarrow    \overline{C}_{ E}^*(\Gamma, \Omega_{\nc, R, E}^p(\Gamma))
\stackrel{\left( \begin{smallmatrix} {\rm res}_1\\ {\rm res}_2 \end{smallmatrix}\right)}
\longrightarrow
\begin{matrix}\overline{C}^*(\Lambda, \Omega_{\nc, R}^p(\Lambda)) \\\oplus \ \overline{C}^*(\Pi, \Omega_{\nc, R}^p(\Pi)) \end{matrix}\longrightarrow 0.
\end{align}
Here, ``${\rm res}_i$" denotes the corresponding projection.

In what follows, letting $p$ vary, we will take colimits of \eqref{shortexactsequence}. Recall that the colimits, along the maps $\theta_{p, R, E}$ or $\theta_{p, R}$ in \eqref{theta-p},  of the middle and the right hand terms of \eqref{shortexactsequence} are $\overline{C}_{\sg, R, E}^*(\Gamma, \Gamma)$ and $\overline{C}^*_{\sg, R}(\Lambda, \Lambda) \oplus \overline{C}^*_{\sg, R}(\Pi, \Pi),$ respectively.  Similarly, we define  $$ \theta_p^M \colon  \overline C^*(M, \Omega_{\text{$\Lambda$-$\Pi$}}^{p}(M)) \longrightarrow  \overline C^*(M, \Omega_{\text{$\Lambda$-$\Pi$}}^{p+1}(M))$$
as follows: for any $f \in \overline C^*(M, \Omega_{\text{$\Lambda$-$\Pi$}}^{p}(M))$, we set
\begin{align*}
\theta^M_p(f) (s\overline a_{1, i} \otimes x \otimes s\overline b_{1, j}) &=(-1)^{|f|} s\overline a_1 \otimes f(s\overline a_{2, i} \otimes x \otimes s\overline b_{1, j}),
\end{align*}
if $i \geq 1$; otherwise, we set
$$
\theta^M_p(f) ( x\otimes s\overline b_{1, j} )=0.
$$
We observe that $\theta^M_p$ is indeed a morphism of cochain complexes for each $p \geq 0$. Similar to the definition of right singular Hochschild cochain complex in Subsection~\ref{subsec:sHcc}, we have an induction system of cochain complexes
$$
\overline C^*(M, M) \xrightarrow{\theta^M_0} \dotsb\longrightarrow \overline C^*(M, \Omega_{\text{$\Lambda$-$\Pi$}}^{p}(M)) \xrightarrow{\theta^M_{p}} \overline C^*(M, \Omega_{\text{$\Lambda$-$\Pi$}}^{p+1}(M))\xrightarrow{\theta^M_{p+1}}\dotsb,
$$
 and denote its colimit by $\overline C_{\sg}^*(M, M)$.

We have the following commutative diagram of cochain complexes with rows being short exact.
\begin{equation}\label{exactsequencecolimitM}
\xymatrix{
\Sigma^{-1}\overline C^{*}(M, \Omega_{\text{$\Lambda$-$\Pi$}}^{p}(M))\ar[r]^-{\rm inc} \ar[d]_{\theta_p^M} &
\overline{C}_{ E}^*(\Gamma, \Omega_{\nc, R, E}^p(\Gamma)) \ar[d]^-{\theta_p^\Gamma} \ar[r]^-{\left( \begin{smallmatrix} {\rm res}_1\\ {\rm res}_2 \end{smallmatrix}\right)} &  \overline{C}^*(\Lambda, \Omega_{\nc, R}^p(\Lambda)) \oplus \overline{C}^*(\Pi, \Omega_{\nc, R}^p(\Pi)) \ar[d]^-{\theta_p^\Lambda\oplus \theta_p^\Pi}\\
\Sigma^{-1}\overline C^{*}(M, \Omega_{\text{$\Lambda$-$\Pi$}}^{p+1}(M))\ar[r]^-{\rm inc} &
\overline{C}_{ E}^*(\Gamma, \Omega_{\nc, R, E}^{p+1}(\Gamma)) \ar[r]^-{\left( \begin{smallmatrix} {\rm res}_1\\ {\rm res}_2 \end{smallmatrix}\right)} &
\overline{C}^*(\Lambda, \Omega_{\nc, R}^{p+1}(\Lambda)) \oplus \overline{C}^*(\Pi, \Omega_{\nc, R}^{p+1}(\Pi))
}
\end{equation}

The following lemma is analogous to Lemma \ref{lem:HH-colimit}. 
\begin{lem}\label{lem:computesgspace}
The cochain map $\theta_p^M$  is a lifting of  the following connecting map
$$ \widehat \theta^M_p\colon {\rm Ext}^*_{\Lambda\mbox{-}\Pi}(M, \Omega^{p}_{\Lambda\mbox{-}\Pi}(M))\longrightarrow  {\rm Ext}^*_{\Lambda\mbox{-}\Pi}(M, \Omega^{p+1}_{\Lambda\mbox{-}\Pi}(M))$$
in the long exact sequence obtained by applying the functor ${\rm Ext}^*_{\text{$\Lambda$-$\Pi$}}(M, -)$ to (\ref{equ:bimod-B}). Consequently, for any $n\in \mathbb Z$  we have an isomorphism
$$
 H^n(\overline C_{\sg}^*(M, M)) \simeq \Hom_{\mathbf{D}_{\rm sg}(\Lambda\otimes\Pi^{\op})} (M, \Sigma^n M).
$$
\end{lem}

\begin{proof} The proof is similar to that of Lemma \ref{lem:HH-colimit}. 
Since the direct colimit commutes with the cohomology functor, we have
\begin{align*}
 H^n(\overline C_{\sg}^*(M, M)) \simeq \varinjlim_{\widetilde \theta^M_p}  \;  {\rm Ext}^n_{\text{$\Lambda$-$\Pi$}}(M, \Omega_{\text{$\Lambda$-$\Pi$}}^{p}(M)),
\end{align*}
where the colimit map $\widetilde \theta_p^M$ is induced by $\theta_p^M$. Apply the functor ${\rm Ext}^n_{\text{$\Lambda$-$\Pi$}}(M, -)$ to (\ref{equ:bimod-B})
$$\dotsb
\to {\rm Ext}^n_{\text{$\Lambda$-$\Pi$}}(M, \mathbb B^{-p}) \to {\rm Ext}^n_{\text{$\Lambda$-$\Pi$}}(M, \Omega_{\text{$\Lambda$-$\Pi$}}^{p}(M)) \to {\rm Ext}^{n+1}_{\text{$\Lambda$-$\Pi$}}(M, \Sigma^{-1}\Omega_{\text{$\Lambda$-$\Pi$}}^{p+1}(M))\to \dotsb
$$
Since ${\rm Ext}^{n+1}_{\text{$\Lambda$-$\Pi$}}(M, \Sigma^{-1}\Omega_{\text{$\Lambda$-$\Pi$}}^{p+1}(M))$ is naturally isomorphic to ${\rm Ext}^{n}_{\text{$\Lambda$-$\Pi$}}(M, \Omega_{\text{$\Lambda$-$\Pi$}}^{p+1}(M))$, the connecting morphism in the long exact sequence induces a map
$$\widehat \theta_p^M \colon{\rm Ext}^n_{\text{$\Lambda$-$\Pi$}}(M, \Omega_{\text{$\Lambda$-$\Pi$}}^{p}(M))  \longrightarrow{} {\rm Ext}^{n}_{\text{$\Lambda$-$\Pi$}}(M, \Omega_{\text{$\Lambda$-$\Pi$}}^{p+1}(M)).$$

We now show that $\widetilde \theta_p^M= \widehat \theta_p^M$ using the similar argument as the proof of Lemma~\ref{lem:HH-colimit}. We write down the definition of the connecting morphism $\widehat \theta_p^M$. Apply the functor ${\rm Hom}_{\text{$\Lambda$-$\Pi$}}(\overline{\Barr}(\Lambda)\otimes_{\Lambda}M\otimes_{\Pi}\overline{\Barr}(\Pi), -)$ to the short exact sequence \eqref{equ:bimod-B}. Then we have the following short exact sequence of complexes with induced maps
\begin{align}
\label{sesinduced}
0 \to \Sigma^{-1}\overline C^{*}\!(M, \Omega_{\text{$\Lambda$-$\Pi$}}^{p+1}(M))
\!\xrightarrow{}
{\rm Hom}_{\text{$\Lambda$-$\Pi$}}(\overline{\Barr}(\Lambda)\!\smallotimes_{\Lambda}\!M\!\smallotimes_{\Pi}\!\overline{\Barr}(\Pi), \mathbb B^{-p})\!
\xrightarrow{} \!
\overline{C}^*\!(M, \Omega_{\text{$\Lambda$-$\Pi$}}^p(M)) \to 0.
\end{align}
Take $f\in {\rm Ext}^n_{\text{$\Lambda$-$\Pi$}}(M, \Omega_{\text{$\Lambda$-$\Pi$}}^{p}(M))$. It may be represented by an element $f\in \overline{C}^n(M,\Omega_{\text{$\Lambda$-$\Pi$}}^{p}(M))$ such that $\delta'(f)=0$ with $\delta'$ the differential of $ \overline{C}^*(M,\Omega_{\text{$\Lambda$-$\Pi$}}^{p}(M))$. Define
$$\overline{f}\in \bigoplus_{\substack{i,j\geq 0\\ i+j=n+p}}{\rm Hom}\big(s\overline\Lambda^{\otimes i}\otimes M\otimes s\overline\Pi^{\otimes j},\mathbb B^{-p}\big)$$
 such that
 $$\overline{f}(s\overline{a}_{1,i}\otimes x \otimes s\overline{b}_{1,j})=1_\Lambda\otimes f(s\overline{a}_{1,i}\otimes x \otimes s\overline{b}_{1,j}).$$
 We have that $f=\widetilde{\eta}^{-p}\circ \overline{f}$, where $\widetilde{\eta}^{-p}$ is given in (\ref{equ:bimod-B}).  We define $\widetilde{f}\in \overline{C}^{n}(M,\Omega_{\text{$\Lambda$-$\Pi$}}^{p+1}(M))$ such that $$\widetilde{f}(s\overline{a}_{1,i}\otimes x \otimes s\overline{b}_{1,j})=(-1)^n s \overline{a}_1\otimes f(s\overline{a}_{2,i}\otimes x \otimes s\overline{b}_{1,j})$$
 for $i\geq 1, j\geq 0$ and $i+j=n+p+1$; otherwise for $i=0$,  we set $\widetilde{f}(x\otimes s\overline{b}_{1,n+p+1})=0$. We observe that
\begin{equation}
\label{equationforzero}
\partial^{-p-1}\circ (1\otimes \mathbf 1) \circ \widetilde{f}=\delta'' (\overline{f}), \end{equation}
where $(1\otimes \mathbf 1)$ is defined in \eqref{1mathbf1} and $\delta''$ is the differential of the middle complex in \eqref{sesinduced}. Actually for $i =0$,  we have   $\widetilde{f}(x\otimes s\overline{b}_{1,n+p+1})=0$ and
\begin{align}
(\delta'' (\overline{f}))(1_\Lambda\otimes x \otimes s\overline{b}_{1,n+p+1}\otimes 1_\Pi)
&=(-1)^n 1_\Lambda \otimes (f(1_\Lambda\otimes x \otimes_{\Pi} d_{ex}(1_\Lambda \otimes s\overline{b}_{1,n+p+1}\otimes 1_\Pi)))\nonumber\\
&=1_\Lambda\otimes (\delta'(f)(1_\Lambda \otimes x \otimes s\overline{b}_{1,n+p+1}\otimes 1_\Pi))\nonumber\\
&=0,
\end{align}
where $f$, $\overline{f}$, $\delta''(\overline{f})$ and $\delta'(f)$ are identified as $\Lambda$-$\Pi$-bimodule morphisms; compare \eqref{identification-bimodule}.  For $i\neq 0$, one can check directly that \eqref{equationforzero} holds. By the general  construction of the connecting morphism, we have $\widehat{\theta}_p^M(f)=\widetilde{f}$. Note that we also have $\widetilde{\theta}_{p}^M(f)=\widetilde{f}$. This shows that $\widetilde \theta_p^M= \widehat \theta_p^M$.

Since   $\widetilde \Barr(\Lambda) \otimes_\Lambda M \otimes_\Pi \widetilde\Barr(\Pi)$ is a projective resolution of $M$, by Lemma~\ref{lem:differentialforms} and \cite[Lemma 2.4]{Kel18}, we have the following isomorphism
$$\varinjlim_{\widehat \theta^M_p} \; {\rm Ext}^i_{\text{$\Lambda$-$\Pi$}}(M, \Omega_{\text{$\Lambda$-$\Pi$}}^{p}(M)) \simeq \Hom_{\mathbf{D}_{\rm sg}(\Lambda\otimes\Pi^{\op})} (M, \Sigma^i M).
$$
Combining the above two isomorphisms we obtain the desired isomorphism.
\end{proof}

Recall from \eqref{equ:twomaps} the maps $\alpha_{\sg}^i$ and $\beta_{\sg}^i$.  Analogous to   \cite[Lemma~4.5]{Kel03}, we have the following result.

 \begin{prop}\label{prop:HH-sg-long}
 Assume that the $\Lambda$-$\Pi$-bimodule $M$ is projective on each side. Then there is an exact sequence  of cochain complexes
\begin{align}\label{equ:colimit-long}
 0 \longrightarrow \Sigma^{-1}\overline C^{*}_{\sg}(M, M) \stackrel{{\rm inc}}\longrightarrow \overline{C}_{\sg, R, E}^*(\Gamma, \Gamma) \stackrel{\left( \begin{smallmatrix} {\rm res}_1\\ {\rm res}_2 \end{smallmatrix}\right)}\longrightarrow  \overline{C}^*_{\sg, R}(\Lambda, \Lambda) \oplus \overline{C}^*_{\sg, R}(\Pi, \Pi) \longrightarrow 0,
 \end{align}
 which yields a long exact sequence
$$
 \dotsb \to  \HH_{\sg}^i(\Gamma, \Gamma) \xrightarrow{\left( \begin{smallmatrix} {\rm res}_1\\ {\rm res}_2 \end{smallmatrix}\right)} \HH_{\sg}^i(\Lambda, \Lambda) \oplus \HH_{\sg}^i(\Pi, \Pi) \xrightarrow{(-\alpha_{\sg}^i,  \beta_{\sg}^i)} \Hom_{\mathbf{D}_{\rm sg}(\Lambda\otimes\Pi^{\op})} (M, \Sigma^{i} M) \to \dotsb.
 $$
 \end{prop}

\begin{proof}
The exact sequence of cochain complexes follows immediately from (\ref{exactsequencecolimitM}), since the three maps $\mathrm{inc}$ and $\mathrm{res}_i$ ($i = 1, 2$) are compatible with the colimits. Then taking cohomology, we  have an induced long exact sequence. However, it is tricky to prove that the maps $\alpha_{\sg}^i$ and $\beta_{\sg}^i$ do appear in the induced sequence. For this, we have to analyze the following induced long exact sequence of (\ref{shortexactsequence}).
\begin{align}\label{equ:HH-long}
 \dotsb \to  \HH^i(\Gamma, \Omega^p_{\nc, R, E}(\Gamma)) \xrightarrow{\left( \begin{smallmatrix} {\rm res}_1\\ {\rm res}_2 \end{smallmatrix}\right)}   \begin{matrix} \HH^i(\Lambda, \Omega^p_{\nc, R}(\Lambda))\\
  \oplus \HH^i(\Pi, \Omega^p_{\nc, R}(\Pi)) \end{matrix}\xrightarrow{(-\alpha_p^i,  \beta_p^i)} {\rm Ext}^i_{\Lambda\mbox{-}\Pi}(M, \Omega^p_{\Lambda\mbox{-}\Pi}(M)) \to \dotsb.
 \end{align}
Here, to see that the connecting morphism is indeed $(-\alpha_p^i, \beta_p^i)$, we use the explicit description (\ref{alphabeta}) of the differential,  and apply Lemmas~\ref{lem:liftingalpha} and \ref{lem:liftingbeta}.

Note that we have the following commutative diagram
\begin{align*}
\xymatrix@C=4pc{
 \mathbf{D}^b(\Lambda^e) \ar[d]\ar[r]^-{-\otimes_\Lambda M} &\mathbf{D}^b(\Lambda \otimes \Pi^{\op})\ar[d] &  \ar[l]_-{M\otimes_\Pi -}  \mathbf{D}^b(\Pi^e)\ar[d] \\
 \mathbf{D}_{\rm sg}(\Lambda^e) \ar[r]^-{-\otimes_\Lambda M} &\mathbf{D}_{\rm sg}(\Lambda \otimes \Pi^{\op}) & \ar[l]_-{M\otimes_\Pi -}  \mathbf{D}_{\rm sg}(\Pi^e),
}
\end{align*}
where the vertical functors are the natural quotients.  This induces the following commutative diagram for each $p \geq 0$.
\begin{align*}
\xymatrix{
 \HH^i(\Pi, \Omega_{\nc, R}^p(\Pi))\ar[r]^-{\beta_p^i} \ar[d]&   \mathrm{Ext}^i_{\Lambda\otimes \Pi^{\op}}(M, \Omega_{\text{$\Lambda$-$\Pi$}}^{p}(M))\ar[d] & \ar[l]_-{\alpha_p^i} \HH^i(\Lambda, \Omega_{\nc, R}^p(\Lambda))\ar[d] \\
 \HH_{\sg}^i(\Pi, \Pi)\ar[r]^-{\beta_{\sg}^i} &  \Hom_{\mathbf{D}_{\rm sg}(\Lambda\otimes\Pi^{\op})} (M, \Sigma^i M) & \ar[l]_-{\alpha_{\sg}^i} \HH_{\sg}^i(\Lambda, \Lambda)
}
\end{align*}
Thus, by Lemmas~\ref{lem:HH-colimit} and \ref{lem:computesgspace} we have that
\begin{align}\label{equ:alpha-beta}
\alpha_{\sg}^i = \varinjlim_p  \alpha^i_p \quad \text{and}\quad \beta_{\sg}^i = \varinjlim_p \beta_p^i
\end{align}
 for any $i \in \mathbb Z$.

  Recall the standard fact that the connecting morphism in the long exact sequence induced from  a short exact sequence of complexes is canonical, and thus is compatible with colimits of short exact sequences of complexes. We infer that the long exact sequence induced from  (\ref{equ:colimit-long}) coincides with  the colimit of (\ref{equ:HH-long}). Then the required statement follows from (\ref{equ:alpha-beta}) immediately.
\end{proof}

\begin{rem}\label{rem:tricky}
We would like to stress that, unlike \cite[Lemma 4.5]{Kel03}, the short exact sequence (\ref{equ:colimit-long}) does not have a canonical splitting. In other words, there is no canonical homotopy cartesian square as in \cite[Lemma 4.5]{Kel03}.

The reason is as follows. Note that for each $p \geq 0$, (\ref{shortexactsequence}) splits canonically as an exact sequence of  graded modules, where the sections are given by the inclusions
\begin{align*}
{\rm inc}_1 & \colon \overline{C}^*(\Lambda, \Omega_{\nc, R}^p(\Lambda))\longrightarrow \overline{C}_{ E}^*(\Gamma, \Omega_{\nc, R, E}^p(\Gamma))\\
 {\rm inc}_2 & \colon \overline{C}^*(\Pi, \Omega_{\nc, R}^p(\Pi))\longrightarrow \overline{C}_{ E}^*(\Gamma, \Omega_{\nc, R, E}^p(\Gamma)).
 \end{align*}
We observe that $\theta_p^\Gamma\circ {\rm inc}_1={\rm inc}_1\circ \theta_p^\Lambda$. Taking the colimit, we obtain an inclusion of graded modules  $$ \overline{C}^*_{\sg, R}(\Lambda, \Lambda)\longrightarrow  \overline{C}_{\sg, R, E}^*(\Gamma, \Gamma),$$
which is generally not compatible with the differentials.
We also have
$\theta_p^M\circ \widetilde\alpha_p=\widetilde \alpha_{p+1}\circ \theta_p^\Lambda.$
Taking the colimit, we obtain a lifting at the cochain complex level
$$\widetilde{\alpha}\colon \overline{C}^*_{\sg, R}(\Lambda, \Lambda) \longrightarrow \overline C^{*}_{\sg}(M, M)$$
of the maps $\alpha_{\sg}^i$.

 However, the situation for ${\rm inc}_2$ and $\widetilde\beta_p$ is different from ${\rm inc}_1$.  In general, we have  $$\theta_p^\Gamma\circ {\rm inc}_2\neq {\rm inc}_2\circ \theta_p^\Pi \mbox{ and } \theta_p^M\circ \widetilde \beta_p \neq \widetilde \beta_{p+1}\circ \theta_p^\Pi$$
 since for any $f \in \overline{C}^*(\Pi, \Omega_{\nc, R}^p(\Pi))$ we have
$$
 (\theta_p^\Gamma\circ {\rm inc}_2- {\rm inc}_2\circ \theta_p^\Pi)(f)  = \mathbf 1_{sM} \otimes f
$$
 and for $f \in \overline{C}^{m-p} (\Pi, \Omega_{\nc, R}^p(\Pi))$ we have
 \begin{align*}
 \big((\theta_p^M\circ \widetilde \beta_p )(f)\big)(x \otimes s \overline b_{1, m+1}) &= 0\\
 \big((\widetilde \beta_{p+1}\circ \theta_p^\Pi)(f)\big)(x \otimes s \overline b_{1, m+1}) &= (-1)^{m-p}x \triangleright  (b_1 \otimes f(s \overline b_{2, m+1}))\neq 0,
 \end{align*}
 where  $x\otimes s\overline{b}_{m+1}$ belongs to $M\otimes s\overline{\Pi}^{\otimes m+1}$ and  $\triangleright$ is given in \eqref{triangleright}.
This means that  the section $\left( \begin{smallmatrix} {\rm inc}_1\\ {\rm inc}_2\end{smallmatrix} \right)$  of (\ref{shortexactsequence})  is not compatible with $\theta_p^\Gamma$ and $\theta_p^\Lambda\oplus \theta_p^\Pi$, so we cannot take the colimit.

The above analysis also shows that we cannot  lift the maps $\beta_{\sg}^i$ at the cochain complex level canonically. This forces us to use the tricky argument in the proof of Proposition~\ref{prop:HH-sg-long}. \end{rem}

We are now in a position to prove Theorem~\ref{thm:quasi-iso-bimod}.

\vskip 3pt

\noindent \emph{Proof of Theorem~\ref{thm:quasi-iso-bimod}}. \quad Since both the maps $\alpha_{\sg}^i$ and $\beta_{\sg}^i$ are isomorphisms, the long exact sequence in Proposition~\ref{prop:HH-sg-long} yields a family of short exact sequences
$$0\longrightarrow \HH_{\sg}^i(\Gamma, \Gamma) \xrightarrow{\left( \begin{smallmatrix} {\rm res}_1\\ {\rm res}_2 \end{smallmatrix}\right)} \HH_{\sg}^i(\Lambda, \Lambda) \oplus \HH_{\sg}^i(\Pi, \Pi) \xrightarrow{(-\alpha_{\sg}^i,  \beta_{\sg}^i)} \Hom_{\mathbf{D}_{\rm sg}(\Lambda\otimes\Pi^{\op})} (M, \Sigma^{i} M)\longrightarrow 0.$$ In other words, we have the following  commutative diagram
\begin{align*}
\xymatrix{
 \HH_{\sg}^i(\Gamma, \Gamma)\ar[r]^-{{\rm res}_1} \ar[d]_{{\rm res}_2}&   \HH^i_{\sg} (\Lambda, \Lambda)\ar[d] ^-{\alpha_{\sg}^i} \\
 \HH_{\sg}^i(\Pi, \Pi)\ar[r]^-{\beta_{\sg}^i} &  \Hom_{\mathbf{D}_{\rm sg}(\Lambda\otimes\Pi^{\op})} (M, \Sigma^i M) ,
}
\end{align*}
which is a pullback diagram and  pushout diagram, simultaneously. We infer that both ${\rm res}_i$ are isomorphisms. Then both projections
$${\rm res}_1\colon  \overline{C}_{\sg, R, E}^*(\Gamma, \Gamma)\longrightarrow  \overline{C}_{\sg, R}^*(\Lambda, \Lambda)\ \mbox{ and } \ {\rm res}_2\colon  \overline{C}_{\sg, R, E}^*(\Gamma, \Gamma)\longrightarrow  \overline{C}_{\sg, R}^*(\Pi, \Pi)$$
are quasi-isomorphisms. It is clear that they are both strict $B_\infty$-morphisms, and thus  $B_\infty$-quasi-isomorphisms. This yields the required isomorphism in ${\rm Ho}(B_\infty)$.
\hfill $\square$

\section{Keller's conjecture and the main results}
\label{section8}

Let $\mathbb k$ be a field, and  $\Lambda$ be a finite dimensional $\mathbb k$-algebra. Denote by $\Lambda_0=\Lambda/{\rm rad}(\Lambda)$ the semisimple quotient algebra of $\Lambda$ by its Jacobson radical. Recall from Example~\ref{exm:dsc} that $\mathbf{S}_{\dg}(\Lambda)$ denotes the dg singularity category of $\Lambda$.

Recently, Keller proves the following remarkable result.

\begin{thm}[\cite{Kel18}]\label{Keller's theorem}
Assume that $\Lambda_0$ is separable over $\mathbb{k}$. Then there is a natural isomorphism of graded algebras between $\HH_{\sg}^*(\Lambda^{\op}, \Lambda^{\op})$ and $\HH^*(\mathbf{S}_{\dg}(\Lambda), \mathbf S_{\dg}(\Lambda))$. \hfill $\square$\end{thm}

The following natural conjecture is proposed by Keller.

\begin{conj}[\cite{Kel18}]\label{Keller's conjecture}
Assume that $\Lambda_0$ is separable over $\mathbb{k}$.  There is an isomorphism in the homotopy category $\mathrm{Ho}(B_{\infty})$ of $B_\infty$-algebras
\begin{align}\label{equ:Keller}\overline{C}_{\sg , L}^*(\Lambda^{\op}, \Lambda^{\op})\longrightarrow C^*(\mathbf S_{\dg}(\Lambda), \mathbf S_{\dg}(\Lambda)).
\end{align}
Consequently, there is an induced isomorphism  of Gerstenhaber algebras between $\HH_{\sg}^*(\Lambda^{\op}, \Lambda^{\op})$ and $\HH^*(\mathbf{S}_{\dg}(\Lambda), \mathbf S_{\dg}(\Lambda))$.
\end{conj}

\begin{rem}
Indeed, there is a stronger version of Keller's conjecture: the natural isomorphism in Theorem~\ref{Keller's theorem} lifts to an isomorphism between $\overline{C}_{\sg , L}^*(\Lambda^{\op}, \Lambda^{\op})$ and $C^*(\mathbf S_{\dg}(\Lambda), \mathbf S_{\dg}(\Lambda))$ in $\mathrm{Ho}(B_{\infty})$.  Here, we treat only the above weaker version.
\end{rem}

We say that an algebra $\Lambda$ \emph{satisfies} Keller's conjecture,  provided that  there is such an isomorphism (\ref{equ:Keller}) for $\Lambda$. It is not clear whether Keller's conjecture is left-right symmetric. More precisely, we do not know whether $\Lambda$ satisfies Keller's conjecture even assuming that $\Lambda^{\rm op}$ does so; compare Remark~\ref{rem:left-right}.

The following invariance theorem provides useful reduction techniques for Keller's conjecture. We recall from Subsection~\ref{subsec:one-point} the one-point coextension $\Lambda'=\begin{pmatrix} \mathbb{k} & M \\ 0 & \Lambda \end{pmatrix}$ and the one-point extension $\Lambda''=\begin{pmatrix} \Lambda & N \\ 0 & \mathbb{k}\end{pmatrix}$ of $\Lambda$.

\begin{thm}\label{thm:Keller-redu}
The following statements hold.
\begin{enumerate}
\item The algebra $\Lambda$ satisfies Keller's conjecture if and only if so does $\Lambda'$.
\item The algebra $\Lambda$ satisfies Keller's conjecture if and only if so does $\Lambda''$.
\item Assume that the algebras $\Lambda$ and $\Pi$ are linked by a singular equivalence with a level. Then $\Lambda$ satisfies Keller's conjecture if and only if so does $\Pi$.
\end{enumerate}
\end{thm}

\begin{proof}
For (1), we combine Lemmas~\ref{lem:opce} and \ref{lem:C-quasi-iso} to obtain an isomorphism
$$ C^*(\mathbf S_{\dg}(\Lambda'), \mathbf S_{\dg}(\Lambda')) \simeq  C^*(\mathbf S_{\dg}(\Lambda), \mathbf S_{\dg}(\Lambda))$$
in  the homotopy category $\mathrm{Ho}(B_{\infty})$. Note that $\Lambda'^{\rm op}$ is the one-point extension of $\Lambda^{\rm op}$. Recall from Lemma~\ref{onepointext}
the  strict $B_\infty$-quasi-isomorphism
$$ \overline{C}_{\sg, L, E'}^*({\Lambda'}^{\rm op}, \Lambda'^{\rm op}) \longrightarrow{}  \overline{C}_{\sg, L, E}^*(\Lambda^{\rm op}, \Lambda^{\rm op}).$$
Now applying Lemma~\ref{lem:L-inclu} to both $\Lambda^{\rm op}$ and ${\Lambda'}^{\rm op}$, we obtain an isomorphism
$$ \overline{C}_{\sg, L}^*(\Lambda'^{\rm op}, \Lambda'^{\rm op}) \simeq  \overline{C}_{\sg, L}^*(\Lambda^{\rm op}, \Lambda^{\rm op}).$$
Then (1) follows immediately.

The argument  for (2) is very similar. We apply Lemmas~\ref{lem:ope} and \ref{lem:C-quasi-iso}  to $\Lambda''$. Then we apply Lemma~\ref{lem:strict-L-quasi} to the opposite algebras of $\Lambda$ and $\Lambda''$.

For (3), we observe that by the isomorphism (\ref{iso:duality}), Keller's conjecture is equivalent to the existence of an isomorphism
$$\overline{C}_{\sg , R}^*(\Lambda, \Lambda)^{\rm opp}\longrightarrow C^*(\mathbf S_{\dg}(\Lambda), \mathbf S_{\dg}(\Lambda)).$$
By Lemmas~\ref{lem:sing-equi1} and \ref{lem:C-quasi-iso}, we have an isomorphism
$$C^*(\mathbf S_{\dg}(\Lambda), \mathbf S_{\dg}(\Lambda))\simeq C^*(\mathbf S_{\dg}(\Pi), \mathbf S_{\dg}(\Pi)).$$
Then we are done by Proposition~\ref{prop:sing-equi2}.
\end{proof}

The following  result confirms Keller's conjecture for an algebra $\Lambda$ with radical square zero. Moreover, it relates, at the $B_\infty$-level,  the singular Hochschild cochain complex of $\Lambda$ to  the Hochschild cochain complex of the Leavitt path algebra.

\begin{thm}\label{thm-main}
Let  $Q$ be a finite quiver without  sinks. Denote by   $\Lambda=\mathbb k Q/J^2$  the  algebra   with radical square zero, and  by $L=L(Q)$ the Leavitt path algebra.  Then we  have the following isomorphisms in  $\mathrm{Ho}(B_{\infty})$
$$ \overline{C}_{\sg , L}^*(\Lambda^{\op}, \Lambda^{\op})
\stackrel{\Upsilon}\longrightarrow   C^*(L, L)
\stackrel{\Delta}\longrightarrow   C^*(\mathbf{S}_{\dg}(\Lambda), \mathbf S_{\dg}(\Lambda)).$$
In particular, there are isomorphisms of  Gerstenhaber algebras  $$\HH_{\sg}^*(\Lambda^{\op}, \Lambda^{\op})\stackrel{}\longrightarrow  \HH^*(L, L)\stackrel{}\longrightarrow  \HH^*(\mathbf{S}_{\dg}(\Lambda), \mathbf{S}_{\dg}(\Lambda)).$$
\end{thm}

\begin{proof}
The isomorphism $\Delta$ is obtained as the following composite
\begin{align*}
C^*(L, L) \xrightarrow{{\rm Lem.} \ref{lem:C-dga}} C^*(\mathbf{per}_{\rm dg}(L^{\rm op}), \mathbf{per}_{\rm dg}(L^{\rm op})) \xrightarrow{{\rm Lem.} \ref{lem:C-quasi-iso} + {\rm Prop.} \ref{prop:CY-Li}} C^*(\mathbf{S}_{\dg}(\Lambda), \mathbf S_{\dg}(\Lambda)).
\end{align*}

 Similarly, the isomorphism $\Upsilon$ is obtained by the following diagram
\begin{align}
\label{diagram1}
\xymatrix{
\overline{C}_{\sg , L}^*(\Lambda^{\op}, \Lambda^{\op}) \ar@{.>}[dd]_{\Upsilon} \ar[rr]^-{{\rm Prop.} \ref{lemma-CL1}} && \overline{C}_{\sg , R}^*(\Lambda, \Lambda)^{\rm opp}  && \overline{C}_{\sg , R, E}^*(\Lambda, \Lambda)^{\rm opp} \ar[ll]_-{{\rm Lem.}  \ref{lemma7.3-split}} \\
&&  &&  \overline{C}_{\sg, R}^*(Q, Q)^{\rm opp} \ar[u]_-{{\rm Thm.}  \ref{thm:radical}}  \ar[d]^-{{\rm Prop.}  \ref{prop:interme}} \\
C^*(L, L)&& \overline{C}^*_E(L, L) \ar[ll]_{{\rm Lem.} \ref{lemma6.1-split}}   && \widehat{C}^*(L, L)^{\rm opp} \ar[ll]_{{\rm Thm.} \ref{prop-B4}}
}
\end{align}
We use the isomorphism (\ref{iso:duality}), which is proved in Proposition \ref{lemma-CL1}. The \emph{combinatorial $B_\infty$-algebra} $\overline{C}_{\sg, R}^*(Q, Q)$ of $Q$ is introduced in Section~\ref{subsection:Algebras-radical}. The \emph{Leavitt $B_\infty$-algebra}  $\widehat{C}^*(L, L)$ is introduced in Section~\ref{Section:9}. Both of them are brace $B_\infty$-algebras.

The proof of Theorem \ref{prop-B4} occupies Sections \ref{Section:10} and \ref{section:11}. We obtain an explicit $A_\infty$-quasi-isomorphism $(\Phi_1, \Phi_2, \cdots)\colon  \widehat{C}^*(L, L)\rightarrow \overline{C}_E^*(L, L)$ in Proposition \ref{proposition-Phi}. We emphasize that  each $\Phi_k$ is given by the brace operation on $\widehat{C}^*(L, L)$. The verification of $(\Phi_1, \Phi_2, \cdots)$ being a $B_\infty$-morphism is essentially using the higher pre-Jacobi identity of $\widehat{C}^*(L, L)$.  The isomorphisms of Gerstenhaber algebras follow from Lemma~\ref{lem:Ger}.
\end{proof}

 Denote by $\mathcal{X}$ the class of finite dimensional algebras $\Lambda$ with the following property:  there exists some finite quiver $Q$ without sinks, such that $\Lambda$ is connected to $\mathbb{k}Q/J^2$  by a finite zigzag of one-point (co)extensions and singular equivalences with levels. For example, if $Q'$ is \emph{any} finite quiver possibly with sinks, then $\mathbb{k}Q'/J^2$ clearly lies in $\mathcal{X}$.

We have the following immediate consequence of Theorems~\ref{thm:Keller-redu} and \ref{thm-main}.

\begin{cor}
Any algebra belonging to the class $\mathcal{X}$ satisfies Keller's conjecture. \hfill $\square$
\end{cor}

By \cite[Theorem~6.3]{CLiuW}, there is a singular equivalence with level between any given Gorenstein quadratic monomial algebra and its associated algebra with radical square zero. It follows that  $\mathcal{X}$ contains all Gorenstein quadratic monomial algebras, and thus Keller's conjecture holds for  them. By \cite{GR}, all finite dimensional gentle algebras are Gorenstein quadratic monomial. We conclude that Keller's conjecture holds for all finite dimensional gentle algebras. Let us mention the connection between gentle algebras and Fukaya categories \cite{HKK}.

\section{Algebras with radical square zero and the combinatorial $B_\infty$-algebra}
\label{subsection:Algebras-radical}

Let $Q$ be a finite quiver without sinks. Let $\Lambda=\mathbb k Q/J^2$ be the corresponding algebra with radical square zero. We will give a combinatorial description of the  singular Hochschild cochain complex of $\Lambda$; see Subsection~\ref{subsection:combinatorial}.  For its $B_\infty$-algebra structure, we describe it as the combinatorial $B_\infty$-algebra $\overline{C}_{\sg, R}^*(Q, Q)$ of $Q$; see Subsection \ref{subsection:combinatorial1}.

\subsection{A combinatorial description of the singular Hochschild cochain complex}\label{subsection:combinatorial}

 Set $E=\mathbb k Q_0$, viewed as a semisimple subalgebra of $\Lambda$. Then $\overline{\Lambda}=\Lambda/(E\cdot 1_\Lambda)$ is identified with $\mathbb k Q_1$.  We will give  a description of  the $E$-relative right singular Hochschild cochain complex $\overline{C}_{\sg,R, E}^*(\Lambda, \Lambda)$ by  parallel paths in the quiver $Q$.

For two subsets $X$ and $Y$ of paths in $Q$, we denote
$$X//Y:=\{ (\gamma, \gamma')\in X\times Y \mid \text{$s(\gamma)=s(\gamma')$ and $t(\gamma)=t(\gamma')$}\}.$$
An element in $Q_m//Q_p$ is called a {\it parallel path} in $Q$. We will abbreviate a path $\beta_m\cdots \beta_2\beta_1\in Q_m$ as $\beta_{m, 1}$. Similarly, a path $\alpha_p\cdots \alpha_2\alpha_1\in Q_p$ is denoted by $\alpha_{p, 1}$.

For a set $X$, we denote by $\mathbb k(X)$ the $\mathbb k$-vector space spanned by elements in $X$. We will view $\mathbb{k}(Q_m//Q_p)$ as a graded $\mathbb k$-space concentrated on degree $m-p$.  For a graded $\mathbb k$-space $A$, let $s^{-1}A$ be the $(-1)$-shifted graded space such that $(s^{-1}A)^i=A^{i-1}$ for $i\in\mathbb Z$. The element in $s^{-1}A$ is denoted by $s^{-1}a$ with $|s^{-1}a|=|a|+1$. Roughly speaking, we have $|s^{-1}|=1$. Therefore, $s^{-1}\mathbb k(Q_m//Q_p)$ is concentrated on degree $m-p+1$.

We will define a $\mathbb k$-linear map (of degree zero) between graded spaces
$$\kappa_{m, p}\colon    \mathbb k(Q_m// Q_p)\oplus s^{-1}  \mathbb k(Q_m//Q_{p+1})\longrightarrow {\rm Hom}_{\text{$E$-$E$}}((s\overline{\Lambda})^{\otimes_E m}, (s\overline{\Lambda})^{\otimes_E p}\otimes_E \Lambda).$$
For  $y=(\alpha_{m, 1}, \beta_{p,1})\in Q_m//Q_p$ and any monomial $x=s\alpha_m'\otimes_E\cdots\otimes_E s\alpha_1'\in (s\overline{\Lambda})^{\otimes_E m}$ with $\alpha_j' \in Q_1$ for any $1\leq j \leq m$,  we set
$$\kappa_{m, p}(y)(x)=
\begin{cases} (-1)^{\epsilon}s\beta_p\otimes_E\cdots\otimes_E s\beta_1\otimes_E 1 & \text{if $\alpha_j = \alpha_j'$ for all $1\leq j\leq m$,}\\
0 & \text{otherwise}.
\end{cases}$$
For $s^{-1}y'=s^{-1}(\alpha_{m,1}, \beta_{p, 0})\in s^{-1}\mathbb k(Q_m//Q_{p+1})$,  we set
$$\kappa_{m, p}(s^{-1}y')(x)= \begin{cases} (-1)^{\epsilon}s\beta_p\otimes_E\cdots\otimes_E s\beta_1\otimes_E \beta_0& \text{if $\alpha_j = \alpha_j' $ for all $1\leq j\leq m,$}\\
0 & \text{otherwise.}
\end{cases}
$$
Here, we denote $\epsilon=(m-p)p+\frac{(m-p)(m-p+1)}{2}$.

\begin{lem}{\rm (\cite[Lemma 3.3]{Wan})}
\label{lemma6.2}
For any $m, p\geq 0$, the above map $\kappa_{m,p}$ is an isomorphism of graded vector spaces.
\hfill $\square$
\end{lem}

We define a graded vector space for each $p\geq 0$,
$$ \mathbb k(Q//Q_p):=\prod_{m\geq 0}  \mathbb k(Q_m//Q_p), $$
where the degree of  $(\gamma, \gamma')$ in $Q_m//Q_p$ is $m-p$. We define a $\mathbb k$-linear map of degree zero
$$\theta_{p, R}\colon  \mathbb k(Q//Q_p) \longrightarrow  \mathbb k(Q//Q_{p+1}), \quad (\gamma, \gamma')\longmapsto \sum_{\{\alpha\in Q_1| s(\alpha)=t(\gamma)\}} (\alpha\gamma, \alpha\gamma').$$
Denote by $\overline{C}^*_{\sg,R,  0}(Q, Q)$ the colimit of the inductive system of graded vector spaces
$$\mathbb  k(Q//Q_0)\xrightarrow{\theta_{0, R}} \mathbb  k(Q//Q_1)\xrightarrow{\theta_{1, R}} \mathbb  k(Q//Q_2)\xrightarrow{\theta_{2, R}} \cdots \xrightarrow{
\theta_{p-1, R}} \mathbb  k(Q//Q_p)\xrightarrow{\theta_{p, R}}\cdots.$$
Therefore,  for any $m\in \mathbb Z$, we have
$$\overline{C}_{\sg ,R, 0}^m(Q, Q)=\lim\limits_{\substack{\longrightarrow\\ \theta_{p, R}}} \mathbb  k(Q_{m+p}//Q_p).$$

We define  a complex
\begin{align}\label{equ:sg-R-Q}
\overline{C}_{\sg , R}^*(Q, Q)=\overline{C}_{\sg ,R,  0}^{ *}(Q, Q) \oplus s^{-1}\overline{C}_{\sg , R, 0}^{ *}(Q,Q),
\end{align}
whose differential $\delta$ is induced by
 \begin{equation}
\label{thediff}\left(\begin{smallmatrix} 0 & D_{m, p}\\ 0 & 0\end{smallmatrix}\right)\colon   \mathbb k(Q_{m}//Q_p)\oplus s^{-1} \mathbb k(Q_{m}// Q_{p+1})\longrightarrow \mathbb k(Q_{m+1}//Q_p)\oplus s^{-1}\mathbb k(Q_{m+1}// Q_{p+1}). \end{equation}
For $(\gamma, \gamma')\in Q_{m}//Q_p$, we have
\begin{equation}
\label{equation-defD}
D_{m, p}((\gamma, \gamma'))=\sum_{\{\alpha\in Q_1 \; |\; s(\alpha)=t(\gamma)\}}s^{-1}(\alpha\gamma, \alpha\gamma')- (-1)^{m-p}\sum_{\{\beta\in Q_1\; |\; t(\beta)=s(\gamma)\}} s^{-1}(\gamma \beta, \gamma' \beta).
\end{equation}
We implicitly use the identity $s^{-1}\theta_{p+1, R}\circ D_{m, p}=D_{m+1, p+1}\circ \theta_{p, R}$. Here if the set $\{\beta\in Q_1\; |\; t(\beta)=s(\gamma)\}$ is empty then we define
$\sum_{\{\beta\in Q_1\; |\; t(\beta)=s(\gamma)\}} s^{-1}(\gamma \beta, \gamma' \beta) =0.$

Recall from Subsection~\ref{subsec:rel-sinHo} that $\Omega_{\nc, R, E}^p(\Lambda)= (s\overline{\Lambda})^{\otimes_E p}\otimes_E \Lambda$. Recall from \eqref{rightblacktriangle} the left $\Lambda$-action $\blacktriangleright.$ Note that we have
\begin{align*}
\beta_{p+1} \blacktriangleright (s\beta_p \otimes_E \dotsb\otimes_E s \beta_1 \otimes_E \beta_0) = \begin{cases}
0 & \text{if $\beta_0 \in Q_1$}\\
(-1)^p s\beta_{p+1} \otimes_E \dotsb\otimes_E s\beta_2 \otimes_E \beta_1\beta_0 & \text{if $\beta_0 \in Q_0$}
\end{cases}
\end{align*}
where $\beta_i \in Q_1 =\overline \Gamma$ for $1\leq i \leq p+1$. Then it is not difficult to show that the map \eqref{thediff} is compatible with the differential $\delta_{ex}$  of $\overline C^*(\Lambda, \Omega_{\nc, R, E}^p(\Lambda))$. More precisely,  the following  diagram is commutative
\begin{equation*}
\xymatrix@C=2pc{
{\rm Hom}_{\text{$E$-$E$}}((s\overline{\Lambda})^{\otimes_E m}, (s\overline{\Lambda})^{\otimes_E p}\otimes_E \Lambda) \ar[r]^-{\delta_{ex}} & {\rm Hom}_{\text{$E$-$E$}}((s\overline{\Lambda})^{\otimes_E m+1}, (s\overline{\Lambda})^{\otimes_E p}\otimes_E \Lambda) \\
\mathbb k(Q_m// Q_p)\oplus s^{-1}  \mathbb k(Q_{m}//Q_{p+1})\ar[u]_-{\cong}^-{\kappa_{m, p}}\ar[r]^-{\left(\begin{smallmatrix} 0 & D_{m, p}\\ 0 & 0\end{smallmatrix}\right)}& \mathbb k(Q_{m+1}// Q_p)\oplus s^{-1}  \mathbb k(Q_{m+1}//Q_{p+1}) \ar[u]_-{\cong}^-{\kappa_{m+1, p}}}
\end{equation*}
where recall that  the formula for  $\delta_{ex}$ is given in Subsection \ref{subsection-dg-HH}.

The above commutative diagram allows us to take the colimit along the isomorphisms $\kappa_{m, p}$ in  Lemma~\ref{lemma6.2}. Therefore, we have the following result.

\begin{lem}\label{lem:kappa}
The isomorphisms $\kappa_{m, p}$ induce  an isomorphism of complexes
\begin{equation*}\label{equation:kappa}
\kappa\colon  \overline{C}_{\sg, R}^*(Q, Q)\xrightarrow{\sim} \overline{C}_{\sg,R, E}^*(\Lambda, \Lambda).
\end{equation*}
\end{lem}

\subsection{The combinatorial $B_\infty$-algebra}\label{subsection:combinatorial1}

 In this subsection, we will transfer, via the isomorphism $\kappa$,  the cup product $-\cup_R-$ and brace operation $-\{-, \dotsc, -\}_R$ of $\overline{C}_{\sg, R, E}^*(\Lambda, \Lambda)$ to $\overline{C}_{\sg, R}^*(Q, Q)$.  We will provide an example for illustration.

 By abuse of notation, we still denote the  cup product and brace operation on $\overline{C}_{\sg, R}^*(Q, Q)$ by $-\cup_R-$ and $-\{-, \dotsc, -\}_R$.

We will use the following \emph{non-standard} sequences to depict parallel paths.
 \begin{enumerate}[(i)]
\item We write
 $s^{-1}x=s^{-1}(\alpha_{m, 1}, \beta_{p, 0})\in s^{-1}\overline{C}_{\sg, R, 0}^*(Q, Q)$ as
\begin{align}\label{nonstandardsequence1}
\xymatrix{
{\color{blue}\xrightarrow{\beta_0}}\xrightarrow{\beta_1}\dotsb\xrightarrow{\beta_p}\xleftarrow{\alpha_m}\dotsb\xleftarrow{\alpha_2}\xleftarrow{\alpha_1}.
}
\end{align}
\item We write $x=(\alpha_{m, 1}, \beta_{p, 1})\in \overline{C}_{\sg, R, 0}^*(Q, Q)$ as
\begin{align}
\label{picture}\xymatrix{
\xrightarrow{\beta_1}\dotsb\xrightarrow{\beta_p}\xleftarrow{\alpha_m}\dotsb\xleftarrow{\alpha_2}\xleftarrow{\alpha_1},
}
\end{align}
Here,  all $ \alpha_1,\dotsc,  \alpha_m, \beta_0, \beta_1, \dotsc, \beta_p $ are arrows in $Q$.
\end{enumerate}
The above sequences have the following feature: the left part consists of rightward arrows, and the right part consists of leftward arrows. Recall that $\Omega^p_{{\rm nc}, R, E}(\Lambda)=(s\overline{\Lambda})^{\otimes_E p}\otimes_E \Lambda = (s\overline{\Lambda})^{\otimes_E p}\otimes_E \overline \Lambda \oplus (s\overline{\Lambda})^{\otimes_E p}\otimes_E  E$, and that the leftmost arrow $\beta_0$ in (i) is an element in the tensor factor $\overline \Lambda$. To emphasize this fact, we color the arrow  blue. These sequences will be quite convenient  to express the cup product and brace operation on $\overline{C}^*_{\sg, R}(Q, Q)$, as we will see below.

Let us first describe $-\cup_R-$  on $\overline{C}^*_{\sg, R}(Q, Q)$. Let
$$s^{-1}x=s^{-1} (\alpha_{m, 1},  \beta_{p, 0})=({\color{blue}\xrightarrow{\beta_0}} \xrightarrow{\beta_1} \cdots\xrightarrow{\beta_p} \xleftarrow{\alpha_m}\cdots\xleftarrow{\alpha_1})$$
$$ s^{-1}y=s^{-1}(\alpha'_{n, 1}, \beta_{q, 0}')=({\color{blue} \xrightarrow{\beta_0'}} \xrightarrow{\beta_1'} \cdots\xrightarrow{\beta_q'}\xleftarrow{\alpha'_n}\cdots\xleftarrow{\alpha'_1})$$ be two elements in $s^{-1}\overline C_{\sg, R, 0}^*(Q, Q)$. Let $$
z=(\alpha_{m, 1}, \beta_{p, 1})=( \xrightarrow{\beta_1}\cdots\xrightarrow{\beta_p} \xleftarrow{\alpha_m}\cdots\xleftarrow{\alpha_1})$$ $$w=(\alpha'_{n, 1}, \beta_{q,1}')=( \xrightarrow{\beta_1'}\cdots\xrightarrow{\beta_q'} \xleftarrow{\alpha'_n}\cdots\xleftarrow{\alpha'_1})$$
be two elements in $\overline{C}^*_{\sg, R, 0}(Q, Q)$.   The cup product $-\cup_R-$ is given by (C1)-(C4).
\begin{enumerate}
\item[(C1)] $(s^{-1}x)\cup_R (s^{-1}y)=0;$

\item[(C2)] The cup product $z\cup_R w$ is given by the following parallel path
$$\delta_{s(\alpha_1), s(\beta_1')}(\underbrace{\xrightarrow{\beta_1}\cdots \xrightarrow{\beta_p} \xleftarrow{\alpha_m}\cdots\xleftarrow{\alpha_1}}_{z}\underbrace{\xrightarrow{\beta_1'}\cdots\xrightarrow{\beta_q'}\xleftarrow{\alpha'_n}\cdots\xleftarrow{\alpha'_1}}_w).$$ Here,  we  replace  the subsequence
$\xleftarrow{\alpha}\xrightarrow{\beta}$ by $\delta_{\alpha, \beta}$ iteratively, till  obtaining  a parallel  path,  that is,  the left part consists of rightward arrows and the right part consists of leftward arrows. More precisely, we have
$$z\cup_R w=\begin{cases} \prod_{i=1}^q\delta_{\beta_i', \alpha_{i}} \; (\alpha_{m, q+1}\alpha'_{n,1}, \beta_{p, 1}) &\mbox{if $q<m$,}\\
\prod_{i=1}^m\delta_{\beta'_{i}, \alpha_{i}} \; (\alpha'_{n, 1}, \beta'_{q, m+1}\beta_{p, 1}) &\mbox{if $q\geq m$,}\end{cases}$$

\item[(C3)]  $(s^{-1}x)\cup_R w$ is obtained by replacing $\xleftarrow{\alpha}\xrightarrow{\beta}$ with $\delta_{\alpha, \beta}$, iteratively
$$
(\underbrace{{\color{blue} \xrightarrow{\beta_0}} \xrightarrow{\beta_1}\cdots \xrightarrow{\beta_p}\xleftarrow{\alpha_m}\cdots\xleftarrow{\alpha_1}}_{s^{-1}x}\underbrace{\xrightarrow{\beta_1'}\cdots\xrightarrow{\beta_q'}\xleftarrow{\alpha'_n}\cdots\xleftarrow{\alpha'_1}}_{w}).
$$
Therefore, we have
$$ (s^{-1}x)\cup_R w=\begin{cases}\prod_{i=1}^q\delta_{\beta_i', \alpha_{i}} \; s^{-1}(\alpha_{m, q+1}\alpha'_{n,1}, \beta_{p, 0}) &\mbox{if $q<m$,}\\
\prod_{i=1}^m\delta_{\beta'_{i}, \alpha_{i}}\;  s^{-1}(\alpha'_{n, 1}, \beta'_{q, m+1}\beta_{p, 0}) &\mbox{if $q\geq m$;}\end{cases}$$

\item[(C4)] $z\cup_R (s^{-1}y)$ is obtained by replacing
$\xleftarrow{\alpha}\xrightarrow{\beta}$ with $\delta_{\alpha, \beta}$, iteratively
$$( {\color{blue}\xrightarrow{\beta_0'}}\underbrace{\xrightarrow{\beta_1} \cdots\xrightarrow{\beta_p} \xleftarrow{\alpha_m}\cdots\xleftarrow{\alpha_1}}_{z} \xrightarrow{\beta_1'}\cdots\xrightarrow{\beta_q'} \xleftarrow{\alpha'_n}\cdots\xleftarrow{\alpha'_1}).$$
Therefore, we have
$$ z\cup_R (s^{-1}y)=\begin{cases}\prod_{i=1}^q\delta_{\beta_i', \alpha_{i}} \; s^{-1}(\alpha_{m, q+1}\alpha'_{n,1}, \beta_{p, 1}\beta_0') &\mbox{if $q<m$,}\\
\prod_{i=1}^m\delta_{\beta'_{i}, \alpha_{i}}  \;  s^{-1}(\alpha'_{n, 1}, \beta'_{q, m+1}\beta_{p, 1}\beta'_0) &\mbox{if $q\geq m$.}
\end{cases}$$
\end{enumerate}
\vskip 5pt

Let us describe the brace operation $-\{-, \dotsc, -\}_R$ on $\overline{C}_{\sg , R}^*(Q, Q)$ in the following  cases (B1)-(B3).
\begin{enumerate}
\item[(B1)] For any $x\in \overline{C}_{\sg , R}^*(Q, Q)$, we have
$$x\{y_1, \dotsc, y_k\}_R=0$$
if there exists some $1\leq j\leq k$ with $y_j\in \overline{C}_{\sg ,R,  0}^*(Q, Q)\subset \overline{C}_{\sg , R}^*(Q, Q)$.

\item[(B2)] If  $s^{-1}y_j\in s^{-1}\overline{C}_{\sg , R, 0}^*(Q, Q)$ is  such that $y_j$ is a parallel path for each $1\leq j\leq k$, and
 $s^{-1}x=s^{-1}(\alpha_{m, 1}, \beta_{p, 0})\in s^{-1}\overline{C}_{\sg , R, 0}^*(Q, Q)$, then
\begin{equation*}\label{equation-radical-B}
(s^{-1}x)\{s^{-1}y_1, \dotsc, s^{-1}y_k\}_R=\sum\limits_{\substack{a+b=k,\; a,b\geq 0\\ 1\leq i_1< i_2 <\cdots < i_a\leq m\\ 1\leq l_1\leq l_2\leq  \cdots \leq l_b\leq p }}  (-1)^{a+\epsilon}\mathfrak{b}^{(i_1, \dotsc, i_a)}_{(l_1, \dotsc, l_{b})}(s^{-1}x; s^{-1}y_1, \dotsc, s^{-1}y_k),
\end{equation*}
where
 $\mathfrak b^{(i_1, \dotsc, i_a)}_{(l_1, \dotsc, l_b)}(s^{-1}x; s^{-1}y_1, \dotsc, s^{-1}y_k)$ is illustrated as follows
\begin{align*}
{\color{blue} \xrightarrow{\beta_0}} \xrightarrow{\beta_1}\cdots \xrightarrow{\beta_{l_1-1}} y_1\xrightarrow{\beta_{l_1}} \cdots \xrightarrow{\beta_{l_{b}-1}}y_{b}\xrightarrow{\beta_{l_b}} \cdots \xrightarrow{\beta_p} \xleftarrow{\alpha_m}\cdots \xleftarrow{\alpha_{i_a}}y_{b+1}\xleftarrow{\alpha_{i_a-1}}\cdots \xleftarrow{\alpha_{i_1}} y_{k} \cdots\xleftarrow{\alpha_1}.
\end{align*}
To save the space, we just use the symbol
$y_j$ to indicate the sequence of the parallel path $y_j$ as in (\ref{picture}) for $1\leq j\leq k$. We replace any subsequence  $\xleftarrow{\alpha}\xrightarrow{\beta}$ by $\delta_{\alpha, \beta}$ iteratively, and then arrive at a well-defined parallel path.

Let us explain the sign $(-1)^{a + \epsilon}$ appeared above.
The sign $$\epsilon = \sum_{r=1}^{b} (|s^{-1}y_r|-1)(m+p-l_{r}+1)+\sum_{r=1}^{a} (|s^{-1}y_{k-r+1}|-1)(i_{r}-1)$$
is obtained via the Koszul sign rule by reordering the positions ($\beta_i^*$ and $\alpha_j$ are of degree one) of the elements $$ \beta_{0}^*, \beta_1^*,  \dotsc \beta_p^*,  \alpha_m,  \dotsc \alpha_1,    y_1,  y_2 ,  \dotsb , y_k;  $$
and the extra sign $(-1)^a$ is to make sure that the brace operation is compatible with the colimit maps $\theta_{*, R}$.

\item[(B3)]If   $s^{-1}y_j\in s^{-1}\overline{C}_{\sg , R, 0}^*(Q, Q)$ is such that $y_j$ is a parallel path for each $1\leq j\leq k$, and $x=(\alpha_{m, 1}, \beta_{p, 1})\in \overline{C}_{\sg , R, 0}^*(Q, Q)$, then
\begin{equation*}
x\{s^{-1}y_1, \dotsc, s^{-1}y_k\}_R=\sum\limits_{\substack{a+b=k,\; a, b\geq 0 \\1\leq i_1< i_2 <\cdots < i_a\leq m\\ 1\leq l_1\leq l_2\leq  \cdots \leq l_{b}\leq p }} (-1)^{a+\epsilon}\mathfrak b^{(i_1, \dotsc, i_a)}_{(l_1, \dotsc, l_{b})}(x; s^{-1}y_1, \dotsc, s^{-1}y_k),
\end{equation*}
where $\mathfrak b^{(i_1, \dotsc, i_a)}_{(l_1, \dotsc, l_{b})}(x; s^{-1}y_1, \dotsc, s^{-1}y_k)$ is  obtained from the following sequence by replacing  $\xleftarrow{\alpha}\xrightarrow{\beta}$ with $\delta_{\alpha, \beta}$ iteratively
\begin{align*}
\xrightarrow{\beta_1} \cdots \xrightarrow{\beta_{l_1-1}} y_1\xrightarrow{\beta_{l_1}} \cdots \xrightarrow{\beta_{l_{b}-1}}y_{b}\xrightarrow{\beta_{l_b}} \cdots \xrightarrow{\beta_p} \xleftarrow{\alpha_m} \cdots \xleftarrow{\alpha_{i_a}}y_{b+1}\xleftarrow{\alpha_{i_a-1}}\cdots \xleftarrow{\alpha_{i_1}} y_{k} \cdots\xleftarrow{\alpha_1},
\end{align*}
and   $\epsilon$ is the same as in (B2).
\end{enumerate}

\begin{thm}\label{thm:radical}
The complex $\overline{C}^*_{\sg, R}(Q, Q)$, equipped with the cup product $-\cup_R-$ and brace operation $-\{-, \dotsc, -\}_R$, is a brace $B_{\infty}$-algebra. Moreover, the isomorphism $\kappa\colon  \overline{C}^*_{\sg, R}(Q, Q)\rightarrow \overline{C}^*_{\sg, R, E}(\Lambda, \Lambda)$ is a strict $B_\infty$-isomorphism.
\end{thm}

The resulted $B_\infty$-algebra $\overline{C}^*_{\sg, R}(Q, Q)$ is called the \emph{combinatorial $B_\infty$-algebra} of $Q$.

\begin{proof}
The above cup product $-\cup_R-$ and brace operation $-\{-, \dotsc, -\}_R$ on $\overline{C}_{\sg, R}^*(Q, Q)$ are  transferred from $\overline{C}_{\sg, R, E}^*(\Lambda, \Lambda)$ via the isomorphism $\kappa$; compare Theorem~\ref{thm-Wan2} and Lemma~\ref{lem:kappa}.  More precisely, for any $x, y, y_1, \dotsc, y_k \in \overline{C}^*_{\sg, R}(Q, Q)$  we may check
\begin{equation}\label{checkbrace}
\begin{split}
\kappa(x \cup_R y) &= \kappa(x) \cup_R \kappa(y)\\
(-1)^{a + \epsilon} \kappa\big(\mathfrak{b}^{(i_1, \dotsc, i_a)}_{(l_1, \dotsc, l_{b})}(x; y_1, \dotsc, y_k)\big) &  = (-1)^{b}B^{(i_1, \dotsc, i_a)}_{(l_1, \dotsc, l_{b})}(\kappa(x); \kappa(y_1), \dotsc, \kappa(y_k)),
\end{split}
\end{equation}
where $\epsilon$ is defined as in (B2) above.

We may check the first identity  case by case. Let $x=(\alpha_{m, 1}, \beta_{p, 1})$ and $y=(\alpha'_{n, 1}, \beta_{q,1}')$. Suppose first that $q<m$. Then for $z\in s\overline{\Lambda}^{\otimes_E m+n-q}$ we have
\begin{align*}
\kappa(x\cup_R y)(z)
&=  \prod_{i=1}^q\delta_{\beta_i', \alpha_{i}} \; \kappa((\alpha_{m, q+1}\alpha'_{n,1}, \beta_{p, 1}))(z) \\
&= \!\begin{cases}
(-1)^{\epsilon_1} \! \prod_{i=1}^q\delta_{\beta_i', \alpha_{i}}\;  s\beta_{p}\smallotimes\cdots\smallotimes s\beta_{1} \smallotimes 1, & \text{~if~ $z =\!s\alpha_{m}\!\smallotimes \!\cdots \!\smallotimes \!s\alpha_{q+1}
\smallotimes s\alpha'_{n}\!\smallotimes\!\cdots \!\smallotimes\! s\alpha'_1$,} \\
0,&  \text{otherwise,}
\end{cases}
\end{align*} where $$\epsilon_1 = (m+n-p-q)p + \frac{(m+n-p-q)(m+n-p-q+1)}{2}.$$ Here the first equality follows from (C2) and the second identity follows from the definition of $\kappa$.  Note that we have $\kappa(x)\in {\rm Hom}_{\text{$E$-$E$}}(s\overline{\Lambda}^{\otimes_E m}, \Omega_{\nc, R,E}^p(\Lambda))$ and $\kappa(y)\in {\rm Hom}_{\text{$E$-$E$}}(s\overline{\Lambda}^{\otimes_E n}, \Omega_{\nc,R,E}^q(\Lambda))$. By the definition of the cup product of $\overline{C}^*_{\sg, R, E}(\Lambda, \Lambda)$ in \eqref{equation-defcupforsg}, we have $\kappa(x)\cup_R\kappa(y)\in {\rm Hom}_{\text{$E$-$E$}}(s\overline{\Lambda}^{\otimes_E m+n}, \Omega_{\nc, R,E}^{p+q}(\Lambda)).$ One may check directly that $$\kappa(x) \cup_R \kappa(y)=(\theta_{p+q-1, R,E}\circ \cdots \circ\theta_{p+1,R,E}\circ\theta_{p,R,E})(\kappa(x \cup_R y) ).$$  Thus we have $\kappa(x\cup_R y)=\kappa(x)\cup_R \kappa(y)$ in $\overline{C}^*_{\sg, R, E}(\Lambda, \Lambda).$ Similarly,  we may check for $q\geq m$. We omit the routine verification for the other three cases, according to (C1), (C3) and (C4).

The second identity in \eqref{checkbrace} follows from the observation that the Deletion Process in Definition~\ref{defn:brace-A} exactly corresponds to the iterative replacement  in (B2) and (B3).    See Example~\ref{exm-paralle} below for a detailed illustration.
\end{proof}

\vskip 5pt

\begin{exm}\label{exm-paralle}
Consider the following  four monomial elements in $\overline{C}^*_{\sg, R}(Q, Q)$
\begin{align*}
s^{-1}x &= s^{-1} (\alpha_5\alpha_4\alpha_3\alpha_2 \alpha_1, \beta_3\beta_2 \beta_1\beta_0)\\
s^{-1}y_1 &= s^{-1}(\alpha_3' \alpha_2'\alpha_1', \beta_1' \beta_0')\\
s^{-1}y_2 &= s^{-1} (\alpha_3''\alpha_2''\alpha_1'', \beta_3''\beta_2''\beta_1''\beta_0'')\\
s^{-1}y_3 &= s^{-1} (\alpha_2'''\alpha_1''', \beta_3'''\beta_2'''\beta_1'''\beta_0''').
\end{align*}
According to \eqref{nonstandardsequence1}, they may be depicted in the following way
\begin{equation*}
\begin{split}
s^{-1}x & = ({\color{blue} \xrightarrow{\beta_0}}\xrightarrow{\beta_1} \xrightarrow{\beta_2}\xrightarrow{\beta_3}\xleftarrow{\alpha_5}\xleftarrow{\alpha_4}\xleftarrow{\alpha_3}\xleftarrow{\alpha_2}\xleftarrow{\alpha_1})\\
s^{-1}y_1 &=({\color{blue}\xrightarrow{\beta_0'}} \xrightarrow{\beta_1'}\xleftarrow{\alpha_3'}\xleftarrow{\alpha_2'}\xleftarrow{\alpha_1'})\\  s^{-1}y_2& =({\color{blue}\xrightarrow{\beta_0''}}\xrightarrow{\beta_1''}\xrightarrow{\beta_2''}\xrightarrow{\beta_3''}\xleftarrow{\alpha_3''}\xleftarrow{\alpha_2''} \xleftarrow{\alpha_1''})\\ s^{-1}y_3&=({\color{blue}\xrightarrow{\beta_0'''}} \xrightarrow{\beta_1'''} \xrightarrow{\beta_2'''}\xrightarrow{\beta_3'''}\xleftarrow{\alpha_2'''}\xleftarrow{\alpha_1'''}).\end{split}\end{equation*}
By Formula (B2), the operation $\mathfrak b^{(2,4)}_{(2)}(s^{-1}x; s^{-1}y_1, s^{-1}y_2, s^{-1}y_3)$ is depicted by
\begin{align*}
({\color{blue}\xrightarrow{\beta_0}}\xrightarrow{\beta_1}\underbrace{\xrightarrow{\beta_0'} \dotsc\xleftarrow{\alpha_1'}}_{s^{-1}y_1}\xrightarrow{\beta_2}\xrightarrow{\beta_3}\xleftarrow{\alpha_5}\xleftarrow{\alpha_4}\underbrace{\xrightarrow{\beta_0''}\dotsc\xleftarrow{\alpha_1''}}_{s^{-1}y_2}
\xleftarrow{\alpha_3}\xleftarrow{\alpha_2}\underbrace{\xrightarrow{\beta_0'''} \dotsc\xleftarrow{\alpha_1'''}}_{s^{-1}y_3}\xleftarrow{\alpha_1}).
\end{align*}
After replacing $\xleftarrow{\alpha}\xrightarrow{\beta}$ with $\delta_{\alpha, \beta}$ iteratively, we get
\begin{align}
\label{lambda}
 \lambda ({\color{blue}\xrightarrow{\beta_0}}\xrightarrow{\beta_1}\xrightarrow{\beta_0'}\xrightarrow{\beta_1'}\xrightarrow{\beta_3''}\xleftarrow{\alpha_3''}\xleftarrow{\alpha_2'''}\xleftarrow{\alpha_1'''}\xleftarrow{\alpha_1}),
\end{align}
where $\lambda=\delta_{\alpha_1', \beta_2}\delta_{\alpha_2', \beta_3}\delta_{\alpha_4, \beta_0''} \delta_{\alpha_5, \beta_1''}\delta_{\alpha_3', \beta''_2}  \delta_{\alpha_2, \beta_0'''}\delta_{\alpha_3, \beta_1'''}    \delta_{\alpha_1'', \beta_2'''} \delta_{\alpha_2'', \beta_3'''}$. Hence,
\begin{align}\label{b242}
\mathfrak b^{(2,4)}_{(2)}(s^{-1}x; s^{-1}y_1, s^{-1}y_2, s^{-1}y_3) =\lambda s^{-1} (\alpha_3''\alpha_2'''\alpha_1'''\alpha_1, \beta_3''\beta_1'\beta_0'\beta_1\beta_0).
\end{align}

\bigskip
Let us check that $\kappa$ preserves the brace operations. Note that
\begin{itemize}
\item[$\bullet$] $f:=\kappa(s^{-1}x)\in \overline{C}_E^{2}(\Lambda, \Omega^3_{\nc, R}(\Lambda))$ is uniquely determined by
$$s\alpha_5\otimes s\alpha_4\otimes s\alpha_3\otimes s\alpha_2\otimes s\alpha_1\mapsto -s\beta_3\otimes s\beta_2\otimes s\beta_1 \otimes \beta_0,$$
i.e. sending any other monomial to zero;
\item[$\bullet$] $g_1:=\kappa(s^{-1}y_1) \in \overline{C}^{2}(\Lambda, \Omega_{\nc, R}^1(\Lambda))$  is uniquely determined  by   $$s\alpha_3'\otimes s\alpha_2'\otimes s\alpha_1'\mapsto -s\beta_1'\otimes \beta_0';$$

    \item[$\bullet$] $g_2:=\kappa(s^{-1}y_2)\in \overline{C}^{0}(\Lambda, \Omega_{\nc, R}^3(\Lambda))$ is uniquely determined  by  $$ s\alpha_3''\otimes s\alpha_2''\otimes s\alpha_1''\mapsto s\beta_3''\otimes s\beta_2''\otimes s\beta_1''\otimes \beta_0'';$$

        \item[$\bullet$] $g_3:=\kappa(s^{-1}y_3)\in \overline{C}^{-1}(\Lambda, \Omega_{\nc, R}^3(\Lambda))$  is uniquely determined  by  $$ s\alpha_2'''\otimes s\alpha_1'''\mapsto -s\beta_3'''\otimes s\beta_2'''\otimes s\beta_1'''\otimes \beta_0'''.$$
\end{itemize}
By Figure~\ref{Brace-action-Sep},  we have that  the  element
$$B_{(2)}^{(2, 4)}(\kappa(s^{-1}x); \kappa(s^{-1}y_1), \kappa(s^{-1}y_2), \kappa(s^{-1}y_3))= B_{(2)}^{(2, 4)}(f; g_1, g_2, g_3)$$
 is depicted by the graph on the right of Figure \ref{Brace-action-Sep-exm}, which is uniquely determined by $$s\alpha_3''\otimes s\alpha_2'''\otimes s\alpha_1'''\otimes s\alpha_1\mapsto \lambda (s\beta_3''\otimes s\beta_2'\otimes s\beta_1'\otimes s\beta_0'\otimes s\beta_1\otimes \beta_0).$$ Here $\lambda$ is the same as the one  in (\ref{lambda}).

By \eqref{b242} we have that $\kappa(\mathfrak b_{(2)}^{(2, 4)}(s^{-1}x; s^{-1}y_1, s^{-1}y_2, s^{-1}y_3))$ is uniquely determined by
$$s\alpha_3''\otimes s\alpha_2'''\otimes s\alpha_1'''\otimes s\alpha_1\mapsto -\lambda (s\beta_3''\otimes s\beta_2'\otimes s\beta_1'\otimes s\beta_0'\otimes s\beta_1\otimes \beta_0).$$
Therefore,  we have  $$\kappa(\mathfrak b_{(2)}^{(2, 4)}(s^{-1}x; s^{-1}y_1, s^{-1}y_2, s^{-1}y_3))= -B_{(2)}^{(2, 4)}(\kappa(s^{-1}x); \kappa(s^{-1}y_1), \kappa(s^{-1}y_2), \kappa(s^{-1}y_3)).$$
This verifies that $\kappa$ preserves the brace operations.
\end{exm}
\begin{figure}[h]
\centering
\includegraphics[height=50mm]{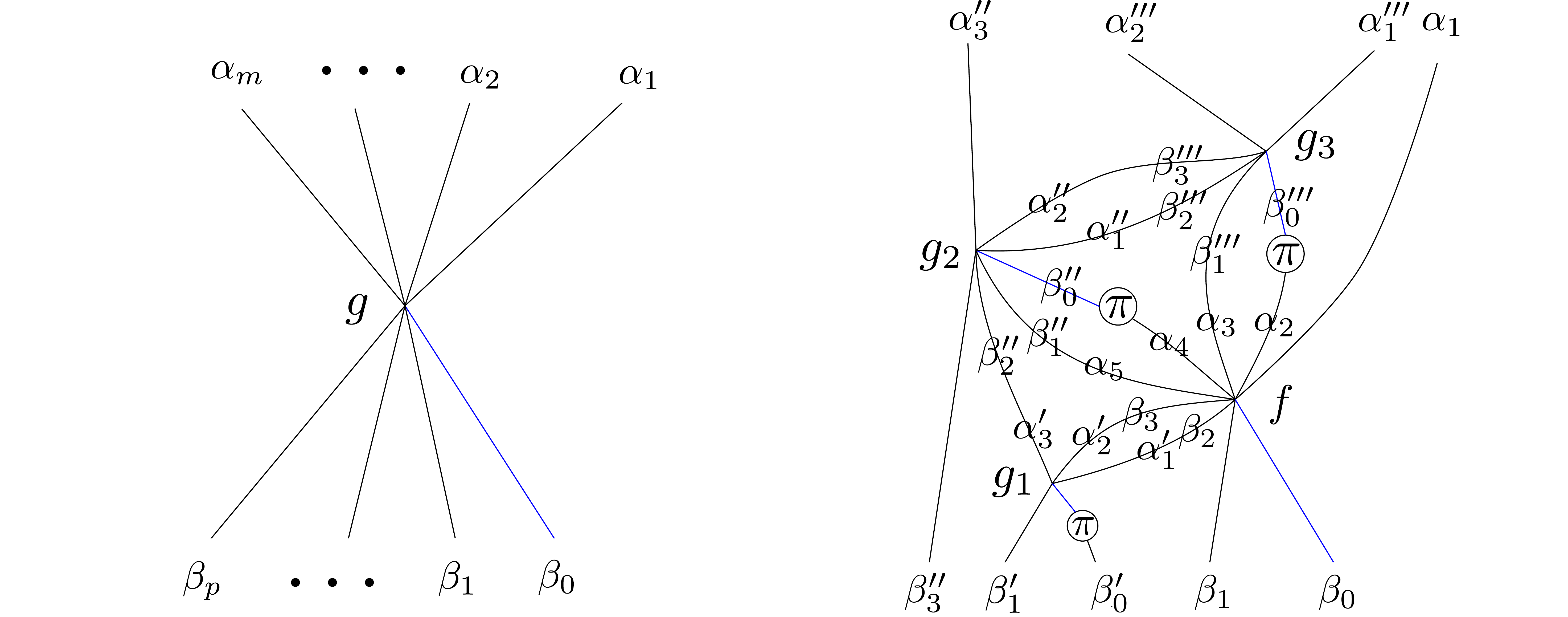}
  \caption{If $\beta_0\in Q_1$, then  the left  graph represents the element  $g=s^{-1}(\alpha_m\cdots \alpha_1, \beta_p\cdots\beta_1\beta_0)\in \overline{C}^*_{\sg, R}(Q, Q)$. If $\beta_0=s(\beta_1)\in Q_0$, then it represents $g=(\alpha_m\cdots \alpha_1, \beta_p\cdots\beta_1)\in \overline{C}^*_{\sg, R}(Q, Q)$. The map represented by the  right graph is nonzero only if the elements in each internal edge coincide (i.e. $\alpha_2''=\beta_3''', \alpha_1''=\beta_2''', \alpha_2= \beta_0'''$ and so on).   }
  \label{Brace-action-Sep-exm}
\end{figure}

\section{The Leavitt $B_{\infty}$-algebra as an intermediate object}\label{Section:9}

Let $Q$ be a finite quiver without sinks. Let $L=L(Q)$ is the Leavitt path algebra of $Q$.  In this section, we introduce the  Leavitt $B_\infty$-algebra  $(\widehat{C}^*(L, L), \delta', -\cup'-; -\{-, \dotsb, -\}')$, which is an intermediate object   connecting  the singular Hochschild cochain complex of $\mathbb k  Q/ J^2$ to the Hochschild cochain complex  of $L$. More precisely,  we will show that  the Leavitt  $B_\infty$-algebra $\widehat{C}^*(L, L)$ is strictly $B_\infty$-isomorphic to $\overline{C}_{\sg ,R}^*(Q, Q)$; see Proposition~\ref{prop:interme} below.

 In Sections~\ref{Section:10} and \ref{section:11}, we will show that there is an explicit non-strict $B_\infty$-quasi-isomorphism between the two $B_\infty$-algebras $\widehat{C}^*(L, L)$ and $\overline{C}_E^*(L, L)$. Namely, we have
$$
\overline{C}^*_{\sg, R, E}(\Lambda, \Lambda) \stackrel{\kappa}{\longleftarrow} \overline{C}_{\sg ,R}^*(Q, Q) \stackrel{\rho}{\longrightarrow} \widehat{C}^*(L, L) \xrightarrow{(\Phi_1, \Phi_2, \dotsb)}\overline{C}_E^*(L, L),
$$
where the left two maps are  strict $B_\infty$-isomorphisms and the rightmost one is a non-strict $B_{\infty}$-quasi-isomorphism. Recall that the leftmost map $\kappa$ is already given in Theorem~\ref{thm:radical}.

\subsection{An explicit complex} \label{subsection:Definition-complex}
We define the following graded vector space
$$\widehat{C}^*(L, L)=\bigoplus_{i\in Q_0} e_iLe_i\oplus \bigoplus_{i\in Q_0} s^{-1} e_iLe_i,$$
where we recall that the degree $|s^{-1}|=1$. The differential $\widehat{\delta}$ of $\widehat{C}^*(L, L)$ is given by  $\left(\begin{smallmatrix}0 &\delta' \\0 &0\end{smallmatrix}\right)$, where $$\delta'(x)=s^{-1}x-(-1)^{|x|}\sum_{\{\alpha\in Q_1\mid  t(\alpha)=i \}}s^{-1}\alpha^*x\alpha$$
 for any $x=e_ixe_i\in e_iLe_i$ and $i\in Q_0$. Note that we have $\widehat{\delta}(s^{-1}y)=0$ for $y\in \bigoplus_{i\in Q_0}e_iLe_i.$ This defines the complex $(\widehat{C}^*(L, L), \widehat{\delta})$.

Recall the complex $\overline{C}_{\sg, R}^*(Q, Q)$ from (\ref{equ:sg-R-Q}). We claim that there is a morphism of complexes
\begin{equation*}
\rho\colon   \overline{C}_{\sg, R}^*(Q, Q) \longrightarrow  \widehat{C}^*(L, L)	
\end{equation*}
given by
 \begin{flalign*}
 \rho((\gamma, \gamma')) & =\gamma'^*\gamma && \text{for $(\gamma, \gamma')\in Q_m// Q_p$}; \\
  \rho(s^{-1}(\gamma, \gamma')) &= s^{-1}\gamma'^*\gamma && \text{for $s^{-1}(\gamma, \gamma')\in s^{-1}\mathbb k(Q_m//Q_{p+1})$.}
 \end{flalign*}
 Indeed, we observe that for $(\gamma,
\gamma')\in Q_m// Q_p$
\begin{equation*}
\rho(\theta_{p,R}(\gamma, \gamma'))=\sum_{\alpha\in Q_1} (\alpha\gamma' )^* \alpha\gamma=\gamma'^*\gamma= \rho((\gamma, \gamma')),
\end{equation*}
where the second equality follows from  $\sum_{\{\alpha \in Q_1|s(\alpha)=i\}}\alpha^*\alpha=e_i$. Similarly, we have $$\rho(\theta_{p, R}(s^{-1}(\gamma,\gamma')))=\rho(s^{-1}(\gamma, \gamma')).$$
 This shows that $\rho$ is well defined. Comparing $D_{m,p}$ in \eqref{equation-defD} and $\delta'$, it is easy to check that  $\rho$ commutes with the differentials.
This proves the claim. Moreover, we have the following result.

\begin{lem}\label{lrho}
The above morphism $\rho$ is an isomorphism of complexes.
\end{lem}

\begin{proof}
This  follows immediately from the definition of $\overline{C}_{\sg, R, 0}^*(Q, Q)$ and Lemma~\ref{lemma:basis}.  \end{proof}

 \subsection{The  Leavitt $B_{\infty}$-algebra} \label{subsection-B-widehat}

We will define the cup product $-\cup'-$ and brace operation $-\{-, \dotsc, -\}'$ on $\widehat{C}^*(L, L)$.

Recall from (\ref{equ:CK2}) that each element in $e_iLe_i \subset \widehat{C}^*(L, L)$ can be written as a linear combination of the following monomials
\begin{equation}\label{mon}
\beta_1^*\beta_2^*\cdots \beta_p^*\alpha_m\alpha_{m-1} \cdots \alpha_1,\end{equation}
where $\beta_p\cdots\beta_2\beta_1$ and $\alpha_m\alpha_{m-1}\cdots \alpha_1$ are paths in $Q$ with lengths $p$ and $m$, respectively. In particular, all $\beta_j$ and $\alpha_k$ belong to $Q_1$. Moreover, we require that $p\geq 1$ and $m\geq 0$, and that $t(\alpha_m)=s(\beta_p^*)=t(\beta_p)$. In case where $m=0$, these $\alpha_i$'s do not appear. The monomial (\ref{mon}) has degree $m-p$.

 Similarly, we write any element in $s^{-1}e_iLe_i\subset \widehat{C}^*(L, L)$ as a linear combination of the following monomials
\begin{equation}\label{montwo}
s^{-1}\beta_0^* \beta_1^*\cdots \beta_p^*\alpha_m\alpha_{m-1}\cdots \alpha_1\end{equation}
where $\alpha_k, \beta_j\in Q_1$ for $1\leq k\leq m$ and $0\leq j\leq p$. The monomial (\ref{montwo}) also has degree $m-p$. The difference here is that we require $p\geq 0$ and $m\geq 0$, since the $\beta_j$'s are indexed from zero.

The cup product $-\cup'-$ on $\widehat{C}^*(L, L)$ is defined by the following (C1')-(C4').

\begin{enumerate}
\item[(C1')] For any $s^{-1}u \in s^{-1}e_iLe_i $ and $ s^{-1}v \in s^{-1}e_jLe_j$ with $i,j\in Q_0$, we have $$s^{-1}u\cup's^{-1}v=0;$$
\item[(C2')]  For any  $u\in e_iLe_i$ and $ v \in e_jLe_j$ with $i,j\in Q_0$, we have
$$u \cup' v =  uv;$$
\item[(C3')] For any $s^{-1}u \in s^{-1}e_iLe_i$ and $v \in e_jLe_j$ with $i,j\in Q_0$, we have  $$(s^{-1}u)\cup' v= s^{-1} uv;$$

\item[(C4')] For any $u \in e_iLe_i$ and $s^{-1}v=s^{-1}\beta_0^*\beta_1^*\cdots \beta_p^* \alpha_m \alpha_{m-1} \cdots \alpha_1\in s^{-1}e_jLe_j$ with $i,j\in Q_0$,  we have
    $$
    u \cup's^{-1} v= \sum_{\alpha\in Q_1} s^{-1}\alpha^* u \alpha v= s^{-1} \beta_0^*u \beta_1^*\beta_2^*\cdots \beta_p^* \alpha_m\alpha_{m-1}\cdots \alpha_1.$$
Here, we use the relations $\alpha\beta^*=\delta_{\alpha,\beta}e_{t(\alpha)}$. Note that there is no Koszul sign caused by swapping $s^{-1}\beta_0^*$ with $u$, as the degree of $s^{-1}\beta_0^*$ is zero.
\end{enumerate}
Then $\widehat{C}^*(L, L)$ becomes a dg algebra with this cup product.

\begin{rem}
\begin{itemize}
\item[(1)] It seems that we cannot extend the cup product {\it naturally} to $L\oplus s^{-1}L$. For instance, take $u\in e_iLe_j$ and $ v\in e_jLe_i$ with $i,j\in Q_0$, $i\neq j$. When we define $u\cup' v=uv$ and extend the differential $\delta'\colon L\xrightarrow{}s^{-1}L$ by $\delta'(u)=s^{-1}u$ and $\delta'(v)=s^{-1}v$, then we have $$\delta'(u\cup' v)=s^{-1}uv-(-1)^{|uv|}\sum_{\{\alpha\in Q_1, t(\alpha)=i\}}s^{-1}\alpha^*uv\alpha.$$
But on the other hand, we have
$$\delta'(u)\cup' v+(-1)^{|u|}u\cup' \delta'(v)\!=\!s^{-1}u\cup' v+(-1)^{|u|}u\cup' s^{-1}v\!=\!s^{-1}\!uv+(-1)^{|u|}\!\sum_{\alpha\in Q_1}\!s^{-1}\!\alpha^*u\alpha v=s^{-1}uv.$$
So we may have  that $\delta'(u\cup' v)\neq \delta'(u)\cup' v+(-1)^{|u|}u\cup' \delta'(v)$. In other words, we do not obtain a dg algebra with the cup product and the differential.
\item[(2)] By (C3') and (C4'), we may view $\bigoplus_{i\in Q_0}s^{-1}e_iLe_i$ as a bimodule over $\bigoplus_{i\in Q_0}e_iLe_i$. According to (C1'), $\widehat{C}^*(L,L)$ is a trivial extension algebra; see \cite[p.78]{ARS}.
\end{itemize}
\end{rem}

Let  $v, u_1, \dotsc, u_k $ be  monomials in $\widehat{C}^*(L, L)$.  Then the brace operation $v\{u_1, \dotsc, u_k\}'$ is defined  by  the following (B1')-(B3').

\begin{enumerate}
\item[(B1')] If $u_j\in \prod_{i\in Q_0} e_iL e_i\subset \widehat{C}^*(L, L) $ for some $1\leq j\leq k$, then  \begin{align}\label{B1'}
v\{u_1, \dotsc, u_k\}'=0.
\end{align}

\item[(B2')] If $s^{-1}u_j\in \prod_{i\in Q_0} s^{-1}e_iL  e_i\subset \widehat{C}^*(L , L ) $ for each $1\leq j\leq k$, and
$$s^{-1}v=s^{-1}\beta_0^*\beta_1^*\cdots \beta_p^* \alpha_m\alpha_{m-1} \cdots \alpha_1\in\prod_{i\in Q_0} s^{-1}e_iL  e_i \subset \widehat{C}^*(L , L )$$  then we define
\begin{equation}\label{equation-brace-formula1}
\begin{split}
 &s^{-1}v \{s^{-1}u_1, \dotsc, s^{-1}u_k\}'=\sum\limits_{\substack{a+b=k,\;  a,b\geq 0 \\ 1\leq i_1< i_2 <\cdots < i_a\leq m\\ 1\leq l_1\leq l_2\leq  \cdots \leq l_b\leq p }} (-1)^{a+\epsilon} \;  \mathbb b^{(i_1, \dotsc, i_a)}_{(l_1, \dotsc, l_b)}(s^{-1}v; s^{-1}u_1, \dotsc, s^{-1}u_k),
 \end{split}
 \end{equation}
where $\mathbb{b}^{(i_1, \dotsc, i_a)}_{(l_1, \dotsc, l_b)}(s^{-1}v; s^{-1}u_1, \dotsc, s^{-1}u_k)\in \prod_{i\in Q_0} s^{-1}e_iL  e_i$ is defined as
\begin{align*}
&s^{-1}\beta_0^*\beta_1^*\cdots\beta^*_{l_1-1} u_1\beta_{l_1}^*\cdots \beta^*_{l_2-1}u_{2} \beta_{l_2}^*\cdots \beta_{l_b-1}^*u_{b} \beta_{l_b}^* \cdots \beta_{p-1}^*\beta_p^*
  \alpha_m\alpha_{m-1}\cdots\alpha_{i_a} u_{b+1} \alpha_{i_a-1}\cdots \\
  &\alpha_{i_2} u_{k-1} \alpha_{i_2-1}\cdots \alpha_{i_1}u_k \alpha_{i_1-1} \cdots \alpha_2 \alpha_1, \end{align*}
 and the sign
$$ \epsilon= \sum_{r=1}^b (|s^{-1}u_r|-1)(m+p-l_{r}+1) +\sum_{r=1}^a (|s^{-1}u_{k-r+1}|-1)(i_{r}-1)$$
is obtained via the Koszul sign rule by reordering the elements ($\beta_i^*$ and $\alpha_i$ are of degree one)
$$
\beta_0^*, \beta_1^*, \dotsc, \beta_p^*; \alpha_m, \alpha_{m-1}, \dotsc, \alpha_1; u_1, \dotsc, u_k
$$

\item[(B3')]  If $s^{-1}u_j\in \prod_{i\in Q_0} s^{-1}e_iL  e_i\subset \widehat{C}^*(L , L ) $ for each $1\leq j\leq k$, and  $v=\beta_1^*\cdots \beta_p^* \alpha_m\cdots \alpha_1\in\prod_{i\in Q_0} e_iL  e_i \subset \widehat{C}^*(L , L )$,   then  \begin{equation}\label{equation-brace-formula2}
\begin{split}
 v\{s^{-1}u_1, \dotsc, s^{-1}u_k\}'=\sum\limits_{\substack{a+b=k, \; a,b \geq 0 \\1\leq i_1< i_2 <\cdots < i_a\leq m\\ 1\leq l_1\leq l_2\leq  \cdots \leq l_b\leq p }} (-1)^{a+\epsilon}\;  \mathbb b^{(i_1, \dotsc, i_a)}_{(l_1, \dotsc, l_b)}(v; s^{-1}u_1, \dotsc, s^{-1}u_k),
 \end{split}
 \end{equation}
 where $\mathbb{b}^{(i_1, \dotsc, i_a)}_{(l_1, \dotsc, l_b)}(v; s^{-1}u_1, \dotsc, s^{-1}u_k)\in \prod_{i\in Q_0} e_iL  e_i$ is defined as
\begin{align*}
&\beta_1^*\beta_2^*\cdots\beta^*_{l_1-1} u_1\beta_{l_1}^*\cdots \beta^*_{l_2-1}u_{2} \beta_{l_2}^*\cdots \beta_{l_{b}-1}^*u_{b} \beta_{l_{b}}^* \cdots \beta_{p-1}^*\beta_p^*
  \alpha_m\alpha_{m-1} \cdots\alpha_{i_a} u_{b+1} \alpha_{i_a-1}\cdots \\
  &\alpha_{i_2} u_{k-1} \alpha_{i_2-1}\cdots \alpha_{i_1}u_k \alpha_{i_1-1} \cdots \alpha_2\alpha_1 \end{align*}
 and $\epsilon$ is the same as in (B2').
 \end{enumerate}

 \vskip 5pt

 The following remarks also apply to (B3').

\begin{rem} \label{remark-admissible}
(1) Each  summand $\mathbb b_{(l_1, \dotsc, l_b)}^{(i_1, \dotsc, i_a)}(s^{-1}v; s^{-1}u_1, \dotsc, s^{-1}u_k)$   is an insertion  of   $u_1, \dotsc, u_k$ (from left to right) into $s^{-1}v=s^{-1}\beta_0^*\beta_1^*\cdots \beta_p^* \alpha_m\alpha_{m-1}\cdots \alpha_1$ as follows
    $$s^{-1}\beta_0^*\cdots \beta_{l_1-1}^* {\color{red}\underbrace{u_1}}\beta_{l_1}^*\cdots \beta_{l_2-1}^*{\color{red}\underbrace{u_{2}}} \beta_{l_2}^* \cdots \beta_{l_{b}-1}^* {\color{red}\underbrace{u_{b}}}\beta_{l_{b}}^*\cdots\beta_p^*
 \alpha_m\cdots\alpha_{i_a} {\color{red}\underbrace{u_{b+1}}} \cdots  \alpha_{i_1} {\color{red}\underbrace{u_{k}}} \alpha_{i_1-1} \cdots \alpha_1.$$
 We are not allowed to  insert  any $u_i$ between $\beta_p^*$ and $\alpha_m$; in case where $m=0$, the insertion on the right of $\beta_p^*$ is not allowed. If $a=0$, there is no insertions into $\alpha_k$'s. Similarly, if $b=0$, there is no insertions into $\beta_j^*$'s.

Since $1\leq l_1\leq l_2\leq  \cdots \leq l_b\leq p$, we are allowed to insert more than one $u_i$'s into $s^{-1}v$ at the same position between $\beta^*_{j-1}$ and $\beta_{j}^*$ for some $1\leq j\leq p$. For example, we might have the following insertion  with $l_2=l_3$
$$s^{-1}\beta_0^*\beta_1^*\cdots \beta_{l_1-1}^*{\color{red}\underbrace{u_1}}\beta_{l_1}^*\cdots \beta_{l_2-1}^* {\color{red}\underbrace{u_{2}u_{3}}} \beta_{l_3}^* \cdots \beta_{l_{b}-1}^*{\color{red}\underbrace{u_{b}}}\beta_{l_{b}}^*\cdots\beta_p^*
 \alpha_m\cdots\alpha_{i_a} {\color{red}\underbrace{u_{b+1}}} \cdots \alpha_{i_1} {\color{red}\underbrace{u_{k}}}\cdots \alpha_1.$$
As $1\leq i_1<i_2<\cdots<i_j\leq m$, we are not allowed to insert more than one $u_i$'s into $s^{-1}v$ at the same position between $\alpha_{j-1}$ and $\alpha_j$ for some $1\leq j\leq m$. For example,   the following insertion is {\it not} allowed
 $$s^{-1}\beta_0^*\beta_1^*\cdots {\color{red}\underbrace{u_1}}\beta_{l_1}^*\cdots {\color{red}\underbrace{u_{b}}}\beta_{l_{b}}^*\cdots  \beta_p^*
 \alpha_m\cdots\alpha_{i_a} {\color{red}\underbrace{u_{b+1}}} \cdots  \alpha_{i_s} {\color{red}\underbrace{u_{k-s+1}u_{k-s+2}}}\cdots \alpha_1.$$

 (2) The brace operation is well defined, that is, it is compatible with the second Cuntz-Krieger relations or (\ref{equ:CK2}). For the proof, one might use the following relation to swap the insertion of $u_b$ into $s^{-1}v$
 $$\sum_{\{\alpha\in Q_1\; |\; s(\alpha)=i\}} \alpha^*\alpha u_b=\sum_{\{\alpha\in Q_1\; |\; s(\alpha)=i\}} u_b\alpha^*\alpha,$$
 where both sides are equal to $\delta_{i,j}u_b$ for $u_b\in e_jLe_j$. Proposition \ref{prop:interme} will provide an alternative proof for the well-definedness.

 (3) We observe that $v\{s^{-1}u_1, \dotsc, s^{-1}u_k\}$  in (\ref{equation-brace-formula2}) is also defined for any    $v \in L $,  not necessarily  $v\in \bigoplus_{i\in Q_0} e_iLe_i$. However, due to (2), it seems to be essential to require that all the $u_j$'s belong to $\bigoplus_{i\in Q_0} e_iLe_i$.
\end{rem}

It seems to be very nontrivial to verify directly that the above data define a $B_\infty$-structure on $\widehat{C}^*(L , L)$. Instead, we use the isomorphism $\rho$ in Lemma~\ref{lrho} to show that the above data are transferred from those in $\overline{C}_{\sg, R}^*(Q, Q)$.

\begin{prop}\label{prop:interme}
The isomorphism $\rho\colon   \overline{C}_{\sg, R}^*(Q, Q) \longrightarrow  \widehat{C}^*(L, L)$ preserves the cup products and the brace operations. In particular, the complex $\widehat{C}^*(L , L )$, equipped with the cup product $-\cup'-$ and the brace operation $-\{-, \dotsc, -\}'$ defined as above, is a $B_{\infty}$-algebra.
\end{prop}

The obtained $B_\infty$-algebra  $\widehat{C}^*(L, L)$ is called the \emph{Leavitt $B_\infty$-algebra}, due to its closed relation to the Leavitt path algebra.   Combining this result with Theorem~\ref{thm:radical}, we infer that $\widehat{C}^*(L , L )$ and $\overline{C}^*_{\sg, R, E}(\Lambda, \Lambda)$ are strictly $B_\infty$-isomorphic.

\begin{proof}
By a routine computation, we verify that  $\rho$ sends  the formulae (C1)-(C4)  to (C1')-(C4'), respectively. The key point in the verfification is the fact that  replacing $\xleftarrow{\alpha} \xrightarrow{\beta}$ by $\delta_{\alpha, \beta}$ in (C2)-(C4) corresponds to the first Cuntz-Krieger relations $\alpha\beta^*=\delta_{\alpha, \beta}e_{t(\alpha)}$, which are implicitly used in the multiplication of $L$ in (C2')-(C4').

It remains to check that $\rho$ is compatible with the brace operations. That is, $\rho$ sends the formulae (B1)-(B3) to (B1')-(B3'), respectively.

Let $x, y_1, \dotsc, y_k$  be  parallel paths either in $\overline{C}_{\sg, R, 0}^*(Q, Q)$ or in $s^{-1}\overline{C}_{\sg, R,0}^*(Q,Q)$. If there exists some  $y_j$ belonging to $\overline{C}_{\sg, R, 0}^*(Q, Q)$, then
$x\{y_1, \dotsc, y_k\}_R=0$. Thus, we have
$$\rho(x\{y_1, \dotsc, y_k\}_R)=0=\rho(x)\{\rho(y_1), \dotsc, \rho(y_k)\}'.$$
This shows that $\rho$ sends  the formula (B1) to  the formula (B1').

Let $x=s^{-1}(\alpha_{m, 1}, \beta_{p, 0})\in s^{-1}\overline{C}_{\sg, R, 0}^*(Q, Q)$ and $y_1, \dotsc, y_k\in s^{-1}\overline{C}_{\sg, R,  0}^*(Q, Q).$ Using the first Cuntz-Krieger relations $\alpha\beta^*=\delta_{\alpha, \beta}e_{t(\alpha)}$, we infer that  $\rho$ sends  the summand $\mathfrak b_{(l_1, \dotsc, l_{k-j})}^{(i_1, \dotsc, i_j)}(x; y_1, \dotsc, y_k)$ of $x\{y_1, \dotsc, y_k\}$ in  (\ref{equation-radical-B})  to the one $\mathbb b_{(l_1, \dotsc, l_{k-j})}^{(i_1, \dotsc, i_j)}(\rho(x); \rho(y_1), \cdots,\rho( y_k))$  of $\rho(x)\{\rho(y_1), \dotsc, \rho(y_k)\}'$ in (\ref{equation-brace-formula1}). See Example \ref{exam:paralle2} below for a detailed illustration.  Thus we have  $$\rho(x\{y_1, \dotsc, y_k\}_R)=\rho(x)\{\rho(y_1), \dotsc, \rho(y_k)\}'.$$
This shows that the formula (B2) corresponds to (B2') under $\rho$.

Similarly, if $x=(\alpha_{m, 1}, \beta_{p, 1})\in \overline{C}_{\sg, R, 0}^*(Q, Q)$ and $y_1, \dotsc, y_k\in s^{-1}\overline{C}_{\sg, R,0}^*(Q, Q),$ we have $$\rho\left(\mathfrak b_{(l_1, \dotsc, l_{k-j})}^{(i_1, \dotsc, i_j)}(x; y_1, \dotsc, y_k)\right)=\mathbb b_{(l_1, \dotsc, l_{k-j})}^{(i_1, \dotsc, i_j)}(\rho(x); \rho(y_1), \cdots,\rho( y_k))$$ and thus $\rho(x\{y_1, \dotsc, y_k\}_R)=\rho(x)\{\rho(y_1), \dotsc, \rho(y_k)\}'.$ This shows that $\rho$ sends  (B3)  to (B3').
\end{proof}

\begin{exm}\label{exam:paralle2}
Consider the following monomial elements in $\overline{C}^*_{\sg, R}(Q, Q)$ as in Example~\ref{exm-paralle}
\begin{align*}
s^{-1}x &= s^{-1} (\alpha_5\alpha_4\alpha_3\alpha_2 \alpha_1, \beta_3\beta_2 \beta_1\beta_0)\\
s^{-1}y_1 &= s^{-1}(\alpha_3' \alpha_2'\alpha_1', \beta_1' \beta_0')\\
s^{-1}y_2 &= s^{-1} (\alpha_3''\alpha_2''\alpha_1'', \beta_3''\beta_2''\beta_1''\beta_0'')\\
s^{-1}y_3 &= s^{-1} (\alpha_2'''\alpha_1''', \beta_3'''\beta_2'''\beta_1'''\beta_0''').
\end{align*}

Let us check that $\rho$ preserves the brace operations. Note that
\begin{align*}
\rho(s^{-1}x)&=s^{-1} \beta_0^*\beta_1^*\beta_2^*\beta_3^*\alpha_5\alpha_4\alpha_3\alpha_2 \alpha_1\\
\rho(s^{-1}y_1)&=s^{-1}\beta_0'^*\beta_1'^*\alpha_3'\alpha_2'\alpha_1'\\
\rho(s^{-1}y_2)&=s^{-1}\beta_0''^*\beta_1''^*\beta_2''^*\beta_3''^*\alpha_3''\alpha_2''\alpha_1''\\
\rho(s^{-1}y_3)&=s^{-1}\beta_0'''^*\beta_1'''^*\beta_2''''^*\beta_3'''^*\alpha_2'''\alpha_1'''.
\end{align*}
Then by Formula (B2') we have that
\begin{align*}
&\mathbb b_{(2)}^{(2, 4)}(\rho(s^{-1}x); \rho(s^{-1}y_1), \rho(s^{-1}y_2), \rho(s^{-1}y_3))\\
 &\!\!\!\!\!\!\!\!\!\!\! = s^{-1} \beta_0^*\beta_1^*\underbrace{\beta_0'^*\beta_1'^*\alpha_3'\alpha_2'\alpha_1'}\beta_2^*\beta_3^*\alpha_5\alpha_4\underbrace{\beta_0''^*\beta_1''^*\beta_2''^*\beta_3''^*\alpha_3''\alpha_2''\alpha_1''}\alpha_3\alpha_2\underbrace{\beta_0'''^*\beta_1'''^*\beta_2''''^*\beta_3'''^*\alpha_2'''\alpha_1''' }\alpha_1\\
 &\!\!\!\!\!\!\!\!\!\!\!= \lambda s^{-1}\beta_0^*\beta_1^*\beta_0'^*\beta_1'^*\beta_3''^*\alpha_3''\alpha_2'''\alpha_1'''\alpha_1\\
 &\!\!\!\!\!\!\!\!\!\!\!=\rho(\mathfrak b_{(2)}^{(2, 4)}(s^{-1}x; s^{-1}y_1, s^{-1}y_2, s^{-1}y_3)),
\end{align*}
where the second identity follows from the second Cuntz-Krieger relations, and the coefficient $\lambda \in \mathbb k$ is defined as above.
Therefore we have
$$\rho(\mathfrak b_{(2)}^{(2, 4)}(s^{-1}x; s^{-1}y_1, s^{-1}y_2, s^{-1}y_3))=\mathbb b_{(2)}^{(2, 4)}(\rho(s^{-1}x); \rho(s^{-1}y_1), \rho(s^{-1}y_2), \rho(s^{-1}y_3)).$$

\end{exm}

\subsection{A recursive formula for the brace operation}  We will give a recursive formula for the brace operation $-\{-, \dotsc, -\}'$ on $\widehat{C}^*(L, L)$, which will be used in the proof of Proposition~\ref{proposition-Phi}.

\begin{prop}\label{prop:recursive}
Let $v=\beta_1^*\cdots\beta_p^*\alpha_m\cdots \alpha_1\in L$ be a monomial with $\beta_i, \alpha_j\in Q_1$  for  $1\leq i\leq p$ and $1\leq j\leq m$, and  let $s^{-1}u_1, \dotsc, s^{-1}u_{k}\in \bigoplus_{i\in Q_0} s^{-1}e_iLe_i$ for $k\geq 1$. Suppose that  $s^{-1}u_k=s^{-1}\gamma_0^*\widetilde{u_k} $ with $\gamma_0\in Q_1$ and $\widetilde{u_k} \in e_{t(\gamma_0)}Le_{s(\gamma_0)}$. Then we have
\begin{align}\label{equation-recursive}
&v\{s^{-1}u_1, \dotsc, s^{-1}u_k\}'\\
={} &\sum_{j=0}^{p-1} (-1)^{(j+|v|+1)\epsilon_k+|u_k|} \Big((\beta_{1, j}^*\gamma_0^*)\{s^{-1}u_1, \dotsc, s^{-1}u_{k-1}\}'\Big)\cdot (\widetilde{u_k}  \beta_{j+1,p}^* \alpha_{m, 1})\nonumber\\
& - \sum_{j=0}^{m-1} (-1)^{(j+1) \epsilon_k+|u_k|} \; \delta_{\alpha_{j+1}, \gamma_0}\Big((\beta_{1, p}^*\alpha_{m, j+2}) \{s^{-1}u_1, \dotsc, s^{-1}u_{k-1}\}'
\Big) \cdot( \widetilde{u_k} \alpha_{j, 1})\nonumber,\end{align}
where $\epsilon_k=|u_1|+\cdots+|u_k|$,  and the dot $\cdot$ indicates the multiplication of $L$.
\end{prop}

For the brace operation $v\{s^{-1}u_1, \dotsc, s^{-1}u_k\}'$ with $v\in L$, we refer to  Remark~\ref{remark-admissible}(3). Here, we write  $\alpha_{j, i}=\alpha_{j}\alpha_{j-1}\cdots \alpha_i$, $\beta_{i, j}^*=\beta_i^*\beta_{i+1}^*\cdots\beta_j^*$ for any $i\leq j$. Moreover, $\beta_{1, 0}^*\gamma_0^*$,  $\widetilde{u_k} \alpha_{0, 1}$ and $\beta_{1, p}^*\alpha_{m, m+1}$ are understood as $\gamma_0^*$, $\widetilde{u_k} $ and $\beta_{1,p}^*$, respectively. In particular, the above proposition also works for $v = \beta_1^*\dotsc \beta_p^*$ and $v= \alpha_m \dotsc \alpha_1.$

\begin{proof}
 We only prove the identify for the cases $m, p>0$. The cases where $m=0$ or $p=0$ can be proved in a similar way.

 We will compare the summands on the right hand side of  (\ref{equation-recursive}) with the summands $\mathbb b^{(i_1, \dotsc, i_a)}_{(l_1, \dotsc, l_b)}(v; s^{-1}u_1, \dotsc, s^{-1}u_k)$ in (\ref{equation-brace-formula2}). We analyze the position in $v=\beta_1^*\beta_2^* \cdots\beta_p^*\alpha_m\alpha_{m-1} \cdots \alpha_1$ where $u_k$ is inserted according to Remark~\ref{remark-admissible}(1).

  For any fixed  $0\leq j\leq p-1$, the  first term on the right hand side of (\ref{equation-recursive})
$$(-1)^{(j+|v|+1)\epsilon_k+|u_k|} \Big((\beta_{1,j}^* \gamma_0^*)\{s^{-1}u_1, \dotsc, s^{-1}u_{k-1}\}'\Big)\cdot (\widetilde{u_k}  \beta^*_{j+1, p} \alpha_{m,1})$$
equals  the following summand
$$\sum_{1\leq l_1\leq l_2\leq \cdots\leq l_{k-1}\leq l_k=j+1}(-1)^{|v|\epsilon_k+\sum_{r=1}^{k-1} (l_r-1)|u_r|+j|u_k|}  \;  \mathbb b_{(l_1, l_2, \dotsc, l_{k-1}, j+1)}^{\emptyset}(v; s^{-1}u_1, \dotsc, s^{-1}u_k),$$
since both of them are the sums of all insertions such that  $u_k$  is inserted into $v$ at the position between $\beta_j^*$ and  $\beta_{j+1}^*$.

To complete the proof, we assume that the insertion of $u_k$ into $v$ is at the position between $\alpha_{j+1}$ and $\alpha_j$ for any fixed $0\leq j\leq m-1$. That is, we are concerned with the following summand
\begin{align}\label{equ:re-br1}
\sum_{\substack{a+b=k,\; a,b \geq 0\\ j+1=i_1<i_2<\cdots < i_a\leq m\\ 1\leq l_1\leq l_2\leq \cdots \leq l_b\leq p}} (-1)^{a+\epsilon} \; \mathbb b_{(l_1, \dotsc, l_b)}^{(j+1, i_2, \dotsc, i_a)}(v; s^{-1}u_1, \dotsc, s^{-1}u_k).
\end{align}
Here, $\epsilon$ is the same as in (\ref{equation-brace-formula2}). We observe that
\begin{align*}
&\mathbb b_{(l_1, \dotsc, l_b)}^{(j+1, i_2, \dotsc, i_a)}(\beta_{1, p}^*\alpha_{m, 1}; s^{-1}u_1, \dotsc, s^{-1}u_k)\\
={} & \delta_{\alpha_{j+1}, \gamma_0} \mathbb b_{(l_1, \dotsc, l_b)}^{(i_2, \dotsc, i_a)}(\beta_{1, p}^*\alpha_{m, j+2}; s^{-1}u_1, \dotsc, s^{-1}u_{k-1}) \cdot (\widetilde{u_k} \alpha_{j, 1}),  \end{align*}
where the insertion of $u_1, \dotsc, u_{k-1}$ into $\beta_1^*\cdots \beta_p^*\alpha_m \cdots \alpha_{j+2}$ is involved in the latter term. It follows that for each $0\leq j\leq m-1$,  (\ref{equ:re-br1}) equals
$$-(-1)^{(j+1) \epsilon_k+|u_k|} \; \delta_{\alpha_{j+1}, \gamma_0}\Big((\beta_{1, p}^*\alpha_{m, j+2}) \{s^{-1}u_1, \dotsc, s^{-1}u_{k-1}\}'
\Big) \cdot( \widetilde{u_k} \alpha_{j, 1}).$$
This is the second term on the right hand side of (\ref{equation-recursive}). Then the required identity follows immediately.
\end{proof}

\section{An $A_\infty$-quasi-isomorphism for the Leavitt path algebra}\label{Section:10}

In this section, we use the homotopy transfer theorem for dg algebras to obtain an explicit $A_\infty$-quasi-isomorphism between the two dg algebras $\widehat{C}^*(L, L)$ and $\overline{C}_E^*(L, L)$; see Propositions~\ref{proposition-A-infinity} and \ref{proposition-Phi}.

\subsection{An explicit $A_{\infty}$-quasi-isomorphism between dg algebras}

In what follows, we apply the functor  $\mathrm{Hom}_{\text{$L$-$L$}}(-, L)$ to the homotopy deformation retract (\ref{equation-hdr-1}). In particular, we have the dg-projective bimodule resolution $P$ of $L$.

Recall from Section~\ref{Section:9} the Leavitt $B_\infty$-algebra $\widehat{C}^*(L, L)$.  We will use the  identification
$$\Hom_{L\text{-}L}(P, L)=(\widehat{C}^*(L, L), \widehat{\delta})$$
by the  following natural isomorphisms
\begin{flalign}\label{2}
&& \Hom_{L\text{-}L}(Le_i\otimes e_iL, L)&\xrightarrow{\cong} e_iLe_i, & &\phi\longmapsto \phi(e_i\otimes e_i); &&\nonumber \\
&& \Hom_{L\text{-}L}(Le_i\otimes s\mathbb k\otimes e_iL, L)&\xrightarrow{\cong} s^{-1}e_iLe_i,& & \phi \longmapsto (-1)^{|\phi|}s^{-1}\phi(e_i\otimes s\otimes e_i).&&
\end{flalign}
It is straightforward to verify that the above isomorphisms are compatible with the differentials.

Recall that $E=\bigoplus_{i\in Q_0} \mathbb{k}e_i$  and that the $E$-relative Hochschild cochain complex  $\overline{C}_E^*(L, L)$ is naturally  identified with ${\rm Hom}_{L\mbox{-}L}(\overline{\rm Bar}_E(L), L)$; compare (\ref{identification-bimodule}). Under the above identifications, (\ref{equation-hdr-1}) yields the following  homotopy deformation retract
\begin{equation}\label{equation-hdr-0}\xymatrix@C=0.0000000000001pc{
(\widehat{C}^*(L, L), \widehat{\delta})   \ar@<0.5ex>[rrrrrrrrrr]^-{\Phi}&&&&&&&&&&(\overline{C}_E^*(L, L), \delta) \ar@<0.5ex>[llllllllll]^-{\Psi} & \ar@(dr, ur)_-{H}}
\end{equation}
with $\Phi={\rm Hom}_{L\mbox{-}L}(\pi, L), \Psi={\rm Hom}_{L\mbox{-}L}(\iota, L)$ and $H={\rm Hom}_{L\mbox{-}L}(h, L)$ satisfying
 \begin{center}
 $\Psi\circ \Phi={\mathbf 1}_{\widehat{C}^*(L,L)}$ and ${\mathbf 1}_{\overline{C}^*_E(L,L)}=\Phi\circ \Psi+\delta\circ H+H\circ \delta$.
 \end{center}

As in Subsection~\ref{subsection-dg-HH}, we  denote the following subspaces of $\overline{C}_E^*(L, L)$ for any  $k\geq 0$\begin{align*}
\overline{C}_E^{*, k}(L, L) &=  \Hom_{\text{$E$-$E$}}((s \overline L)^{\otimes_E k}, L)\\
\overline{C}_E^{*, \geq k}(L, L) &= \prod_{i\geq k} \Hom_{\text{$E$-$E$}}((s\overline L)^{\otimes_E i}, L)\\
\overline{C}_E^{*, \leq k}(L, L) &= \prod_{0\leq i\leq k} \Hom_{\text{$E$-$E$}}((s\overline L)^{\otimes_E i}, L).
\end{align*}
In particular, we have $\overline{C}_E^{*, 0}(L, L) = \Hom_{\text{$E$-$E$}}(E, L) = \bigoplus_{i\in Q_0} e_iLe_i$.

 Let us describe  the above homotopy deformation retract (\ref{equation-hdr-0}) in more detail.

\begin{enumerate}
\item The surjection $\Psi$ is given by
\begin{flalign}\label{equ:Psi}
&&\Psi(x)&=x  && \mbox{for $x\in \overline{C}_E^{*, 0}(L, L)=\bigoplus_{i\in Q_0} e_iLe_i$};\nonumber \\
&&\Psi(f)&=-\sum_{\alpha\in Q_1} s^{-1}\alpha^*f(s\overline\alpha) &&
\mbox{for $f\in \overline{C}_E^{*,1}(L, L)$;}&&\\ \nonumber
&&\Psi(g)&=0 &&\mbox{for $g\in \overline{C}_E^{*, \geq 2}(L, L)$}.&&
\end{flalign}
\item The injection $\Phi$ is given by

\begin{equation}\label{phiformulaimage}
\begin{split}
\Phi(u)&=u \quad\quad\quad\quad\quad\quad\quad \text{for $u \in \prod_{i\in Q_0} e_iLe_i\subset \widehat{C}^*(L, L);$}\\
\Phi(s^{-1}u)& \in \overline{C}_E^{*, 1}(L, L)  \quad\quad \ \quad \text{for $s^{-1}u\in \prod_{i\in Q_0} s^{-1}e_iLe_i\subset \widehat{C}^*(L, L)$}.
\end{split}
\end{equation}
Here, in the first identity  we use the identification $\overline{C}_E^{*, 0}(L, L) = \bigoplus_{i\in Q_0} e_iLe_i$. The explicit formula of $\Phi(s^{-1}u)$ will be given  in Lemma~\ref{lem-Phi} below.
\end{enumerate}

\begin{enumerate} \setcounter{enumi}{2}
\item The homotopy $H$ is given by
\begin{equation}
\label{H-vanish}
\begin{split}
H|_{\overline{C}_E^{*, \leq 1}(L, L)}& = 0\\
H(f)(s\overline{a}_{1, n})& =(-1)^{\epsilon} f(s\overline{a}_{1, n-1}\otimes_E \overline{\iota\pi}(1\otimes_Es\overline{a_n}\otimes_E1)]
\end{split}
\end{equation}
 for any $f\in \overline{C}_E^{*, n+1}(L, L)$ with $n\geq 1$, where $\epsilon=1+|f|+\sum_{i=1}^{n-1} (|a_i|-1)$. Here, for convenience we use the following notation
\begin{flalign*}
&& f(s\overline{a}_{1, n+1}\otimes_Ex]:=f(s\overline{a}_{1, n+1})x, &&
\text{for  $  s\overline{a}_{1, n+1}\in (s\overline L)^{\otimes_E n+1}$ and $x\in L$},
\end{flalign*}
and we simply write $s\overline{a}_{1, n+1}:= s\overline{a_1}\otimes_E s\overline{a_2}\otimes_E \dotsb \otimes_E s\overline{a_{n+1}}.$
\end{enumerate}

The following lemma provides the formula of $\Phi(s^{-1}u)$ in (\ref{phiformulaimage}).

\begin{lem}\label{lem-Phi}
For  any $s^{-1}u \in \bigoplus_{i\in Q_0}s^{-1}e_iLe_i\subset \widehat{C}^*(L, L)$, we have
$$\Phi(s^{-1}u)(s\overline v)= (-1)^{(|v|-1)|u|}v\{s^{-1}u\}', $$
where  $v\in L$ and $v\{s^{-1}u\}'$ is given by (\ref{equation-brace-formula2}).
\end{lem}

\begin{proof}
 Let $v=\beta_1^{*}\cdots \beta_p^{*}\alpha_m\cdots \alpha_1\in e_iLe_j$ be a monomial, where $i, j\in Q_0$.  We first assume that $m, p>0$. Under the identification (\ref{2}), the element $s^{-1}u$ corresponds to a morphism of $L$-$L$-bimodules of degree $|u|-1$
\begin{flalign*}
&&\phi_{s^{-1}u}\colon Le_i\otimes s\mathbb k\otimes e_iL\longrightarrow{} L, && a\otimes s\otimes b \longmapsto (-1)^{(|a|+1)(|u|-1)}aub.&&
\end{flalign*}
 Then we have  $\Phi(s^{-1}u)(s\overline v)=(\phi_{s^{-1}u}\circ \pi)(1\otimes s\overline{v}\otimes 1).$
  By Remark~\ref{rem:derivation}, we have
\begin{equation*}
\begin{split}
\Phi(s^{-1}u)(s\overline v)
={} &(\phi_{s^{-1}u}\circ \pi)(1\otimes s\overline{v}\otimes 1)
=\phi_{s^{-1}u}(D(v))\\
={} & (-1)^{|u|}uv+
\sum_{l=1}^{p-1}(-1)^{|u|(l+1)}\beta_1^*\cdots \beta_l^*u\beta_{l+1}^*\cdots \beta_p^*\alpha_m\cdots \alpha_1\\
&\!\!\!\! + \sum_{l=1}^{m-1} (-1)^{|u| (m+p-l-1)+1}\beta_1^*\cdots \beta_p^*\alpha_m\cdots \alpha_{l+1} u\alpha_l\cdots \alpha_1 +(-1)^{(|v|+1)|u|+1}vu.
\end{split}
\end{equation*}
It follows from the definition of the brace operation in \eqref{equation-brace-formula2} that \begin{align*}
v\{s^{-1}u\}'
 ={} &(-1)^{|v||u|}uv + \sum_{l=1}^{p-1}(-1)^{|v||u|+|u|l}\beta_1^*\cdots \beta_{l}^*u\beta_{l+1}^*\cdots \beta_p^*\alpha_m\cdots\alpha_1\\
 &+\sum_{l=1}^{m-1}(-1)^{|v||u|+1+|u|(|v|-l)}\beta_1^*\cdots \beta_p^*\alpha_m\cdots \alpha_{l+1}u\alpha_{l}\cdots\alpha_1-vu.
 \end{align*}
 By comparing the signs of the above two formulae, we infer
 $$
 \Phi(s^{-1}u)(s\overline v) = (-1)^{(|v|-1)|u|}v\{s^{-1}u\}'.
$$
Similarly, one can prove the statement for either $p=0$ or $m=0$.
\end{proof}

\begin{rem}\label{remphi1alpha}
Note that for $\alpha \in Q_1$ we have  $$\Phi(s^{-1}u) (s \overline \alpha) =\alpha \{s^{-1} u\}' = -\alpha u,$$
where the second identity is due to Remark \ref{remark-admissible}(3).
The  formula of $\Phi=\Phi_1$ will be generalized to $\Phi_k$ for $k >1$ by using $-\{\underbrace{-, \dotsc, -}_k\}'$; see Proposition \ref{proposition-Phi} below.
\end{rem}

The following simple lemma on the homotopy $H$  will  be used in Lemma \ref{lemma-homotopy-zero} below.

\begin{lem}\label{facts}
Let $\alpha\in Q_1$ and $f\in \overline{C}_E^{*, n+1}(L , L)$ with $n\geq 1$. Then we have
$$H(f)(s\overline{a_1}\otimes_E\cdots \otimes_E s\overline{a_{n-1}}\otimes_E s\overline{\alpha})=0$$
for any $a_1, \dotsc, a_{n-1}\in L$.
\end{lem}

\begin{proof} By  (\ref{H-vanish}) we have
\begin{align*}
H(f)(s\overline{a}_{1, n-1}\otimes_E s\overline{\alpha}) & =(-1)^{\epsilon} f(s\overline{a}_{1, n-1}\otimes_E \overline{\iota\pi}(1\otimes_E s\overline{\alpha}\otimes_E 1)]\\
& =(-1)^{\epsilon} f(s\overline{a}_{1, n-1}\otimes_E s\overline{e_{t(\alpha)}} \otimes_E s\overline{\alpha})\\
&=0,
\end{align*}
where the last identity comes from the fact that     $\overline{e_{t(\alpha)}} = 0 $ in $\overline L=L/(E\cdot 1)$.
\end{proof}

The following lemma shows that the homotopy deformation retract (\ref{equation-hdr-0}) satisfies the assumption \eqref{assumptionhomtopytransfer} of Corollary \ref{corollary-hdr}.
 \begin{lem}\label{lemma-homotopy-zero}
 For any $g_1, g_2\in  \overline{C}_E^*(L, L)$, we have $$H(g_1\cup H(g_2))=0=\Psi(g_1\cup H(g_2)).$$\end{lem}
\begin{proof}
Throughout the proof, we assume without loss of generality that $$g_1\in \overline{C}_E^{*, m}(L, L) \quad \text{and} \quad g_2\in \overline{C}_E^{*, n}(L, L)\quad \text{for some $m, n \geq 0$}.$$
Note that if $n \leq 1$ then  $H(g_2) =0$ by \eqref{H-vanish} and the desired identities hold.  So in the following we may further assume that  $n\geq 2$.

Let us first verify $\Psi(g_1\cup H(g_2))=0$. Since $\Psi(g)=0$ for any $g\in \overline{C}_E^{*, \geq 2}(L, L)$, we only need to verify $\Psi(g_1\cup H(g_2))=0$ when $m=0$ and $n=2$. In this case,  $g_1\in \overline{C}_E^{*, 0}(L, L)$ is viewed as  an element in  $\bigoplus_{i\in Q_0} e_iLe_i$. Then we have
 \begin{equation*}
 \begin{split}
 \Psi(g_1\cup H(g_2))=-\sum_{\alpha\in Q_1} s^{-1} (\alpha^*g_1)\cdot \Big( H(g_2)(s\overline{\alpha})\Big)=0, \end{split}
 \end{equation*}
 where the second equality follows from Lemma~\ref{facts} since $\alpha \in Q_1$. In order to avoid confusion, we sometimes use the dot $\cdot$ to  emphasize the multiplication of $L$.

 It remains to  verify $H(g_1\cup H(g_2))=0$. For this, we have
 \begin{equation*}
 \begin{split}
 &H(g_1\cup H(g_2))(s\overline{a}_{1, m+n-2})\\
 ={}& (-1)^{\epsilon} (g_1\cup H(g_2))(s\overline{a}_{m+1, m+n-3} \otimes_E \overline{\iota\pi}(1\otimes_E s\overline a_{m+n-2}\otimes_E 1)]\\
 ={}& (-1)^{\epsilon+\epsilon'+1} \sum_i  \sum_{\substack{\alpha \in Q_1}} g_1(s\overline a_{1, m})\cdot  H(g_2)( s\overline a_{m+1, m+n-3} \otimes_E s\overline{x_i \alpha^*}\otimes_E s\overline \alpha) \cdot y_i \\
 ={}&0,\end{split}
 \end{equation*}
 where the last equality follows from~Lemma \ref{facts} as $\alpha \in Q_1$, and we simply write $
\pi(1\otimes_E s\overline a_{m+n-2}\otimes_E 1)=\sum_i x_i\otimes_E s\otimes_E y_i $; compare \eqref{pi}. Here, the signs are given by
 $$
 \epsilon=|g_1|+ |g_2|+\sum_{i=1}^{m+n-3}(|a_i|-1)\quad \text{and} \quad \epsilon'=(|g_2|-1) \left(\sum_{i=1}^m (|a_i|-1)\right).
 $$

\end{proof}

Thanks to Lemma~\ref{lemma-homotopy-zero}, we can apply Corollary~\ref{corollary-hdr} to the homotopy deformation retract (\ref{equation-hdr-0}). We obtain an $A_{\infty}$-algebra structure $(m_1=\widehat\delta, m_2, \cdots)$ on $\widehat{C}^*(L, L)$ and an $A_{\infty}$-quasi-isomorphism $(\Phi_1=\Phi, \Phi_2, \cdots)$ from $(\widehat{C}^*(L, L), m_1, m_2, \cdots) $ to $(\overline{C}_E^{*}(L, L), \delta, -\cup-)$.  More precisely, we have the following recursive formulae for $k \geq 2$; see Remark \ref{ind}
\begin{align}
\label{recursivephimk}
\Phi_k(a_1\otimes \cdots \otimes a_k)  & =(-1)^{k-1} \; H(\Phi_{k-1}(a_1\otimes \cdots\otimes a_{k -1}) \cup \Phi(a_k));\\
\label{secondrecursive}
m_k(a_1\otimes \cdots\otimes a_k)  & = (-1)^{\frac{(k-1)(k-2)}{2}} \; \Psi( \Phi_{k-1}(a_1\otimes \cdots\otimes a_{k-1})\cup \Phi(a_k)).
\end{align}

The following lemma provides some basic properties of $\Phi_k$.
\begin{lem} \label{lem:phicondition}
\begin{enumerate}
\item  For $k \geq 1$,  we have
\begin{equation}
\label{equation:phi1}
\Phi_k(s^{-1}u_1\otimes \cdots\otimes s^{-1} u_k) \in \overline{C}_E^{*, 1}(L, L)
\end{equation}
if $s^{-1}u_j \in \bigoplus_{i\in Q_0}s^{-1}e_iLe_i\subset \widehat{C}^*(L, L)$ for all $1\leq j \leq k$; 

\item For $k \geq 2$, we have
\begin{equation}
\label{equation:phi}
\Phi_k(a_1\otimes \cdots\otimes a_k)=0
\end{equation}
if there exists some $1\leq j\leq k$ such that $a_j\in \bigoplus_{i\in Q_0} e_iLe_i\subset \widehat{C}^*(L, L)$.
\end{enumerate}
\end{lem}
\begin{proof}
Let us prove the first assertion by induction on $k$.
For $k =1$ it follows from \eqref{phiformulaimage}. For $k>1$,  by   \eqref{recursivephimk} we have the following recursive formula
$$\Phi_{k}(s^{-1}u_1\otimes  \cdots\otimes  s^{-1}u_k) = (-1)^{k-1} H(\Phi_{k-1}(s^{-1}u_1\otimes  \cdots\otimes  s^{-1}u_{k-1})\cup \Phi(s^{-1}u_k)). $$
By the induction hypothesis, we have   $\Phi(s^{-1}u_k), \Phi_{k-1}(s^{-1}u_1\otimes  \cdots\otimes  s^{-1}u_{k-1}) \in \overline{C}_E^{*, 1}(L, L). $  Then we obtain $\Phi_{k-1}(s^{-1}u_1\otimes  \cdots\otimes  s^{-1}u_{k-1})\cup \Phi(s^{-1}u_k) \in \overline{C}_E^{*,2}(L,L)$. It follows from  \eqref{H-vanish} that $\Phi_{k}(s^{-1}u_1\otimes  \cdots\otimes  s^{-1}u_k) \in \overline{C}_E^{*, 1}(L, L)$.

Similarly, we may prove the second assertion by induction on $k$. For $k=2$ we have
$$\Phi_2(a_1\otimes a_2) = H(\Phi(a_1)\cup \Phi(a_2)) .$$
By  \eqref{H-vanish} we have $H|_{\overline{C}_E^{*, \leq 1}(L, L)}=0$. It follows from \eqref{phiformulaimage}  that
$\Phi_2(a_1 \otimes  a_2) = 0$ if $a_1$ or $a_2$ lies  in $\bigoplus_{i\in Q_0} e_iLe_i\subset \widehat{C}^*(L, L)$.

Now we consider the case for $k>2$. By the induction hypothesis, we have  $\Phi_{k-1}(a_1\otimes  \cdots\otimes  a_{k-1})=0 $  if there exists $1 \leq j \leq k-1$ such that $a_j$ lies in $\bigoplus_{i\in Q_0} e_iLe_i$. Then by \eqref{recursivephimk} we have   $\Phi_{k}(a_1\otimes  \cdots \otimes  a_{k})= 0$. Otherwise, by assumption $a_k$ must be in $\bigoplus_{i\in Q_0} e_iLe_i$ and thus $\Phi(a_k) \in \overline{C}_E^{*, 0}(L, L)$. Since the elements $a_1, \dotsc, a_{k-1}$ are in $\bigoplus_{i\in Q_0} s^{-1}e_iLe_i$, by the first assertion we obtain $\Phi_{k-1}(a_1\otimes  \cdots\otimes  a_{k-1})  \in \overline{C}_E^{*, 1}(L, L)$.
 By  \eqref{recursivephimk} again, we infer  $\Phi_{k}(a_1\otimes  \cdots\otimes  a_{k})=0$. \end{proof}

A prior, the higher $A_\infty$-products $m_k$ for $k \geq 3$ might be nonzero; see \eqref{secondrecursive}. From Lemma~\ref{lem:phicondition} we have seen that the maps $\Phi_k$ satisfy a nice  degree condition, i.e. for each $k \geq 2$,  the image of $\Phi_k$ only lies in $\overline{C}_E^{*, 1}(L, L)$. This actually will lead to the fact that  $m_k =0$ for $k \geq 3$.
Moreover, we will show that $m_2 =-\cup'-$. Recall from Subsection~\ref{subsection-B-widehat} the cup product $-\cup'-$ on $\widehat{C}^*(L, L)$.

\begin{prop}\label{proposition-A-infinity}
 The product $m_2$ on $\widehat{C}^*(L, L)$ coincides with the cup product $-\cup'-$, and the higher products $m_k$ vanish for all $k>2$.

Consequently, the collection of maps  $(\Phi_1=\Phi, \Phi_2, \dotsb)$ is an $A_\infty$-quasi-isomorphism   from the dg algebra $(\widehat{C}^*(L, L), \delta', -\cup'-)$ to the dg algebra $(\overline{C}_E^{*}(L, L), \delta, -\cup-)$.
 \end{prop}

\begin{proof}
Let us first prove that $m_2$ coincides with $-\cup'-$. Let $u,  v \in \prod_{i\in Q_0}e_iLe_i$. Then we view $s^{-1}u,  s^{-1}v$ as elements in $  \prod_{i\in Q_0}s^{-1}e_iLe_i$. We need to consider the following four cases corresponding to (C1')-(C4'); see Subsection \ref{subsection-B-widehat}.
\begin{enumerate}
\item For (C1'), since $\Phi(s^{-1}u),  \Phi(s^{-1}v)\in \overline{C}_E^{*, 1}(L, L)$ and $\Psi|_{\overline{C}_E^{*, 2}(L, L)}=0$, we have
$$m_2(s^{-1}u \otimes s^{-1}v)=\Psi(\Phi(s^{-1}u)\cup \Phi(s^{-1}v)) = 0 =s^{-1} u \cup' s^{-1}v.$$

\item For (C2'), since $\Psi(u)=u$ and $ \Psi(v)=v$, we have
$$m_2(u \otimes v)=\Psi(\Phi(u)\cup \Phi(v))= \Psi(uv)= uv=u \cup' v.$$

\item For (C3'), we have
\begin{alignat*}{2}
m_2(s^{-1} v\otimes u)
&= \Psi(\Phi(s^{-1}v)\cup \Phi(u))
 &{}={}& -\sum_{\alpha\in Q_1}s^{-1}\alpha^*\Phi(s^{-1}v)(s\overline \alpha)\cdot u \\
 &=\sum_{\alpha\in Q_1}s^{-1}\alpha^*\alpha v u
 \ &={}& \
s^{-1}v \cup' u,
\end{alignat*}
where the third equality follows from Remark~\ref{remphi1alpha}, and the last one is due to the second Cuntz-Krieger relations.

\item Similarly, for (C4') we have
\begin{alignat*}{2}
m_2(u\otimes s^{-1}v)
&= \Psi(\Phi(u)\cup \Phi(s^{-1}v))
&{} ={}& -\sum_{\alpha\in Q_1} s^{-1}\alpha^*(u \cup\Phi(s^{-1}v)) (s\overline \alpha)\\
&=  \sum_{\alpha\in Q_1}s^{-1} \alpha^*u\alpha v
 \ &={}&\  u \cup' (s^{-1}v),
\end{alignat*}
where the third equality follows from Remark~\ref{remphi1alpha}.
\end{enumerate}
This shows that $m_2$ coincides with $-\cup'-$.

Now let us prove $m_k=0$ for $k>2$. Assume by way of contradiction that $m_k(a_1 \otimes  \cdots \otimes  a_k) \neq 0$ for some $a_1, \dotsc, a_k \in\widehat{C}^*(L, L)$.  By  \eqref{secondrecursive}, we have
\begin{align}
\label{align-mk}
m_k (a_1\otimes \cdots\otimes a_k)= (-1)^{\frac{(k-1)(k-2)}{2}}\Psi(\Phi_{k-1}(a_1\otimes\cdots\otimes a_{k-1})\cup \Phi(a_k)),
\end{align}
It follows from Lemma~\ref{lem:phicondition}  that $\Phi_{k-1}(a_1\otimes\cdots\otimes a_{k-1})\in \overline{C}_E^{*, 1}(L, L)$.  Since $\Psi|_{\overline{C}_E^{*, \geq 2}(L, L)}=0$, we infer that   $\Phi(a_k)$ must be in $\overline{C}_E^{*, 0}(L, L) = \bigoplus_{i\in Q_0} e_iLe_i$. Thus, by (\ref{align-mk}) again we get
\begin{align*}
& m_k(a_1 \otimes \dotsb \otimes a_k)\\
={} & -(-1)^{\frac{(k-1)(k-2)}{2}}\sum_{\alpha \in Q_1} \alpha^* \Phi_{k-1}(a_1\otimes  \cdots\otimes  a_{k-1})(s\overline \alpha)\cdot \Phi(a_k) \\
={} &-(-1)^{\frac{(k+1)(k-2)}{2}}\sum_{\alpha \in Q_1} \alpha^* H\Big(\Phi_{k-2}(a_1\otimes  \cdots\otimes  a_{k-2})\cup \Phi(a_{k-1})\Big)(s\overline \alpha) \cdot \Phi(a_k)\\
={} & 0
\end{align*}
where the first equality uses (\ref{equ:Psi}),  the second one uses (\ref{recursivephimk}), and the third one follows from Lemma~\ref{facts}. A contradiction! This shows that $m_k(a_1\otimes  \cdots \otimes  a_k)=0$ for $k>2$.
\end{proof}

\subsection{The $A_\infty$-quasi-isomorphism via the brace operation}
\label{subsection:A-quasi-brace}
It follows from Proposition~\ref{proposition-A-infinity} that we have an $A_{\infty}$-quasi-isomorphism
$$(\Phi_1=\Phi, \Phi_2, \cdots)\colon (\widehat{C}^*(L, L), \widehat{\delta}, -\cup'-)\longrightarrow (\overline{C}_E^*(L, L), \delta, -\cup-)$$
 between the two dg algebras. In this subsection, we will give an explicit formula for $\Phi_k$.

\begin{prop}\label{proposition-Phi}
Let $k\geq 1$.
For  any $s^{-1}u_1, \dotsc, s^{-1}u_k\in \bigoplus_{i\in Q_0}s^{-1}e_iLe_i\subset \widehat{C}^*(L, L)$, we have
$$\Phi_k(s^{-1}u_1\otimes \cdots\otimes s^{-1}u_k)(s\overline v)= (-1)^{(|v|-1)\epsilon_k+\sum\limits_{i=1}^{k-1} (|u_i|-1)(k-i)}v\{s^{-1}u_1, \dotsc, s^{-1}u_k\}', $$
where  $v\in L$ and $v\{s^{-1}u_1, \dotsc, s^{-1}u_k\}'$ is given by (\ref{equation-brace-formula2}). Here,  we denote $\epsilon_k=\sum_{i=1}^k|u_i|$.
\end{prop}
\begin{proof}
We prove this identity  by induction on $k$.  By Lemma \ref{lem-Phi} this holds for $k =1$.

For $k>1$ and $v=\beta_1^*\beta_2^*\cdots \beta_p^*\alpha_m\alpha_{m-1}\cdots \alpha_1\in  L$,  we have
\begin{align}\label{equation-prop612}
&\Phi_k(s^{-1}u_1\otimes \cdots\otimes s^{-1}u_k)(s\overline v)\\
={}&(-1)^{k-1}H\big(\Phi_{k-1}(su_{1, k-1})\cup \Phi(s^{-1}u_k)\big)(s\overline v)\nonumber \\
={}&(-1)^{1+\epsilon_k+(k-1)}\big(\Phi_{k-1}(su_{1, k-1})\cup \Phi(s^{-1}u_k)\big)(\overline{\iota\pi}(1\otimes s\overline v\otimes 1)] \nonumber \\
={}&- \sum_{\alpha\in Q_1}\sum_{j=0}^{p-1} (-1)^{\epsilon_k+|u_k|j+(k-1)}\big(\Phi_{k-1}(su_{1, k-1})(s\overline{\beta_{1, j}^*\alpha^*})\big) \cdot \big(\Phi(s^{-1}u_k)(s\overline \alpha)\big) \cdot \big(\beta^*_{j+1, p}\alpha_{m,1}\big) \nonumber\\
&\!\! + \sum_{\alpha\in Q_1}\sum_{j=0}^{m-1} (-1)^{\epsilon_k+|u_k|(m+p-j)+(k-1)}\big(\Phi_{k-1}(su_{1, k-1})(s\overline{\beta_{1, p}^*\alpha_{m, j+1}\alpha^*}\big)\cdot \big(\Phi(s^{-1}u_k)(s\overline \alpha)\big) \cdot \big(\alpha_{j, 1}\big),  \nonumber
\end{align}
where the first equality follows from \eqref{recursivephimk}, the second one from \eqref{H-vanish}, and  the third one from  Remark~\ref{rem:derivation}.  Here, we simply write $\Phi_{k-1}(su_{1, k-1})= \Phi_{k-1}(s^{-1}u_1 \otimes \cdots\otimes s^{-1}u_{k-1})$,  and write  $\alpha_{j, i}=\alpha_{j}\alpha_{j-1}\cdots \alpha_i$, $\beta_{i, j}^*=\beta_i^*\beta_{i+1}^*\cdots\beta_j^*$ for any $i<j$.

 Write $u_k=\gamma_0^*\widetilde {u_k}$ with $\gamma_0\in Q_1$ and $\widetilde{u_k}\in e_{t(\gamma_0)} Le_{s(\gamma_0)}$. Then by  (\ref{equation-brace-formula2}) and  the case  where $k=1$, we have
\begin{align*}
\Phi(s^{-1}u_k)(s\overline \alpha)=\alpha\{u_k\}'=-\alpha \gamma_0^*\widetilde{u_k}=-\delta_{\alpha,\gamma_0} \widetilde{u_k}, \quad \text{for $\alpha \in Q_1$}.
\end{align*}
Substituting the above identity into (\ref{equation-prop612}), we  get
\begin{equation*}
\begin{split}
&\Phi_k(s^{-1}u_1\otimes \cdots\otimes s^{-1}u_k)(s\overline v)\\
={} &\sum_{j=0}^{p-1} (-1)^{\sum\limits_{i=1}^{k-1} (|u_i|-1)(k-i)+ j \epsilon_k+|u_k|} \big((\beta_{1, j}^* \gamma_0^*)\{su_{1, k-1}\}'\big)\cdot (\widetilde{u_k}  \beta_{j+1, p}^* \alpha_{m,1})  \nonumber\\
&+ \sum_{j=0}^{m-1} (-1)^{\sum\limits_{i=1}^{k-1} (|u_i|-1)(k-i)+(p+m-j) \epsilon_k+|u_k|+1} \delta_{\alpha_{j+1}, \gamma_0}\big((\beta_{1,p}^*\alpha_{m, j+2}) \{su_{1, k-1}\}'
\big) \cdot \big( \widetilde{u_k} \alpha_{j, 1}\big)\\
={} &(-1)^{\sum\limits_{i=1}^{k-1} (|u_i|-1)(k-i)+(|v|-1)\epsilon_k}v\{s^{-1}u_1, \dotsc, s^{-1}u_k\}.\end{split}
\end{equation*}
Here,  to save the space, we simply write $\{s^{-1}u_1, s^{-1}u_2, \dotsc, s^{-1}u_{k-1}\}'$ as $\{su_{1, k-1}\}'$. The first equality follows from the induction hypothesis,  and the second one is exactly due to the identity in Proposition~\ref{prop:recursive}.
\end{proof}

\section{Verifying the $B_\infty$-morphism}\label{section:11}

The final goal is to prove that  the $A_\infty$-quasi-isomorphism obtained in the previous section is indeed a $B_\infty$-morphism. The proof relies on the higher pre-Jacobi identity of the Leavitt $B_\infty$-algebra $\widehat{C}^*(L, L)$; see Remark~\ref{rem-specialB}. For the opposite $B_{\infty}$-algebra $A^{\rm opp}$ of a $B_{\infty}$-algebra $A$, we refer to Definition~\ref{defnopposite}.

 \begin{thm}\label{prop-B4}
The $A_{\infty}$-morphism $(\Phi_1, \Phi_2, \cdots)$ is  a $B_{\infty}$-quasi-isomorphism from the $B_{\infty}$-algebra $\widehat{C}^*(L, L)$ to the opposite $B_{\infty}$-algebra $\overline{C}_E^*(L, L)^{\rm opp}$.
\end{thm}

\begin{proof}
By Lemma \ref{lemma-infinity-morphism} it suffices to verify the identity (\ref{equation-infinity-morphism}). That is,
for any
$x=u_1\otimes u_2 \otimes \cdots \otimes u_p\in \widehat{C}^*(L, L)^{\otimes p}$ and $y =v_1\otimes v_2 \otimes \cdots \otimes v_q\in \widehat{C}^*(L, L)^{\otimes q},$
we need to verify
\begin{align}\label{equation-infinity-morphismnew}
&\sum_{r\geq 1} \sum_{i_1+\dotsb+ i_r=p} (-1)^{\epsilon}\; \widetilde{\Phi}_q(sv_{1, q}) \{\widetilde{\Phi}_{i_1}(su_{1, i_1}), \widetilde{\Phi}_{i_2}(su_{i_1+1, i_1+i_2}), \dotsb, \widetilde{\Phi}_{i_r}(su_{i_1+\dotsb+i_{r-1}+1, p})\}\nonumber\\
={}&\sum (-1)^{\eta}\;
\widetilde{\Phi}_t(sv_{1, j_1}\smallotimes s(u_1\{v_{j_1+1, j_1+l_1}\}')\smallotimes  sv_{j_1+l_1+1, j_2}\smallotimes s(u_2\{v_{j_2+1, j_2+l_2}\}') \smallotimes v_{j_2+l_2+1} \smallotimes  \nonumber\\
& \quad\quad \dotsb\smallotimes sv_{j_p} \smallotimes s(u_{p}\{v_{j_p+1, j_p+l_p}\}')\smallotimes sv_{j_p+l_p+1, q}),
\end{align}
where  the sum on the right hand side is over all nonnegative integers  $(j_1, \dotsc, j_{p}; l_1, \dotsc, l_p)$ such that  $$0 \leq j_1 \leq j_1+l_1 \leq j_2 \leq j_2+l_2\leq \dotsb \leq j_p \leq j_p+l_p \leq q,$$ and $t=p+q-l_1-\cdots-l_p$. Here, $\widetilde \Phi_k$ is defined by \eqref{formula-sec} and   the signs are given by
\begin{align*}
\epsilon &= (|u_1|+\cdots+|u_p|-p)(|v_1|+\cdots+|v_q|-q),  \\
\eta& =\sum_{i=1}^p (|u_i|-1)((|v_1|-1)+(|v_2|-1) + \dotsb +(|v_{j_i}|-1)).
\end{align*}

To verify (\ref{equation-infinity-morphismnew}), we observe that if there exists  $1\leq j\leq p$ (or $1\leq l\leq q$) such that  $u_j$ (or $v_l$)  lies in   $\bigoplus_{i\in Q_0} e_iLe_i\subset \widehat{C}^*(L, L)$, then by \eqref{equation:phi} and \eqref{B1'}
both the left and  right hand sides  of (\ref{equation-infinity-morphismnew}) vanish. So we may and will assume that all $u_j$'s and $v_l$'s are in  $\bigoplus_{i\in Q_0} s^{-1}e_iLe_i\subset \widehat{C}^*(L, L)$.

It follows from \eqref{formula-sec} and Proposition \ref{proposition-Phi} that for any $v_1,\dotsc, v_q\in \bigoplus_{i\in Q_0}s^{-1}e_iLe_i$,
\begin{align}
\label{formula}
\widetilde{\Phi}_q(sv_{1, q})(s\overline a):={}&(-1)^{|v_1|(q-1)+|v_2|(q-2)+\dotsb + |v_{q-1}|}\;  \Phi_q(v_1\otimes \cdots\otimes v_q)(s\overline a)\nonumber\\
={}&(-1)^{(|a|-1)(|v_1|+\dotsb + |v_q|-q)} a\{v_1, v_2, \dotsc, v_q\}'.
\end{align}
Here, we stress that the elements $sv_1, \dotsc, sv_q$ in $\widetilde{\Phi}_q(sv_{1, q})$ are viewed  in the component  $s(\oplus_{i\in Q_0}s^{-1}e_iLe_i)$  of $ s\widehat C^*(L, L)$,  rather than in $\oplus_{i\in Q_0}e_iLe_i \subset \widehat C^*(L, L)$.

It follows from \eqref{equation:phi1} that  $\widetilde{\Phi}_{q}(sv_{1, q}) \in   \overline{C}_E^{*,  1}(L, L)= \Hom_{\text{$E$-$E$}}(s\overline L, L).$ Thus, by \eqref{equation:brace} we note that
$$\widetilde{\Phi}_q(sv_{1, q})\{\widetilde{\Phi}_{i_1}(su_{1, i_1}), \widetilde{\Phi}_{i_2}(su_{i_1+1, i_1+i_2}), \dotsc, \widetilde{\Phi}_{i_r}(su_{i_1+\dotsb+i_{r-1}+1, p})\}=0$$
if $r \neq 1$. Therefore, the left hand side of (\ref{equation-infinity-morphismnew}), denoted by $\mathrm{LHS}$,  equals
\begin{equation*}
\begin{split}
&\mathrm{LHS}=(-1)^{\epsilon}\; \widetilde{\Phi}_q(sv_{1, q})\{\widetilde{\Phi}_p(su_{1, p})\}.
\end{split}
  \end{equation*}
   Applying the above to an arbitrary  element $s\overline a \in s\overline L$, we have
    \begin{equation}\label{equ-important-0}
\begin{split}
\mathrm{LHS}(s\overline a)={}&(-1)^{\epsilon} \;  \widetilde{\Phi}_q(sv_{1, q})( s\widetilde{\Phi}_p(su_{1, p})(s\overline a)) \\
={}&(-1)^{\epsilon+(|a|-1)(|u_1|+\dotsb + |u_p|-p)}\; \widetilde{\Phi}_q(sv_{1, q})\big( s (a\{u_1,\dotsc, u_p\}')\big)\\
={}&(-1)^{\epsilon_1}(a\{u_1, \dotsc, u_p\}')\{v_1, \dotsc, v_q\}',  \end{split}
   \end{equation}
   where $
   \epsilon_1 = (|a|-1)(|u_1|+\cdots+|u_p|-p+|v_1|+\cdots+|v_q|-q)$,  and the second and third equalities follow from \eqref{formula}.

For the right hand side of (\ref{equation-infinity-morphismnew}), denoted by $\mathrm{RHS}$,  we   use  \eqref{formula} again and  have
 \begin{align}\label{equ-important-1}
   \mathrm{RHS}(s\overline a)
 = \sum &(-1)^{\eta+\eta_1} a \{v_{1, j_1},   u_1\{v_{j_1+1,j_1+l_1}\}',  v_{j_1+l_1+1, j_2}, u_2\{v_{j_2+1, j_2+l_2}\}', v_{j_2+l_2+1},  \nonumber \\
&\dotsc, v_{j_p},  u_{p}\{v_{j_p+1, j_p+l_p}\}',  v_{j_p+l_p+1, q}\}',\end{align}
where $\eta_1 =(|a|-1)(|u_1|+\dotsb + |u_p|-p+|v_1|+\cdots+|v_q|-q).$

Comparing  (\ref{equ-important-0}) and (\ref{equ-important-1}) with  the higher pre-Jacobi identity in Remark~\ref{rem-specialB} for the Leavitt $B_\infty$-algebra $\widehat{C}^*(L, L)$,
we obtain $$\mathrm{LHS}(s\overline a)=\mathrm{RHS}(s\overline a).$$
This verifies the identity (\ref{equation-infinity-morphismnew}), completing the proof.
\end{proof}

\vskip 10pt
\noindent {\bf Acknowledgements.}\quad We are very grateful to Bernhard Keller for many inspiring discussions. Chen thanks Lleonard Rubio y Degrassi for the reference \cite{BRyD} and  Guisong Zhou for the reference \cite{Masu}. Li thanks  Pere Ara, Jie Du, Roozbeh Hazrat and Steffen Koenig for their support and help. Wang thanks Henning Krause for explaining the details in \cite{Kra} and Alexander Voronov for many useful discussions on $B_{\infty}$-algebras.  He also thanks Guodong Zhou for his constant support and encouragement.

A large part of this work was carried out while Wang was supported by the Max-Planck Institute for Mathematics in Bonn.  He thanks the institute for its hospitality and financial support. This work is also supported by the National Natural Science Foundation of China (Nos.  11871071 and 11971449) and the Australian Research Council grant DP160101481.

 {\footnotesize \noindent Xiao-Wu Chen\\
 CAS Wu Wen-Tsun Key Laboratory of Mathematics,\\
  University of Science and Technology of China, Hefei, Anhui, 230026, PR China\\
  xwchen$\symbol{64}$mail.ustc.edu.cn}

\vskip 5pt

{\footnotesize \noindent Huanhuan Li\\
School of Mathematical Sciences, Anhui University\\
Hefei 230601, Anhui, PR China\\
lihuanhuan2005$\symbol{64}$163.com}

\vskip 5pt

{\footnotesize \noindent Zhengfang Wang\\
Institute of Algebra and Number Theory, University of Stuttgart\\
Pfaffenwaldring 57, 70569 Stuttgart, Germany \\
zhengfangw$\symbol{64}$gmail.com}

\end{document}